\documentclass[11pt]{amsart}
\usepackage[utf8]{inputenc}
\usepackage[margin=1in]{geometry}
\usepackage{amsfonts,amssymb,amsmath,amsthm,tikz,comment,mathtools,setspace,float,stmaryrd,datetime,bbm,adjustbox}
\usepackage[toc]{appendix}
\usepackage{enumitem}
\usepackage[hidelinks]{hyperref}
\usepackage{scalerel} 
\numberwithin{equation}{section}
\usepackage{chngcntr}
\counterwithin{figure}{section}

\allowdisplaybreaks

\newcommand{\N}{\mathbb{N}}
\newcommand{\R}{\mathbb{R}}
\newcommand{\Q}{\mathbb{Q}}
\newcommand{\Z}{\mathbb{Z}}
\newcommand{\F}{\mathcal{F}}

\newcommand{\X}{\mathcal{X}}
\newcommand{\Y}{\mathcal{Y}}
\newcommand{\B}{\mathcal{B}}

\newcommand{\D}{\mathcal{D}}
\newcommand{\E}{\mathcal{E}}

\newcommand{\Nor}{\mathcal N}
\newcommand{\CRpin}{C_{\text{pin}}(\R)}
\newcommand{\wt}{\widetilde}
\newcommand{\wh}{\widehat}
\newcommand{\Ss}{\mathcal{S}}
\newcommand{\T}{\mathcal{T}}
\newcommand{\deq}{\overset{d}{=}}

\newcommand{\ve}{\varepsilon}
\newcommand{\sqtwopi}{\sqrt{2\pi}}
\newcommand{\fraclambda}{\frac{\lambda}{\sqrt 2}}
\newcommand{\zfraclambda}{\frac{z -\lambda}{\sqrt 2}}
\newcommand{\negzfraclambda}{-\frac{z + \lambda}{\sqrt 2}}
\newcommand{\f}{\frac}
\newcommand{\tf}{\tfrac}
\newcommand{\mbf}{\mathbf}

\newcommand{\LM}{\operatorname{LeftMax}}

\newcommand{\Pp}{\mathbb P}
\newcommand{\Ee}{\mathbb E}
\newcommand{\NU}{\operatorname{NU}}
\newcommand{\h}{h}
\newcommand{\vv}{v}

\newcommand{\li}{\;{\le}_{\rm inc}\;}
\newcommand{\gi}{\;{\ge}_{\rm inc}\;}
\newcommand{\Xcomp}{\wh \X}
\newcommand{\Ycomp}{\wh \Y}
\newcommand{\Var}{\mathbb V ar}
\newcommand{\1}{\mathbbm 1}

\newcommand{\Se}{\succ}
\newcommand{\SeS}{\succcurlyeq}
\newcommand{\Ne}{\prec}
\newcommand{\NeN}{\preccurlyeq}
\newcommand{\supp}{\operatorname{supp}}

\newcommand{\Compset}{S}
\newcommand{\Exp}{\operatorname{Exp}}
\newcommand{\CI}{\operatorname{CI}}
\newcommand{\Busedc}{\Theta}

\newcommand{\Dq}{D_d}
\newcommand{\DqD}{\D_d}

\newcommand{\Yd}{\Y^d}
\newcommand{\Xd}{\X^d}

\newcommand{\nud}{\nu_d}
\newcommand{\mud}{\mu_d}

\newcommand{\be}{\begin{equation}}
\newcommand{\ee}{\end{equation}}

\def\tspa{\hspace{0.7pt}}
\def\tspb{\hspace{0.9pt}}

\newcommand{\sig}{{\scaleobj{0.8}{\boxempty}}} 
\newcommand{\sigg}{{\scaleobj{0.9}{\boxempty}}} 

\newcommand\abullet{{\raisebox{2pt}{\scaleobj{0.5}{\bullet}}}}  

\newtheorem{theorem}{Theorem}[section]

\newtheorem{corollary}[theorem]{Corollary}
\newtheorem{lemma}[theorem]{Lemma}

\theoremstyle{definition}
\newtheorem{definition}[theorem]{Definition}
\newtheorem{example}[theorem]{Example}

\theoremstyle{remark}
\newtheorem{remark}[theorem]{Remark}

\title[Geodesics and competition interfaces in BLPP]{Global structure of semi-infinite geodesics and competition interfaces in Brownian last-passage percolation}

\author{Timo Sepp{\"a}l{\"a}inen}
\address{Timo Sepp{\"a}l{\"a}inen, University of Wisconsin-Madison, Mathematics Department, Van Vleck Hall, 480
Lincoln Dr., Madison WI 53706-1388, USA.}
\email{seppalai@math.wisc.edu}
\author{Evan Sorensen}
\address{Evan Sorensen, University of Wisconsin-Madison, Mathematics Department, Van Vleck Hall, 480
Lincoln Dr., Madison WI 53706-1388, USA.}
\email{elsorensen@wisc.edu}

\date{\today}

\begin{document}

\begin{abstract}
    In Brownian last-passage percolation (BLPP), the Busemann functions $\B^{\theta}(\mbf x,\mbf y)$ are indexed by two points $\mbf x,\mbf y \in \Z \times \R$, and a direction parameter $\theta > 0$. We derive the joint distribution of Busemann functions across all directions. The set of directions where the Busemann process is discontinuous, denoted by $\Busedc$, provides detailed information about the uniqueness and coalescence of semi-infinite geodesics.  The uncountable set of initial points in BLPP gives rise to new phenomena not seen in discrete models. For example, in every direction $\theta > 0$, there exists a countably infinite set of initial points $\mbf x$ such that there exist two $\theta$-directed geodesics that split but eventually coalesce. Further, we define the competition interface in BLPP and show that the set of initial points whose competition interface is nontrivial has Hausdorff dimension $\f{1}{2}$. From each of these exceptional points, there exists a random direction $\theta \in \Busedc$ for which there exists two $\theta$-directed semi-infinite geodesics that split immediately and never meet again. Conversely, when $\theta \in \Busedc$,  from every initial point $\mbf x \in \Z \times \R$, there exists two $\theta$-directed semi-infinite geodesics that eventually separate.  Whenever $\theta \notin \Busedc$, all $\theta$-directed semi-infinite geodesics coalesce.
\end{abstract}
\maketitle

\tableofcontents
\section{Introduction}

\subsection{Broad goals of the project}

This work is part of an effort to understand global geometric properties of random growth of the first- and last-passage type. In these stochastic models, growth progresses in space along paths called {\it geodesics} that optimize an energy functional. Of particular interest are the semi-infinite geodesics,  their existence,  uniqueness, multiplicity and coalescence, and the competition interfaces that separate non-unique geodesics in a given direction.  Semi-infinite geodesics are hard to study because they look at the environment all the way to infinity. 

The novelty of the present paper lies in its semi-discrete, or partial continuum, setting.   In contrast with lattice models,   new features arise and new  methods are needed.  The main tool for accessing these geometric properties is the {\it Busemann process}. We establish analytic and probabilistic properties of the Busemann process and then use those to derive properties of the geodesics and the competition interfaces.    

The specific model we work with is the {\it Brownian last-passage percolation model} (BLPP) that lives in the space $\Z\times\R$.  BLPP arose in queuing theory in the 1980s and 1990s, in the work of Harrison and Williams, and Glynn and Whitt \cite{glynn1991,Harrison1985,harrison1990,harrison1992}.  In the 2000s BLPP and its positive temperature counterpart, the semi-discrete Brownian polymer or O'Connell-Yor polymer  \cite{brownian_queues}, have occupied a place among the exactly solvable models in which properties of the Kardar-Parisi-Zhang (KPZ) class can be fruitfully studied. We refer the reader to the introduction of our previous paper \cite{Seppalainen-Sorensen-21a} for more on the history and context, and concentrate here on the new features and connections.  

Beyond the present work, the next natural stages of this project involve  studying geodesics (i) in the semi-discrete setting of the positive temperature Brownian polymer and (ii)  in the full continuum settings  of the stochastic heat equation and the directed landscape. The novel methods developed in this paper are applied to the directed landscape in~\cite{Busani-Seppalainen-Sorensen-2022}, which appeared after the first version of the present paper.

\subsection{The third work of a series}
Our paper is the third in a series on the Busemann functions and semi-infinite geodesics of BLPP. While we rest on the foundation provided by the two earlier works, our introduction and main results  are presented in a  self-contained manner. 

 In the first stage Alberts, Rassoul-Agha, and Simper~\cite{blpp_utah} proved the almost sure existence of a Busemann function  in BLPP from a fixed pair of initial points into a fixed direction. This limit appears in equation \eqref{B120} below.  In~\cite{Seppalainen-Sorensen-21a}, the present  authors extended the individual Busemann functions to a full Busemann process $\B^{\theta\sig}(\mbf x, \mbf y)$ in BLPP,  indexed by all initial points $\mbf x, \mbf y\in\Z\times\R$, directions represented by positive reals $\theta$, and signs $\sig\in\{-,+\}$ that keep track of discontinuities. From this construction, \cite{Seppalainen-Sorensen-21a}  derived the following results on semi-infinite geodesics: 
\begin{enumerate}
    \item On a single event of probability one, every semi-infinite geodesic has an asymptotic direction, and from every initial space-time point $\mbf x$ and in every direction $\theta$, there exists a semi-infinite geodesic.
    \item Given a direction $\theta$,  all $\theta$-directed semi-infinite geodesics coalesce on a $\theta$-dependent full-probability event. 
    \item Similarly, given a  northeast and a  southwest direction, there are almost surely no bi-infinite geodesics in those directions.
\end{enumerate}

The present paper takes 
 \cite{Seppalainen-Sorensen-21a} as a starting point to go deeper into the Busemann process and the  semi-infinite geodesics in BLPP. Next we go over some highlights and relate them to past literature.  The organization of the paper is explained in Section \ref{sec:org}.  
 
 \subsection{A jump process of coupled Brownian motions with drift}   Among our main results and also our main tool for studying geodesics is the joint distribution of the  Busemann process across space and asymptotic directions.
 On a fixed level $\{m\}\times\R$ of the space $\Z\times\R$, the Busemann process $\{ \B^{\theta \sig}((m,0),(m,t)): t\in\R, \,\sigg\in\{-,+\}\}$ appears as a coupled family of two-sided Brownian motions with drift. 
 Here, the real coordinate on the $\R$ component of $\Z \times \R$  plays the role of the time variable $t$ of the Brownian motions.  Figure \ref{fig:Buse_process_all} depicts a simulation restricted to the right half-line $\R_{\ge0}$.  The drift is determined by the   direction parameter $\theta$.    Any two trajectories coincide in a neighborhood of the origin and separate at some point.  As we move away from the origin, the trajectories move  further away from each other.  The separation time  is not memoryless, and hence the coupled processes are not jointly Markovian. 
 
 When the spatial location (time variable of the Brownian motions)  is fixed, marginally 
 in the direction parameter $\theta$,  we see a monotone jump process with stationary but dependent increments.  This corresponds to jumping vertically from trajectory to trajectory in Figure \ref{fig:Buse_process_all}.    Explicit distributions of the increments and expected numbers of jumps are given in Section \ref{section:Busepp}. 
 
 Busani~\cite{Busani-2021} recently constructed what is termed the \textit{stationary horizon}, as the scaling limit of the Busemann process along a horizontal line in the exponential lattice corner growth model (CGM). This object is a random collection of continuous functions $\{G_\alpha\}_{\alpha \in \R}$, where $G_\alpha$ is a two-sided variance $4$ Brownian motion with drift $\alpha$. A precise description is given in Definition~\ref{def:SH}. It is expected that the stationary horizon is a universal object in the KPZ class. Our work in Section~\ref{sec:stat_horiz} gives additional evidence for this claim. In Theorem~\ref{thm:dist_of_SH}, we show that, after a simple reflection, the horizontal Busemann process for BLPP is equal in distribution to the stationary horizon, restricted to nonnegative drifts. Therefore, the distributional calculations of Section~\ref{section:Busepp} give additional quantitative information about the stationary horizon, beyond what is given in~\cite{Busani-2021}. Furthermore, we show in Theorem~\ref{thm:conv_to_SH} that under KPZ scaling, the BLPP Busemann process converges to the stationary horizon, in the sense of finite-dimensional distributions.

 \subsection{Non-uniqueness of semi-infinite geodesics} 

The Busemann process can be used to define a family of semi-infinite geodesics that we then call {\it Busemann geodesics}. This construction from \cite{Seppalainen-Sorensen-21a} is repeated in Section \ref{section:SIG}. Due to planar monotonicity, Busemann geodesics bound arbitrary semi-infinite geodesics.  Hence, uniqueness of Busemann geodesics in a given direction translates into uniqueness of all semi-infinite geodesics in that direction.

BLPP has {\it two} sources of non-uniqueness of Busemann geodesics. A given Busemann function $\B^{\theta\sig}$ produces distinct leftmost and rightmost geodesics from a random countable set of initial points, denoted by $\NU_0^{\theta \sig}$. The leftmost and rightmost geodesics are labeled by $L$ and $R$. Additionally, if a direction $\theta$ is a jump point of the Busemann process, there is non-uniqueness represented by a $\theta\pm$  distinction.  The random  set $\Busedc$ of Busemann function discontinuities is countably infinite and dense in the set of directions.

When $\theta \notin \Busedc$, all $\theta$-directed semi-infinite geodesics are Busemann geodesics, including the cases of $L/R$ non-uniqueness. It is an open question whether directions $\theta \in \Busedc$ have semi-infinite geodesics that are not Busemann geodesics. Only in   the exponential lattice CGM is it presently known that the Busemann geodesics account for all semi-infinite geodesics
\cite[Section 3.2]{Janjigian-Rassoul-Seppalainen-19}.

The $L/R$ distinction is a continuum feature that is not present in the exponential lattice CGM, while the $\theta\pm$  distinction is a similar phenomenon as on the lattice. The $L/R$ distinction  occurs only at a random countable set of initial points $\mbf x$.  It turns out that the $L/R$ distinction is {\it local} in the sense that it disappears after a while.  This is illustrated in Figure \ref{fig:multiple_sig} where multiple geodesics emanate from $\mbf x$, but by location $\mbf z$ they have rejoined into a single path. {\it The nontrivial fact is that after $\mbf z$, they remain together forever.} This follows from the fact that after meeting at $\mbf z$ these geodesics become portions of the unique geodesic started from a point without the $L/R$ distinction. See Theorem~\ref{thm:general_coalescence}\ref{itm:return_point} and the proof in Section~\ref{sec:pf_coal}.

\subsection{Coalescence of geodesics}

For each direction $\theta$ in the discontinuity set $\Busedc$, the $\theta+$ geodesics from all (uncountably many) initial points coalesce, and the same is true for $\theta-$.  If $\theta\notin\Busedc$, there is no $\pm$ distinction and again coalescence holds. 
We present a {\it new coalescence proof} that utilizes the regularity   of the Busemann process.  This argument is in Section \ref{sec:pf_coal} where it culminates in the proof of Theorem \ref{thm:general_coalescence}. 

Previously, two approaches to coalescence of planar geodesics  were available. (i)  A proof given by Licea and  Newman \cite{licea1996} used a modification argument followed by a Burton-Keane type lack-of-space argument. (ii) In \cite{Timo_Coalescence} the first author developed a softer proof that utilized the tree of dual geodesics and relied on properties of the stationary version of the growth process. This latter proof we applied to BLPP  in \cite{Seppalainen-Sorensen-21a}.

 \subsection{Fractal sets} Finite geodesics from a given initial point $\mbf x$ to all the points to its north and east begin by either a horizontal or a vertical step.  These two collections of finite geodesics are separated by an infinite path   that emanates from $\mbf x$, called the {\it competition interface}. See Figure \ref{fig:the competition interface}.  In the lattice CGM, the competition interface has a random direction into the open quadrant from each lattice point.  By contrast, in BLPP, the typical competition interface is {\it trivial} in the sense that it is an infinite vertical line. Geometrically, this means that {\it all} geodesics emanating from $\mbf x$ start with a horizontal step.  
 
 However, there is a Hausdorff dimension $\tfrac12$ set of exceptional initial points, called  $\CI$,  from which the competition interface has a nontrivial limiting slope. Even though the set $\CI$ is uncountable, the set of possible limiting slopes is countable.  These limiting slopes are characterized by the Busemann process (Theorem \ref{thm:ci_and_Buse_directions}). 
 
 Random fractals related to geodesics  appear in the KPZ fixed point and the directed landscape. Article \cite{Basu-Ganguly-Hammond-21} studies the {\it Airy difference profile},  the scaling limit of the process $z \mapsto L_{(0,-n^{2/3}),(z,n)}$  $- L_{(0,n^{2/3}),(z,n)}$, where $L$ denotes BLPP time (see Section~\ref{section:def of BLPP}). The  limiting object is a continuous nondecreasing process that is locally constant, except on a set of Hausdorff dimension $\f{1}{2}$. The result of~\cite{Basu-Ganguly-Hammond-21} is applied to the directed landscape in  \cite{Bates-Ganguly-Hammond-22} and is used to study  the set of pairs $y$ such that there exist two geodesics between $(-1,0)$ and $(1,0)$ whose only common points are the endpoints. This set is exactly the set of local variation of the Airy difference profile and therefore has Hausdorff dimension $\f{1}{2}$. See also~\cite{KPZ_violate_Johansson,Ganguly-Hegde-2021} for further study of random fractal sets that appear in the continuum models of the KPZ class. 
 
 \subsection{Queues} 
 Properties of the Brownian queue are central to our arguments. The spatial evolution of the Busemann process implies that it obeys transformations that arise in the queuing context. Our characterization of the distribution of the Busemann process  relies on a uniqueness theorem of Cator, L\'{o}pez and Pimentel~\cite{Cator-Lopez-Pimentel-2019} for the invariant distribution of a particular queuing transformation, stated in the present paper as Theorem~\ref{convergence Theorem in Cator 2019}.
 
  \subsection{Geodesics in the KPZ scaling}  Hammond made a detailed study of point-to-point geodesics in BLPP in the KPZ scaling regime, that is, geodesics between  points $\left(0,2n^{2/3} x\right)$ and $(n,2n^{2/3} y)$, where $x,y \in \R$ and $n$ is a large integer \cite{Hammond1,Hammond2,Hammond3,Hammond4}. This work has been valuable  for understanding the directed landscape. See, for example, \cite{Directed_Landscape,Basu-Ganguly-Hammond-21,Bates-Ganguly-Hammond-22,Ganguly-Hegde-2021}.  The setting of our work  is different, since we study BLPP 
 globally  instead of through a thin  $n\times n^{2/3}$ scaling window. However,  related themes arise. Theorem 1.1 in~\cite{Hammond2} gives explicit asymptotic bounds for the probability that there are $k$ disjoint BLPP geodesics between two intervals of size $n^{2/3}$.    Proposition 6.1 of~\cite{Hammond4}  establishes that,  with high probability, two geodesics from two sufficiently close initial points (in scaled coordinates) to the same terminal point coalesce well before the endpoint.

 Very recently, Rahman and Vir\'ag~\cite{Rahman-Virag-21} proved the existence of semi-infinite geodesics and Busemann functions in the directed landscape, the continuum scaling limit of the KPZ universality class. They first prove the existence of semi-infinite geodesics, then use the geodesics to define Busemann functions. Conversely, in our work, we construct the semi-infinite geodesics from the Busemann functions. While the models are different, there are some analogous results that appear. For example, Theorem 5 of~\cite{Rahman-Virag-21} states that, for a fixed direction and a fixed horizontal line, with probability one, there exists a random, at most countable, set of points from which the geodesic in the fixed direction is not unique. However, all geodesics in that fixed direction coalesce, so the splitting geodesics eventually come back together. This is the same phenomenon we observe in BLPP, proved in~\cite{Seppalainen-Sorensen-21a}, Theorem 3.1(iii),(vii). The present work describes the geometry of all geodesics, simultaneously across all directions, whereas~\cite{Rahman-Virag-21} focuses on a single, fixed direction. For example, in the present paper, we show that this ``bubble" phenomenon occurs simultaneously in  every direction (see Theorems~\ref{thm:NU} and~\ref{thm:general_coalescence}).

 The authors, together with Ofer Busani, apply the techniques developed in this paper to derive corresponding results for the directed landscape (DL)  in~\cite{Busani-Seppalainen-Sorensen-2022}. In particular, the new proof technique for coalescence is crucial, and analogous results hold on  non-uniqueness and random fractal sets. 

 \subsection{Relation to the lattice corner growth model} 
 The results we prove, and our approach, are related to  the work on  the  lattice CGM in \cite{Fan-Seppalainen-20} and~\cite{Janjigian-Rassoul-Seppalainen-19}.  BLPP is technically more challenging than  the lattice situation, and the present work benefits greatly from the direction provided by the prior work on discrete models.  We discuss the relations between the two models  in Section \ref{sec:CGM}. Originally, Glynn and Whitt \cite{glynn1991} derived BLPP as a weak limit of the lattice CGM. We show that under this same scaling, for two fixed directions, the Busemann process of  the exponential  CGM converges weakly to its BLPP counterpart. However, we prove all our results for the Busemann process and semi-infinite geodesics directly from the BLPP model, without importing results from the discrete model and appealing to the limit.
 
 \subsection{Organization of the paper} \label{sec:org} 
 Section~\ref{section:Defs_results} provides definitions and terminology and then states the main theorems. Section~\ref{sec:Buse_jd_intro} gives a detailed description of the distribution of the Busemann process. Section~\ref{section:main_geometry} describes the global structure of the semi-infinite geodesics and competition interfaces. Section~\ref{sec:CGM} illustrates connections to the corner growth model and the stationary horizon. In Section~\ref{sec:op}, we state several open problems. The remainder of the paper is devoted to the proofs. The proofs of the results for the Busemann process, including Theorem~\ref{thm:Theta properties} are in Section~\ref{sec:Buse_proofs}. The proofs of the description of the geodesics and competition interfaces, along with the proofs of the Theorems~\ref{thm:Busedc_class} and~\ref{thm:Haus_comp_interface}, are contained in Section~\ref{sec:detailed_SIG_proofs}. Section~\ref{sec:CGM_proofs} proves the results of Section~\ref{sec:CGM}. The Appendices contain some technical results and inputs from the literature.
 
 \subsection{Acknowledgements}
 Although the techniques and proofs are our own, the results in Section~\ref{sec:stat_horiz} came after the first version of this paper after discussions with Ofer Busani. The authors thank also Tom Alberts, Erik Bates, Wai Tong (Louis)
Fan, Sean Groathouse, Chris Janjigian, Leandro Pimentel, and Firas Rassoul-Agha for helpful discussions. T. Sepp\"al\"ainen
was partially supported by National Science Foundation grants DMS-1602846 and DMS-1854619 and by
the Wisconsin Alumni Research Foundation. E. Sorensen was partially supported by T. Sepp\"al\"ainen via
National Science Foundation grants DMS-1602846 and DMS-1854619.

 \section{Definitions and main results}
  \label{section:Defs_results}
  \subsection{Notation} \label{section:notation}
The following notation and conventions are used throughout the paper.
\begin{enumerate} [label=\rm(\roman{*}), ref=\rm(\roman{*})]  \itemsep=3pt 
\item $\Z$, $\Q$ and $\R$ are restricted by subscripts, as in for example $\Z_{> 0}=\{1,2,3,\dotsc\}$.  

\item \label{SE notation} We use two orderings of space-time points. In the standard coordinatewise ordering,   $(m,s) \le (n,t)$ means that  $m \le n \in \Z$ and $s \le t \in \R$.  
In the  down-right, or southeast, ordering,   $(r,s) \Ne (m,t)$ means that  $m \le r \in \Z$ and  $s < t \in \R$, as in Figure~\ref{fig:initial_point_config}. The weak version $\mbf (r,s) \NeN \mbf (m,t)$ means that  $m \le r$ and  $s \le t$.  
\item  $X \sim \Nor(\mu,\sigma^2)$ indicates that the random variable $X$ has normal distribution with mean $\mu$ and variance $\sigma^2$. For $\alpha > 0$, $X \sim \operatorname{Exp}(\alpha)$ indicates that $X$ has exponential distribution with rate $\alpha$, or equivalently, mean $\alpha^{-1}$.
\item  Equality in distribution between random variables and  processes is denoted by $\deq$.  
\item A two-sided Brownian motion is a continuous random process $\{B(t): t \in \R\}$ such that $B(0) = 0$ almost surely and such that $\{B(t):t \ge 0\}$ and $\{B(-t):t \ge 0\}$ are two independent standard Brownian motions on $[0,\infty)$. For $c > 0$, we call $\{\sqrt c B(t): t\in \R\}$ a Brownian motion of variance $c$.  
\item For $\lambda \in \R$,   $\{Y(t): t \in \R\}$ is a two-sided Brownian motion with drift $\lambda$ if the process $\{Y(t) - \lambda t: t \in \R\}$ is a two-sided Brownian motion. 
\item The square $\sigg$ as a superscript represents a sign: $-$ or $+$. 
\item \label{inc notation}
Increments of  a function $f$ are denoted by  $f(s,t)=f(t) - f(s)$.
\item \label{ginc notation} Increment ordering of two functions $Z,\widetilde Z:\R \rightarrow \R$ is defined as follows:  $Z \li \widetilde Z$ if  $Z(s,t) \leq \widetilde Z(s,t)$ whenever $s < t$.
\item \label{pinned notation}
  The space of continuous functions ``pinned" at $0$ is denoted by 
\[\CRpin = \{f \in C(\R): f(0) = 0\}.\]
\item \label{increment stationarity}
A stochastic process $(X(t))_{t\in\R}$ indexed by the real line $\R$  is increment-stationary if, for each $s \in \R$, this process-level equality in distribution holds: 
\[
(X(0,t))_{t \in \R} \deq (X(s,t + s))_{t \in \R}.
\]
A vector-valued process  $(X^1,\ldots,X^n)$    is jointly increment-stationary if, for each $s \in \R$,
\[
(X^1(0,t),\ldots,X^n(0,t))_{t \in \R} \deq (X^1(s,t + s),\ldots,X^n(s,t + s))_{t \in \R}.
\]
\end{enumerate}

\subsection{Geodesics in Brownian last-passage percolation} \label{section:def of BLPP}
The Brownian last-passage process is defined as follows. On a probability space $(\Omega, \F,\Pp)$, let $\mathbf B = \{B_r\}_{r \in \Z}$ be a field of independent, two-sided Brownian motions. For $(m,s) \le (n,t)$, define the set 
\[
\Pi_{(m,s),(n,t)} := \{\mbf s_{m,n} = (s_{m - 1},s_m,\ldots,s_n) \in \R^{n - m + 2}: s = s_{m - 1} \le s_m \le \cdots \le s_n = t   \}.
\]
 Denote the energy of a sequence $\mbf s_{m,n} \in \Pi_{(m,s),(n,t)}$ by
\be\label{E10} 
\E(\mbf s_{m,n}) = \sum_{r = m}^n B_r(s_{r - 1},s_r).
\ee
Now, for $\mbf x = (m,s) \le (n,t) = \mbf y$, define the Brownian last-passage time as
\begin{equation} \label{BLPP formula}
L_{\mbf x,\mbf y}=L_{\mbf x,\mbf y}(\mbf B) = \sup\{\E(\mbf s_{m,n}): \mbf s_{m,n} \in \Pi_{\mbf x,\mbf y}\}.
\end{equation}

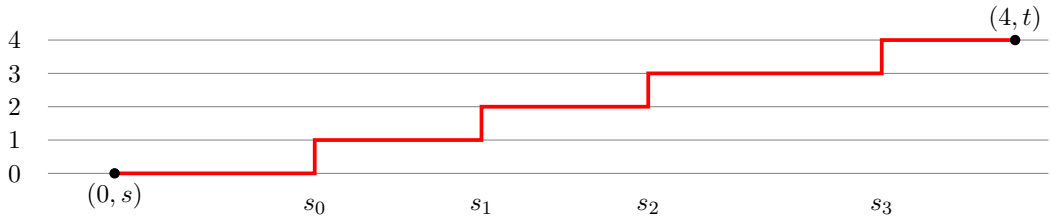
\begin{figure}[t]
\begin{adjustbox}{max totalsize={5.5in}{5in},center}
\begin{tikzpicture}
\draw[gray,thin] (0.5,0) -- (15.5,0);
\draw[gray,thin] (0.5,0.5) --(15.5,0.5);
\draw[gray, thin] (0.5,1)--(15.5,1);
\draw[gray,thin] (0.5,1.5)--(15.5,1.5);
\draw[gray,thin] (0.5,2)--(15.5,2);
\draw[red,ultra thick] (1.5,0)--(4.5,0)--(4.5,0.5)--(7,0.5)--(7,1)--(9.5,1)--(9.5,1.5)--(13,1.5)--(13,2)--(15,2);
\filldraw[black] (1.5,0) circle (2pt) node[anchor = north] {$(0,s)$};
\filldraw[black] (15,2) circle (2pt) node[anchor = south] {$(4,t)$};
\node at (4.5,-0.5) {$s_0$};
\node at (7,-0.5) {$s_1$};
\node at (9.5,-0.5) {$s_2$};
\node at (13,-0.5) {$s_3$};
\node at (0,0) {$0$};
\node at (0,0.5) {$1$};
\node at (0,1) {$2$};
\node at (0,1.5) {$3$};
\node at (0,2) {$4$};
\end{tikzpicture}
\end{adjustbox}
\caption{\small Example of a planar path from $(0,s)$ to $(4,t)$, represented by the sequence \\$(s=s_{-1}, s_0, s_1, s_2, s_3, s_4=t)\in\Pi_{(0,s),(4,t)}$.}
\label{fig:BLPP_geodesic}
\end{figure}

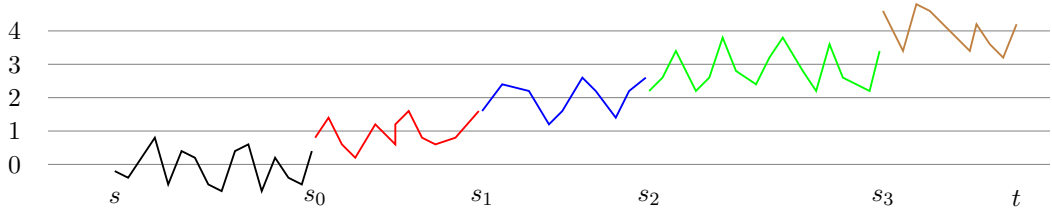
\begin{figure}[t]
\begin{adjustbox}{max totalsize={5.5in}{5in},center}
    \begin{tikzpicture}
\draw[gray,thin] (0.5,0) -- (15.5,0);
\draw[gray,thin] (0.5,0.5) --(15.5,0.5);
\draw[gray, thin] (0.5,1)--(15.5,1);
\draw[gray,thin] (0.5,1.5)--(15.5,1.5);
\draw[gray,thin] (0.5,2)--(15.5,2);
\draw[black,thick] plot coordinates {(1.5,-0.1)(1.7,-0.2)(1.9,0.1)(2.1,0.4)(2.3,-0.3)(2.5,0.2)(2.7,0.1)(2.9,-0.3)(3.1,-0.4)(3.3,0.2)(3.5,0.3)(3.7,-0.4)(3.9,0.1)(4.1,-0.2)(4.3,-0.3)(4.45,0.2)};
\draw[red,thick] plot coordinates {(4.5,0.4)(4.7,0.7)(4.9,0.3)(5.1,0.1)(5.4,0.6)(5.7,0.3)(5.7,0.6)(5.9,0.8)(6.1,0.4)(6.3,0.3)(6.6,0.4)(6.95,0.8)};
\draw[blue,thick] plot coordinates {(7,0.8)(7.3,1.2)(7.7,1.1)(8,0.6)(8.2,0.8)(8.5,1.3)(8.7,1.1)(9,0.7)(9.2,1.1)(9.45,1.3)};
\draw[green,thick] plot coordinates
{(9.5,1.1)(9.7,1.3)(9.9,1.7)(10.2,1.1)(10.4,1.3)(10.6,1.9)(10.8,1.4)
(11.1,1.2)(11.3,1.6)(11.5,1.9)(11.8,1.4)(12,1.1)(12.2,1.8)(12.4,1.3)(12.6,1.2)(12.8,1.1)(12.95,1.7)};
\draw[brown,thick] plot coordinates {(13,2.3)(13.3,1.7)(13.5,2.4)(13.7,2.3)(13.9,2.1)(14.1,1.9)(14.3,1.7)(14.4,2.1)(14.6,1.8)(14.8,1.6)(15,2.1)};
\node at (1.5,-0.5) {$s$};
\node at (4.5,-0.5) {$s_0$};
\node at (7,-0.5) {$s_1$};
\node at (9.5,-0.5) {$s_2$};
\node at (13,-0.5) {$s_3$};
\node at (15,-0.5) {$t$};
\node at (0,0) {$0$};
\node at (0,0.5) {$1$};
\node at (0,1) {$2$};
\node at (0,1.5) {$3$};
\node at (0,2) {$4$};
\end{tikzpicture}
\end{adjustbox}
    \caption{\small The Brownian increments $B_r(s_{r - 1},s_r)$ for $r=0,\dotsc,4$ in \eqref{E10} that make up the energy of the path depicted in Figure \ref{fig:BLPP_geodesic}. 
    }
    \label{fig:BLPP maximizing path}
\end{figure}

Each element $\mbf s_{m,n} \in \Pi_{(m,s),(n,t)}$ represents a unique continuous path $\Gamma$ in $\R^2$ from $\mbf x$ to $\mbf y$ as follows: $\Gamma$ consists of horizontal segments $\{r\}\times[s_{r-1}, s_r]$ on level $r$ for $r=m,\dotsc,n$, connected by vertical unit segments $[r,r+1]\times\{s_r\}$ for $r=m,\dotsc,n-1$. See Figure \ref{fig:BLPP_geodesic}. Because of this bijection, we regard  $\Pi_{\mbf x,\mbf y}$ equivalently as the space of such up-right paths from $\mbf x$ to $\mbf y$. 
For $(m,t) \in \Z \times \R$, we graphically represent the $t$-coordinate as the horizontal coordinate (the time coordinate of the Brownian motions) and the $m$-coordinate as the vertical coordinate (level) on the plane. This is a convention that is taken from~\cite{blpp_utah}, although it disagrees with the standard $x-y$ labelling of the coordinate axes. By continuity and compactness, for all $(m,s) = \mbf x \le \mbf y = (n,t) \in \Z \times \R$, there exists   $\mbf s_{m,n} \in \Pi_{\mbf x,\mbf y}$ such that $\E(\mbf s_{m,n}) = L_{\mbf x,\mbf y}$.  We call a  maximizer $\mbf s_{m,n}$  and its associated path a \textit{geodesic} between $\mbf x$ and $\mbf y$.

The following lemma 
establishes uniqueness of finite geodesics for a fixed initial and terminal point. 
\begin{lemma}[\cite{Hammond4}, Theorem B.1] \label{lemma:uniqueness of LPP time}
Fix endpoints $\mbf x \le \mbf y \in \Z \times \R$. Then, with probability one, there is a unique path whose energy achieves $L_{\mbf x,\mbf y}(\mbf B)$.
\end{lemma}
However, it is also true that for each fixed initial point $\mbf x \in \Z \times \R$, with probability one, there exist points $\mbf y \ge \mbf x$ such that the geodesic between $\mbf x$ and $\mbf y$ is not unique.
We show how to construct such points in Lemma~\ref{lemma:mult_geod} and derive a bound on the number of geodesics in Lemma~\ref{lemma:geodesic_bound}.
The following important lemma  is a deterministic statement which holds for last-passage percolation across any field of continuous functions, hence in particular for Brownian motions.

\begin{lemma}[\cite{Directed_Landscape}, Lemma 3.5] \label{existence of leftmost and rightmost geodesics}
Between any two points $(m,s) \le (n,t) \in \Z \times \R$,  there is a rightmost and a leftmost Brownian last-passage geodesic between the two points. That is, there exist $\mbf s_{m,n}^L,\mbf s_{m,n}^R \in \mbf \Pi_{(m,s),(n,t)}$, that are maximal for $\E(\mbf s_{m,n})$, such that, for any other maximal sequence $\mbf s_{m,n}$, $s_{r}^L \le s_r \le s_r^R$ for $m \le r \le n$.
\end{lemma}

To an infinite sequence, $s = s_{m - 1} \le s_m \le s_{m + 1} \le \cdots$ we similarly associate a semi-infinite path. It is possible that $s_r = \infty$ for some $r \ge m$, in which case the last segment of the path is the ray $ \{r\}\times [s_{r - 1},\infty)$, where $r$ is the first index with $s_r = \infty$. The infinite path has direction $\theta \in [0,\infty]$ if 
\[
\lim_{n \rightarrow \infty}\f{s_n}{n} \qquad\text{exists and equals }\theta.
\]
We call an up-right semi-infinite path a \textit{semi-infinite geodesic} if, for any two points $\mbf x\le \mbf y \in \Z \times \R$ that lie on the path, the portion of the path between the two points is a geodesic. 

For a semi-infinite, up-right path $\Gamma$ starting from $\mbf x \in \Z \times \R$,   the coordinate-wise ordering $\le$ is a complete ordering of the set $\Gamma$. This motivates the following definition.
\begin{definition} \label{def:coalescence}
 Two semi-infinite, up-right paths $\Gamma_1$ and $\Gamma_2$ {\it coalesce} if there exists a point $\mbf z \in \Gamma_1\cap\Gamma_2$ such that for all $\mbf w \ge \mbf z$, 
 $\mbf w \in \Gamma_1$ iff  $\mbf w \in \Gamma_2$. We call the minimal such $\mbf z$ the {\it coalescence point}. See Figure~\ref{fig:coalpt}. 
\end{definition}    

\noindent
The following states the coalescence into a {\it fixed} direction 
 proved in~\cite{Seppalainen-Sorensen-21a}.  Theorem~\ref{thm:general_coalescence} of the present paper extends this result to coalescence of all Busemann geodesics with the same direction $\theta > 0$ and sign $\sigg \in \{-,+\}$ (see Section~\ref{sec:geometry_sub} for the precise definitions).
\begin{theorem}[\cite{Seppalainen-Sorensen-21a}, Theorem 3.1(vii)] \label{thm:fixedcoal}
Fix $\theta > 0$. Then, with probability one, all $\theta$-directed semi-infinite geodesics coalesce. 
\end{theorem}

\begin{figure}[t]
 \centering
            \begin{tikzpicture}
            \draw[gray,thin] (0,0)--(10,0);
            \draw[gray,thin] (0,1)--(10,1);
            \draw[gray,thin] (0,2)--(10,2);
            \draw[gray,thin] (0,3)--(10,3);
            \draw[gray,thin] (0,4)--(10,4);
            \draw[gray,thin] (0,5)--(10,5);
            \draw[red,ultra thick,->] plot coordinates {(1,0)(2,0)(2,1)(3,1)(3,2)(4.5,2)(4.5,3)(7,3)(7,4)(8,4)(8,5)(10,5)(10,5.5)};
            \draw[red,ultra thick] plot coordinates {(0,2)(2.5,2)(2.5,3)(4,3)(4,4)(8,4)};
            \filldraw[black] (1,0) circle (2pt) node[anchor = north] {$\mbf x$};
            \filldraw[black] (0,2) circle (2pt) node[anchor = south] {$\mbf y$};
            \filldraw[black] (7,4) circle (2pt) node[anchor = south] {$\mbf z$};
            \end{tikzpicture}
            \caption{\small Coalescence of two up-right paths at $\mbf z$.}
            \label{fig:coalpt}
            \bigskip
        \end{figure}
\subsection{Main theorems} \label{sec:main_thm}

The geometric properties of BLPP obtained in this paper rest on studying the Busemann process $\{\B^{\theta \sig}(\mbf x,\mbf y): \mbf x,\mbf y \in \Z \times \R,\,\theta > 0,\,\sigg \in \{-,+\}\}$, defined for all points and directions simultaneously.
The $\sigg \in \{-,+\}$ distinction records the left- and right-continuous versions of the process as a function of $\theta$.   Theorem~\ref{thm:summary of properties of Busemanns for all theta} provides a detailed summary of the properties of this process. 
 The immediate connection between the Busemann process and the last-passage percolation process is the following limit, stated in 
Theorem~\ref{thm:summary of properties of Busemanns for all theta}\ref{busemann functions agree for fixed theta}:  for a fixed direction $\theta > 0$, with probability one, for all $\mbf x,\mbf y \in \Z \times \R$,
\[
\B^{\theta -}(\mbf x,\mbf y) = \B^{\theta +}(\mbf x,\mbf y) = \lim_{n \to \infty} \bigl[L_{\mbf x,(n,n\theta)} - L_{\mbf y,(n,n\theta)}\bigr].
\]

However, in general across all directions $\theta > 0$, it does not hold that $\B^{\theta -} = \B^{\theta +}$ as functions $(\Z \times \R)^2 \to \R$. The finer geometric properties of BLPP turn out to be intimately related  to the random set of discontinuities of the Busemann process, defined as follows: 
\be \label{eqn:Theta}
\Busedc_{\mbf x,\mbf y} = \{\theta > 0: \B^{\theta-}(\mbf x,\mbf y ) \neq \B^{\theta +}(\mbf x,\mbf y)\}\qquad \text{and}\qquad
\Busedc = \bigcup_{\mbf x,\mbf y \in \Z \times \R} \Busedc_{\mbf x,\mbf y}.
\ee
As the discontinuity set of a function of locally bounded variation (see Remark~\ref{rmk:dense_theta}), $\Busedc_{\mbf x,\mbf y}$ is at most countable. When it is understood that $\theta \notin \Busedc$, we write $\B^\theta$ without the $\pm$ distinction in the superscript. 
 Our first main result is a description of the random set of discontinuities of the Busemann process. 

\begin{theorem} \label{thm:Theta properties}
For each fixed $\theta > 0$, $\Pp(\theta \in \Busedc) = 0$. Further, the following hold on a single event of probability one. 
\begin{enumerate} [label=\rm(\roman{*}), ref=\rm(\roman{*})]  \itemsep=3pt
    \item \label{itm:Theta_count_intro}The set $\Busedc$ is countably infinite and dense in $\R_{>0}$. 
    \item \label{itm:const} For each $\mbf x\neq \mbf y \in \Z \times \R$, the set $\Busedc_{\mbf x,\mbf y}$ is infinite and either has a single limit point at $0$ or no limit points. Furthermore, on each open interval $I \subseteq (0,\infty) \setminus \Busedc_{\mbf x,\mbf y}$, the function  $\theta \mapsto \B^{\theta-}(\mbf x,\mbf y) =\B^{\theta+}(\mbf x,\mbf y)$ is constant on $I$.
    \item \label{itm:Theta=Vm} For each $m \in \Z$, the set $\Busedc_{(m,-t),(m,t)}$ is nondecreasing in $t \in \R_{\ge 0}$. For any $m \in \Z$ and any sequence $t_k \to \infty$,
    \be \label{eqn:line_dc}
    \Busedc = \bigcup_k \Busedc_{(m,-t_k),(m,t_k)}
    \ee
\end{enumerate}
\end{theorem}
\begin{remark}
  Part \ref{itm:Theta=Vm} above says that the entire set of discontinuities appears in the discontinuities of $\theta\mapsto \B^{\theta}((m,-t),(m,t))$ for $t$ outside any large bounded interval $[-T,T]$, on each horizontal level $m$ of the lattice. 
\end{remark}

\begin{remark}
Theorem~\ref{thm:Theta properties}\ref{itm:const} states that $\theta \mapsto \B^{\theta \pm}(\mbf x,\mbf y)$ are the right- and left-continuous versions of a jump process. This condition implies strong results about the collection of semi-infinite geodesics. In particular, the set $\Busedc$ classifies directions in which the collection of semi-infinite geodesics in that direction all coalesce. This is described in the next theorem. 
\end{remark}
\begin{theorem} \label{thm:Busedc_class}
The following hold on a single event of full probability.
\begin{enumerate}[label=\rm(\roman{*}), ref=\rm(\roman{*})]  \itemsep=3pt
\item \label{itm:BLPP_good_dir_coal} When $\theta \notin \Busedc$, all $\theta$-directed semi-infinite geodesics {\rm(}from each initial point{\rm)} coalesce. There is a countably infinite random set of initial points, outside of which, the semi-infinite geodesic in each direction $\theta\notin\Theta$ is unique.
\item \label{itm:BLPP_bad_dir_split} When $\theta \in \Busedc$, there are at least two coalescing families of $\theta$-directed semi-infinite geodesics, called the $\theta-$ and $\theta +$ geodesics. From each initial point $\mbf x \in \Z \times \R$, there exists at least one $\theta -$ geodesic and at least one $\theta +$ geodesic, which separate at some point $\mbf y \ge \mbf x$ and never come back together.
\end{enumerate}
\end{theorem}
\begin{remark}
There are two types of non-uniqueness present in Theorem~\ref{thm:Busedc_class}. The type mentioned in Part~\ref{itm:BLPP_good_dir_coal} is temporary in the sense that geodesics must come back to coalesce. This type of non-uniqueness occurs in every direction, but only from a countably infinite set of initial points. The second type of non-uniqueness  in Part~\ref{itm:BLPP_bad_dir_split} occurs from every initial point, but only in a countable dense set of directions. Unlike the previous type, the geodesics that separate do not come back together.  See Section~\ref{sec:non_unique_coal} for more discussion on non-uniqueness. In the case $\theta \in \Busedc$, we do not know whether there are more than two coalescing of families of geodesics, but we expect that this is not the case. In exponential last-passage percolation, it was shown in~\cite{Janjigian-Rassoul-Seppalainen-19} that there can be no more than two such families, using machinery from~\cite{Coupier-11} that relies on the connection to TASEP. See Remark~\ref{rmk:Busegeod} for further discussion.
\end{remark}

Due to the geometry of the space $\Z \times\R$, when the splitting of geodesics described in Theorem~\ref{thm:Busedc_class}\ref{itm:BLPP_bad_dir_split} occurs at a point $\mbf y$, one geodesic must make an upward step from $\mbf y$ while the other moves horizontally from $\mbf y$.  
The \textit{competition interface} from an initial point $\mbf y$ in discrete lattice models separates points $\mbf z \ge \mbf y$ depending on whether the geodesic from $\mbf y$ to $\mbf z$ makes an initial horizontal or vertical step. In BLPP, this concept is much more delicate. This is because, for a fixed initial point $(m,s) \in \Z \times \R$, with probability one, for \textit{every} point $(n,t)$ with $n \ge m$ and $t >s$, all geodesics from $(m,s)$ to $(n,t)$ travel initially along the horizontal line at level $m$. However, there is a random exceptional set of points at which this is not the case, defined as follows:
\begin{align}
\CI &= \{(m,s) \in \Z \times \R: \text{ for some } (n,t) \text{ with }n \ge m, t > s, \text{ there exists a}  \label{CI_intro}\\
&\qquad\qquad\text{ geodesic from $(m,s)$ to $(n,t)$ that makes an initial vertical step}\}. \nonumber 
\end{align}
Refer to Figure~\ref{fig:CI_Def} for clarity. 
\begin{figure}[t]
    \centering
    \begin{adjustbox}{max totalsize={.9\textwidth}{\textheight},center}
\begin{tikzpicture}
\draw[gray,thin] (0.5,0) -- (15.5,0);
\draw[gray,thin] (0.5,1) --(15.5,1);
\draw[gray, thin] (0.5,2)--(15.5,2);
\draw[gray,thin] (0.5,3)--(15.5,3);
\draw[gray,thin] (0.5,4)--(15.5,4);
\draw[red,ultra thick] (1.5,0)--(4.5,0)--(4.5,1)--(7,1)--(7,2)--(9.5,2)--(9.5,3)--(13,3)--(13,4)--(15,4);
\filldraw[black] (1.5,0) circle (2pt) node[anchor = north] {$(m,s)$};
\filldraw[black] (15,4) circle (2pt) node[anchor = south] {$(n,t)$};
\filldraw[black] (9.5,2) circle (2pt) node[anchor = north] {$(k,u)$};
\end{tikzpicture}
\end{adjustbox}
\caption{\small An example of a typical point $(m,s)$ that is \textit{not} in $\CI$. However, the point $(k,u)$ \textit{does} lie in $\CI$ since the geodesic from $(k,u)$ to $(n,t)$ makes an immediate vertical step.}
\label{fig:CI_Def}
\end{figure}
The following theorem describes this exceptional set.
 \begin{theorem} \label{thm:Haus_comp_interface}
 With probability one, for each level $m \in \Z$, the set 
    $
    \CI_m := \{s \in \R: (m,s) \in \CI\}
    $
    has Hausdorff dimension $\f{1}{2}$ and is dense in $\R$.  Hence, the set $\CI$ itself has Hausdorff dimension $\f{1}{2}$. For each $\mbf y \in \Z \times \R$, $\Pp(\mbf y \in \CI) = 0$. The set $\CI$ also has an equivalent description as the set of  $\mbf x \in \Z \times \R$ for which there exists a random direction $\theta > 0$ such that there are two semi-infinite geodesics from $\mbf x$ in direction $\theta$, whose only common point is the initial point $\mbf x$. 
 \end{theorem}
   \begin{remark}
    There are in fact many more equivalent ways to describe the set $\CI$. These are detailed in Theorem~\ref{thm:ci_equiv}. Compare Theorem~\ref{thm:Busedc_class}\ref{itm:BLPP_bad_dir_split} with Theorem~\ref{thm:Haus_comp_interface}. On one hand, when $\theta \in \Busedc$, from \textit{every} initial point $\mbf x$, there exist two semi-infinite geodesics in direction $\theta$ that eventually split. On the other hand, for all $\mbf x \notin \CI$, the two geodesics do not split immediately. See Figure~\ref{fig:intro_split}. This is in contrast to the exponential corner growth model studied in~\cite{Janjigian-Rassoul-Seppalainen-19}, where \textit{every} initial point has a random direction in which there are two semi-infinite geodesics in that direction that split immediately.  
  \end{remark}

\begin{figure}[t]
            \begin{tikzpicture}
            \draw[red,ultra thick,->] plot coordinates {(1,0)(4,0)(4,1)(6,1)(6,2)(8,2)(8,3)(9.5,3)(9.5,4)(12,4)};
            \draw[red, ultra thick] plot coordinates {(0,2)(2.5,2)(2.5,3)(4,3)(4,4)(9.5,4)};
            \draw[blue,thick,->] plot coordinates
            {(0,2)(2.5,2)(2.5,3)(3,3)(3,4)(3.5,4)(3.5,5)(8,5)(8,5.5)};
            \draw[blue,thick] plot coordinates {(1,0)(4,0)(4,1)(5,1)(5,2)(5.5,2)(5.5,3)(6,3)(6,4)(7,4)(7,5)};
            \filldraw[black] (1,0) circle (2pt) node[anchor = north] {\small $\mbf x$};
            \filldraw[black] (0,2) circle (2pt) node[anchor = east] {\small $\mbf y$};
            \filldraw[black] (3,3) circle (2pt) node[anchor = north] {\small $\mbf w$};
            \filldraw[black] (5,1) circle (2pt) node[anchor = north] {\small $\mbf v$};
            \end{tikzpicture}
            \caption{\small An illustration of the behaivor of Theorem~\ref{thm:Busedc_class}\ref{itm:BLPP_bad_dir_split} and Theorem~\ref{thm:Haus_comp_interface}. The red/thick paths are the $\theta +$ geodesics and the blue/thin paths are the $\theta-$ geodesics, for $\theta \in \Busedc$. The $\theta-$ geodesics all coalesce, and the $\theta +$ geodesics all coalesce.  From every initial point, there are two distinct semi-infinite geodesics in direction $\theta$, but the geodesics can only split at points lying in the Hausdorff dimension $\f{1}{2}$ set $\CI$ (In this example, the splitting points are $\mbf v$ and $\mbf w$). }
            \label{fig:intro_split}
        \end{figure}

\section{The distribution of the Busemann process} \label{sec:Buse_jd_intro}

 As alluded to in the previous section, Busemann functions give  the asymptotic difference of last-passage times from all pairs of   starting points to a common terminal point that travels to $\infty$ in a given direction. The direction is indexed by a parameter $\theta > 0$. See Figure~\ref{fig:Busemann functions} and Theorem~\ref{thm:summary of properties of Busemanns for all theta}\ref{busemann functions agree for fixed theta} below. Alberts, Rassoul-Agha, and Simper~\cite{blpp_utah} proved the existence of Busemann functions for fixed initial points and directions. In~\cite{Seppalainen-Sorensen-21a}, we extended this to the full Busemann process, indexed by all lattice pairs $(\mbf x, \mbf y)$, directions $\theta>0$ and signs $\pm$, that records also the discontinuities in the direction parameter.  This is our starting point. In order to clearly indicate whether a probability one statement applies globally or to fixed parameters, we refer to several full probability events that were constructed in~\cite{Seppalainen-Sorensen-21a}, namely $\Omega_2$, $\Omega^{(\theta)}$, and $\Omega_{\mbf x}^{(\theta)}$. 
 
   \begin{figure}[t]
  \centering
            \begin{tikzpicture}
            \draw[black,thick] (0,0)--(8,0);
            \draw[black,thick] (0,0)--(0,3.5);
            \draw[red, ultra thick] plot coordinates {(1.5,1)(2.7,1)(2.7,1.5)(3,1.5)(3,2)(4.6,2)};
            \draw[red, ultra thick,arrows = ->] plot  coordinates {(2.1,0.1)(3,0.1)(3,0.5)(3.6,0.5)(3.6,1)(4.1,1)(4.1,1.5)(4.6,1.5)(4.6,2)(5.5,2)(5.5,2.5)(6.1,2.5)(6.1,3)(6.9,3)};
            \draw[black,thick,arrows = ->] (7,3) --(8,3.5);
            \filldraw[black] (1.5,1) circle (2pt) node[anchor =  south] {\small $\mbf y$};
            \filldraw[black] (2,0) circle (2pt) node[anchor = south] {\small $\mbf x$};
            \filldraw[black] (7,3) circle (2pt);
            \node at (7,2.7) {\small $(n,n\theta)$};
            \draw[black,thick] (-0.2,3)--(0.2,3);
            \node at (-0.4,3) {\small $n$};
            \draw[black,thick] (7,-0.2)--(7,0.2);
            \node at (7,-0.4) {\small $n\theta$};
            \end{tikzpicture}
            \caption{\small Geodesics from $\mbf x$ and $\mbf y$ to a common terminal point $(n,n\theta)$. The Busemann limit sends $n\to\infty$.}
            \label{fig:Busemann functions}
        \end{figure}
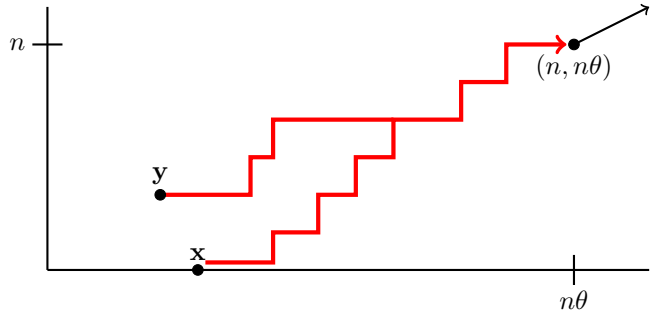

\begin{theorem}[\cite{Seppalainen-Sorensen-21a}, Theorems 3.5 and 3.7] \label{thm:summary of properties of Busemanns for all theta}
On $(\Omega, \F,\Pp)$, there exists a process
\[
\{\B^{\theta \sig}(\mbf x,\mbf y): \theta > 0,\sigg \in \{-,+\}, \mbf x,\mbf y \in \Z \times \R\}
\]
with the following properties. Below, vertical and horizontal Busemann increments are abbreviated by 
\begin{align} 
    &\vv_{m}^{\theta\sig}(t) := \B^{\theta \sig}((m - 1,t),(m,t)), \text{ and} \label{vertical Busemann simple expression} \\
    &\h_m^{\theta \sig}(t) := \B^{\theta \sig}((m,0),(m,t)). \label{horizontal Busemann simple expression}
\end{align}
    \begin{enumerate} [label=\rm(\roman{*}), ref=\rm(\roman{*})]  \itemsep=3pt 
    \item{\rm(}Additivity{\rm)} On $\Omega_2$, whenever $\mathbf x,\mathbf y,\mathbf z \in (\Z \times \R)$, $\theta > 0$, and $\sigg \in \{-,+\}$,
    \[
    \B^{\theta \sig}(\mathbf x, \mathbf y) + \B^{\theta \sig}(\mathbf y,\mathbf z) = \B^{\theta\sig}(\mathbf x, \mathbf z). \label{general additivity Busemanns}
    \] 
    \item{\rm(}Monotonicity{\rm)} \label{general monotonicity Busemanns} On $\Omega_2$, whenever $0 <\gamma < \theta < \infty$, $m \in \Z$, and $t \in \R$,
    \begin{align*}
    0 &\le \vv_{m}^{\gamma -}(t) \leq \vv_{m}^{\gamma +}(t) \leq \vv_{m}^{\theta -}(t) \le \vv_{m}^{\theta +}(t), \text{ and } \\
    B_m&\li \h_m^{\theta +} \li \h_m^{\theta -} \li \h_m^{\gamma +} \li   \h_m^{\gamma -}. 
    \end{align*}
\item{\rm(}Convergence{\rm)} \label{general uniform convergence Busemanns} On $\Omega_2$, for every $m \in \Z$, $\theta > 0$ and $\sigg \in \{-,+\}$, 
    \begin{enumerate} [label=\rm(\alph{*}), ref=\rm(\alph{*})]  \itemsep=3pt 
        \item \label{general uniform convergence:limits from left} As $\gamma \nearrow \theta$,  $\B^{\gamma \sig}(\mbf x,\mbf y)$ converges, uniformly on compact subsets of $(\Z \times \R)^2$, to $\B^{\theta -}(\mbf x,\mbf y)$.  
        \item \label{general uniform convergence:limits from right} As $\delta \searrow \theta$, $\B^{\delta \sig}(\mbf x,\mbf y)$ converges, uniformly on compact subsets of $(\Z \times \R)^2$, to $\B^{\theta +}(\mbf x,\mbf y)$. 
        \item \label{general uniform convergence:limits to infinity}As $\gamma \nearrow \infty$, $\h_m^{\gamma \sig}$ converges,  uniformly on compact subsets of $\R$, to $B_m$. 
        \item \label{general uniform convergence:limits to 0} As $\delta \searrow 0$, $\vv_m^{\delta \sig}$ converges, uniformly on compact subsets of $\R$, to $0$. 
    \end{enumerate}
\item{\rm(}Continuity{\rm)} \label{general continuity of Busemanns}  
For all $r,m \in \Z$, $\theta > 0$, and $\sigg \in \{-,+\}$, $(s,t) \mapsto \B^{\theta \sig}((m,s),(r,t))$ is a continuous function $\R^2 \rightarrow \R$. 
\item{\rm(}Limits{\rm)}  \label{limits of B_m minus h m + 1}
    For each $\theta > 0$ and $\sigg \in \{-,+\}$,
    \[
    \lim_{s \rightarrow \pm \infty} \bigl[B_m(s) - \h_{m + 1}^{\theta \sig}(s)\bigr] = \mp \infty. 
    \]
    \item{\rm(}Queuing relationships for Busemann functions{\rm)}
    \label{general queuing relations Busemanns}
    For all $m \in \Z$, $\theta > 0$, and signs $ \sigg \in \{-,+\}$,  
    \[
    \vv_{m + 1}^{\theta \sig} = Q(\h_{m + 1}^{\theta \sig},B_m) \qquad\text{and}\qquad \h_m^{\theta \sig} =D(\h_{m + 1}^{\theta \sig},B_m), 
    \]
    where $Q$ and $D$ are defined in~\eqref{definition of Q}--\eqref{definition of D}. 
    \item{\rm(}Busemann limit in a fixed direction{\rm)} \label{busemann functions agree for fixed theta} Fix  $\theta > 0$. Then, on the event $\Omega^{(\theta)}$, for all $\mbf x,\mbf y \in \Z \times \R$ and all sequences $\{t_n\}$ satisfying $t_n/n \rightarrow \theta$ as $n\to\infty$,
\be\label{B120} 
\B^{\theta-}(\mbf x,\mbf y) = \lim_{n \rightarrow \infty} \bigl[L_{\mbf x,(n,t_n)} - L_{\mbf y,(n,t_n)}\bigr] =  \B^{\theta+}(\mbf x,\mbf y). 
\ee
    \item{\rm(}Independence{\rm)}
    For any $m \in \Z$,
    \[
    \{\h_r^{\theta \sig}: \theta > 0, \sigg \in \{-,+\}, r > m\} \text{ is independent of } \{B_r: r \le m\} \label{independence structure of Busemann functions on levels}.
    \]
    \item{\rm(}Marginal distributions{\rm)} \label{Buse_marg_dist}For each $\theta > 0$, the process $t \mapsto \h_m^\theta(t)$ is a two-sided Brownian motion with drift $\f{1}{\sqrt \theta}$. The process $t \mapsto \vv_m^\theta(t)$  is a stationary and reversible strong Markov process such that, for each $t \in \R$, $\vv_m^\theta(t) \sim \operatorname{Exp}\bigl(\f{1}{\sqrt \theta}\bigr)$.
\item{\rm(}Shift invariance{\rm)} \label{itm:shift_invariance}
For each $\mbf z \in \Z \times \R$,
\begin{align*}
&\{\B^{\theta \sig}(\mbf x,\mbf y):\mbf x,\mbf y \in \Z \times \R,\theta > 0, \sigg \in \{-,+\}\} \\\deq &\{\B^{\theta \sig}(\mbf x + \mbf z,\mbf y + \mbf z):\mbf x,\mbf y \in \Z \times \R,\theta > 0, \sigg \in \{-,+\}\}. 
\end{align*}
\end{enumerate}
\end{theorem}

\begin{remark}
 Since $h_m^{\theta \sig}(0) = 0$ for $m \in \Z$, $\theta > 0$, and $\sigg \in \{-,+\}$, the monotonicity of Part~\ref{general monotonicity Busemanns} implies that, for $m \in \Z$, $t > 0$, and $\gamma < \theta < \delta$, 
 \be \label{weakmont}
h_m^{\delta -}(t) \le h_m^{\theta +}(t) \le h_m^{\theta - }(t) \le h_m^{\gamma+}(t), \qquad\text{and for }t < 0,\text{ all inequalities reverse}.
 \ee
Note that Part~\ref{general monotonicity Busemanns} is much stronger than~\eqref{weakmont}, as {\it all}  increments of $h_m^{\theta - }$ dominate those of $h_m^{\theta +}$. This property is used often in the sequel. 
\end{remark}

\begin{remark} \label{rmk:dense_theta}  Theorem~\ref{thm:summary of properties of Busemanns for all theta}\ref{busemann functions agree for fixed theta} implies that we can fix an arbitrary  countable dense subset $\Lambda$ of directions in $\R_{>0}$ and then include in any  full-probability event  the condition that the limit \eqref{B120} holds for all $\theta\in\Lambda$. In particular, then $\B^{\theta -}(\mbf x,\mbf y) = \B^{\theta +}(\mbf x,\mbf y)$ for all $\mbf x,\mbf y \in \Z \times \R$ and all $\theta\in\Lambda$. This and the left and right limits in Part \ref{general uniform convergence Busemanns} then imply that   $\theta\mapsto\B^{\theta-}(\mbf x,\mbf y)$ and $\theta\mapsto\B^{\theta+}(\mbf x,\mbf y)$ are the left- and right-continuous versions of the same function of locally bounded variation, and a jump happens at any given $\theta$ with probability zero. 

When we prove our new results, we choose $\Lambda=\Q_{>0}$. This comes in the definition~\eqref{omega4} of the full-probability event $\Omega_4$ in the proofs section. As a result, rational directions $\theta$ will occupy a special role in some statements.

A key point is the distinction between the global view and the view into a fixed direction $\theta$. Only the global view reveals the $\pm$ distinction.  On the event  $\Omega^{(\theta)}$ we do not see the   $\pm$ distinction, and hence we can drop the sign from  the superscript and write $\B^\theta, h_m^{\theta}$, and $v_m^{\theta}$.  Note also that the limit in \eqref{B120} has not been established simultaneously in all directions.  
\end{remark}

The term ``Busemann increment'' is justified by the fact that $\B^{\theta \sig}(\mathbf x, \mathbf y)=\B^{\theta \sig}(\mathbf 0, \mathbf y)-\B^{\theta \sig}(\mathbf 0, \mathbf x)$.

The geometric properties of geodesics and competition interfaces explained in   Section~\ref{section:main_geometry} are proved from properties of the  distribution of the Busemann process $\B^{\theta \sig}(\mbf x,\mbf y)$, to which we now turn. 
Through the queuing transformations (Theorem~\ref{thm:summary of properties of Busemanns for all theta}\ref{general queuing relations Busemanns}), additivity (Theorem~\ref{thm:summary of properties of Busemanns for all theta}\ref{general additivity Busemanns}) and stationarity, 
  in principle we can understand the entire Busemann process by restricting our attention to the Busemann process  on a  single  horizontal level $m$:  $\{h_m^{\theta \sig}(t):  \theta > 0, \tspa \sig\in\{-,+\}, \tspa t \in \R \}$. 


\subsection{Horizontal Busemann functions as transforms of Brownian motions with drift}
\label{section:fdd}
The joint distribution of finitely many horizontal Busemann functions is constructed by applying queuing transformations to  independent Brownian motions with drift. 
We define first the path spaces, then the mappings, and lastly the distributions. Recall the pinned function space $\CRpin$ from Section \ref{section:notation}\ref{pinned notation}. 
Set 
\be\label{Yndef}\begin{aligned} 
\ \;  \Y_n := \Bigg\{\mathbf Z = (Z^1,\ldots, Z^n) \in \CRpin^n: \text{ for } 1 \le i \le n, \lim_{t \rightarrow \infty} \frac{Z^i(t)}{t} \text{ exists and lies in }\R_{> 0},  \\%
\qquad\qquad\text{and for } 2 \leq i \leq n, \lim_{t \rightarrow \infty} \frac{Z^i(t)}{t} > \lim_{t \rightarrow \infty} \frac{Z^{i - 1}(t)}{t}    \Bigg\},
\end{aligned}\ee
and
\begin{align} \label{Xndef}
\X_n := \Bigg\{\eta = (\eta^1,\ldots,\eta^n) \in \CRpin^n:  \eta^i \gi \eta^{i - 1}  \text{ for } 2 \leq i \leq n,  
\text{ and } \liminf_{t \rightarrow \infty} \frac{\eta^1(t)}{t} > 0 \Bigg\}.
\end{align}
Two larger spaces $\Ycomp_n$ and $\Xcomp_n$ are defined as above except that the lowest limits 
\[
\lim_{t\rightarrow \infty}
t^{-1}Z^1(t)
\qquad\text{ and }\qquad\liminf_{t\rightarrow \infty} t^{-1}\eta^1(t)
\]
are permitted to be $0$ while the other inequalities are still required to be strict. 
These four  spaces  
are Borel subsets of the space $C(\R)^n$ (see Section~\ref{sec:Buse_proofs}) and in particular separable metric spaces under the topology  of uniform convergence on compact subsets of $\R$. 


 For two functions $Z,B \in \CRpin$ satisfying $\limsup_{t \rightarrow \infty} [B(t) - Z(t)] = -\infty$, define the following mappings. 
\begin{align}
Q(Z,B)(t) = &\sup_{t \le s < \infty}\{B(t,s)-Z(t,s)\}, \label{definition of Q} \\
D(Z,B)(t) = &B(t) + \sup_{0 \le s < \infty}\{B(s) - Z(s)\} - \sup_{t \le s < \infty}\{B(s) - Z(s)\}, \label{definition of D}\\
R(Z,B)(t) = &Z(t) + \sup_{t \le s < \infty}\{B(s) - Z(s)\} - \sup_{0 \le s < \infty}\{B(s) - Z(s)\}. \label{definition of R}
\end{align}
Equivalently, $D(Z,B)(t) = Z(t) + Q(Z,B)(0) - Q(Z,B)(t)$, and $R(Z,B)(t) = B(t) + Q(Z,B)(t) - Q(Z,B)(0)$.
In queuing terms, the increments of $Z$ denote the arrivals process to the queue, while the increments of $B$ denote the service process. For outputs, $Q(Z,B)$ is the queue-length process, and the increments of $D(Z,B)$ form the departures process. See Section 5.3 and Appendix C of~\cite{Seppalainen-Sorensen-21a} for a more detailed description of the connections to queuing theory. For $0 \le a < b$ and $0 \le c < d$, the pair $(D,R)$ is bijective on the following space of functions, denoted $\Y_2^{(a,b),(c,d)}$:
\begin{align*}
 \Bigg\{(Z,B) \in \CRpin^2: \lim_{t \rightarrow \infty} \frac{B(t)}{t} = a,
\lim_{t \rightarrow \infty} \frac{Z(t)}{t} = b, \lim_{t \rightarrow -\infty}\frac{B(t)}{t} = c,
\lim_{t \rightarrow -\infty} \frac{Z(t)}{t} = d \Bigg\}.
\end{align*}
This is presented as Theorem D.1 in~\cite{Seppalainen-Sorensen-21a}, although some extra care is needed to show that $(D,R)$ and its inverse preserve the space $\Y_2^{(a,b),(c,d)}$.  We do not use the bijectivity of the map $(D,R)$ in the present paper, so we omit the full details. A proof that the map $(D,R)$ preserves limits as $t \rightarrow \infty$ is presented as Lemma~\ref{D and R preserve limits}.

We iterate the mapping $D$ as follows: first, set $D^{(1)}(Z) = Z$, and for $n \ge 2$,
\begin{align}
    D^{(n)}(Z^n,Z^{n - 1},\ldots,Z^1) = D(D^{(n - 1)}(Z^n,\ldots,Z^{2}),Z^1). \label{D iterated}
\end{align}
Next define a transformation  $\D^{(n)}$ that maps $\Y_n$ into $\X_n$ and $\Ycomp_n$ into $\Xcomp_n$. For $\mathbf Z  = (Z^1,\ldots,Z^n)\in \Ycomp_n$, the image $\D^{(n)}(\mbf Z) = \eta = (\eta^1,\ldots, \eta^n) \in \Xcomp_n$ is defined as follows: 
\begin{equation} \label{definition of script D}
\eta^i = D^{(i)}(Z^i,\ldots,Z^1) \text{ for } 1 \le i \le n.
\end{equation}
We used decreasing indexing in \eqref{D iterated} to match the main definition \eqref{definition of script D}.

As discussed above, these mappings have their origin in queuing theory. This goes back to the work of Harrison and Williams~\cite{Harrison1985,harrison1990,harrison1992}, but the particular formulation of these mappings matches more closely that in~\cite{brownian_queues}. The iterated mapping $\D^{(n)}$ has analogues in discrete queuing systems. See Theorem 2.1 in~\cite{Ferrari-Martin-2007} and Equation (3-3) in~\cite{Fan-Seppalainen-20}.


\begin{lemma} \label{image of script D lemma}
The mapping $\D^{(n)}$ satisfies the following properties: 
\begin{enumerate} [label=\rm(\roman{*}), ref=\rm(\roman{*})]  \itemsep=3pt 
    \item \label{itm:image} $\D^{(n)}$ maps $\Y_n$ into $\X_n \cap \Y_n$ and $\Ycomp_n$ into $\Xcomp_n \cap \Ycomp_n$.
    \item \label{itm:Ynlim} If $(Z^1,\ldots,Z^n) \in \Ycomp_n$ satisfies
    \[
    \lim_{t \rightarrow \infty} \f{Z^i(t)}{t} = a_i\qquad\text{for }  1 \le i \le n,
    \]
    then the image $(\eta^1,\ldots,\eta^n) = \D^{(n)}(Z^1,\ldots,Z^n)$ also satisfies
    \[
    \lim_{t \rightarrow \infty} \f{\eta^i(t)}{t} = a_i\qquad\text{for }  1 \le i \le n.
    \]
\end{enumerate}
\end{lemma}

\begin{definition} \label{definition of v lambda and mu lambda}
Given $\lambda = (\lambda_1,\ldots,\lambda_n)$ with $0 < \lambda_1 < \cdots < \lambda_n$, define the probability measure $\nu^\lambda$ on $\Y_n$ as follows: the vector $\mathbf Z = (Z^1,\ldots,Z^n)$ has distribution $\nu^\lambda$ if the components of $\mathbf Z$ are independent and each $Z^i$ is a standard, two-sided Brownian motion with drift $\lambda_i$. The measure $\nu^\lambda$ is extended to $\Ycomp_n$ when $\lambda_1 = 0$. 
Define the measure $\mu^\lambda$ on $\X_n$ (or $\Xcomp_n$) as $\mu^\lambda = \nu^\lambda \circ (\D^{(n)})^{-1}$.
\end{definition}

\begin{lemma} \label{weak continuity and consistency}
The following properties of the measures $\mu^\lambda$ hold: 
\begin{enumerate} [label=\rm(\roman{*}), ref=\rm(\roman{*})]  \itemsep=3pt 
    \item \label{weak continuity} {\rm(}Weak continuity{\rm)}
    Let $\lambda = (\lambda_1,\ldots,\lambda_n)$ with $0 \le \lambda_1 < \cdots < \lambda_n$. For $1 \le i \le n$, let $\lambda_i^k \ge 0$ be sequences satisfying $\lim_{k \rightarrow \infty}\lambda_i^k = \lambda_i$. Then, if $\lambda^k = (\lambda_1^k,\lambda_2^k,\ldots,\lambda_n^k)$, $\mu^{\lambda^k} \rightarrow \mu^\lambda$ weakly, as probability measures on $\Xcomp_n$. 
    \item \label{consistency}{\rm(}Consistency{\rm)} If $(\eta^1,\ldots,\eta^n) \in \Xcomp_n$ has distribution $\mu^{(\lambda_1,\ldots,\lambda_n)}$ for $0 \le \lambda_1 < \cdots < \lambda_n$, then any subsequence $(\eta^{j_1},\ldots,\eta^{j_k})$ has distribution $\mu^{(\lambda_{j_1},\ldots,\lambda_{j_k})}$. 
    \item \label{scaling relations}{\rm(}Scaling relations{\rm)} Let $0 \le \lambda_1 < \cdots < \lambda_n$ and $c> 0, \nu \in \R$. If $(\eta^1,\ldots,\eta^n)$ has distribution $\mu^{(\lambda_1,\ldots,\lambda_n)}$ and $(\wt \eta^1,\ldots,\wt \eta^n)$ has distribution $\mu^{(c(\lambda_1 + \nu),\ldots,c(\lambda_n + \nu))}$, then
\[
\{(\eta^1(t),\ldots,\eta^n(t )):t \in \R\}\deq \Big\{\bigl(c \wt \eta^1(t/c^2) - \nu t,\ldots,c \wt \eta^n(t/c^2)-\nu t\bigr):t \in \R\Big\}.
\]
\end{enumerate}
\end{lemma}

\noindent  Now, we can give the following description of finitely many horizontal   Busemann functions on a given level. 
There is no $\theta\pm$ distinction in the statement because it involves only finitely many $\theta$-values, and for a given $\theta$ and $m$, the functions $h_{m}^{\theta-}$ and $h_{m}^{\theta+}$ almost surely coincide.

\begin{theorem} \label{dist of Busemann functions and Bm}
Let $\theta_1 > \theta_2 > \cdots > \theta_n > 0$ and set $\lambda_i = \f{1}{\sqrt \theta_i}$ for $1 \le i \le n$. Then, for each level $m \in \Z$, the $(n+1)$-tuple of functions 
$ 
(B_m,h_{m}^{\theta_1},\ldots,h_{m}^{\theta_n})
$ 
 lies 
almost surely in the space $\Xcomp_{n+1}\cap\Ycomp_{n+1}$ and has probability distribution $\mu^{(0,\lambda_1,\ldots,\lambda_n)}$. 
\end{theorem}

\subsection{Fixed time marginal process across directions}
\label{section:Busepp}
 
 In this section we study the process $h_0^{\theta \sig}(t)$ for a fixed $t$, as $\theta$ varies. 
 While $\theta$ is the geometrically natural parameter because it represents the asymptotic direction of semi-infinite geodesics, we will also find the parameter $\lambda := \f{1}{\sqrt \theta}$ useful. In particular, for $\lambda > 0$, $t\mapsto h_0^{1/\lambda^2}(t)$ is a Brownian motion with drift $\lambda$. 
 When $\lambda$ is the index, it is convenient to have the alternative notation   
 \[
X(\lambda;t) := h_0^{(1/\lambda^2)-}(t)
\qquad \text{for }  \lambda > 0 \text{ and } t \in \R,
\]
so that $\lambda \mapsto X(\lambda;t)$ is a cadlag process, and $\Ee[X(\lambda;t)] = \lambda t$.
 In light of Theorem~\ref{dist of Busemann functions and Bm}, it makes sense to extend the definition to $\lambda=0$ by setting    $X(0;t):= B_0(t)$. Next, we describe  the behavior of the process $\{X(\lambda;t): \lambda \ge 0\}$ for fixed $t\in\R$. Since the Busemann functions satisfy $h_m^\theta(0) = 0$ and $h_m^\theta \li h_m^\gamma$ for $\theta > \gamma$ (Theorem~\ref{thm:summary of properties of Busemanns for all theta}),   $\lambda \mapsto X(\lambda;t)$ is a nondecreasing process for each $t > 0$.

 \begin{remark} \label{rmk:any_interval_same}
From $h_0^{\theta \sig}(t) = \B^{\theta \sig}((0,0),(0,t))$ and shift invariance (Theorem~\ref{thm:summary of properties of Busemanns for all theta}\ref{itm:shift_invariance}), we have this equality  in distribution of $\lambda$-indexed processes, for each $s \in \R$ and $t > 0$:
\be \label{eqn:deq}
\{X(\lambda;t): \lambda \ge 0\} \deq \{X(\lambda;t + s) - X(\lambda;s): \lambda \ge 0\}.
\ee
Hence, while we focus our attention on the distribution of the left-hand side of~\eqref{eqn:deq}, our results apply as well to the right. Results for negative $t$ are obtained by noting that,  
with $s=-t<0$, \eqref{eqn:deq}  gives the  distributional equality $\{X(\lambda;t): \lambda \ge 0\} \deq \{-X(\lambda;-t): \lambda \ge 0\}$.  
\end{remark}
 A nondecreasing process $\{Y(\lambda): \lambda \in[0,\infty)\}$ is a \textit{jump process} if, with probability one, for every interval $[a,b] \subseteq [0,\infty)$, $Y$ has finitely many points of increase in $[a,b]$. The process $\lambda \mapsto X(\lambda;t)$ is a jump process, as described in Theorem \ref{thm:Busemann jump process intro version} and illustrated in  Figure~\ref{fig:jump process for X}. 

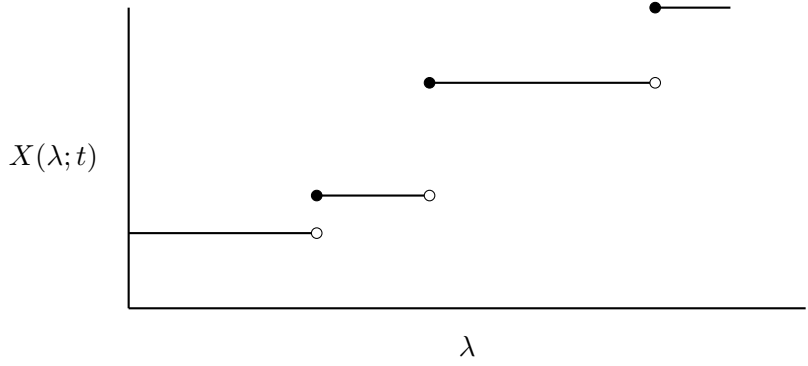
\begin{figure}[t]
    \centering
            \begin{tikzpicture}
            \draw[black,thick] (0,0)--(9,0);
            \draw[black,thick] (0,0)--(0,4);
            \draw[black,thick] (0,1)--(2.5,1);
            \draw[black,fill = black] (2.5,1.5) circle (2pt);
            \draw[black,thick] (2.55,1.5)--(4,1.5);
            \draw[black,fill = black] (4,3) circle (2pt);
            \draw[black,thick] (4.05,3)--(7,3);
            \draw[black,fill = black] (7,4) circle (2pt);
            \draw[black,thick] (7.05,4)--(8,4);
            \draw[black,fill = white] (2.5,1) circle (2pt);
            \draw[black,fill = white] (4,1.5) circle (2pt);
            \draw[black,fill = white] (7,3) circle (2pt);
            \node at (4.5,-0.5) {$\lambda$};
            \node at (-1,2) {$X(\lambda;t)$};
            \end{tikzpicture}
            \caption{\small A graphical description of the process $\lambda \mapsto X(\lambda;t)$ for a fixed $t>0$}
            \label{fig:jump process for X}
        \end{figure}

\begin{theorem} \label{thm:Busemann jump process intro version} 
Fix $t > 0$. Then, $\{ X(\lambda;t): \lambda\ge0\}$ is a nondecreasing real-valued process with stationary increments. With probability one, the path of the process is a step function whose jump locations are a discrete subset of $[0,\infty)$, there exists $\ve  >0$ such that $X(\lambda;t)=B_0(t)$ for $\lambda\in[0,\ve)$,
and $\lim_{\lambda\to\infty}X(\lambda;t)=\infty$.  The expected number of jumps   in an interval $[a,b] \subseteq [0,\infty)$ is given by
\[ \Ee\bigl[ \#\{\lambda\in[a,b]:  X(\lambda-;t) < X(\lambda+;t) \}\bigr] = 2(b - a) \sqrt{{t}/{\pi}}. \]
\end{theorem}
\begin{remark} In terms of jump directions of the Busemann process, 
the last statement of Theorem~\ref{thm:Busemann jump process intro version} is equivalent to the following: For $m \in \Z$, $0 < \gamma < \delta \le \infty$ and $s < t \in \R$, the expected number of directions $\theta \in (\gamma,\delta)$ satisfying $h_m^{\theta+}(s,t) < h_m^{\theta -}(s,t)$ is given  by 
\be\label{Th670} 
\Ee\bigl[ \#\bigl(\Busedc_{(m,s),(m,t)}\cap(\gamma,\delta)\bigr) \bigr] = 2\sqrt{\f{t - s}{\pi}} \Big(\f{1}{\sqrt \gamma} - \f{1}{\sqrt \delta}\Big).
\ee
\end{remark}

\noindent Theorem~\ref{thm:Busemann jump process intro version} is proved by first showing increment-stationarity and then analyzing the distribution of an increment of the process. By the increment-stationarity of Theorem~\ref{thm:Busemann jump process intro version}, the distribution of $X(\lambda_2;t) - X(\lambda_1;t)$ is the same as the distribution of $X(\lambda;t) - X(0;t)$, where $\lambda = \lambda_2 - \lambda_1$.  Denote the distribution function of this increment by
\[ F(z;\lambda,t) = \Pp(X(\lambda;t) - X(0;t) \le z) \qquad\text{for } z\ge 0. \]


\begin{theorem} \label{thm:Buse_inc}
For $z \ge 0$, $t > 0$, and $\lambda > 0$,
\begin{equation} \label{BuseCDF}
    F(z;\lambda,t) = \Phi\Bigl(\frac{z - \lambda t}{\sqrt{2 t}}\Bigr) + e^{\lambda z}\biggl( (1 + \lambda z + \lambda^2 t)\Phi\Bigl(-\frac{z + \lambda t}{\sqrt{2 t}}\Bigr) - \lambda \sqrt{{t}/{ \pi}\tspb}\,e^{-\frac{(z + \lambda t)^2}{4 t}}    \biggr).
    \end{equation}
\end{theorem}
\begin{remark}  
  Using~\eqref{BuseCDF}, the distribution of $X(\lambda;t) - X(0;t)$ can be written as a mixture of probability measures
    \[
    p \delta_0 + (1 - p) \pi,
    \]
    where $\delta_0$ is the point mass at the origin, 
    \be \label{eqn:0inc}
    p = F(0;\lambda,t)= \Pp(X(\lambda;t) - X(0;t)=0) = (2 + \lambda^2 t) \Phi\bigl(-\lambda\sqrt{{t}/{2}}\,\bigr) - \lambda e^{-\f{\lambda^2 t}{4}} \sqrt{{t}/{\pi}\,}
    \ee
    and $\pi$ is a continuous probability measure supported on $[0,\infty)$ with density \[
    (1 - p)^{-1}\Big[\dfrac{\partial}{\partial z}F(z;\lambda,t)\Big]\1(z > 0).
    \]
    Since $\lambda\mapsto X(\lambda;t) $ is nondecreasing, \eqref{eqn:0inc} implies that   $\lambda\mapsto F(0;\lambda,t)$ is nonincreasing. Further,  from~\eqref{eqn:0inc}, we can compute
    \be \label{eqn:Dlambda}
    \dfrac{\partial }{\partial \lambda} F(0;\lambda,t) = 2 \lambda t \Phi(-\lambda \sqrt{t/2}) - 2e^{-\f{\lambda^2 t}{4}}\sqrt{t/\pi}.
    \ee
    By Theorem 1.2.6 in~\cite{Durrett}, for all $y > 0$, $\int_{-\infty}^{-y} e^{-x^2/2}\,dx < y^{-1} e^{-y^2/2}$ (The theorem is stated with a weak inequality, but the proof shows that the equality is strict).  Applying this to~\eqref{eqn:Dlambda}, we see that for $t > 0$, $\lambda \mapsto F(0;\lambda,t)$ is strictly decreasing. 
    Hence, $p = F(0;\lambda,t) > 0$ for all $\lambda > 0$ and $t > 0$. From~\eqref{eqn:0inc}, for each $t > 0$,
    $\lim_{\lambda \rightarrow \infty} F(0;\lambda,t) = 0$. 
     
     The random variable $X(\lambda;t) - X(0;t)$ has the following Laplace transform/Moment generating function. For $\alpha \in \R_{\neq \lambda}$,
    \begin{align*}
&\Ee\Big[\exp\Big(-\alpha(X(\lambda;t) - X(0;t))\Big)\Big] \\
&=e^{\alpha^2 t - \alpha \lambda t}\Phi\Big((\lambda - 2\alpha)\sqrt{\f{t}{2}}\Big) \Big(1 - \f{\alpha^2}{(\lambda - \alpha)^2}\Big) \\
&\qquad\qquad\qquad+\Phi\Big(-\lambda \sqrt{\f{t}{2}}\Big)\Big(1 + \f{\alpha \lambda}{(\lambda - \alpha)^2} - \f{\alpha(1 + \lambda^2 t)}{\lambda - \alpha}\Big) +\f{\lambda \alpha}{(\lambda - \alpha)}\sqrt{\f{t}{\pi}} e^{-\lambda^2 t/4}.
\end{align*}
This is computed in Section~\ref{sec:global_Buse_proofs}.
\end{remark}

Recall that, for $s < t$, $h_m^{1/\lambda^2}(s,t)$ has the  $\mathcal N(\lambda (t - s), t - s)$ distribution and hence, by the monotonicity in Theorem \ref{thm:summary of properties of Busemanns for all theta}\ref{general monotonicity Busemanns}, $h_m^{1/\lambda^2}(s,t)\to +\infty$ as $\lambda\to\infty$. Part \ref{itm:convNor} below refines this statement. 

\begin{corollary} \label{cor:inc_dist_conseq}
The following hold. 
\begin{enumerate} [label=\rm(\roman{*}), ref=\rm(\roman{*})]  \itemsep=3pt
\item \label{itm:convNor} For fixed $s < t \in \R$, as $\lambda \rightarrow \infty$, $h_m^{1/\lambda^2}(s,t) - B_m(s,t) - \lambda (t- s)$ converges in distribution to a normal random variable with mean zero and variance $2(t - s)$. 
\item \label{itm:non_independent} For $t > 0$ and $0 \le \lambda_1 \le \lambda_2$, $X(\lambda_2;t) - X(\lambda_1;t)$ is not independent of $X(\lambda_1;t)$. Furthermore, the process $\lambda \mapsto X(\lambda;t)$ does not have independent increments.
\end{enumerate}
\end{corollary}

\begin{remark}
In addition to Corollary~\ref{cor:inc_dist_conseq}, numerical calculations give more information about the structure of this non-independence. Specifically, it appears that for $t > 0$ and $0 < \lambda_1 < \lambda_2$,
\[
\Pp\bigl(X(\lambda_2;t) = X(\lambda_1;t)\tspb\big\vert\tspb  X(\lambda_1;t) = X(0;t)\bigr) < \Pp\bigl(X(\lambda_2;t) = X(\lambda_1;t)\bigr).
\]
In other words, conditioning on  no jumps in the interval $[0,\lambda_1]$ increases the probability of a jump in $(\lambda_1,\lambda_2]$.
\end{remark}

\subsection{Coupled Brownian motions with drift}

On a fixed horizontal level $m$ of $\Z\times\R$, the Busemann functions  $h_m^{\theta \sig}(t)$ form an infinite family of coupled Brownian motions with drift.  
This section describes the structure of this family.

  We return to the parameter $\theta = \f{1}{\lambda^2}$   of the direction of semi-infinite geodesics. Recall that  $t \mapsto h_m^\theta(t)$ is a Brownian motion with drift $\f{1}{\sqrt \theta}$. It is convenient to extend the range of the parameter $\theta$ to infinity  by defining  $h_m^{\infty} = B_m$. By stationarity it is enough to consider the level $m=0$.  As pointed out in Remark~\ref{rmk:any_interval_same}, it is sufficient to restrict attention to nonnegative times $t\ge0$, and then  Theorem~\ref{thm:qualitative_Buse}  captures also the properties of the restarted process  $t \mapsto h_m^{\theta \sig}(t + s) - h_m^{\theta \sig}(s)$ for any fixed $(m,s) \in \Z \times \R$. Recall that $\Busedc$ is the set of discontinuities of the Busemann process defined in \eqref{eqn:Theta}. 


\begin{theorem} \label{thm:qualitative_Buse}
The following hold on a single event of probability one. 
\begin{enumerate} [label=\rm(\roman{*}), ref=\rm(\roman{*})]  \itemsep=3pt
    \item \label{itm:dist_incr} For  $0 < \gamma < \delta \le \infty$  and $\sigg_1,\sigg_2 \in \{-,+\}$, the difference  $h_0^{\gamma \sig_1}(t) - h_0^{\delta \sig_2}(t)$ between the two trajectories  is nonnegative and nondecreasing as a function of $t \ge 0$. For $\theta > 0$, the same is true of the difference $h_0^{\theta -}(t) - h_0^{\theta +}(t)$ as a function of $t \ge 0$.  
    \item \label{itm:stick_split} For  $0 < \gamma < \delta \le \infty$ and $\sigg_1,\sigg_2 \in \{-,+\}$, there exists a random time \\$S = S(\gamma\sigg_1,\delta\sigg_2) > 0$ such that $h_0^{\gamma \sig_1}(t) = h_0^{\delta \sig_2}(t)$ for $t \in [0,S]$, and $h_0^{\gamma \sig_1}(t) > h_0^{\delta \sig_2}(t)$ for $t > S$.
    \item \label{itm:theta_split_h} For every such value of $S = S(\gamma\sigg_1,\delta\sigg_2)$, there exists $\theta \in [\gamma,\delta]\cap \Busedc$ such that $h_0^{\theta -}(t) = h_0^{\theta +}(t)$ for $t \in [0,S]$, and $h_0^{\theta -}(t) > h_0^{\theta +}(t)$ for $t > S$. 
    \item \label{itm:count_trajectories}
    For each $T > 0$, the set of distinct trajectories  $\{t \mapsto h_m^{\theta \sig}(t): t\in[0,T],  \, \theta \in(0,\infty], \,\sigg \in \{-,+\}\}$ is countably infinite. 
    
    \item \label{itm:discrete_trajectories} At each fixed time $T > 0$, the set of values $\{h_m^{\theta \sig}(T): \theta \in (0,\infty],\sigg \in\{-,+\}\}$
    is a countably infinite subset of $\R$, bounded from below but unbounded from above, and has no limit points in $\R$. In particular, for every $\ve > 0$, there exists $\eta = \eta(\ve) > 0$ such that for all $0< \theta \le \eta$, $h_0^{\theta -}(\ve) \ge h_0^{\theta+}(\ve) > B_0(\ve)$.
\end{enumerate}
\end{theorem}

\begin{remark}
For  $0 < \gamma < \delta \le \infty$ the distribution of the separation time is given by 
\[  \Pp[ S(\gamma,\delta) > t] = \Pp[h_0^{\gamma}(t) = h_0^{\delta}(t)] =  F(0;\tf{1}{\sqrt \gamma} - \tf{1}{\sqrt \delta},t)  \quad \text{ for }  t>0,  \]   where $F(0;\lambda,t)$ is from~\eqref{eqn:0inc}.  There is no $\pm$ distinction because for a fixed   $\theta > 0$,  $h_m^{\theta +}(t) = h_m^{\theta -}(t)$ for all $t \in \R$ with probability one. 
 \end{remark}
 \begin{figure}[t]
    \centering
    \includegraphics[width = 5.5in]{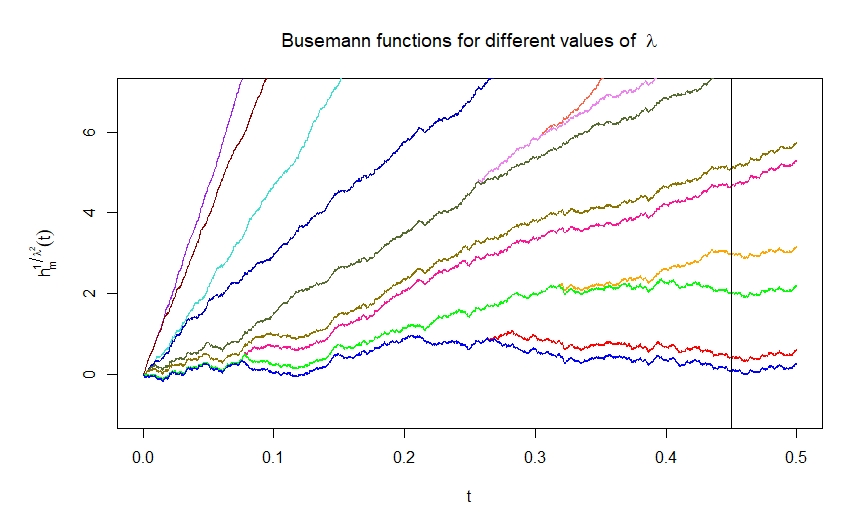}
    \caption{\small A simulation of the branching structure of the Busemann functions. Higher trajectories correspond to larger values of the drift $\lambda>0$, or equivalently, to smaller values of the direction parameter $\theta=\lambda^{-2}$ and thereby to geodesic directions approaching the vertical.}
    \label{fig:Buse_process_all}
\end{figure}
\begin{remark} \label{rmk:time_to_split}
Figure~\ref{fig:Buse_process_all} presents a simulation of the trajectories $\{h_0^\theta(t):t\ge 0\}$ for various values of the direction parameter 
$\theta=\lambda^{-2}$.
We see a visual representation of the statements of  Theorem~\ref{thm:qualitative_Buse}. The lowest (blue) trajectory  is the Brownian motion $h_0^{\infty} =B_0$ with direction $\theta=\infty$ and drift $\lambda = 0$. Trajectories move together from the origin and then split, and the distance between them is nondecreasing (Theorem~\ref{thm:qualitative_Buse}\ref{itm:dist_incr}--\ref{itm:stick_split}). Part~\ref{itm:theta_split_h} implies that when two trajectories split, there exists $\theta \in \Busedc$ such that $h_0^{\theta -}$ follows the upper trajectory and $h_0^{\theta+}$ follows the lower trajectory. We expect that three distinct trajectories do not split at the same time, but we do not have a proof and leave it as an open problem.

As one travels upward along the vertical line at $T = 0.45$ in the figure, one observes the process $\lambda \mapsto h_0^{(1/\lambda^2) \mp}(T) = X(\lambda \pm;T)$. Specifically, let $0<\lambda_1 <\lambda_2<\cdots$ be the jump times of this process. Then, for $0 \le \lambda < \lambda_1$, $X(\lambda \pm;T)$ is equal to the vertical coordinate of the bottom curve (blue). At $\lambda = \lambda_1$, $X(\lambda_1-;T)$ is still equal to the vertical coordinate of the bottom curve (blue), but $X(\lambda_1+;T)$ is equal to the vertical coordinate of the second-lowest curve (red). For $\lambda_1 < \lambda < \lambda_2$, $X(\lambda \pm;T)$ is equal to the vertical coordinate of the red curve, and so on. Lastly, in the figure, we see some trajectories splitting from $B_0$ very close to $t = 0$, as guaranteed by Part~\ref{itm:discrete_trajectories}. 
\end{remark}
\begin{remark}
Corollary 2 of Rogers and Pitman~\cite{Rogers-Pitman-81} describes another coupling of two  Brownian motions with drift
such that they agree for a finite amount of time.  Their result   is related to our work because it is used to show the stability of the Brownian queue (see, for example, page 289 in~\cite{brownian_queues}). However, the Rogers-Pitman coupling is different from ours, because, for example, theirs does not satisfy the monotonicity of increments given in Theorem~\ref{thm:qualitative_Buse}\ref{itm:dist_incr}.
\end{remark}

\section{Global geometry of geodesics and the competition interfaces} \label{section:main_geometry}
\subsection{Busemann geodesics} \label{section:SIG}

In addition to Theorems~\ref{thm:Busedc_class} and~\ref{thm:Haus_comp_interface}, the results of this section characterize uniqueness and coalescence of semi-infinite geodesics across all directions and initial points. These geometric properties are accessed through analytic and probabilistic  properties of the Busemann process. As in~\cite{Seppalainen-Sorensen-21a}, the following demonstrates how to construct semi-infinite geodesics from Busemann functions. 

\begin{definition} \label{def:semi-infinite geodesics}
For each initial point $(m,t) \in \Z\times\R$, direction $\theta>0$ and sign $\sig \in \{-,+\}$, let $\mbf T^{\theta\sig}_{(m,t)}$ denote the set of real sequences 
\[
t = \tau_{m - 1} \le \tau_m \le \tau_{m + 1}\le \cdots \le \tau_{r}\le \cdots
\]
that satisfy
\begin{equation} \label{sijump}
 B_{r}(\tau_{r})- \h_{r + 1}^{\theta \sig}(\tau_{r})  = \sup_{s \in [\tau_{r - 1},\infty)}\{B_r(s) - \h_{r + 1}^{\theta \sig}(s) \} \qquad\text{for each $r \ge m$.}  
\end{equation} 
 Theorem~\ref{thm:summary of properties of Busemanns for all theta}\ref{general continuity of Busemanns}--\ref{limits of B_m minus h m + 1} guarantees  that such sequences exist. Equality of two elements $(\tau_r)_{r\ge m-1}$ and $(\tau'_r)_{r\ge m-1}$ of $\mbf T_{\mbf x}^{\theta \sig}$ means that  $\tau_r=\tau'_r$ for all $r\ge m-1$. At each level $r$, there exist finite leftmost and rightmost maximizers.  Let 
\[
t = \tau_{(m,t),m - 1}^{\theta \sig,L} \le \tau_{(m,t),m}^{\theta\sig,L}  \le \cdots \qquad\text{and}\qquad t = \tau_{(m,t),m - 1}^{\theta \sig,R}\le \tau_{(m,t),m}^{\theta\sig,R}   \le \cdots
\]
denote the leftmost and rightmost sequences in $\mbf T_{(m,t)}^{\theta\sig}$. Since an increasing sequence of jump times  determines an infinite planar path,   as illustrated in   Figure~\ref{fig:BLPP_semi-infinite_geodesic},  we think of  $\mbf T_{\mbf x}^{\theta \sig}$ equivalently as the set of semi-infinite paths determined by its elements. For $S \in \{L,R\}$,   let $\Gamma_{(m,t)}^{\theta \sig,S}$ be the continuous  semi-infinite path on the plane defined by the jump times $\{\tau_{(m,t),r}^{\theta \sig,S}\}_{r\ge m-1}$. Finally, let 
$
\mbf T_{(m,t)}^{\theta} := \mbf T_{(m,t)}^{\theta +} \cup \mbf T_{(m,t)}^{\theta -} 
$
 denote the collection of all the sequences (or paths) associated to the direction parameter $\theta$. 
\end{definition}

\begin{figure}[t]
\begin{adjustbox}{max totalsize={5.5in}{5in},center}
\begin{tikzpicture}
\draw[gray,thin] (0.5,0) -- (15.5,0);
\draw[gray,thin] (0.5,0.5) --(15.5,0.5);
\draw[gray, thin] (0.5,1)--(15.5,1);
\draw[gray,thin] (0.5,1.5)--(15.5,1.5);
\draw[gray,thin] (0.5,2)--(15.5,2);
\draw[red, ultra thick,->] (1.5,0)--(4.5,0)--(4.5,0.5)--(7,0.5)--(7,1)--(9.5,1)--(9.5,1.5)--(13,1.5)--(13,2)--(15.5,2);
\filldraw[black] (1.5,0) circle (2pt) node[anchor = north] {$(m,t)$};
\node at (4.5,-0.5) {$\tau_{m}^\theta$};
\node at (7,-0.5) {$\tau_{m + 1}^\theta$};
\node at (9.5,-0.5) {$\tau_{m + 2}^\theta$};
\node at (13,-0.5) {$\cdots$};
\node at (0,0) {$m$};
\node at (0,0.5) {$m + 1$};
\node at (0,1) {$m + 2$};
\node at (0,1.5) {$\vdots$};
\end{tikzpicture}
\end{adjustbox}
\caption{\small Example of an element of $\mbf T_{(m,t)}^\theta$}
\label{fig:BLPP_semi-infinite_geodesic}
\end{figure}
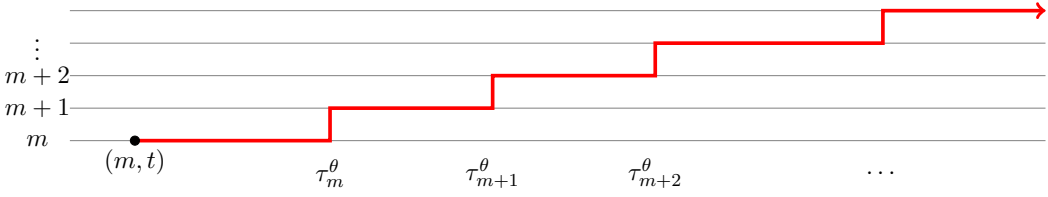

\begin{remark} \label{rmk:max_seq}
We make the observation that in Definition~\ref{def:semi-infinite geodesics}, if, at any step $r$, the function $B_r(u) - h_{r 
+ 1}^{\theta \sig}(u)$ has more than one maximizer over $u \in [\tau_{r-1},\infty)$, then any choice $\tau_{r}$  of maximizer continues the sequence as an element of $\mbf T_{\mbf x}^{\theta \sig}$, regardless of the past steps. In terms of paths, if $\Gamma \in  \mbf T_{\mbf x}^{\theta \sig}$, then for any point $\mbf y \in \Gamma\cap(\Z\times\R)$, the portion of $\Gamma$ above and to the right of $\mbf y$ is an element of $\mbf T_{\mbf y}^{\theta \sig}$. 
\end{remark}

 It was proved in~\cite{Seppalainen-Sorensen-21a} 
 that every element of $\mbf T_{(m,t)}^\theta$ is a $\theta$-directed  semi-infinite geodesic  out of $\mbf x=(m,t)$. We call these {\it Busemann geodesics}.   In general, $\Gamma_{(m,t)}^{\theta-,L}$ is the leftmost among {\it all} $\theta$-directed semi-infinite geodesics out of $(m,t)$ and $\Gamma_{(m,t)}^{\theta+,R}$ is the rightmost.
These properties, along with other previously proved facts, are recorded below.

\begin{theorem}[\cite{Seppalainen-Sorensen-21a}, Theorems 3.1(iv)--(v), 4.3 and 4.5(ii)] \label{existence of semi-infinite geodesics intro version}
The following hold on the full-probability event $\Omega_2$, unless stated otherwise.
 \begin{enumerate} [label=\rm(\roman{*}), ref=\rm(\roman{*})]  \itemsep=3pt 
     \item{\rm(}Existence{\rm)} \label{energy of path along semi-infinte geodesic} For all $\mbf x \in \Z \times \R$, $ \theta > 0$, and $\sigg \in \{-,+\}$, every element of $\mbf T_{\mbf x}^{\theta \sig}$ defines a semi-infinite geodesic starting from $\mbf x$. More specifically, for any two points $\mbf y \le \mbf z$ along a path in $\mbf T_{\mbf x}^{\theta\sig}$, the energy of this path between $\mbf y$ and $\mbf z$ equals $\B^{\theta\sig}(\mbf y,\mbf z)$, and this energy is maximal over all paths between $\mbf y$ and $\mbf z$. 
     \item{\rm(}Leftmost and rightmost finite geodesics along paths{\rm)} \label{Leftandrightmost} If, for some $\theta > 0$, $\sigg \in \{-,+\}$, and $\mbf x \in \Z \times \R$, the points $\mbf y \le \mbf z \in \Z \times \R$ both lie on $\Gamma_{\mbf x}^{\theta \sig,L}$, then the portion of $\Gamma_{\mbf x}^{\theta \sig,L}$ between $\mbf y$ and $\mbf z$ coincides with  the leftmost finite geodesic between these two points. Similarly, $\Gamma_{\mbf x}^{\theta \sig,R}$ is the rightmost geodesic between any two of its points.  
     \item{\rm(}Monotonicity{\rm)} \label{itm:monotonicity of semi-infinite jump times} The following inequalities hold.
     \begin{enumerate} [label=\rm(\alph{*}), ref=\rm(\alph{*})]  \itemsep=3pt 
         \item \label{itm:monotonicity in theta} For all $0 < \gamma < \theta$, $S \in \{L,R\}$,  $(m,t) \in \Z \times \R$, and $r \ge m$,
    \[
     t \le \tau_{(m,t),r}^{\gamma -,S} \le \tau_{(m,t),r}^{\gamma +,S} \le \tau_{(m,t),r}^{\theta -,S} \le \tau_{(m,t),r}^{\theta +,S}.
    \]
    \item\label{itm:monotonicity in t} For all $\theta > 0$, $m \le r \in \Z$, $s < t \in \R$, and $\sig \in \{-,+\}$, 
    \[
    \tau_{(m,s),r}^{\theta \sig,L} \le \tau_{(m,t),r}^{\theta\sig,L}\qquad \text{and} \qquad \tau_{(m,s),r}^{\theta \sig,R} \le \tau_{(m,t),r}^{\theta \sig, R}.
    \] 
     \item \label{itm:strong monotonicity in t} For $\theta > 0$, on the $\theta$-dependent full-probability event $\Omega^{(\theta)}$,
    for all pairs of  initial points $(m,s)$ and $(m,t)$ in $\Z\times\R$ that satisfy $s<t$,
      we have 
    \[
    \tau_{(m,s),r}^{\theta,R} \le \tau_{(m,t),r}^{\theta,L} \quad \text{ for all $r \ge m$.} 
    \]
     \end{enumerate}
    
    \item{\rm(}Convergence{\rm)} \label{itm:convergence of geodesics}
    For all $(m,t) \in \Z \times \R$, $r \ge m$, $\theta > 0$, $\sigg \in \{-,+\}$, and $S \in \{L,R\}$, the following limits hold. 
    \begin{align}
        \lim_{\gamma \nearrow \theta} \tau_{(m,t),r}^{\gamma \sig,L} = \tau_{(m,t),r}^{\theta -,L}\qquad&\text{and}\qquad \lim_{\delta \searrow \theta} \tau_{(m,t),r}^{\delta \sig,R} = \tau_{(m,t),r}^{\theta +,R}, \label{eqn:limits in theta} \\
        \lim_{\theta \searrow 0} \tau_{(m,t),r}^{\theta \sig,S} = t\qquad&\text{and}\qquad\lim_{\theta \rightarrow \infty} \tau_{(m,t),r}^{\theta \sig,S} = \infty, \label{eqn:limits in theta to infinity}  \\
        \lim_{s \nearrow t} \tau_{(m,s),r}^{\theta \sig,L} =  \tau_{(m,t),r}^{\theta \sig,L}\qquad&\text{and}\qquad \lim_{u \searrow t} \tau_{(m,u),r}^{\theta \sig,R} = \tau_{(m,t),r}^{\theta \sig,R}. \label{eqn:limits in t}
    \end{align}
    \item \label{general limits for semi-infinite geodesics} {\rm(}Directedness{\rm)} For all $\mbf x \in \Z \times \R$, $\theta > 0$, $\sigg \in \{-,+\}$, and all $\{\tau_r\}_{r \ge m} \in \mbf T_{\mbf x}^{\theta \sig}$, 
    \[
    \lim_{n\rightarrow \infty} \f{\tau_n}{n} = \theta. 
    \]
    \item{\rm(}General directedness{\rm)}  \label{itm:gen_directedness} All semi-infinite geodesics, whether they are Busemann geodesics or not, are $\theta$-directed for some $\theta \in [0,\infty]$. The only $0$- or $\infty$-directed semi-infinite geodesics are vertical and horizontal lines, respectively. 
    \item{\rm(}Control of semi-infinite geodesics{\rm)}  \label{itm:all semi-infinite geodesics lie between leftmost and rightmost} If, for some $\theta > 0$ and $(m,t) \in \Z \times \R$, any other geodesic {\rm(}constructed from the Busemann functions or not{\rm)} is defined by the sequence $t = t_{m - 1} \le t_m \le \cdots$, starts at $(m,t)$, and has direction $\theta$, then for all $r \ge m$,
    \[
    \tau_{(m,t),r}^{\theta -,L} \le t_r \le \tau_{(m,t),r}^{\theta +,R}.  
    \]
    \end{enumerate}
    \end{theorem}
    \begin{remark}
    Theorem~\ref{thm:geod_stronger} and Remark~\ref{rmk:incomparable} below strengthen Theorem~\ref{existence of semi-infinite geodesics intro version} in the following ways. 
    Part~\ref{itm:monotonicity of semi-infinite jump times}\ref{itm:monotonicity in theta} is clarified to show that, in general, $\tau_{(m,t)}^{\gamma \sig,R}$ and $\tau_{(m,t)}^{\theta \sig,L}$ are incomparable for $\gamma < \theta$. Part~\ref{itm:monotonicity of semi-infinite jump times}\ref{itm:strong monotonicity in t} is strengthened to an almost sure statement simultaneously over all directions. The limits in Equation~\eqref{eqn:limits in theta} are strengthened to allow us to interchange $L$ and $R$ in both statements. The limits in Equation~\eqref{eqn:limits in t} are strengthened to allow both $L$ and $R$ in the converging jump time.   
    This illustrates how knowledge of the joint distribution of the Busemann process leads to almost sure structural results for geodesics. 
    \end{remark}

 The following result is new to the present paper and strengthens the regularity properties of the geodesics as functions of the direction and the initial point. 






\begin{theorem} \label{thm:geod_stronger}
On a single event of full probability, the following hold. 
\begin{enumerate}[label=\rm(\roman{*}), ref=\rm(\roman{*})]  \itemsep=3pt
    \item  \label{itm:stronger convergence theta} For each $n \ge m$ in $\Z$, $\theta > 0$, and compact subsets $K \subseteq \R$, there exists a random $\ve = \ve(m,n,K,\theta) > 0$ such that, whenever $t \in K$, $\theta - \ve < \gamma < \theta < \delta < \theta + \ve$, $m \le r \le n$, $\sigg \in\{-,+\}$, and $S \in \{L,R\}$,
    \be\label{tau134}
    \tau_{(m,t),r}^{\gamma\sig,S} = \tau_{(m,t),r}^{\theta -,S}\qquad\text{and}\qquad \tau_{(m,t),r}^{\delta\sig,S} = \tau_{(m,t),r}^{\theta +,S}.
    \ee
    \item \label{itm:global strong monotonicity in t}
    For each $\theta  > 0$, $\sigg \in \{-,+\}$, $m \le r \in \Z$ and $s < t \in \R$,
    $
    \tau_{(m,s),r}^{\theta \sig,R} \le \tau_{(m,t),r}^{\theta \sig,L}. 
    $
    \item \label{itm:stronger convergence in t} For each $\theta > 0$, $\sigg \in \{-,+\}$, $m \le r\in \Z$, $s \in \R$, and $S \in \{L,R\}$,
    \[
    \lim_{u \nearrow s} \tau_{(m,u),r}^{\theta \sig,S} = \tau_{(m,s),r}^{\theta \sig,L},\qquad\text{and}\qquad \lim_{t \searrow s} \tau_{(m,t),r}^{\theta \sig,S} = \tau_{(m,s),r}^{\theta \sig,R}.
    \]
\end{enumerate}
\end{theorem}


\begin{remark} \label{rmk:incomparable}
In general, the inequalities of Theorem~\ref{existence of semi-infinite geodesics intro version}\ref{itm:monotonicity of semi-infinite jump times}\ref{itm:monotonicity in theta} and the Equalities of~\eqref{tau134} cannot be extended to  mix $L$ with $R$. In fact, for every  $(m,t) \in \NU_1$, defined below, and $\ve  > 0$ there exist $\theta - \ve < \gamma < \theta$ such that 
\be \label{eqn:pm_counterexample}
 \tau_{(m,t),m}^{\theta -,L} = \tau_{(m,t),m}^{\theta +,L} < \tau_{(m,t),m}^{\gamma -,R} = \tau_{(m,t),m}^{\gamma +,R}.
\ee
To show this, choose $(m,t) \in \NU_1$ and $\ve > 0$. As noted in Remark~\ref{rmk:dense_theta}, there exists an event of probability one, on which, for all $\theta \in \Q_{>0}$, the $\pm$ distinction is not present. By Theorem~\ref{thm:NU}, there exists $\theta \in \Q_{>0}$ such that $\tau_{(m,t),m}^{\theta,L} < \tau_{(m,t),m}^{\theta,R}$. By Theorem~\ref{thm:geod_stronger}\ref{itm:stronger convergence theta}, there exists $\gamma  \in (\theta - \ve,\theta) \cap \Q_{>0}$ such that $\tau_{(m,t),m}^{\gamma,R} = \tau_{(m,t),m}^{\theta ,R}$, and~\eqref{eqn:pm_counterexample} holds.
\end{remark}

\subsection{Non-uniqueness and coalescence of geodesics} \label{sec:non_unique_coal}



In contrast with lattice LPP with continuous weights, in BLPP every direction has random exceptional initial points from which the directed geodesic is not unique. Let $\NU_0^{\theta \sig}$ be the set of space-time points from which emanate at least two semi-infinite Busemann geodesics with the same direction $\theta$ and sign $\sigg$. Let $\NU_1^{\theta \sig}$ be the subset of $\NU_0^{\theta \sig}$ of those points out of which  two  $\theta\sigg$ geodesics separate  on the first level.
    Precisely, 
 \begin{align} 
    \NU_0^{\theta \sig} &=\{(m,t) \in \Z \times \R: \tau_{(m,t),r}^{\theta \sig, L} < \tau_{(m,t),r}^{\theta \sig,R} \text{ for some }r \ge m \}, \text{ and } \label{NU_0 theta def} \\
    \NU_1^{\theta \sig} &= \{(m,t) \in \NU_0^{\theta \sig}: \tau_{(m,t),m}^{\theta \sig ,L} < \tau_{(m,t),m}^{\theta \sig,R} \}. \label{NU_1 theta def}
\end{align}
Define their unions over directions and signs as 
\begin{align}
\label{NU_0_1} \NU_0 &= \bigcup_{\theta > 0, \,\sigg \in \{-,+\}} \NU_0^{\theta \sig} \qquad\text{and}\qquad 
\NU_1 = \bigcup_{\theta > 0,\, \sigg \in \{-,+\}} \NU_1^{\theta \sig}.
\end{align}
The reason for singling out the subset $\NU_1$ of $\NU_0$ becomes evident later in Theorem \ref{thm:ci_and_NU}, where membership in $\NU_1$  connects with the behavior of the competition interface.

 \begin{remark} \label{rmk:NU_Sets} The role of the sets $\NU_0^{\theta\sig}$ in the  non-uniqueness of  $\theta$-directed geodesics is somewhat subtle.  It depends on the direction $\theta$ and on whether we take the $\theta$-specific view or the global view, that is,  the choice of the  full-probability event on which we view the situation.  
 
 (a) The crudest situation is that we fix $\theta$ and work on the event $\Omega^{(\theta)}$ of Theorem~\ref{thm:summary of properties of Busemanns for all theta}\ref{busemann functions agree for fixed theta}.    On this event,  the $\pm$ distinction is not visible, and we can drop the sign and  write $\NU_0^{\theta}=\NU_0^{\theta\pm}$. Now, on the full probability event  $\Omega^{(\theta)} \cap \Omega_2$, the set $\NU_0^\theta$ is {\it exactly} the set of initial points $\mbf x \in \Z \times \R$ out of which the $\theta$-directed geodesic is not unique. This follows because when there is no $\pm$ in Theorem~\ref{existence of semi-infinite geodesics intro version}\ref{itm:all semi-infinite geodesics lie between leftmost and rightmost},  non-uniqueness out of $\mbf x$ in direction $\theta$  happens iff the $L$ and $R$ geodesics in $\mbf T^{\theta}_{\mbf x}$ do not agree. 
 
 (b) If we want a global view, that is, a consideration of all directions simultaneously on a single full-probability event, then we must consider the 
  random set   $\Busedc$ of Busemann function jump points, defined in~\eqref{eqn:Theta}. The set $\Busedc$ is countably infinite (Theorem \ref{thm:Theta properties}\ref{itm:Theta_count_intro}).   If $\theta\notin\Busedc$, then  the $\pm$ distinction is not present and  the situation is as in point (a).  However, if $\theta\in\Busedc$, then out of \textit{every} $\mbf x \in \Z \times \R$ there is more than one $\theta$-directed semi-infinite geodesic. (See Theorem~\ref{thm:geod_description} and Remark \ref{rmk:splitting_coalesceing_general}.) Yet, the sets $\NU_0^{\theta -}$ and $\NU_0^{\theta + }$ are only countably infinite (Theorem \ref{thm:NU}).  Thus, for $\theta\in\Busedc$ the  union $\NU_0^{\theta -} \cup \NU_0^{\theta + }$ dramatically fails to catch all the initial points $\mbf x$ out of which the   $\theta$-directed geodesic is not unique. The reason is that $\NU_0^{\theta -} \cup \NU_0^{\theta + }$ 
  accounts only for the $L/R$ distinction of geodesics and not the $\pm$ distinction. 

 \end{remark}

The following theorem describes the sets $\NU_0$ and $\NU_1$. The countable unions in \eqref{decomp} below are technically crucial because they allow us to rule out left/right nonuniqueness from rational initial  points $\mbf x$ simultaneously  in {\it all} directions without accumulating uncountably many zero-probability events. Recall that $\NU_1^{\theta \sig} \subseteq \NU_0^{\theta \sig}$.

\begin{theorem} \label{thm:NU} The following hold on a single event of full probability. 
\begin{enumerate}[label=\rm(\roman{*}), ref=\rm(\roman{*})]  \itemsep=3pt
    \item \label{itm:allcount} For every $\theta > 0$ and $\sigg \in \{-,+\}$, the sets $\NU_0^{\theta \sig}$ and $\NU_1^{\theta \sig}$ are both countably infinite.
    \item \label{itm:union} The sets $\NU_0$ and $\NU_1$ are both countably infinite. Specifically,  
    \be \label{decomp}
    \NU_0 = \bigcup_{\theta \in \Q_{>0}} \NU_0^\theta\qquad \text{ and }\qquad\NU_1 = \bigcup_{\theta \in \Q_{> 0}} \NU_1^\theta.
    \ee
    \item \label{itm:NUp0} 
    For each $\mbf x \in \Z \times \R$, $\Pp(\mbf x \in \NU_0) = 0$. In particular, the full probability event of the theorem is constructed so that, for all $\mbf x\in\Z\times\Q$, $\mbf x \notin \NU_0$, or in other words, $\mbf T^{\theta\sig}_{\mbf x}$ contains a single sequence for all $\theta > 0$ and $\sigg \in \{-,+\}$. 
    
    \item \label{itm:only_endpoint}   The sets $\NU_0^{\theta \sig}$ and $\NU_1^{\theta \sig}$ can be described as follows, 
    \begin{align*}
    \NU_0^{\theta \sig} &=\{(m,t) \in \Z \times \R: t = \tau_{(m,t),r}^{\theta \sig, L} < \tau_{(m,t),r}^{\theta \sig,R} \text{ for some }r \ge m \}, \text{ and }  \\
    \NU_1^{\theta \sig} &= \{(m,t) \in \NU_0^{\theta \sig}:  t = \tau_{(m,t),m}^{\theta \sig ,L} < \tau_{(m,t),m}^{\theta \sig,R} \}.
    \end{align*}
    In other words, Busemann geodesics emanating from $(m,t)$ can separate only along the upward vertical ray from  $(m,t)$. 
\end{enumerate}
\end{theorem}

\begin{remark}
Theorems~\ref{thm:ci_and_NU} and~\ref{thm:ci_and_Buse_directions} later in the section give more details about the nature of the sets $\NU_0$ and $\NU_1$ and relate them to the geometry of semi-infinite geodesics. Theorem~\ref{thm:ci_and_NU}\ref{itm:ci_int} gives some intuition into why we can write $\NU_1$ as a union over just a dense set of directions: When $(m,s) \in \NU_1$, then $(m,s) \in \NU_1^{\theta \sig}$ for all $\theta$ in a nonempty interval. 
\end{remark}

\begin{remark} \label{rmk:number_sIG} We spell out the geometric consequences of 
Theorem~\ref{thm:NU}, combined with some other facts. Refer to Figure~\ref{fig:multiple_sig}. Consider the set $\mbf T_{(m,t)}^{\theta \sig}$ of  semi-infinite $\theta\sigg$ Busemann geodesics out of the point $(m,t)$. Each of them is $\theta$-directed by Theorem~\ref{existence of semi-infinite geodesics intro version}\ref{general limits for semi-infinite geodesics}, and hence must eventually exit the vertical ray $\{(x,t):x \ge m\}$ that emanates from $(m,t)$.  
In particular, this is true of the  leftmost geodesic   $\Gamma_{(m,t)}^{\theta \sig,L}$ and consequently there is a level $r$ such that no $\theta\sigg$ Busemann geodesic out of $(m,t)$ goes through the point $(k,t)$ for $k \ge r$.

By Part~\ref{itm:NUp0} of Theorem~\ref{thm:NU}, there exists an event of full probability on which $\mbf T_{(r,q)}^{\theta \sig}$ contains a single element for each $(r,q) \in \Z \times \Q,\theta > 0$, and $\sigg \in \{-,+\}$. On this event, for each $r \ge m$, at most one $\mbf T_{(m,t)}^{\theta \sig}$ geodesic can  move to the right from the point $(r,t)$. Otherwise, two such geodesics pass through $(r,q)$ for some rational $q >t$ and produce   two   geodesics in the set $\mbf T_{(r,q)}^{\theta \sig}$, a contradiction. In other words, the $\mbf T_{(m,t)}^{\theta \sig}$  geodesics branch one by one from  the upward vertical ray emanating from the initial point $(m,t)$. 

One conclusion of the above is that 
  $\mbf T_{(m,t)}^{\theta \sig}$ is a finite set. By Theorem \ref{thm:general_coalescence}\ref{itm:split_return}, from some point onwards, all $\mbf T_{(m,t)}^{\theta \sig}$ geodesics are back together.

 Three interrelated open problems are left in this situation. (i) Do there exist  initial   points $(m,t)$ with more than two elements in $\mbf T_{(m,t)}^{\theta \sig}$?   (ii) Is branching at the first level the {\it only} possibility when $\mbf T_{(m,t)}^{\theta \sig}$ is not a singleton?  If the answer is negative, further questions arise. (iii) Is the difference $\NU_0\setminus\NU_1$ empty, or equivalently, does any branching from $(m,t)$  imply   branching already at level $m$?   If $\NU_0\setminus\NU_1\ne\varnothing$, how large is it?    
\end{remark}

\begin{figure}[t]
 \centering
            \begin{tikzpicture}
            \draw[gray,thin] (0,0)--(10,0);
            \draw[gray,thin] (0,1)--(10,1);
            \draw[gray,thin] (0,2)--(10,2);
            \draw[gray,thin] (0,3)--(10,3);
            \draw[gray,thin] (0,4)--(10,4);
            \draw[gray,thin] (0,5)--(10,5);
            \draw[red,ultra thick] plot coordinates {(0,0)(3.2,0)(3.2,1)};
            \draw[red,ultra thick] plot coordinates {(0,1)(6,1)(6,2)(7,2)(7,3)(7.5,3)(7.5,4)(9,4)(9,5)};
            \draw[red,ultra thick] plot coordinates {(0,2)(5,2)(5,3)(6,3)(7,3)};
            \draw[red,ultra thick,->] plot coordinates {(0,0)(0,4)(3,4)(3,5)(10,5)};
            \filldraw[black] (0,0) circle (2pt) node[anchor = north] {$\mbf x$};
            \filldraw[black] (9,5) circle (2pt) node[anchor = south] {$\mbf z$};
            \end{tikzpicture}
            \caption{\small In this example, out of the point $\mbf x$, there are multiple elements of $\mbf T_{\mbf x}^{\theta \boxempty}$, i.e., multiple Busemann geodesics from $\mbf x$ with the same direction and sign. The geodesics all split from each other at the vertical line containing the initial point. The point $\mbf x$ lies in the set $\NU_1$ (and thus also $\NU_0$) because there exist geodesics that split immediately at the initial point. Furthermore, although there are multiple $\theta \boxempty$ geodesics from the same point, they eventually come back together and agree from some point $\mbf z$ onwards by Theorem~\ref{thm:general_coalescence}\ref{itm:split_return}. We know that there exist points $\mbf x \in \Z \times \R$ with two elements of $\mbf T_{\mbf x}^{\theta \boxempty}$, but whether there are points with more than two such geodesics, as in the figure, is an open problem. }
            \label{fig:multiple_sig}
            \bigskip
        \end{figure}
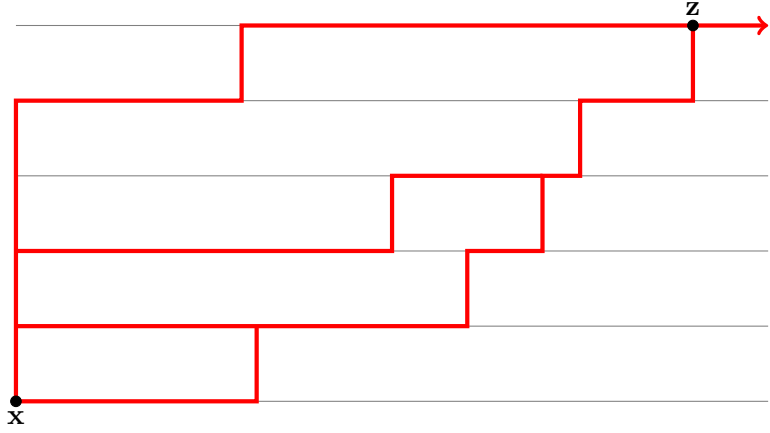
        \begin{figure}[t]
    \centering
    \begin{tikzpicture}
    \draw[black,thick] (-0.5,0)--(7,0);
    \draw[black,thick] (0,-0.5)--(0,5);
    \filldraw[black](2,3) circle (2pt) node[anchor = west] {$\mbf y = (r,s)$};
    \filldraw[black] (5,3) circle (2pt) node[anchor = west] {$\mbf z = (r,t)$};
    \filldraw[black] (2,1) circle (2pt) node[anchor = west] {$\mbf w = (m,s)$};
    \filldraw[black] (5,1) circle (2pt) node[anchor = west] {$\mbf x = (m,t)$};
    \node at (2,-0.25) {$s$};
    \node at (5,-0.25) {$t$};
    \node at (-0.5,1) {$m$};
    \node at (-0.5,3) {$r$};
    \end{tikzpicture}
    \caption{\small In this figure, we have $\mbf y \Ne \mbf x$, $\mbf y \Ne \mbf z$, and $\mbf w \Ne \mbf x$. We also have $\mbf y \NeN \mbf w$ but $\mbf y \not \Ne \mbf w$ and $\mbf z \NeN \mbf x$ but $\mbf z \not \Ne \mbf x$. The points $\mbf w$ and $\mbf z$ satisfy the coordinatewise ordering $\mbf w\le\mbf z$, but they are incomparable under the ordering $\SeS$. }
    \label{fig:initial_point_config}
\end{figure}

\noindent The next theorem gives coalescence of all Busemann geodesics with the same direction $\theta > 0$ and sign $\sigg \in \{-,+\}$.  Recall the southeast ordering $\mbf x \SeS \mbf y$ from Section~\ref{section:notation}, referring to  Figure~\ref{fig:initial_point_config} for clarity.

\begin{theorem} \label{thm:general_coalescence}
The following hold on a single event of full probability, simultaneously for all directions  $\theta > 0$ and signs $\sigg \in \{-,+\}$. 
\begin{enumerate} [label=\rm(\roman{*}), ref=\rm(\roman{*})]  \itemsep=3pt
    \item \label{itm:allsignscoalesce} 
    Whenever $\mbf x,\mbf y \in \Z\times \R$, any two geodesics   $\Gamma_1 \in \mbf T_{\mbf x}^{\theta \sig}$ and $\Gamma_2 \in \mbf T_{\mbf y}^{\theta \sig}$  coalesce. If $\mbf x \Se \mbf y$, then the minimal point of intersection is the coalescence point. 
    \item \label{itm:return_point} 
    If $\Gamma_1,\Gamma_2 \in \mbf T_{(m,t)}^{\theta \sig}$ are distinct, their coalescence point is the minimal point of the set 
$ 
    (\Gamma_1 \cap \Gamma_2) \setminus \{(x,t):x \in [m,\infty)\}.
   $ 
    In other words, the coalescence point of $\Gamma_1$ and $\Gamma_2$ is the first point of intersection {\rm after} these  geodesics split.
    \item \label{itm:split_return} 
    For each $(m,t) \in \Z \times \R$, there exists a level $r \ge m$ and a sequence $s_r \le s_{r + 1} \le \cdots$ such that  every sequence $\{\tau_k\}_{k = m - 1}^\infty \in \mbf T_{(m,t)}^{\theta \sig}$ satisfies  $\tau_k = s_k$ for $k \ge r$. Equivalently, there exists a point $\mbf z \ge \mbf (m,t)$ and a semi-infinite geodesic $\Gamma_0 \in \mbf T_{\mbf z}^{\theta \sig}$ such that all geodesics in $\mbf T_{(m,t)}^{\theta \sig}$ agree with $\Gamma_0$ above and to the right of $\mbf z$.  See Figure~\ref{fig:multiple_sig}.
\end{enumerate}
\end{theorem}

\begin{remark}
A consequence of Theorem~\ref{thm:general_coalescence} is that the non-uniqueness of semi-infinite geodesics captured by the $L/R$ distinction is temporary, and does not separate the geodesics all the way to $\infty$. That is, while there may be two geodesics in $\mbf T_{\mbf x}^{\theta \sig}$ that separate, they must come back together, as in Figure~\ref{fig:multiple_sig}.  By contrast, in the $\pm$ distinction, geodesics with the same direction split and never come back together. This is explained further in  Remarks~\ref{rmk:splitting_coalesceing_general}, \ref{rmk:split}, and~\ref{rmk:CI not NU 0}.
\end{remark}

\begin{remark}
In~\cite{Seppalainen-Sorensen-21a}, we proved that, for a fixed direction $\theta > 0$ with probability one, all $\theta$-directed geodesics (whether constructed by the Busemann functions or not) coalesce. This is recorded in the current paper as Theorem~\ref{thm:fixedcoal}. This theorem was proven by defining southwest semi-infinite geodesics in a dual environment. It was shown that, if two geodesics with the same direction $\theta$ are disjoint, there must exist a bi-infinite geodesic whose northwest and southeast directions are $\theta$. Then, it was shown that, for fixed northeast and southwest directions, there are almost surely no bi-infinite geodesics in those directions (Theorem 3.1(vi) in~\cite{Seppalainen-Sorensen-21a}). Theorem~\ref{thm:general_coalescence}\ref{itm:allsignscoalesce} does not rely on the result from~\cite{Seppalainen-Sorensen-21a} and provides a new method of proof. Theorem~\ref{thm:geod_description}\ref{allBuse} below states that for all $\theta \notin \Busedc$, all $\theta$-directed semi-infinite geodesics are Busemann geodesics, and Theorem~\ref{thm:Theta properties} states that $\Pp(\theta \in \Busedc) = 0$ for all $\theta > 0$. Therefore, Theorem~\ref{thm:fixedcoal} follows as a special case of Theorem~\ref{thm:general_coalescence}\ref{itm:allsignscoalesce}.
\end{remark}

\begin{figure}[t]
 \centering
            \begin{tikzpicture}
            \draw[gray,thin,dashed] (0,0)--(7,0);
            \draw[gray,thin,dashed] (0,1)--(7,1);
            \draw[gray,thin,dashed] (0,2)--(7,2);
            \draw[gray,thin,dashed] (0,3)--(7,3);
            \draw[gray,thin,dashed] (0,4)--(7,4);
            \draw[gray,thin,dashed] (0,5)--(7,5);
            \draw[red,ultra thick,->] plot coordinates {(0,0)(0,4)(3,4)(3,5)(7,5)};
            \draw[blue, thick] plot coordinates {(0,3)(4,3)(4,4)(5,4)(5,5)};
            \filldraw[black] (0,0) circle (2pt) node[anchor = north] {$\mbf x$};
            \filldraw[black] (0,3) circle (2pt) node[anchor = east] {$\mbf y$};
            \filldraw[black] (5,5) circle (2pt) node[anchor = south] {$\mbf w$};
            \end{tikzpicture}
            \caption{\small The red/thick path is $\Gamma_{\mbf x}^{\theta \boxempty,L}$ and 
            the blue/thin path is $\Gamma_{\mbf y}^{\theta \boxempty,R}$}
            \label{fig:meet_and_split}
            \bigskip
        \end{figure}
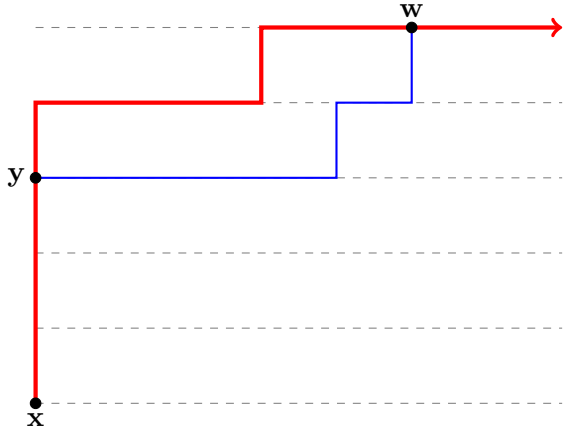

\begin{remark} \label{rmk:coal_pt_not_minpt}
 Without the strict ordering $\mbf x \Se \mbf y$   in Theorem~\ref{thm:general_coalescence}\ref{itm:allsignscoalesce}, the intersection point of two geodesics is not known to be the same as the coalescence point.   Whether the following occurs with positive probability  for a random  weakly ordered pair  $\mbf x = (m,s) \SeS \mbf y = (r,s)$ is left as an open problem: First, $\Gamma_{\mbf x}^{\theta \sig,L}$ moves vertically  from $(m,s)$ to $(r,s)$ and meets $\Gamma_{\mbf y}^{\theta \sig,R}$. After that, if $\Gamma_{\mbf x}^{\theta \sig,L}$ makes a vertical step to $(r + 1,s)$, but $\Gamma_{\mbf y}^{\theta \sig,R}$ makes a horizontal step, then the geodesics become separated. In this case  $\mbf z^{\theta \sig}(\mbf x,\mbf y) = \mbf y$ is the minimal point of intersection but not the coalescence point $\mbf w$, as illustrated in
 Figure~\ref{fig:meet_and_split}. 
\end{remark}

\subsection{Geodesics and the discontinuities of the Busemann process}
\label{sec:geometry_sub}
Recall the   discontinuity sets $\Busedc_{\mbf x,\mbf y}$ and $\Busedc$ defined in~\eqref{eqn:Theta}.


 \medskip 

Suppose that, for some initial points $\mbf x$ and $\mbf y$, signs $\sigg_1,\sigg_2 \in \{-,+\}$, and directions  $\gamma < \theta$, two geodesics $\Gamma_1 \in \mbf T_{\mbf x}^{\gamma \sig_1}$ and $\Gamma_2 \in \mbf T_{\mbf y}^{\gamma \sig_1}$ coalesce at some point $\mbf z$, and two other geodesics $\Gamma_3 \in \mbf T_{\mbf x}^{\theta \sig_2}$ and $\Gamma_4 \in \mbf T_{\mbf y}^{\theta \sig_2}$ also coalesce at that same point $\mbf z$. Then by additivity of the Busemann functions and Theorem~\ref{existence of semi-infinite geodesics intro version}\ref{energy of path along semi-infinte geodesic},
\be \label{coaldiff}
\B^{\gamma \sig_1}(\mbf x,\mbf y) = \B^{\gamma \sig_1}(\mbf x,\mbf z) - \B^{\gamma \sig_1}(\mbf y,\mbf z) = L_{\mbf x,\mbf z} - L_{\mbf y,\mbf z} = \B^{\theta \sig_2}(\mbf x,\mbf y).
\ee
In light of Theorem~\ref{thm:Theta properties}\ref{itm:const}, it is natural to ask whether the converse holds: that is, whether $\B^{\theta \sig_2}(\mbf x,\mbf y) = \B^{\gamma \sig_1}(\mbf x,\mbf y)$ implies that the $\theta$- and $\gamma$-directed geodesics out of $\mbf x$ and $\mbf y$  share a common coalescence point. The answer is affirmative for lattice LPP with continuous weights  (see section 3 of \cite{Janjigian-Rassoul-Seppalainen-19}). We show in Remark~\ref{rmk:coalcounter} that this does {\it not} hold in general for BLPP. However, such a statement does hold when restricted to certain configurations of initial points and when the $L/R$ type of the geodesics is specified appropriately. This is manifested in the following definition.

\begin{definition} \label{def:coalpt}
For $\mbf x \SeS \mbf y$ and $\theta > 0$, we define $\mbf z^{\theta \sig}(\mbf x,\mbf y) \in \Z \times \R$ to be the minimal point of intersection of the semi-infinite geodesics $\Gamma_{\mbf x}^{\theta \sig,L}$ and $\Gamma_{\mbf y}^{\theta \sig,R}$. 
\end{definition}
\begin{remark}
By Theorem~\ref{thm:general_coalescence}\ref{itm:allsignscoalesce}, $\mbf z^{\theta \sig}(\mbf x,\mbf y)$ is well-defined, and if $\mbf x$ and $\mbf y$ satisfy the strict ordering $\mbf x \Se \mbf y$, then $\mbf z^{\theta \sig}(\mbf x,\mbf y)$ is also the coalescence point of $\Gamma_{\mbf x}^{\theta \sig,L}$ and $\Gamma_{\mbf y}^{\theta \sig,R}$. In the general case $\mbf x \SeS \mbf y$, the definition of $\mbf z^{\theta \sig}(\mbf x,\mbf y)$ as the minimal intersection point (instead of the coalescence point) is required for the following theorems. Also, the condition $\mbf x \SeS \mbf y$ and the $L/R$ distinctions are essential. 
\end{remark}

\begin{theorem} \label{thm:common_coalesce}
On a single event of probability one, whenever $\gamma < \delta$, and $\mbf x \SeS \mbf y \in \Z \times \R$, the following are equivalent.
\begin{enumerate}[label=\rm(\roman{*}), ref=\rm(\roman{*})]  \itemsep=3pt
    \item \label{itm:sameBuse}$\B^{\gamma +}(\mbf x,\mbf y) = \B^{\delta-}(\mbf x,\mbf y)$
    \item \label{itm:samecoal}$\mbf z^{\gamma +}(\mbf x,\mbf y) = \mbf z^{\delta -}(\mbf x,\mbf y)$.
    \item \label{itm:allsamepath} There exists $\mbf z \in \Z \times \R$ and finite-length, upright paths $\Gamma_1$ {\rm(}connecting $\mbf x$ and $\mbf z${\rm)} and $\Gamma_2$ {\rm(}connecting $\mbf y$ and $\mbf z${\rm)} such that for all $\theta \in (\gamma,\delta)$ and $\sigg \in \{-,+\}$,  $\Gamma_1$ agrees with $\Gamma_{\mbf x}^{\theta \sig,L}$, $\Gamma_{\mbf x}^{\gamma+,L}$ and $\Gamma_{\mbf x}^{\delta-,L}$ from $\mbf x$ to $\mbf z$, and $\Gamma_2$ agrees with  $\Gamma_{\mbf y}^{\theta \sig,R}$, $\Gamma_{\mbf y}^{\gamma+,R}$, and $\Gamma_{\mbf y}^{\delta-,R}$ from $\mbf y$ to $\mbf z$. The paths $\Gamma_1$ and $\Gamma_2$ are disjoint before they reach the point $\mbf z$. 
\end{enumerate}
\end{theorem}
\begin{remark}
If we remove condition~\ref{itm:allsamepath} from Theorem~\ref{thm:common_coalesce}, then we get a more general equivalence where we do not need to assume the signs $-$ and $+$ for the $\delta$ and $\gamma$ geodesics, respectively. See Theorem~\ref{thm:equiv_coal_pt}. 
\end{remark}

\begin{remark} \label{rmk:Buse_mont}
Set $\mbf x = (m,t)$, $\mbf y = (r,s)$ and $\mbf w = (r,t)$. The assumption $\mbf x \SeS \mbf y$ requires $r \ge m$ and $s \le t$. By the additivity of Theorem~\ref{thm:summary of properties of Busemanns for all theta}\ref{general additivity Busemanns}, we may write 
\be \label{hvdecomp}
\B^{\theta \sig}(\mbf x,\mbf y) = \B^{\theta \sig}(\mbf x,\mbf w) + \B^{\theta \sig}(\mbf w,\mbf y) = \sum_{k = m + 1}^{r} v_k^{\theta \sig}(t)  -h_r^{\theta \sig}(s,t).
\ee
By the monotonicity of Theorem~\ref{thm:summary of properties of Busemanns for all theta}\ref{general monotonicity Busemanns}, the process $\theta \mapsto \B^{\theta \sig}(\mbf x,\mbf y)$ is nondecreasing whenever $\mbf x \SeS \mbf y$. Hence, Condition~\ref{itm:sameBuse} of Theorem~\ref{thm:common_coalesce} is equivalent to the stronger statement that, for $\theta \in (\gamma,\delta)$ and $\sigg \in \{-,+\}$,
\[
\B^{\theta \sig}(\mbf x,\mbf y) = \B^{\gamma +}(\mbf x,\mbf y) = \B^{\delta-}(\mbf x,\mbf y).
\]
\end{remark}

\begin{theorem} \label{thm:thetanotsupp}
On a single event of full probability, whenever $\theta > 0$ and $\mbf x \SeS \mbf y \in \Z \times \R$, the following are equivalent. 
\begin{enumerate}[label=\rm(\roman{*}), ref=\rm(\roman{*})]  \itemsep=3pt
    \item \label{itm:pmBuse}$\B^{\theta -}(\mbf x,\mbf y) = \B^{\theta +}(\mbf x,\mbf y)$
    \item \label{itm:pmcoal}$\mbf z^{\theta -}(\mbf x,\mbf y) = \mbf z^{\theta+}(\mbf x,\mbf y)$.
    \item \label{itm:disjointgeod} $\Gamma_{\mbf x}^{\theta+,L} \cap \Gamma_{\mbf y}^{\theta-,R} \neq \varnothing$. 
\end{enumerate}
\end{theorem}

\noindent We conclude this subsection with a theorem that characterizes the properties of $\theta$-directed semi-infinite geodesics depending on whether $\theta \in\Busedc$.

\begin{theorem} \label{thm:geod_description}
On a single event of probability one, the following are equivalent.
\begin{enumerate} [label=\rm(\roman{*}), ref=\rm(\roman{*})]  \itemsep=3pt
    \item  \label{notTheta} $\theta \notin \Busedc$.
    \item \label{LRsame} $\Gamma_{\mbf x}^{\theta +,R} = \Gamma_{\mbf x}^{\theta -,R}$ and $\Gamma_{\mbf x}^{\theta+,L} = \Gamma_{\mbf x}^{\theta - ,L}$ for all $\mbf x \in \Z \times \R$.
    \item \label{coal} All $\theta$-directed semi-infinite geodesics coalesce {\rm(}whether they are Busemann geodesics or not{\rm)}. 
    \item \label{countonegeod}  For all $\mbf x \in (\Z \times \R) \setminus \NU_0$, there exists a single $\theta$-directed semi-infinite geodesic starting from $\mbf x$. 
    \item \label{onegeod} There exists $\mbf x \in \Z \times \R$ such that there is exactly one $\theta$-directed semi-infinite geodesic starting from $\mbf x$. 
    \item \label{Ronegeod} There exists $\mbf x\in \Z \times \R$ such that $\Gamma_{\mbf x}^{\theta +,R} = \Gamma_{\mbf x}^{\theta-,R}$.
    \item \label{Lonegeod} There exists $\mbf x \in \Z \times \R$ such that $\Gamma_{\mbf x}^{\theta+,L} = \Gamma_{\mbf x}^{\theta-,L}$. 
\end{enumerate}
Under these equivalent conditions, the following also holds.
\begin{enumerate}[resume, label=\rm(\roman{*}), ref=\rm(\roman{*})]  \itemsep=3pt
    \item \label{allBuse} For all $\mbf x \in \Z \times \R$, all $\theta$-directed semi-infinite geodesics starting from $\mbf x$ are in the set $\mbf T_{\mbf x}^{\theta}$, i.e., they are all Busemann geodesics.
\end{enumerate}
\end{theorem}

\begin{remark} \label{rmk:splitting_coalesceing_general}
By Theorem~\ref{existence of semi-infinite geodesics intro version}\ref{Leftandrightmost}, for each $\mbf x \in \Z \times \R$, $\theta > 0$, and $\sigg \in \{-,+\}$, $\Gamma_{\mbf x}^{\theta \sig,L}$ is the leftmost geodesic between any two of its points.  Hence, by~\ref{Lonegeod}$\Leftrightarrow$\ref{notTheta}  of Theorem~\ref{thm:geod_description}, for each $\theta \in \Busedc$ and $\mbf x \in \Z \times \R$, there exists a point $\mbf v \ge \mbf x$ such that $\Gamma_{\mbf x}^{\theta -,L}$ and $\Gamma_{\mbf x}^{\theta +,L}$ split at $\mbf v$ and never come back together. The same is true with $L$ replaced with $R$ and ``left" replaced with ``right." On the other hand, for a given sign $\sigg \in \{-,+\}$, all $\theta \sigg$ geodesics coalesce. See Figure~\ref{fig:two_splitting}.
\end{remark}

\begin{figure}[t]
            \begin{tikzpicture}
            \draw[red,ultra thick,->] plot coordinates {(1,0)(4,0)(4,1)(6,1)(6,2)(8,2)(8,3)(9.5,3)(9.5,4)(12,4)};
            \draw[red, ultra thick] plot coordinates {(0,2)(2.5,2)(2.5,3)(4,3)(4,4)(9.5,4)};
            \draw[blue,thick,->] plot coordinates
            {(0,2)(2.5,2)(2.5,3)(3,3)(3,4)(3.5,4)(3.5,5)(8,5)(8,5.5)};
            \draw[blue,thick] plot coordinates {(1,0)(4,0)(4,1)(5,1)(5,2)(5.5,2)(5.5,3)(6,3)(6,4)(7,4)(7,5)};
            \filldraw[black] (1,0) circle (2pt) node[anchor = north] {\small $\mbf x$};
            \filldraw[black] (0,2) circle (2pt) node[anchor = east] {\small $\mbf y$};
            \filldraw[black] (3,3) circle (2pt) node[anchor = north] {\small $\mbf w$};
            \filldraw[black] (5,1) circle (2pt) node[anchor = north] {\small $\mbf v$};
            \filldraw[black] (7,5) circle (2pt) node[anchor = south] {\small $\mbf z^{\theta-}(\mbf x,\mbf y)$};
            \filldraw[black] (9.5,4) circle (2pt) node[anchor = south] {\small $\mbf z^{\theta +}(\mbf x,\mbf y)$};
            \end{tikzpicture}
            \caption{\small The red/thick paths are the $\theta +$ geodesics and the blue/thin paths are the $\theta-$ geodesics. In this figure, we consider the geodesics $\Gamma_{\mbf x}^{\theta-,L}$ and $\Gamma_{\mbf x}^{\theta+,L}$ out of $\mbf x$, and we consider the geodesics $\Gamma_{\mbf y}^{\theta-,R}$ and $\Gamma_{\mbf y}^{\theta+,R}$ out of $\mbf y$. The picture would be qualitatively the same if we changed the $L/R$ distinctions of the geodesics, as long as the choice of $L/R$ is the same at each initial point. We choose this particular configuration in the figure so that the coalescence points are $\mbf z^{\theta -}(\mbf x,\mbf y)$ and $\mbf z^{\theta +}(\mbf x,\mbf y)$, as defined in Definition~\ref{def:coalpt}. }
            \label{fig:two_splitting}
        \end{figure}
        
        \begin{figure}[t]
        \begin{adjustbox}{max totalsize={5.5in}{5in},center}
\begin{tikzpicture}
\draw[blue,ultra thick,->] (6,2)--(7,2)--(7,3)--(9,3)--(9,4)--(12,4);
\draw[red,ultra thick,->] (8,2)--(11,2)--(11,3)--(14,3)--(14,4)--(15,4);
\draw[red,ultra thick,->] (3,2)--(3,3)--(5,3)--(5,4)--(8,4)--(8,5);
\filldraw[black] (0.5,0.5) circle (1pt);
\filldraw[black] (1.5,1) circle (1pt);
\filldraw[black] (2.5,1.5) circle (1pt);
\filldraw[black] (-0.5,0) circle (2pt) node[anchor = north] {$\mbf x$};
\end{tikzpicture}
\end{adjustbox}
\caption{\small The picture in the (unknown) case that there exists a semi-infinite geodesic $\Gamma$ that is not a Busemann geodesic. After a certain level, there is a single $\theta-$ geodesic and a single $\theta+$ geodesic, and $\Gamma$ (blue, middle) lies strictly between the $\theta-$ Busemann geodesic (red, left) and the $\theta +$ Busemann geodesic (red, right). The diagonal dots are used to indicate that initially, there may be more than one $\theta \boxempty$ Busemann geodesic, but these all coalesce into a single geodesic at a certain level. }
\label{fig:nonBuse}
\end{figure}

\begin{remark} \label{rmk:Busegeod}
Theorem~\ref{thm:geod_description}\ref{allBuse} is proved by showing that if an arbitrary $\theta$-directed semi-infinite geodesic $\Gamma$ coalesces with a Busemann geodesic, then $\Gamma$ must also be a Busemann geodesic. The same proof can be applied to show the following statement:

Let $\theta > 0$ (possibly in $\Busedc$), and assume $\Gamma$ is an arbitrary $\theta$-directed semi-infinite geodesic from $\mbf x$. Also, assume that for some Busemann geodesic $\wh \Gamma$, there exists a sequence of points  $(m_n,t_n) \in \Gamma \cap \wh \Gamma$ with $m_n,t_n \rightarrow \infty$. Then, $\Gamma$ is a Busemann geodesic.

Therefore, by Theorems~\ref{existence of semi-infinite geodesics intro version}\ref{itm:all semi-infinite geodesics lie between leftmost and rightmost} and~\ref{thm:general_coalescence}\ref{itm:split_return}, if there exists some $\theta \in \Busedc$ and a $\theta$-directed semi-infinite geodesic $\Gamma$ that is not a Busemann geodesic, then there exists a level $r$ such that, above level $r$, $\Gamma_{\mbf x}^{\theta \sig,L}$ agrees with  $\Gamma_{\mbf x}^{\theta \sig,R}$ for $\sigg \in \{-,+\}$, and $\Gamma$ lies strictly between the $\theta-$ and $\theta +$ Busemann geodesics. See Figure~\ref{fig:nonBuse}. Whether such a geodesic $\Gamma$ exists with positive probability for some random direction is currently unknown and is left as an open problem.
\end{remark}

\subsection{The competition interfaces} \label{section:ci}
Loosely speaking,  the competition interface from a given initial point $(m,s) \in \Z \times \R$ separates the  points $(n,t) \ge (m,s)$ into two sets, differentiated by  whether the geodesic between $(m,s)$ and $(n,t)$ passes through $(m + 1,s)$.
However, since point-to-point geodesics are not unique in general (see Lemma~\ref{existence of leftmost and rightmost geodesics}), we introduce separate left and right competition interfaces. 
 For $(m,s) \le (n,t)$, let $\Gamma_{(m,s),(n,t)}^L$ and $\Gamma_{(m,s),(n,t)}^R$ denote, respectively, the leftmost and rightmost geodesics from $(m,s)$ to $(n,t)$. 
For $n > m$, define 
\[
\sigma_{(m,s),n}^L := \sup\Big\{t \ge s: \Gamma_{(m,s),(n,t)}^L \text{ passes through }(m + 1,s)\Big\}
\]
and
\[
\sigma_{(m,s),n}^R := \sup\Big\{t \ge s: \Gamma_{(m,s),(n,t)}^R \text{ passes through }(m + 1,s)\Big\}.
\]

\begin{remark}
We can represent the sequence $s = \sigma_{(m,s),m}^L \le \sigma_{(m,s),m + 1}^L \le \cdots$ as an infinite path on the plane as follows. See Figure~\ref{fig:the competition interface} for clarity. For $r \in \Z$, we let $r^\star=r - \f{1}{2}$. The path starts at $((m + 1)^\star,s)$, and for each $r > m$, the path takes a horizontal segment from $(r^\star,\sigma_{(m,s),r - 1}^L)$ to $(r^\star,\sigma_{(m,s),r}^L)$ and then a vertical segment from $(r^\star,\sigma_{(m,s),r}^L)$ to $((r + 1)^\star,\sigma_{(m,s),r}^L)$. The same rule is applied to the sequence $s = \sigma_{(m,s),m}^R \le \sigma_{(m,s),m + 1}^R \le \sigma_{(m,s),m + 2}^R \le \cdots$.
\end{remark}

\begin{definition}
The left competition interface from $(m,s)$ is the path  $s = \sigma_{(m,s),m}^L \le \sigma_{(m,s),m + 1}^L \le \sigma_{(m,s),m + 2}^L \le \cdots$.   The right competition interface is defined similarly, replacing $L$ in all superscripts with $R$. 
%
If $\sigma_{(m,s),n}^L = s$ for all $n \ge m$, we say that $(m,s)$ has \textit{trivial} left competition interface.  We say the same with $L$ replaced with $R$ and ``left" replaced with ``right." In this case, the associated path is the upward  vertical ray started at $((m +1)^\star,s)$.
\end{definition}

\begin{example}
In BLPP, it is fairly simple to construct a point $(m,s)$ whose right and left competition interfaces differ. Refer to Figure~\ref{fig:two_ci} for clarity. For each $m \in \Z$, the proof of Lemma~\ref{lemma:mult_geod}\ref{3 geod} constructs a random pair $s < t \in \R$ such that there exist exactly three geodesics between $(m,s)$ and $(m + 1,t)$: one that makes an immediate vertical step, one that makes a horizontal step from $(m,s)$ to $(m,t)$ and then a vertical step to $(m + 1,t)$, and one geodesic that jumps to level $m +1$ at an intermediate time $u \in (s,t)$. Since 
\[
L_{(m,s),(m + 1,t)} = \sup_{v \in [s,t]}\{B_m(s,v) + B_{m + 1}(v,t)\},
\]
the three geodesics correspond to the existence of exactly three maximizers of the function $B_m(v) - B_{m + 1}(v)$ over $v \in [s,t]$, namely $v = s, v = u$, and $v = t$. Thus, for $v\in[s,u)$ there is a unique geodesic between $(m,s)$ and $(m + 1,v)$, and this geodesic passes through $(m + 1,s)$. For $v\in[u,t)$ there are exactly two geodesics between $(m,s)$ and $(m + 1,v)$,  the left one jumps at $s$ and the right one at $u$. 
For $\hat t > t$, no geodesic between $(m,s)$ and $(m + 1,\hat t)$ can pass through $(m + 1,s)$. Otherwise, $u$ and $t$ would both be  maximizers of  $B_m(v) - B_{m + 1}(v)$ over $v \in [s,\hat t]$ lying in the interior of the interval, a contradiction to Corollary~\ref{cor:max_in_interior}. Therefore, $u = \sigma_{(m,s),m + 1}^R < \sigma_{(m,s),m + 1}^L = t$.  
\end{example}

\begin{figure}[t]
\begin{adjustbox}{max totalsize={5.5in}{5in},center}
\begin{tikzpicture}
\draw[gray,thin] (-0.5,0)--(15,0);
\draw[gray,thin] (-0.5,1)--(15,1);
\draw[gray,thin] (-0.5,2)--(15,2);
\draw[gray,thin] (-0.5,3)--(15,3);
\draw[gray,thin] (-0.5,4)--(11,4);
\draw[gray,thin] (13,4)--(15,4);
\draw[gray,thin,dashed] (-0.5,0.5)--(15,0.5);
\draw[gray,thin,dashed] (-0.5,1.5)--(15,1.5);
\draw[gray,thin,dashed] (-0.5,2.5)--(15,2.5);
\draw[gray,thin,dashed] (-0.5,3.5)--(15,3.5);

\draw[blue,thick,<-] (15,4.5)--(11,4.5)--(11,3.5)--(6,3.5)--(6,2.5)--(4.5,2.5)--(4.5,1.5)--(1.5,1.5)--(1.5,0.5)--(-0.5,0.5);
\draw[red,ultra thick] (-0.5,0)--(2,0)--(2,1)--(5,1)--(5,2)--(7,2)--(7,3)--(14,3)--(14,4)--(15,4);
\draw[red,ultra thick] (-0.5,0)--(-0.5,1)--(1,1)--(1,2)--(3,2)--(3,3)--(5,3)--(5,4)--(10,4);
\filldraw[black] (-0.5,0) circle (2pt) node[anchor = north] {$(m,s)$};
\filldraw[black] (10,4) circle (2pt) node[anchor = south] {$(n,t_1)$};
\filldraw[black] (15,4) circle (2pt) node[anchor = north] {$(n,t_2)$};
\filldraw[black] (11,4) circle (2pt) node[anchor = west] {$(n,\sigma_{(m,s),n})$};
\end{tikzpicture}
\end{adjustbox}
\caption{\small The competition interface based at $(m,s)$ (blue/thin path) separates points $(n,t)$ based on whether the geodesic (red/thick path) between $(m,s)$ and $(n,t)$ makes an initial vertical or horizontal step. The solid horizontal lines mark levels in $\Z$ and the dotted lines mark levels in $\Z-\tfrac12$. This figure ignores the distinction between left and right competition interface.}
\label{fig:the competition interface}
\end{figure}

\begin{figure}[t]
    \centering
    \begin{tikzpicture}
    \draw[gray,thin] (0,0)--(7,0);
    \draw[gray,thin] (0,1)--(7,1);
    \draw[red, ultra thick] plot coordinates {(4,0)(1,0)(1,1)(4,1)(4,0)(6,0)(6,1)(4,1)};
    \draw[blue,thick,->] plot coordinates {(1,0.55)(3.9,0.55)(3.9,1.5)};
    \draw[blue,dashed,->] plot coordinates {(1,0.45)(5.9,0.45)(5.9,1.5)};
    \filldraw[black] (1,0) circle (2pt) node[anchor = north] {$(m,s)$};
    \filldraw[black] (4,1) circle (2pt);
    \node at (3,1.25) {$(m + 1,u)$};
    \filldraw[black] (6,1) circle (2pt);
    \node at (7,1.25) {$(m + 1,t)$};
    \end{tikzpicture}
    \caption{\small The three geodesics (red/thick), the left competition interface (blue/dashed), and the right competition interface (blue/thin)} 
    \label{fig:two_ci}
\end{figure}


For $S \in \{L,R\}$ and $\sigg \in \{-,+\}$, set  $\tau_{(m,s),m}^{0\sig,S} = s$, and for $(m,s) \in \Z \times \R$, define
\be \label{eqn:intro_ci_dir}
\theta_{(m,s)}^L := \sup\{\theta \ge 0: \tau_{(m,s),m}^{\theta\sig,L} = s\},\;\;\text{and}\;\; \theta_{(m,s)}^R := \sup\{\theta \ge 0: \tau_{(m,s),m}^{\theta\sig,R} = s\}.  
\ee
 These quantities are the asymptotic directions of the competition interfaces, and they characterize points with nontrivial competition interface, as will be seen in the theorems that follow. By the monotonicity and limits of  Theorem~\ref{existence of semi-infinite geodesics intro version}\ref{itm:monotonicity of semi-infinite jump times}\ref{itm:monotonicity in theta} and Equation~\eqref{eqn:limits in theta}, $\theta_{(m,s)}^R$ and $\theta_{(m,s)}^L$ are independent of the choice of sign $\sigg \in\{-,+\}$ used in the definition \eqref{eqn:intro_ci_dir}.
 
 \begin{remark}
 From the definition, it immediately follows that, for $(m,s) \in \Z \times \R$ and $n > m$, 
 \be  \label{eqn:LR_ci}
 s \le \sigma_{(m,s),n}^R \le \sigma_{(m,s),n}^L.
 \ee 
 From~\eqref{eqn:intro_ci_dir} and the monotonicity of Theorem~\ref{existence of semi-infinite geodesics intro version}\ref{itm:monotonicity of semi-infinite jump times}\ref{itm:monotonicity in theta}, 
 \be \label{eqn:LR_ci_dir} 
 0 \le \theta_{(m,s)}^R \le \theta_{(m,s)}^L.
 \ee
 At first glance, these inequalities may seem strange, but the left competition interface is to the right of the right competition interface, because the modifiers ``left" and ``right" refer to the geodesics that are separated. 
 \end{remark}
 \begin{lemma} \label{lemma:ci_finite}
 On a single event of probability one, for every $(m,s) \in \Z \times \R$, $\sigma_{(m,s),n}^L$ and $\theta_{(m,s)}^L$ are finite. 
 \end{lemma}
 
 Now that we know all four quantities are finite, we can state the result for the asymptotic directions.

\begin{theorem} \label{thm:limiting direction of competition interface} On a single event of full probability, 
the following limits hold for each $(m,s) \in \Z \times \R$: 
\be\label{th-lim} 
\theta_{(m,s)}^L = \lim_{n\rightarrow \infty} \f{\sigma_{(m,s),n}^L}{n}\qquad\text{and}\qquad\theta_{(m,s)}^R =\lim_{n \rightarrow \infty} \f{\sigma_{(m,s),n}^R}{n}.
\ee
\end{theorem}

The next theorem provides  characterizations of  nontrivial competition interfaces, simultaneously, for all initial points.  In particular, the equivalences imply a sharp geometric dichotomy: either a competition interface is trivial, or it has a strictly positive limiting slope in \eqref{th-lim}. The latter case is triggered by having even one {\it finite} geodesic that starts with a vertical step before moving to the right, as in the definition of the random set $\CI$ from~\eqref{CI_intro}. 
 
 \begin{theorem} \label{thm:ci_equiv}
 On  a single event of full probability, for every $(m,s) \in \Z \times \R$,  the following are equivalent. 
 \begin{enumerate}[label=\rm(\roman{*}), ref=\rm(\roman{*})]  \itemsep=3pt
\item \label{itm:some_vert_geod} $(m,s) \in \CI$. That is, for some $n > m$ and $t > s$, at least one geodesic between $(m,s)$ and $(n,t)$ passes through $(m + 1,s)$.      
     \item \label{itm:Ltriv} $\sigma_{(m,s),n}^L > s$ for some $n > m$, i.e., $(m,s)$ has nontrivial left competition interface.
     \item \label{itm:Rtriv} $\sigma_{(m,s),n}^R > s$ for some $n > m$, i.e., $(m,s)$ has nontrivial right competition interface. 
     \item \label{itm:thetaL > 0}$\theta_{(m,s)}^L > 0$.
     \item \label{itm:thetaR > 0} $\theta_{(m,s)}^R > 0$.
     \item \label{itm:Lvert} There exists $\theta > 0$ and $\sigg \in \{-,+\}$ such that $\tau_{(m,s),m}^{\theta \sig,L} = s$, i.e., $\Gamma_{\mbf (m,s)}^{\theta \sig,L}$ makes an immediate vertical step to $(m + 1,s)$. In this case, $\theta \le \theta_{(m,s)}^L$, and the statement holds for all $\theta < \theta_{(m,s)}^L$.
     \item \label{itm:Rvert} There exists $\theta > 0$ and $\sigg \in \{-,+\}$ such that $\tau_{(m,s),m}^{\theta \sig,R} = s$, i.e., $\Gamma_{\mbf (m,s)}^{\theta \sig,R}$ makes an immediate vertical step to $(m + 1,s)$. In this case, $\theta \le \theta_{(m,s)}^R$, and the statement holds for all $\theta < \theta_{(m,s)}^R$.
     \item \label{itm:Lsplit} There exists $\theta > 0$ such that $\Gamma_{(m,s)}^{\theta-,L} \cap \Gamma_{(m,s)}^{\theta +,L} = \{(m,s)\}$. In other words, there exists $\theta > 0$ such that $\Gamma_{(m,s)}^{\theta-,L}$ makes a vertical step to $(m + 1,s)$, while $\Gamma_{(m,s)}^{\theta+,L}$ makes an initial horizontal step, and the two geodesics never meet again after the initial point. In this case, the following hold: 
     \begin{itemize}
         \item 
     $\theta = \theta_{(m,s)}^L$, 
     \item for $\gamma < \theta$ both  $\Gamma_{(m,s)}^{\gamma-,L}$ and $\Gamma_{(m,s)}^{\gamma+,L}$  take an initial  vertical step to $(m + 1,s)$,
     \item for $\gamma > \theta$ both  $\Gamma_{(m,s)}^{\gamma-,L}$ and $\Gamma_{(m,s)}^{\gamma+,L}$ take an initial horizontal step.
     \end{itemize} 
    In other words, $\theta_{(m,s)}^L$ is the unique direction $\gamma$ such that $\Gamma_{(m,s)}^{\gamma-,L} \cap \Gamma_{(m,s)}^{\gamma+,L}$ is a finite subset of the plane.
     \item \label{itm:Rsplit} Condition~\ref{itm:Lsplit} with superscript $L$ replaced with $R$.
     \item \label{itm:v = 0} $v_{m + 1}^{\theta \sig}(s)= 0$ for some $\theta > 0$ and $\sigg \in \{-,+\}$. In this case, $\theta \le \theta_{(m,s)}^L$, and $v_{m + 1}^{\theta \sig}(s) = 0$ for all $\theta < \theta_{(m,s)}^L$.
 \end{enumerate}
 \end{theorem}
 \begin{remark} \label{rmk:split}
 For the implication \ref{itm:thetaR > 0}$\Rightarrow$\ref{itm:Rsplit}, refer to Figure~\ref{fig:semi-infinte geodesics split at initial point} for clarity. By this result and \ref{notTheta}$\Leftrightarrow$\ref{LRsame} of Theorem~\ref{thm:geod_description}, if  $\theta_{(m,s)}^R > 0$, then $\theta_{(m,s)}^R \in \Busedc$. The same is true if $R$ is replaced with $L$ in the superscript. For all other $\theta \in \Busedc$, the $\theta+$ and $\theta-$ geodesics from $(m,s)$ eventually split by~\ref{Ronegeod}$\Leftrightarrow$\ref{notTheta} and~\ref{Lonegeod}$\Leftrightarrow$\ref{notTheta} of Theorem~\ref{thm:geod_description}, but they travel together for some time before splitting. 
\end{remark}

\begin{figure}[t]
\begin{adjustbox}{max totalsize={5.5in}{5in},center}
\begin{tikzpicture}
\draw[gray,thin] (-0.5,0)--(15,0);
\draw[gray,thin] (-0.5,1)--(15,1);
\draw[gray,thin] (-0.5,2)--(15,2);
\draw[gray,thin] (-0.5,3)--(15,3);
\draw[gray,thin] (-0.5,4)--(15,4);
\draw[gray,thin] (-0.5,4)--(15,4);
\draw[gray,thin] (-0.5,5)--(15,5);
\draw[gray,thin,dashed] (-0.5,0.5)--(15,0.5);
\draw[gray,thin,dashed] (-0.5,1.5)--(15,1.5);
\draw[gray,thin,dashed] (-0.5,2.5)--(7,2.5);
\draw[gray,thin,dashed] (9,2.5)--(15,2.5);
\draw[gray,thin,dashed] (-0.5,3.5)--(3,3.5);
\draw[gray,thin,dashed] (5,3.5)--(15,3.5);
\draw[gray,thin,dashed] (-0.5,4.5)--(15,4.5);
\draw[blue,thick,<-] (14.5,5)--(14.5,4.5)--(13,4.5)--(13,3.5)--(6,3.5)--(6,2.5)--(4.5,2.5)--(4.5,1.5)--(1.5,1.5)--(1.5,0.5)--(-0.5,0.5);
\draw[red,ultra thick,->] (-0.5,0)--(2,0)--(2,1)--(5,1)--(5,2)--(7,2)--(7,3)--(14,3)--(14,4)--(15,4)--(15,4.5);
\draw[red,ultra thick,->] (-0.5,0)--(-0.5,1)--(1,1)--(1,2)--(3,2)--(3,3)--(5,3)--(5,4)--(10,4)--(10,5)--(13,5);
\node at (8,2.5) {$\Gamma_{(m,s)}^{\wh \theta+,R}$};
\node at (4,3.5) {$\Gamma_{(m,s)}^{\wh \theta-,R}$};
\filldraw[black] (-0.5,0) circle (2pt) node[anchor = north] {$(m,s)$};
\filldraw[black] (-0.5,1) circle (2pt) node[anchor = east] {$(m + 1,s)$};
\end{tikzpicture}
\end{adjustbox}
\caption{\small When the competition interface direction $\widehat\theta = \theta_{(m,s)}^R > 0$, $\Gamma_{(m,s)}^{\wh \theta-,R}$ (upper red/thick path) immediately splits from $\Gamma_{(m,s)}^{\wh \theta +,R}$ (lower red/thick path). These paths never touch after the initial point, and the competition interface (blue/thin) lies between the paths.}
\label{fig:semi-infinte geodesics split at initial point}
\end{figure}

 \noindent Recall the countable sets $\NU_0$ and $\NU_1$ of initial points of $L/R$ non-uniqueness of  Busemann geodesics for a given $\theta\sigg$, defined in  \eqref{NU_0_1}. The following relates these sets to the set of points with nontrivial competition interface.
 \begin{theorem} \label{thm:ci_and_NU}
 The following hold on a single event of full probability. \begin{enumerate} [label=\rm(\roman{*}), ref=\rm(\roman{*})]  \itemsep=3pt
     \item \label{contain} $\NU_0 \subseteq \CI$. 
     \item \label{itm:cisd} $
\NU_1 = \{(m,s) \in \Z \times \R: \theta_{(m,s)}^R \neq \theta_{(m,s)}^L\} = \{(m,s) \in \Z \times \R: 0 < \theta_{(m,s)}^R < \theta_{(m,s)}^L\}$.
\item \label{itm:ci_int} The following classifies the directions and signs for which $(m,s) \in \NU_1^{\theta \sig}$ {\rm(}with the convention $(a,a] = [a,a) = \varnothing$ for $a \in \R${\rm)}.
\begin{enumerate} [label=\rm(\alph{*}), ref=\rm(\alph{*})]  \itemsep=3pt
    \item $(m,s) \in \NU_1^{\theta -}$ if and only if $\theta \in (\theta_{(m,s)}^R,\theta_{(m,s)}^L]$.
    \item $(m,s) \in \NU_1^{\theta +}$ if and only if $\theta \in [\theta_{(m,s)}^R,\theta_{(m,s)}^L)$.
\end{enumerate} 
\item \label{itm:NU_dense_self} The set $\NU_1$ is dense in itself. Specifically, for $(m,s) \in \NU_1^{\theta \sig}$ and every $\ve > 0$, there exists $t \in (s - \ve,s)$ such that $(m,t) \in \NU_1^{\theta \sig}$. Further, if $(m,s) \in \NU_1$, then for each $\theta< \theta_{(m,s)}^R$, $\sigg \in \{-,+\}$ \rm{(}or $\theta = \theta_{(m,s)}^R$ and $\sigg = -$\rm{)} and $\ve > 0$, there exists $t \in (s,s + \ve)$ such that $(m,t) \in \NU_1^{\theta \sig}$.
\item \label{itm:NU_struct} For all $(m,s) \in \Z \times \R$, there exists a random $\ve = \ve(m,s) > 0$ such that for all $\theta > \theta_{(m,s)}^R$, $\sigg \in \{-,+\}$, and $t \in (s,s + \ve]$,  $(m,t) \notin \NU_1^{\theta \sig}$. The statement also holds for $\theta = \theta_{(m,s)}^R$ if $\sigg = +$.
\item \label{itm:CI_struct} For all $(m,s) \in \CI$ and all $\ve > 0$, there exists $t \in (s,s + \ve)$ such that $\theta_{(m,t)}^R = \theta_{(m,s)}^R > 0$. 
 \end{enumerate}
 \end{theorem}

 

\begin{remark} \label{rmk:CI not NU 0} The set $ \CI \setminus \NU_0$ has Hausdorff dimension $\f{1}{2}$, since $\CI$ has Hausdorff dimension $\f{1}{2}$ (Theorem~\ref{thm:Haus_comp_interface}) and $\NU_0$ is countable (Theorem~\ref{thm:NU}\ref{itm:union}). For all $\mbf x \in \CI \setminus \NU_0$, $\theta > 0$ and $\sigg \in \{-,+\}$, there is exactly one geodesic in $\mbf T_{\mbf x}^{\theta \sig}$. However, by \ref{itm:thetaL > 0}$\Rightarrow$\ref{itm:Lsplit} of Theorem~\ref{thm:ci_equiv}, for $\wh \theta = \theta_{\mbf x}^{L} = \theta_{\mbf x}^R$, the $\wh \theta -$ geodesic travels initially vertically while the $\wh \theta +$ geodesic travels initially horizontally, and the two geodesics never meet again. 
\end{remark}

\noindent The final theorem of this section relates the directions of the competition interfaces to the exceptional directions of  the Busemann process and  sharpens the weak inequalities $0 \le \theta_{(m,s)}^R \le \theta_{(m,s)}^L$ with a trichotomy. Before stating the theorem, we introduce some notation and make some remarks. For each $(m,s) \in \Z \times \R$, define the following closed subset of $\R_{\ge0}$: 
\be \label{Compset_def}
\Compset_{(m,s)} := \{\theta > 0: v_{m + 1}^{\theta -}(s) < v_{m + 1}^{\theta +}(s), \text{ and } h_{m}^{\theta +}(s,t) < h_m^{\theta -}(s,t) \text{ for all }t > s\} \cup A_{(m,s)},
\ee
where $A_{(m,s)} = \{0\}$ if $v_{m + 1}^{\theta \sig}(s) > 0$ for all $\theta > 0$ and $\sigg \in \{-,+\}$, and $A_{(m,s)} = \varnothing$ otherwise.
\begin{remark}
 We note that the set of $\theta > 0$ such that $h_m^{\theta +}(s,t) < h_m^{\theta -}(s,t)$ is decreasing as $t \searrow s$. This follows from the monotonicity of Theorem~\ref{thm:summary of properties of Busemanns for all theta}\ref{general monotonicity Busemanns}, which implies that, for $s < t < T$,
\[
0 \le h_m^{\theta -}(s,t) - h_m^{\theta +}(s,t) \le h_m^{\theta-}(s,T) - h_m^{\theta +}(s,T).
\]
The set $\Compset_{(m,s)}$ describes the directions that are simultaneously jump points for the Busemann function of the pair $(m,s)$ and $(m +1,s)$ and all the pairs $(m,s)$ and $(m,t)$ for $t > s$.
If $v_{m  +1}^{\theta \sig}(s) > 0$ for all $\theta > 0$ and $\sigg \in\{-,+\}$, then we include $0$ as a jump point because $v_{m + 1}^{\theta \sig}(s) \searrow 0$ as $\theta \searrow 0$ (Theorem~\ref{thm:summary of properties of Busemanns for all theta}\ref{general uniform convergence Busemanns}\ref{general uniform convergence:limits to 0}). Additionally, by Theorem~\ref{thm:coupled_BMs_technical}\ref{itm:infinite_inc_points}, for all $t > s$, $0$ is an accumulation point of the set of $\theta > 0$ such that $h_m^{\theta +}(s,t) < h_m^{\theta -}(s,t)$. 
\end{remark} 

\begin{remark}
The set $\Compset_{(m,s)}$ can be described as an intersection of supports of Lebesgue-Stieltjes measures on $[0,\infty)$. By Theorem~\ref{thm:Theta properties}\ref{itm:const}, we have
\be \label{eqn:measures}
\Compset_{(m,s)} = \supp \mu_{(m,s),(m + 1,s)} \cap \bigcap_{t:t > s}\supp \mu_{(m,t),(m,s)},
\ee
where, for $\mbf x \SeS \mbf y$, $\mu_{\mbf x,\mbf y}$ is the Lebesgue-Stieltjes measure of the nondecreasing function $\theta \mapsto \B^{\theta \sig}(\mbf x,\mbf y)$ (see Remark~\ref{rmk:Buse_mont}). While these functions are only defined for $\theta > 0$, we extend the measures to $[0,\infty)$ simply by defining $\mu_{\mbf x,\mbf y}\{0\} = 0$. Then, by the previous remark, $0$ is a point on the right-hand side of~\eqref{eqn:measures} if and only if $A_{(m,s)} = \{0\}$. In  this sense, the following theorem is the BLPP analogue of the corresponding result for lattice LPP with continuous weights, given in Theorem 3.7 of~\cite{Janjigian-Rassoul-Seppalainen-19}. The left/right distinction in the following creates a new phenomenon not present in the lattice model with continuous weights.  
\end{remark}

\begin{theorem} \label{thm:ci_and_Buse_directions}
The following hold on a single event of probability one.
\begin{enumerate}[label=\rm(\roman{*}), ref=\rm(\roman{*})]  \itemsep=3pt
    \item \label{itm:cicount} Recall the set $\Busedc$ of exceptional directions where Busemann functions jump, defined in~\eqref{eqn:Theta}.  Then,
\[
 \{\theta_{(m,s)}^L\}_{(m,s) \in \CI}  = \{\theta_{(m,s)}^R\}_{(m,s) \in \CI}  = \Busedc.
\]
In particular, there are only countably infinitely many distinct asymptotic directions of the competition interfaces across all initial points in $\Z\times\R$. 
\item \label{itm:ci_dir_supp}
For each $(m,s) \in \Z \times \R$, $\Compset_{(m,s)} = \{\theta_{(m,s)}^R\} \cap \{\theta_{(m,s)}^L\} $. Thus, by Theorems~\ref{thm:ci_equiv} and~\ref{thm:ci_and_NU}\ref{itm:cisd}, there are three possibilities for $S_{(m,s)}$:
\begin{enumerate}[label=\rm(\alph{*}), ref=\rm(\alph{*})]  \itemsep=3pt
    \item $(m,s) \in \NU_1$, $0 < \theta_{(m,s)}^R < \theta_{(m,s)}^L$, and $\Compset_{(m,s)} = \varnothing$ 
    \item $(m,s) \in \CI \setminus \NU_1$, $\theta_{(m,s)}^L = \theta_{(m,s)}^R > 0$, and $\Compset_{(m,s)} = \{\theta_{(m,s)}^L\} = \{\theta_{(m,s)}^R\}$.
    \item $(m,s) \notin \CI$, $\theta_{(m,s)}^L = \theta_{(m,s)}^R = 0$, and $\Compset_{(m,s)} = \{0\}$.
\end{enumerate}
\end{enumerate}
\end{theorem}
\begin{remark}
Theorem~\ref{thm:ci_and_Buse_directions}\ref{itm:cicount} is particularly interesting because $\CI$ is uncountable (Theorem~\ref{thm:Haus_comp_interface}).
\end{remark}

\section{Connections to exponential last-passage percolation and the stationary horizon} \label{sec:CGM}

\subsection{Busemann process in the exponential corner growth model}
Fan and the first author~\cite{Fan-Seppalainen-20} derived the joint distribution of Busemann functions for the exponential corner growth model (CGM). In Section 5.2 of~\cite{Seppalainen-Sorensen-21a}, we outlined the analogies between the construction of semi-infinite geodesics in the discrete case and  that of BLPP. Here, we discuss connections between the joint distribution of the Busemann functions and prove a weak convergence result from the exponential CGM to BLPP. 

Let $\{Y_{\mbf x}\}_{\mbf x \in \Z^2}$ be a collection of nonnegative i.i.d random variables, each associated to a vertex on the integer lattice.  For $\mbf x \le \mbf y \in \Z \times \Z$, define the last-passage time as
\[
G_{\mbf x,\mbf y} = \sup_{\mbf x_\centerdot \in \Pi_{\mbf x,\mbf y}} \sum_{k = 0}^{|\mbf y - \mbf x|_1} Y_{\mbf x_k}, 
\]
where $\Pi_{\mbf x, \mbf y}$ is the set of up-right paths $\{\mbf x_k\}_{k = 0}^{n}$ that satisfy $\mbf x_0 = \mbf x,\mbf x_{n} = \mbf y$, and $\mbf x_k - \mbf x_{k - 1} \in \{\mbf e_1,\mbf e_2\}$. In what follows, we will take $Y_{\mbf x} \sim \operatorname{Exp}(1)$ and refer to this model as the exponential CGM. In this case, Busemann functions exist and are indexed by  direction vectors $\xi$. It is convenient to index the direction vectors as follows, in terms of a real parameter $\alpha\in(0,1)$: 
\[
\xi(\alpha) = \Bigg(\frac{\alpha^2}{\alpha^2 + (1 - \alpha)^2} , \frac{(1 - \alpha)^2}{\alpha^2 + (1 - \alpha)^2}\Bigg).
\]
Then for  a fixed   $\alpha \in (0,1)$  we have the  almost sure Busemann  limit
\[
U^\alpha(\mbf x,\mbf y) = \lim_{n \rightarrow \infty} G_{\mbf x,n\xi(\alpha)} - G_{\mbf y,n\xi(\alpha)}. 
\]

Under the assumption that $Y_0$ has finite second moment, Glynn and Whitt   introduced BLPP as a universal scaling limit of the discrete CGM~\cite[Theorem 3.1 and Corollary 3.1]{glynn1991}. That is, the process 
\[
\{L_{(m,s),(n,t)}: (m,s) \le (n,t) \in \Z \times \R\}
\]
is the functional limit in distribution, as $k \rightarrow \infty$, of the properly interpolated version of the process
\[
\Bigl\{\f{1}{\sqrt k}\Big( G_{(m,sk),(n,tk)} - (t - s)k\Big): (m,s) \le (n,t) \in \Z \times \R \Bigr\}.
\]
Thus, it is reasonable to expect that the Busemann functions for BLPP can be obtained by a limit of the Busemann functions in the exponential CGM. Heuristically, for $\lambda > 0$,
\begin{align}
    \h_0^{1/\lambda^2}(t) &= \lim_{n \rightarrow \infty} L_{(0,0),(n,n/\lambda^2)} - L_{(0,t),(n,n/\lambda^2)} \nonumber \\
    &= \lim_{n \rightarrow \infty} \lim_{k \rightarrow \infty} \frac{1}{\sqrt{k}} \Big(G_{(0,0),(nk/\lambda^2,n)} - G_{(tk,0),(nk/\lambda^2,n)} - tk\Big) \nonumber \\
    &^*\!\!=^* \lim_{k \rightarrow \infty} \lim_{n \rightarrow \infty}\frac{1}{\sqrt{k}} \Big(G_{(0,0),(nk/\lambda^2,n)} - G_{(tk,0),(nk/\lambda^2,n)} - tk\Big)  \nonumber \\
    &= \lim_{k \rightarrow \infty} \f{1}{\sqrt k} \Big(U^{\frac{\sqrt k}{\sqrt k + \lambda}}((0,0),(tk,0)) - tk\Big). \label{Buselimk}
\end{align}
The $^*\!\!=^*$ notation is used to indicate that the order of the limits was changed without justification.

\noindent Similar to the construction in~\cite{Seppalainen-Sorensen-21a} for Busemann functions in BLPP, the Busemann functions for exponential last-passage percolation can be extended simultaneously to all directions, either as a cadlag or caglad version (see~\cite{Sepp_lecture_notes}).
\begin{theorem}[\cite{Fan-Seppalainen-20}, Theorem 3.4] \label{thm:CGM_jump_process}
The process $\alpha \mapsto U^{\alpha}((0,0),(0,1))$  is a jump process and can be explicitly described as follows. Take an inhomogeneous Poisson point process (with a rate function not specified here) that defines jump points for the process. Then, at each jump point, take an independent exponential random variable (whose parameter depends on the location of the jump) to determine the size of the jump. This process has independent, but not stationary, increments. 
\end{theorem}

The BLPP Busemann jump process $\lambda \mapsto X(\lambda;t)$ described by Theorem~\ref{thm:Busemann jump process intro version}   is more complicated than its discrete analogue described in  Theorem~\ref{thm:CGM_jump_process}:  in particular,  the increments   are not independent, as recorded in Corollary~\ref{cor:inc_dist_conseq}\ref{itm:non_independent}. However, the increments of $\lambda\mapsto X(\lambda;t)$ are stationary, in contrast with Theorem~\ref{thm:CGM_jump_process}. Furthermore, the set of jumps of $\lambda \mapsto X(\lambda;t)$ is \textit{not} a Poisson point process. Indeed, if $W$ is the location of the first jump of the process $\lambda \mapsto X(\lambda;t)$, then $\Pp(W > \lambda)$ is given by~\eqref{eqn:0inc}, which is not exponential.

While this stark contrast may seem strange, it is in fact not natural to expect that all properties of the process in Theorem~\ref{thm:CGM_jump_process}  transfer to the BLPP setting. The process described in that theorem only considers Busemann functions across a single horizontal edge. In the scaling between the discrete Busemann functions and BLPP Busemann functions~\eqref{Buselimk}, one considers the Busemann function $U^\alpha$ across $tk$ edges for large $k$. One can show using tools from~\cite{Fan-Seppalainen-20} that for an integer $k > 1$, the increments of the process $\alpha \mapsto U^\alpha((0,0),(k,0))$ are not independent. It remains an open problem to develop an explicit description of the process $\lambda \mapsto X(\lambda;t)$ in the BLPP setting. 

However, the finite-dimensional distributions of the Busemann functions between the two models have a very similar structure, and in fact, the two-dimensional BLPP Busemann process can be obtained from a limit of two-dimensional Busemann processes for the exponential CGM. The proof of the following result, as well as a detailed description of the queuing setup for the discrete model, can be found in Section~\ref{sec:CGM_proofs}.

\begin{theorem} \label{thm:convCGM}
Fix $0 \le \lambda_1 \le \lambda_2$. As an appropriately interpolated sequence of continuous functions, the process
\begin{equation} \label{eqn:CGM_conv}
\f{1}{\sqrt k}\Big(U^{\f{\sqrt k}{\sqrt k + \lambda_1}}((0,0),(\abullet k,0)) -  \abullet k, U^{\f{\sqrt k}{\sqrt k + \lambda_2}}((0,0),(\abullet k,0)) - \abullet k\Big),
\end{equation}
converges in distribution to $(h_0^{1/\lambda_1^2}(\abullet), h_0^{1/\lambda_2^2}(\abullet))$, in the sense of uniform convergence on compact sets.
\end{theorem}

We expect that the full Busemann process of exponential last-passage percolation converges to the Busemann process for BLPP. This would require more technical tightness results on a richer space. Such a type of argument was recently achieved by Busani~\cite{Busani-2021}, who showed that under KPZ scaling, the entire Busemann process for the exponential CGM has a limit, termed the stationary horizon. It is expected that the stationary horizon is a universal object in the KPZ universality class. We explore this further in the following section.

\subsection{The BLPP Busemann process and the stationary horizon} \label{sec:stat_horiz}
To describe the stationary horizon, we introduce some notation from~\cite{Busani-2021}. The map $\Phi:C(\R) \times C(\R) \to C(\R)$ is defined as \[
\Phi(f,g)(t) = \begin{cases}
f(t) + \Big[W_0(f - g) + \inf_{0 \le s \le t} (f(s) - g(s))\Big]^{-} &t \ge 0 \\
f(t) - \Big[W_t(f - g) + \inf_{t < s \le 0} \Big(f(s) - f(t) - [g(s) - g(t)]\Big)\Big]^{-} &t < 0,
\end{cases}
\]
where 
\[
W_t(f) = \sup_{s \le t}[f(t) - f(s)].
\]
We note that the map $\Phi$ is well-defined only on the appropriate space of functions where the supremums are all finite.
This map extends to maps $\Phi^k:C(\R)^k \to C(\R)^k$ as follows. \begin{enumerate}\itemsep=3pt
    \item $\Phi^1(f_1)(t) = f_1(t).$ 
    \item $\Phi^2(f_1,f_2)(t) = [f_1(t),\Phi(f_1,f_2)(t)]$,\qquad\text{and for }$k \ge 3,$ 
    \item $\Phi^k(f_1,\ldots,f_k)(t) = \bigl[f_1(t),\Phi\bigl(f_1,[\Phi^{k - 1}(f_2,\ldots,f_k)]_1\bigr)(t),\ldots,\Phi\bigl(f_1,[\Phi^{k -1}(f_2,\ldots,f_k)]_{k - 1}\bigr)(t)\bigr]$.
\end{enumerate}

\begin{definition} \label{def:SH}
The stationary horizon $\{G_{\alpha}:\alpha \in \R\}$ is a process with state space $C(\R)$ and with paths in the Skorokhod space $D(\R,C(\R))$ of right-continuous functions $\R \to C(\R)$ with left limits. $C(\R)$ has the Polish topology of uniform convergence on compact sets.  
The law of the stationary horizon is characterized as follows: for real numbers $\alpha_1 < \cdots < \alpha_k$, the finite-dimensional vector $(G_{\alpha_1},\ldots,G_{\alpha_k})$ has the same law as $\Phi^k(f_1,\ldots,f_k)$, where $f_1,\ldots,f_k$ are independent two-sided variance $4$ Brownian motions with drifts $\alpha_1,\ldots,\alpha_k$.
\end{definition} 
We now present two theorems that relate the stationary horizon to the BLPP Busemann process. For a function $f:\R \rightarrow \R$, define $\wt f:\R \rightarrow \R$ by $\wt f(t) = -f(-t)$.  Define the map $\mathcal R_k: C(\R)^k \to C(\R)^k$ by $\mathcal R_k(f_1,\ldots,f_k)= (\wt f_1,\ldots,\wt f_k)$. Recall the measures $\mu^\lambda$ from Definition~\ref{definition of v lambda and mu lambda}.
\begin{theorem} \label{thm:dist_of_SH}
 For $\alpha = (\alpha_1,\ldots,\alpha_k)$ with $-\infty <\alpha_1 < \cdots < \alpha_k < \infty$, the $k$-tuple of functions $
 (G_{\alpha_1},\ldots,G_{\alpha_k})$
 has distribution $\mu^\alpha \circ \mathcal R_k^{-1}$. Furthermore, as random elements of the Skorokhod space $D(\R_{\ge 0},C(\R))$, the following distributional equality holds: 
 \[
 \{\wt h_0^{(1/\lambda^2)-}(4 \tspb\abullet)\}_{\lambda \ge 0} \deq \{G_{4\lambda}(\abullet)\}_{\lambda \ge 0}.
 \]
 That is, the reflected and scaled horizontal BLPP Busemann process is equal in distribution to the stationary horizon, restricted to nonnegative drifts $\lambda$.
 \end{theorem}
  \begin{remark}
 The reflection $\mathcal R_k$ appears because in the present work geodesics travel northeast, while the stationary horizon is constructed in~\cite{Busani-2021} as the scaling limit of the Busemann process of~\cite{Fan-Seppalainen-20} where 
 geodesics travel to the southwest. Using Theorem~\ref{thm:dist_of_SH}, the distributional information of the Busemann process obtained in Section~\ref{section:Busepp} applies as well to the stationary horizon. For each fixed $t \in \R$, Theorem~\ref{thm:dist_of_SH} and Remark~\ref{rmk:any_interval_same} give the following distributional equality without the reflection:
 \[
  \{h_0^{(1/\lambda^2)-}(4t)\}_{\lambda \ge 0} \deq \{G_{4\lambda}(t)\}_{\lambda \ge 0}.
 \]
 \end{remark}
 \begin{remark}
 To fit the space of Busemann functions in BLPP, the measures $\mu^\alpha$, defined in Definition~\ref{definition of v lambda and mu lambda} require the vector of drifts to be all nonnegative. However, we can still define the measures for all sequences $\alpha_1 < \cdots < \alpha_k$, as long as the inequalities are all strict. The extensions of Lemmas~\ref{image of script D lemma} and~\ref{weak continuity and consistency} to arbitrary real-valued sequences of drifts follow by the same proofs. 
 \end{remark}
 
 \begin{theorem} \label{thm:conv_to_SH}
Let $\alpha_1 < \cdots < \alpha_k$ be a sequence of real numbers. The process
\[
-n^{-1/3}(h_0^{1 - 2 n^{-1/3}\alpha_1}(-4n^{2/3} \abullet) + n^{2/3}\abullet,\ldots,h_0^{1 - 2n^{-1/3}\alpha_k}(-4n^{2/3}\abullet)+ n^{2/3}\abullet)
\]
converges in distribution to $(G_{4\alpha_1},\ldots,G_{4\alpha_k})$.
\end{theorem}
 
\noindent Theorems~\ref{thm:dist_of_SH} and~\ref{thm:conv_to_SH} are proved by showing that the map $\Phi$ is a reflected version of the map $D$. The details are in Section~\ref{sec:SH_proofs}.
  We believe that one should be able to show convergence in the Skorokhod space, as was done for the exponential CGM in~\cite{Busani-2021}. This requires some tightness arguments to guarantee that jumps of the Busemann process do not happen too close together. We leave such details to future work.

\section{Open problems} \label{sec:op}
Before moving to the proofs, we state a list of open problems. 
\begin{enumerate} [label=\rm(\roman{*}), ref=\rm(\roman{*})]  \itemsep=3pt 
    \item Can Theorem~\ref{thm:geod_description}\ref{allBuse} be extended to show that {\it all}  semi-infinite geodesics with direction $\theta \in \Busedc$ are also Busemann geodesics? In the exponential CGM, Coupier~\cite{Coupier-11} showed that, with probability one, there is no direction with more than two semi-infinite geodesics from the same point. The proof of this fact relied on the coupling with TASEP. In~\cite{Janjigian-Rassoul-Seppalainen-19}, this fact was used to give a complete description of the number of geodesics  in each direction. Because of the $L/R$ distinction in BLPP, there are points and directions with more than two geodesics. For example, by Theorems~\ref{thm:ci_equiv}\ref{itm:Lsplit} and~\ref{thm:ci_and_NU}\ref{itm:ci_int}, if $(m,s) \in \NU_1$ and $\theta = \theta_{(m,s)}^L$, then $\Gamma_{(m,s)}^{\theta-,L},\Gamma_{(m,s)}^{\theta+,L},$ and $\Gamma_{(m,s)}^{\theta-,R}$ are three distinct geodesics, but $\Gamma_{(m,s)}^{\theta-,L}$ and $\Gamma_{(m,s)}^{\theta-,R}$ coalesce by Theorem~\ref{thm:general_coalescence}\ref{itm:allsignscoalesce}.  Perhaps something can be said about the maximal number of non-coalescing geodesics.  
    \item With positive probability, is there some random initial point $\mbf x \in \Z \times \R$ such that $\mbf T_{\mbf x}^{\theta \sig}$ contains more than two sequences for some $\theta > 0$ and $\sigg \in\{-,+\}$? See Remark~\ref{rmk:number_sIG}.
    \item By Theorem~\ref{thm:NU}, we know that the sets $\NU_1 \subseteq \NU_0$ are both countably infinite. The same is true if we choose any $\theta > 0$ and $\sigg \in\{-,+\}$ and consider the sets $\NU_1^{\theta\sig} \subseteq \NU_0^{\theta \sig}$. Theorem~\ref{thm:NU_paper1}\ref{non-discrete or dense} shows that for each $\theta > 0$, the set $\NU_1^\theta$ is neither discrete nor dense, and Theorem~\ref{thm:ci_and_NU}\ref{itm:NU_dense_self} shows that the set $\NU_1$ is dense in itself. However, presently we do not know whether the set $\NU_0 \setminus \NU_1$ is nonempty. This raises several questions.
    \begin{enumerate}
        \item Is $\NU_0 \setminus \NU_1$ nonempty?
        \item If the answer to the previous question is yes, is the set discrete?
        \item Are either of the sets $\NU_0$ or $\NU_1$ dense in $\Z \times \R$?
    \end{enumerate}
    We note that the existence of a point $(m,s) \in \NU_0^{\theta \sig} \setminus \NU_1^{\theta \sig}$ implies that the phenomenon described in Remark~\ref{rmk:coal_pt_not_minpt} occurs. 
    \item We know from Theorem~\ref{thm:ci_and_NU}\ref{itm:cisd} that the countable set $\NU_1$ is exactly the set of initial points whose left and right competition interfaces have different limiting directions. Are there random points whose left and right competition interfaces are different, but have the same direction? If so, do they coalesce? 
    \item It is widely expected that for models in the KPZ universality class, there exist no bi-infinite geodesics with probability one. This was proven for the exponential CGM using two different approaches~\cite{SlyNonexistenceOB,Balzs2019NonexistenceOB}. In~\cite{Seppalainen-Sorensen-21a}, we proved that for fixed northwest and southeast directions, there are almost surely no bi-infinite geodesics in BLPP. Can one show that there are almost surely no bi-infinite geodesics in BLPP without the fixed directional constraint? 
    \item Is there a sharper bound for the maximal number of finite geodesics in BLPP than that in Lemma~\ref{lemma:geodesic_bound}?
    \item Theorem~\ref{thm:CGM_jump_process} records an explicit description of the Busemann process across a single horizontal edge for the exponential CGM. In the BLPP setting, can one describe the process $\lambda \mapsto X(\lambda;t)$ in Theorem~\ref{thm:Busemann jump process intro version} more explicitly?
    \item Are independent Brownian motions with drift the unique invariant distribution of the multiline process defined in Section~\ref{section:multiline process}?  See Theorem~\ref{multiline invariant distribution} for the invariance statement. 
    \item As discussed in Remark~\ref{rmk:time_to_split}, can three or more distinct Busemann trajectories split at the same time?
    \item In Theorem~\ref{thm:conv_to_SH}, we showed convergence of the horizontal BLPP Busemann process to the stationary horizon, in the sense of finite-dimensional distributions. Show that the BLPP Busemann process converges to the stationary horizon in the Skorokhod space $D(\R,C(\R))$. 
\end{enumerate}

 \section{Proof of Results from Section~\ref{sec:Buse_jd_intro} and Theorem~\ref{thm:Theta properties}} \label{sec:Buse_proofs}
 Theorem~\ref{thm:Theta properties} is proved at the very end of this section.
 
\begin{proof}[Proof of Measurability of the state spaces $\X_n$ and $\Y_n$]
We show that $\X_n, \Xcomp_n,\Y_n$, and $\Ycomp_n$ are all Borel subsets of $C(\R)^n$. For $\Y_n$, and $\Ycomp_n$, it is sufficient to show that  
\[
\liminf_{t \rightarrow \infty}\f{Z^i(t)}{t}\qquad\text{and}\qquad\limsup_{t \rightarrow \infty}\f{Z^i(t)}{t}
\]
are both random variables. For $a \in \R$, 
\[
\Bigg\{\liminf_{t \rightarrow \infty} \f{Z^i(t)}{t} > a    \Bigg\} = \bigcup_{k \in \N}\bigcup_{N \in \N} \bigcap_{t \ge N} \Big\{\f{Z^i(t)}{t} \ge a + \f{1}{k}\Big\}.
\]
By continuity of $Z^i$, the intersection over $t \ge N$ can be changed to an intersection over $t \in [t,\infty) \cap \Q$. Hence, the set on the left is measurable.
This also shows that 
$\limsup_{t\rightarrow \infty}\f{Z^i(t)}{t}$ is measurable.
 By continuity, we may write $\X_n$ as 
\begin{multline*}
\X_n = \{\eta: \eta^i(0) = 0 \text{ for } 2 \le i \le n \} \cap \Big\{\eta: \liminf_{t \rightarrow \infty} \f{\eta^1(t)}{t} > 0\Big\} \\
\cap \bigcap_{i = 2}^n \bigcap_{s < t, s,t \in \Q}\Big\{\eta^i(s,t) \ge \eta^{i-1}(s,t)\Big\},
\end{multline*}
which is measurable. By replacing the inequality in the middle event with a weak inequality, $\Xcomp_n$ is also measurable.
\end{proof}

\subsection{Lemmas and identities for queuing mappings} All results of this subsection are deterministic facts related to the queuing mappings. Recall  the notation that $Z \li \wt Z$ if $Z(s,t) \le \wt Z(s,t)$ whenever $s \le t$. We say $f \in \CRpin$ if $f:\R\to\R$ is continuous and $f(0) = 0$. 

\begin{lemma} \label{identity for multiple queuing mappings flipped}
Let $\mathbf Z = (Z^1,\ldots,Z^n) \in \CRpin^n$. Define
\[
A_{n}^{\mathbf Z}(t) = \sup_{t \leq t_1\leq t_{2} \cdots \leq t_{n - 1} < \infty} \sum_{i = 1}^{n - 1} (Z^i(t_i) - Z^{i + 1}(t_i))   
\]
and assume that $A_n^{\mathbf Z}(t)$ is finite. Then, for $n \geq 2$,
\[
D^{(n)}(Z^n,Z^{n - 1}\ldots,Z^1)(t) = Z^1(t) + A_{n}^{\mathbf Z}(0) - A_{n}^{\mathbf Z}(t).
\]
\end{lemma}
\begin{proof}
The case $n = 2$ is the definition of $D$.
Now, assume that the statement holds for some $n \geq 2$. Then,  we have that
\begin{align*}
    &D^{(n + 1)}(Z^{n + 1},Z^{n },\ldots,Z^1)(t) = D(D^{(n)}(Z^{n + 1},Z^{n},\ldots,Z^{2}),Z^1)(t) \\
    &= Z^1(t) + \sup_{0 \leq s <\infty}\{Z^1(s)-D^{(n)}(Z^{n + 1},\ldots,Z^{2})(s)    \}
    \\&\qquad\qquad\qquad\qquad
    -\sup_{t \leq s <\infty}\{ Z^1(s)-D^{(n)}(Z^{n + 1},\ldots,Z^{2})(s)  \} \\
    &= Z^1(t) + \sup_{0 \le s < \infty}\Bigg\{Z^1(s)-Z^{2}(s) + \sup_{s \leq t_2 \leq \cdots\leq t_{n}  < \infty}\sum_{i = 2}^n (Z^i(t_i) - Z^{i + 1}(t_i))   \Bigg\} \\
    &\qquad\qquad\qquad- \sup_{t \le s < \infty}\Bigg\{Z^1(s) -Z^{2}(s) + \sup_{s \leq t_2 \leq \cdots\leq t_{n}  < \infty}\sum_{i = 2}^n (Z^i(t_i) - Z^{i + 1}(t_i)) \Bigg\} \\
    &= Z^1(t) + A_{n + 1}^{\mathbf Z}(0) - A_{n + 1}^{\mathbf Z}(t). \qquad\qquad\qquad\qquad\qquad\qquad\qquad\qquad\qquad\qquad\mbox{\qedhere}
\end{align*}
\end{proof}

\begin{lemma} \label{uniform convergence of queueing mapping}
Let $0 \le a < b$, and let $Z_k,B_k \in \CRpin$ be sequences such that $Z_k \rightarrow Z$ and $B_k \rightarrow B$ uniformly on compact sets. Assume further that 
\begin{equation} \label{eqn:lim_cond}
\limsup_{\substack{t \rightarrow \infty \\ k \rightarrow \infty} } \biggl|\f{1}{t} Z_k(t) - b \,\biggr| = 0 = \limsup_{\substack{t \rightarrow \infty \\ k \rightarrow \infty} } \biggl|\f{1}{t} B_k(t) - a \,\biggr|.
\end{equation}
Then, $Z' := D(Z,B),B' := R(Z,B), Z_k' := D(Z_k,B_k)$, and $ B_k' := R(Z_k,B_k)$ are well-defined for sufficiently large $k$. Furthermore,  
\be
\lim_{k \rightarrow \infty} Z_k' = Z' \qquad\text{and}\qquad \lim_{k \rightarrow \infty} B_k' = B',
\ee
in the sense of uniform convergence on compact subsets of $\R$, and
\be
\limsup_{\substack{t \rightarrow \infty \\ k \rightarrow \infty} } \biggl|\f{1}{t} Z_k'(t) - b \biggr| = 0 = \limsup_{\substack{t \rightarrow \infty \\ k \rightarrow \infty} } \biggl|\f{1}{t} B_k'(t) - a \biggr|.
\ee
\end{lemma}
\begin{proof}
The conditions~\eqref{eqn:lim_cond} guarantee that for all $k$ sufficiently large,
\[
\limsup_{t \rightarrow \infty} [B_k(t) - Z_k(t)] = \limsup_{t \rightarrow \infty} [B(t) - Z(t)] = -\infty. 
\]
By definition of the maps $D$ and $R$, 
\begin{align*}
 Z_k'(t)
= B_k(t) + \sup_{0 \le s <\infty}\{B_k(s) - Z_k(s)\} - \sup_{t \le s < \infty}\{B_k(s) - Z_k(s)\}, \text{ and } \\
B_k'(t) = Z_k(t) + \sup_{t \le s < \infty}\{B_k(s) - Z_k(s)\} - \sup_{0 \le s < \infty}\{B_k(s) - Z_k(s)\}.
\end{align*}
It therefore suffices to show that $\sup_{t \le s < \infty}\{B_k(s) - Z_k(s)\}$ converges uniformly, on compact subsets of $\R$, to $\sup_{t \le s < \infty} \{B(s) - Z(s)\}$, and that 
\be \label{eqn:limit_pres}
\limsup_{\substack{t \rightarrow \infty \\ k \rightarrow \infty} } \biggl|\f{1}{t}\sup_{t \le s < \infty}\{B_k(s) - Z_k(s)\}  - (a - b) \biggr| = 0.
\ee
We first prove pointwise convergence. Let $s_t$ be a maximizer of $B(s) - Z(s)$ over $s \in [t,\infty)$. Then, 
\begin{align*}
\liminf_{k \rightarrow \infty}\sup_{t \le s < \infty}\{B_k(s) - Z_k(s)\} &\ge 
\liminf_{k \rightarrow \infty} [B_k(s_t) - Z_k(s_t)] \\&= B(s_t)- Z(s_t) = \sup_{t \le s < \infty}\{B(s) - Z(s)\}.
\end{align*}
For the converse, for $k$ sufficiently large, let $s_t^k$ be a maximizer of $B_k(s) - Z_k(s)$ over $s \in [t,\infty)$. If
\be \label{eqn:lb}
\limsup_{k \rightarrow \infty}\sup_{t \le s < \infty}\{B_k(s) - Z_k(s)\} > \sup_{t \le s < \infty}\{B(s) - Z(s)\},
\ee
then by the uniform convergence of $B_k$ to $B$ and $Z_k$ to $Z$, it must be the case that $s_t^{k_j} \rightarrow \infty$ along some subsequence $k_j$. Then, by the assumption~\eqref{eqn:lim_cond},
\[
\limsup_{j \rightarrow \infty} \sup_{t \le s < \infty}\{B_{k_j}(s) - Z_{k_j}(s)\} = \limsup_{j \rightarrow \infty} [B_{k_j}(s_t^{k_j}) - Z_{k_j}(s_t^{k_j})] = -\infty,
\]
a contradiction to~\eqref{eqn:lb}.
Therefore, $s_t^k$ is a bounded sequence and for each $t \in \R$, there exists some $R_t \in \R$ such that, for all $k$ sufficiently large, 
\begin{align*}
\sup_{t \le s < \infty} \{B_k(s) - Z_k(s)\} &= \sup_{t \le s \le R_t}\{B_k(s) - Z_k(s)\}, \text{ and} \\
\sup_{t \le s < \infty} \{B(s) - Z(s)\} &= \sup_{t \le s \le R_t}\{B(s) - Z(s)\}.
\end{align*}
The quantity $\sup_{t \le s \le R_t}\{B_k(s) - Z_k(s)\}$ converges to $\sup_{t \le s < \infty} \{B(s) - Z(s)\}$ by the assumed uniform convergence on compact subsets $Z_k \rightarrow Z$ and $B_k \rightarrow B$. The uniform convergence on compact subsets of the function $t \mapsto \sup_{t \le s < \infty} \{B_k(s) - Z_k(s)\}$ is as follows: for $N\in \R$,
\begin{align*}
\limsup_{k \rightarrow \infty}&\sup_{t \in [-N,N]} \; \bigl\lvert \,\sup_{t \le s < \infty} \{B_k(s) - Z_k(s)\} - \sup_{t \le s < \infty} \{B(s) - Z(s)\}\bigr\rvert \\
= \limsup_{k \rightarrow \infty}&\sup_{t \in [-N,N]} |\sup_{t \le s \le R_N} \{B_k(s) - Z_k(s)\} - \sup_{t \le s \le R_N} \{B(s) - Z(s)\}| = 0.
\end{align*}
It remains to prove~\eqref{eqn:limit_pres}. Let $0 < \ve < \f{b - a}{2}$. By assumption~\eqref{eqn:lim_cond}, there exists $K,T \ge 0$ such that for all $k \ge K, t \ge T$, $B_k(t) \le (a + \ve)t$ and $Z_k(t) \ge (b - \ve)t$. Then, for such $t,k$,
\[
\sup_{t \le s < \infty}\{B_k(s) - Z_k(s)\} \le \sup_{t \le s < \infty}\{(a + \ve)s - (b - \ve)s\} = (a - b + 2\ve)t.
\]
For the reverse inequality, we simply apply the assumption~\eqref{eqn:lim_cond} to
\[
\sup_{t \le s < \infty}\{B_k(s) - Z_k(s)\} \ge B_k(t) - Z_k(t). \qedhere
\]
\end{proof}

The following is a direct corollary of Lemma~\ref{uniform convergence of queueing mapping}, setting $B_k = B$ and $Z_k = Z$ for all $k$.
\begin{lemma} \label{D and R preserve limits}
Let $(B,Z) \in \Y_2$ satisfy the following limits:
\[
\lim_{t \rightarrow \infty}\frac{Z(t)}{t} = b\qquad\text{and}\qquad \lim_{t \rightarrow \infty}\frac{B(t)}{t} = a,
\]
with $b > a$. Then, the mappings $D(Z,B)$ and $R(Z,B)$ are well-defined, and 
\[
\lim_{t \rightarrow \infty}\frac{D(Z,B)(t)}{t} = b \qquad \text{and}\qquad \lim_{t \rightarrow \infty}\frac{R(Z,B)(t)}{t} = a.
\]
\end{lemma}

\begin{lemma} \label{D preserves liminf}
Assume that $Z,B \in \CRpin$ are such that 
\[
b := \liminf_{t \rightarrow \infty} \frac{Z(t)}{t}  > 0 \qquad\text{and}\qquad\lim_{t \rightarrow \infty} \frac{B(t)}{t} = 0.
\]
Then,
\[
\liminf_{t \rightarrow \infty} \frac{D(Z,B)(t)}{t} \ge b.
\]
\end{lemma}
\begin{proof}
Since
\[
D(Z,B)(t) = B(t) + \sup_{0 \le s < \infty}\{B(s) - Z(s)\} - \sup_{t \le s < \infty}\{B(s) - Z(s)\},
\]
it is sufficient to prove that 
\[
\limsup_{t \rightarrow \infty} \frac{\sup_{t \le s < \infty}\{B(s) - Z(s)\}}{t} \le -b.
\]
For all $\varepsilon > 0$, there is $T$ large enough so that for $t \ge T$, 
\[
Z(t) \ge bt - \varepsilon t \qquad\text{and}\qquad B(t) \le \varepsilon t.
\]
Then, for such $t$ and  $2\varepsilon < b$,
\[
\sup_{t \le s < \infty}\{B(s) - Z(s)\} \le \sup_{t \le s < \infty}\{2\varepsilon s - bs\} = 2\varepsilon t - bt. \qedhere
\]
\end{proof}

\begin{lemma} \label{ordering of queuing mappings flipped}
Assume that $Z,Z',B \in \CRpin$ satisfy $Z \li Z'$ and
\[
\limsup_{t \rightarrow \infty} B(t) - Z(t) = -\infty.
\]
Then $B \li D(Z,B) \li D(Z',B)$. 
\end{lemma}
\begin{proof} 
By definition of $D$~\eqref{definition of D}, for $s \leq t$,
\begin{align*} \label{DZB}
D(Z,B)(s,t) &= B(s,t) + \sup_{s \leq u < \infty}\{ B(u) - Z(u)\} - \sup_{t \leq u < \infty}\{B(u) - Z(u)\} \nonumber \\
&=B(s,t) + \Big(\sup_{s \leq u \le t}\{ B(u) - Z(u)\} - \sup_{t \leq u < \infty}\{B(u) - Z(u)\}\Big)^+ \\
&= B(s,t) +\Big(\sup_{s \leq u \le t}\{ B(u) + Z(u,t)\} - \sup_{t \leq u < \infty}\{B(u) - Z(t,u)\}\Big)^+,
\end{align*}
and both inequalities follow from the last line. 
\end{proof}


\begin{proof}[Proof of Lemma~\ref{image of script D lemma}]

\medskip \noindent \textbf{Part~\ref{itm:image}:}
We prove that $\D^{(n)}:\Y_n \rightarrow \X_n \cap \Y_n$. By the same reasoning, $\D^{(n)}:\Ycomp_n \rightarrow \Xcomp_n \cap \Ycomp_n$. Lemma~\ref{identity for multiple queuing mappings flipped} gives us that since $Z^1(0) = 0$, we also have $\eta^i(0)  = D^{(i)}(Z^i,\ldots,Z^1)(0)= 0$ for $1 \le i \le n$. 
Since $\eta^1 = Z^1$, the requirement that 
\[
\liminf_{t \rightarrow \infty} \f{\eta^1(t)}{t} > 0
\]
is immediately satisfied. 
By  Lemma~\ref{ordering of queuing mappings flipped},  
$\eta^2 = D(Z^2,Z^1) \gi Z^1 = \eta^1$.
Assume inductively that
\[
\eta^{i} = D^{(i)}(Z^{i},\ldots,Z^1) \gi D^{(i - 1)}(Z^{i - 1},\ldots,Z^1) = \eta^{i - 1}.
\]
Then, after applying this assumption with $Z^2,\ldots,Z^{i + 1}$ in place of $Z^1,\ldots,Z^i$ and using Lemma~\ref{ordering of queuing mappings flipped}, we get that 
\begin{align*}
\eta^{i + 1} = &D^{(i + 1)}(Z^{i + 1},\ldots,Z^1) = D(D^{(i)}(Z^{i + 1},\ldots,Z^2),Z^1) \\
\gi &D(D^{(i - 1)}(Z^{i},\ldots,Z^2),Z^1) = D^{(i)}(Z^i,\ldots,Z^1) = \eta^i.  
\end{align*}
The fact that $\D^{(n)}$ preserves $\Y_n$ follows from Part~\ref{itm:Ynlim}.

\medskip \noindent \textbf{Part~\ref{itm:Ynlim}}: This follows by definition of $\D^{(n)}$ and repeated application of Lemma~\ref{D and R preserve limits}. \end{proof}

The following lemma has the most technical proof of the section, but is key to proving Theorem~\ref{dist of Busemann functions and Bm}. 
\begin{lemma} \label{analogue of Lemma 4.4 in Joint Buse Paper}
Let $(B^1,Z^1,Z^2) \in \Ycomp_3$  and set $B^2 = R(Z^1,B^1)$.
Then,
\[
D(D(Z^2,B^2),D(Z^1,B^1)) = D(D(Z^2,Z^1),B^1).
\]
\end{lemma}
\begin{proof}
We first note that by Lemma~\ref{identity for multiple queuing mappings flipped},
\begin{align}
    D(D(Z^2,Z^1),B^1)(t) &= D^{(3)}(Z^2,Z^1,B^1)(t) \nonumber \\
    &= B^1(t) + \sup_{0 \le s \le u < \infty}\{B^1(s) - Z^1(s) + Z^1(u) - Z^2(u)\} \nonumber \\
    &\qquad\qquad- \sup_{t \le s \le u < \infty}\{B^1(s) - Z^1(s) + Z^1(u) - Z^2(u)\}. \label{expression for rhs in 4.4 analogue}
\end{align}
On the other hand, by definitions of the mappings $D$ and $R$,
\begin{align}
   &\;\;\;\;\;\;D(D(Z^2,B^2),D(Z^1,B^1))(t) \nonumber \\
   &= D(Z^1,B^1)(t) + \sup_{0 \le s < \infty}\{D(Z^1,B^1)(s) - D(Z^2,B^2)(s)  \} \nonumber \\
   &- \sup_{t \le s < \infty}\{D(Z^1,B^1)(s) - D(Z^2,B^2)(s)  \} \nonumber \\
   &= B^1(t) + \sup_{0 \le s < \infty}\{B^1(s) - Z^1(s)\} - \sup_{t \le s < \infty}\{B^1(s) - Z^1(s)\} \nonumber \\
    &+ \sup_{0 \le s < \infty}\Big[B^1(s) - \sup_{s \le u < \infty}\{B^1(u) - Z^1(u)\} - B^2(s) + \sup_{s \le u <\infty}\{B^2(u) - Z^2(u)\}   \Big] \nonumber \\
    &\qquad- \sup_{t \le s < \infty}\Big[B^1(s) - \sup_{s \le u < \infty}\{B^1(u) - Z^1(u)\} - B^2(s) + \sup_{s \le u <\infty}\{B^2(u) - Z^2(u)\}   \Big] \nonumber \\
    &= B^1(t) + \sup_{0 \le s < \infty}\{B^1(s) - Z^1(s)\} - \sup_{t \le s < \infty}\{B^1(s) - Z^1(s)\} \nonumber  \\
    &\qquad+ \sup_{0 \le s < \infty}\Big[B^1(s) - Z^1(s) - 2\sup_{s \le u < \infty}\{B^1(u) - Z^1(u)\} \nonumber \\
    &\qquad \qquad\qquad\qquad\qquad+ \sup_{s \le u \le v <\infty}\{B^1(v) - Z^1(v) + Z^1(u) -  Z^2(u)\}   \Big] \nonumber \\
    &- \sup_{t \le s < \infty}\Big[B^1(s) - Z^1(s) - 2\sup_{s \le u < \infty}\{B^1(u) - Z^1(u)\} \nonumber \\
    &\qquad \qquad\qquad\qquad\qquad+ \sup_{s \le u \le v <\infty}\{B^1(v) - Z^1(v) + Z^1(u) -  Z^2(u)\}   \Big] \label{expression for lhs in 4.4 analogue}.
\end{align}
Comparing~\eqref{expression for rhs in 4.4 analogue} with~\eqref{expression for lhs in 4.4 analogue}, it is sufficient to show that, for arbitrary $t \in \R$,
\begin{align}
    &\sup_{t \le s < \infty}\{B^1(s) - Z^1(s) \}  +\sup_{t \le s < \infty} \Big[B^1(s) - Z^1(s)- 2\sup_{s \le u < \infty} \{B^1(u) - Z^1(u)  \} \label{lhs} \\
&\qquad \qquad\qquad\qquad\qquad\qquad+\sup_{s \le u \le v < \infty}\{Z^1(u) - Z^2(u) + B^1(v) - Z^1(v)\}  \Big] \nonumber\\
      = &\sup_{t \le s \le u < \infty}\{B^1(s) - Z^1(s) + Z^1(u) - Z^2(u) \}. \label{rhs}
\end{align}

We will first prove that $\eqref{lhs} \leq\eqref{rhs}$. We note that
\begin{align} 
\eqref{lhs} \leq &\sup_{t \le u < \infty}\{B^1(u) - Z^1(u)\} + \sup_{t \le s < \infty} \Big[B^1(s) - Z^1(s) - 2\sup_{s \le u < \infty}\{B^1(u) - Z^1(u)\} \nonumber \\
&\qquad \qquad\qquad\qquad\qquad+\sup_{s \le u < \infty}\{Z^1(u) - Z^2(u)\} + \sup_{s \le u < \infty}\{B^1(u) - Z^1(u)\}\Big] \nonumber\\
=&\sup_{t \le u < \infty}\{B^1(u) - Z^1(u)\} + \sup_{t \le s < \infty}\Big[B^1(s) - Z^1(s) \nonumber \\
&\qquad \qquad\qquad\qquad\qquad- \sup_{s \le u < \infty}\{B^1(u) - Z^1(u)\}+\sup_{s \le u < \infty}\{Z^1(u) - Z^2(u)\} \Big] \label{greater than lhs}. 
\end{align}
Now, we let $s^*\geq t$ be a point such that 
\begin{multline*}
B^1(s^*) - Z^1(s^*) - \sup_{s^* \le u < \infty}\{B^1(u) - Z^1(u)\}+\sup_{s^* \le u < \infty}\{Z^1(u) - Z^2(u)\}\\
=\sup_{t \le s < \infty}\Big[B^1(s) - Z^1(s) - \sup_{s \le u < \infty}\{B^1(u) - Z^1(u)\}+\sup_{s \le u < \infty}\{Z^1(u) - Z^2(u)\} \Big]. 
\end{multline*}
We consider two cases.

\noindent \textbf{Case 1:} 
\[
\sup_{s^* \le u < \infty}\{B^1(u) - Z^1(u)\} = \sup_{t \le u < \infty}\{B^1(u) - Z^1(u)\}.
\]
Then,
\begin{align*}
    \eqref{lhs} \le \eqref{greater than lhs} &= \sup_{t \le u < \infty}\{B^1(u) - Z^1(u)\} + B^1(s^*) - Z^1(s^*) \\
    &\qquad \qquad\qquad- \sup_{s^* \le u < \infty}\{B^1(u) - Z^1(u)\} + \sup_{s^* \le u < \infty}\{Z^1(u) - Z^2(u)\}  \\
    &= B^1(s^*) - Z^1(s^*) + \sup_{s^* \le u < \infty}\{Z^1(u) - Z^2(u)\} \\
    &\le \sup_{t \le s \le u < \infty}\{B^1(s) - Z^1(s) + Z^1(u) - Z^2(u) \} = \eqref{rhs}.
\end{align*}

\noindent \textbf{Case 2:} 
\[
\sup_{s^* \le u < \infty}\{B^1(u) - Z^1(u)\} < \sup_{t \le u < \infty}\{B^1(u) - Z^1(u)\}.
\]
Then, we have that 
\[
\sup_{t \le u < \infty}\{B^1(u) - Z^1(u)\} = \sup_{t \le u \le s^*}\{B^1(u) - Z^1(u)\},
\]
so, noting that $B^1(s^*) - Z^1(s^*) \le \underset{s^* \le u < \infty}{\sup}\{B^1(u) - Z^1(u)\}$,
\begin{align*}
    \eqref{lhs} \le \eqref{greater than lhs} &= \sup_{t \le u \le s^*}\{B^1(u) - Z^1(u)\} + B^1(s^*) - Z^1(s^*) \\
    &\qquad \qquad- \sup_{s^* \le u < \infty}\{B^1(u) - Z^1(u)\} + \sup_{s^* \le u < \infty}\{Z^1(u) - Z^2(u)\}  \\
    &\le \sup_{t \le u \le s^*}\{B^1(u) - Z^1(u)\} + \sup_{s^* \le u < \infty}\{Z^1(u) - Z^2(u)\} \\
    &= \sup_{t \le u \le s^* \le v< \infty}\{B^1(u) - Z^1(u) + Z^1(v) - Z^2(v)  \} \\
    &\le \sup_{t \le u \le v <\infty}\{B^1(u) - Z^1(u) + Z^1(v) - Z^2(v)  \} = \eqref{rhs}.
\end{align*}

Now, we prove that $\eqref{rhs} \le \eqref{lhs}$. Let $t \le s^* \le u^* < \infty$ be such that 
\[
B^1(s^*) - Z^1(s^*)+Z^1(u^*) - Z^2(u^*) = \sup_{t \le s \le u < \infty}\{B^1(s) - Z^1(s) + Z^1(u) - Z^2(u) \}.
\]
We consider two new cases.

\noindent \textbf{Case 1:} 
\[
\sup_{s^* \le v < \infty}\{B^1(v) - Z^1(v)\} = \sup_{u^* \le v < \infty}\{B^1(v)  -Z^1(v)\}.
\]
Then,
\begin{align*}
    \eqref{rhs}
    = &B^1(s^*) - Z^1(s^*) + Z^1(u^*) - Z^2(u^*) \\
    &\qquad \qquad\qquad+ \sup_{u^* \le v < \infty}\{B^1(v) - Z^1(v)\} -\sup_{s^* \le v < \infty}\{B^1(v) - Z^1(v)\} \\
    \le &\sup_{t \le s < \infty}\Big[B^1(s) - Z^1(s) - \sup_{s \le u < \infty}\{B^1(u) - Z^1(u)\}\\
    &\qquad \qquad\qquad+ \sup_{s \le u \le v < \infty}\{Z^1(u) - Z^2(u) + B^1(v) - Z^1(v)\}  \Big] \\
    \le &\sup_{t \le s < \infty}\{B^1(s) - Z^1(s)\} \\
    &+\sup_{t \le s < \infty}\Big[B^1(s) - Z^1(s) - 2\sup_{s \le u < \infty}\{B^1(u) - Z^1(u)\}\\
    &\qquad \qquad\qquad+\sup_{s \le u \le v < \infty}\{Z^1(u) - Z^2(u) + B^1(v) - Z^1(v)  \Big],
\end{align*}
and this equals \eqref{lhs}.

\noindent \textbf{Case 2:}
\[
\sup_{s^* \le v < \infty}\{B^1(v) - Z^1(v)\} > \sup_{u^* \le v < \infty}\{B^1(v)  -Z^1(v)\}.
\]
Then, we have that
\[
\sup_{s^* \le v < \infty}\{B^1(v) - Z^1(v)\} = \sup_{s^* \le v \le u^*}\{B^1(v) - Z^1(v)\} > B^1(u^*) - Z^1(u^*).
\]
In other words, there is a point $v^* \in [s^*,u^*)$ such that 
\[
B^1(v^*) - Z^1(v^*) = \sup_{s^* \le v < \infty}\{B^1(v) - Z^1(v)\}. 
\]
Next, we define $w^\star \in (v^\star,u^\star]$ as
\[
w^* = \sup\Big\{w \in (v^*,u^*]: B^1(w) - Z^1(w) = \sup_{u^* \le v < \infty}\{B^1(v) - Z^1(v) \} \Big\}.
\]\\
To see that the set over which the supremum is taken is nonempty, note that
\begin{align*}
B^1(v^*) - Z^1(v^*) &= \sup_{s^* \le v < \infty}\{B^1(v) - Z^1(v)\} \\
&> \sup_{u^* \le v < \infty}\{B^1(v)  -Z^1(v)\} \ge B^1(u^*) - Z^1(u^*),
\end{align*}
and use the Intermediate Value Theorem. By continuity, we observe that 
\begin{equation} \label{sup_eq}
B^1(w^*) - Z^1(w^*) = \sup_{u^* \le v < \infty}\{B^1(v) - Z^1(v)\} = \sup_{w^* \le v < \infty}\{B^1(v) - Z^1(v)  \}.
\end{equation}

\begin{figure}[t]
\centering
    \begin{tikzpicture}
    \draw[black,thick] (10,0)--(20,0);
    \draw[blue,thick] (10,1.5)--(20,1.5);
    \filldraw[black] (12,0) circle (2pt) node[anchor = north] {$v^*$};
    \filldraw[black] (15,0) circle (2pt) node[anchor = north] {$w^*$};
    \filldraw[black] (17,0) circle (2pt) node[anchor = north] {$u^*$};
    \draw [red] plot [smooth] coordinates {(12,2)(13,1)(14,1.75)(15,1.5)(17,0.5)(19,1.5)(20,1)};
    \end{tikzpicture}
    \caption{\small Example graph of the function $B^1(w) - Z^1(w)$. The upper (blue) line represents the value of $\underset{u^\star \le v < \infty}{\sup}\{B^1(v) - Z^1(v)\}$.}
    \label{fig:existence of w^*}
\end{figure}
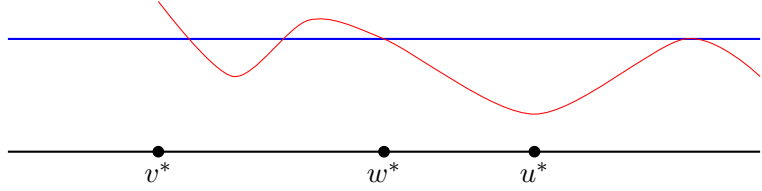

Refer to figure~\ref{fig:existence of w^*} for clarity. Now, by~\eqref{sup_eq}, we have 
\begin{align*}
    \eqref{rhs} = &B^1(s^*) - Z^1(s^*) + Z^1(u^*) - Z^2(u^*) \\
    = &B^1(s^*) - Z^1(s^*) + B^1(w^*) - Z^1(w^*) - 2\sup_{w^* \le v < \infty}\{B^1(v) - Z^1(v)\} \\
    &\qquad+Z^1(u^*) - Z^2(u^*) + \sup_{u^* \le v < \infty}\{B^1(v) - Z^1(v)\} \\
    &\le \sup_{t \le s < \infty}\{B^1(s) - Z^1(s)\}\\
    + &\sup_{t \le w  <\infty}\Big\{B^1(w) - Z^1(w) - 2\sup_{w \le v < \infty}\{B^1(u) - Z^1(u)\} \\
    &\qquad\qquad\qquad + \sup_{w \le u \le v < \infty}\{Z^1(u) - Z^2(u) + B^1(v) - Z^1(v)\} \Big\} = \eqref{lhs}. 
\end{align*}
This concludes all cases of the proof.
\end{proof}

\begin{theorem} \label{Alternate Rep of Iterated Queues}
Let $n \ge 2$, and assume $(B^1,Z^1,Z^{2},\ldots,Z^n) \in \Ycomp_{n + 1}$.   For $2 \le j \le n$ define
$
B^j = R(Z^{j - 1},B^{j - 1}).
$  
Then, for $1 \le k \le n - 1$,
\begin{align*}
    &D^{(n + 1)}(Z^n,Z^{n - 1},\ldots,Z^1,B^1) \\
    &\qquad\qquad = D^{(k + 1)}(D^{(n - k + 1)}(Z^n,\ldots,Z^{k + 1},B^{k + 1}),D(Z^k,B^k),\ldots,D(Z^1,B^1)).
\end{align*}
\end{theorem}
\begin{proof}
With Lemma~\ref{analogue of Lemma 4.4 in Joint Buse Paper} in place, we can now follow the argument of Theorem 4.5 in~\cite{Fan-Seppalainen-20}.
By Lemma~\ref{multiline process well-defined}, all the given operations are well-defined. Lemma~\ref{analogue of Lemma 4.4 in Joint Buse Paper} gives us the statement for $n = 2$. Assume, by induction, that the statement is true for some $n - 1 \ge 2$. We will show the statement is also true for $n$. We first prove the case $k = 1$. 
\begin{align*}
    &\;\;\;D^{(2)}(D^{(n)}(Z^n,\ldots,Z^2,B^2),D(Z^1,B^1)) \\
    &= D(D(D^{(n -1)}(Z^n,\ldots,Z^2),B^2),D(Z^1,B^1)) \\
    &= D(D(D^{(n - 1)}(Z^n,\ldots,Z^2),Z^1),B^1) \\
    &= D(D^{(n)}(Z^n,\ldots,Z^1),B^1) \\
    &= D^{(n + 1)}(Z^n,\ldots,Z^1,B^1).
\end{align*}
The second equality above was a consequence of Lemma~\ref{analogue of Lemma 4.4 in Joint Buse Paper}.
Now, let $2 \le k \le n - 1$. Then, applying the definition of $D^{(k + 1)}$ followed by the induction assumption,
\begin{align*}
    &D^{(k + 1)}(D^{(n - k + 1)}(Z^n,\ldots,Z^{k + 1},B^{k + 1}),D(Z^k,B^k),\ldots,D(Z^1,B^1)) \\
    = &D(D^{(k)}(D^{(n - k + 1)}(Z^n,\ldots,Z^{k + 1},B^{k + 1}),D(Z^k,B^k),\ldots,D(Z^2,B^2)),D(Z^1,B^1)) \\
    = &D(D^{(n)}(Z^n,\ldots,Z^2,B^2),D(Z^1,B^1)) = D^{(2)}(D^{(n)}(Z^n,\ldots,Z^2,B^2),D(Z^1,B^1)).
\end{align*}
Hence, we have reduced this to the $k = 1$ case. 
\end{proof}
\noindent We note that the case $k= n - 1$ of Theorem~\ref{Alternate Rep of Iterated Queues} gives us
\begin{equation} \label{intertwining}
    D^{(n + 1)}(Z^n,\ldots,Z^1,B^1) = D^{(n)}(D(Z^n,B^n),\ldots,D(Z^1,B^1)).
\end{equation}

\subsection{Multiline process}
\label{section:multiline process}
The multiline process is a discrete-time Markov chain on the state space $\Y_n$ of \eqref{Yndef}. The analogous process is defined in a discrete setting in~\cite{Fan-Seppalainen-20}. 
The transition from the time $m-1$ state $\mathbf Z_{m - 1} =\mathbf Z =  (Z^1,Z^2,\ldots, Z^n) \in \Y_n$ to the time $m$ state 
\begin{equation*} 
\mathbf Z_{m} = \overline{\mathbf Z} = (\overline Z^1,\overline Z^2,\ldots, \overline Z^n) \in \Y_n
\end{equation*}
is defined as follows. The driving force is an auxiliary   function $B \in \CRpin$ that satisfies
\[
\lim_{t \rightarrow \infty} t^{-1}B(t) = 0. 
\]
First, set $B^1 = B$,  and $\overline Z^1 = D(Z^1,B^1)$. 
    Then, iteratively for $i = 2,3,\ldots,n$:
    \be \label{multiline process}
    B^i = R(Z^{i - 1},B^{i - 1}), \qquad\text{and}\qquad
    \overline Z^i = D(Z^i,B^i).
    \ee

\begin{lemma} \label{multiline process well-defined}
The multiline process~\eqref{multiline process} is well-defined on the state space $\Y_n$.
\end{lemma}
\begin{proof}
This follows from Lemma~\ref{D and R preserve limits}: Inductively, each $B^i$ satisfies
\[
\lim_{t \rightarrow \infty}\frac{B^i(t)}{t} = 0,
\]
so since $\mathbf Z \in \Y_n$, we have that, for $1 \le i \le n$,
\[
\limsup_{t \rightarrow \infty}B^i(t) - Z^i(t) = -\infty. \qedhere
\]
\end{proof}

\begin{theorem} \label{multiline invariant distribution}
For each $\lambda = (\lambda_1,\ldots,\lambda_n) \in \R^n_{> 0}$ with $0 < \lambda_1 < \cdots < \lambda_n$, the measure $\nu^\lambda$ on $\Y_n$ is invariant for the multiline process~\eqref{multiline process} if the driving function $B$ at each step of the evolution is taken to be an independent standard, two-sided Brownian motion. 
\end{theorem}
\begin{proof}
Assume that $\mathbf Z = (Z^1,\ldots, Z^n) \in \Y_n$ has distribution $\nu^\lambda$. We will show that $\overline{\mathbf Z}$ also has distribution $\nu^\lambda$. The assumption on $\mbf Z$ means that $Z^1,\ldots,Z^n$ are independent two-sided Brownian motions with drift $\lambda_i$. By Theorem~\ref{O Connell Yor BM independence theorem queues}, 
$
\overline Z^1 = D(Z^1,B^1)
$
is a two-sided Brownian motion with  drift $\lambda_1$, independent of
$
B^2 = R(Z^1,B^1),
$
which is a two-sided Brownian motion with zero drift. Hence, the random paths $\overline Z^1, B^2,Z^2,\ldots, Z^n$ are mutually independent. We iterate this process as follows: Assume, for some $2 \leq k \leq n - 1$, that the random paths $\overline Z^1,\ldots,\overline Z^{k - 1},B^k,Z^k,\ldots, Z^n$ are mutually independent, where for $1 \leq i \leq k - 1$, $\overline Z^i$ is a Brownian motion with drift $\lambda_i$. Then, by another application of Theorem~\ref{O Connell Yor BM independence theorem queues},
$
\overline Z^k = D(Z^k,B^k)
$
is a two-sided Brownian motion with drift $\lambda_k$, independent of 
$
B^{k + 1} = R(Z^k,B^k),
$
which is a two-sided Brownian motion with zero drift.
Since $(\overline Z^k,B^{k + 1})$ is a function of $(B^k,Z^k)$, we have that $\overline Z^1,\ldots,\overline Z^k,B^{k + 1},Z^{k + 1},\ldots,Z^n$ are mutually independent, completing the proof. 
\end{proof}
\begin{remark}  Presently, it is open whether  $\nu^\lambda$ is the unique invariant measure with  asymptotic limits   $(\lambda_1,\ldots,\lambda_n)$.  Later, we establish  uniqueness of the invariant distributions  for the Markov chain that describes the Busemann functions.   
\end{remark}

\subsection{Busemann Markov chain}

Define a Markov chain $\eta := (\eta_m)_{m \in \Z_{\ge 0}} = ((\eta_m^1,\ldots,\eta_m^n))_{m \in \Z_{\ge 0}}$ with state space $\X_n$ as follows. It is essential that the state space is $\X_n$ and not $\Xcomp_n$ so that the evolution of this chain is well-defined. Henceforth,  $\mbf F =\{F_m\}_{m \ge 1}$ denotes an i.i.d.\ sequence of two-sided Brownian motions with zero drift, independent of the initial configuration $\eta_0\in\X_n$. At each discrete time step $m \ge 1$, set $ F_m$ to be the driving Brownian motion. Given the time $m-1$ state $\eta_{m-1}$,   define the time $m$   state of the chain as  
\begin{equation} \label{Busemann Markov chain}
\eta_{m} = \bigl(D(\eta^1_{m - 1},F_m),D(\eta^2_{m - 1},F_m),\ldots,D(\eta^n_{m - 1},F_m)\bigr).
\end{equation}
Lemmas~\ref{D preserves liminf} and~\ref{ordering of queuing mappings flipped} imply that if $\eta_{m -1} \in \X_n$, then $\eta_{m} \in \X_n$ as well. 
\begin{theorem} \label{existence of an invariant measure for Busemann MC}
The measure $\mu^\lambda$ of Definition~\ref{definition of v lambda and mu lambda} is invariant for the Markov chain~\eqref{Busemann Markov chain}.
\end{theorem}
\begin{proof}
This follows by an intertwining argument originating for particle systems in~\cite{Ferrari-Martin-2007} and carried out for exponential last-passage percolation in~\cite{Fan-Seppalainen-20}. Let $\mathbf Z \sim \nu^\lambda$. Assume that $\eta$ has distribution $\mu^\lambda$--the distribution of $\D^{(n)}(\mathbf Z)$. Without loss of generality, we assume that $\eta = \D^{(n)}(\mathbf Z)$. Then, for Brownian motion $B$, let $\Ss^B$ denote the mapping of a single evolution step of $\mathbf Z$ according to the multiline process~\eqref{multiline process} and $\T^B$ denote the mapping of a single evolution step of $\eta$ according to the Markov chain~\eqref{Busemann Markov chain}. Then, using the definition of $D^{(k)}$ and Equation~\eqref{intertwining},
\begin{align*}
\T_k^B(\eta) = &D(\eta^k,B) = D(D^{(k)}(Z^k,\ldots,Z^1),B^1) = D^{(k + 1)}(Z^k,\ldots,Z^1,B^1) \\
= &D^{(k)}(D(Z^k,B^k),D(Z^{k - 1},B^{k - 1}),\ldots,D(Z^1,B^1))  \\
= &D^{(k)}(\Ss_k^B(\mathbf Z),\Ss_{k - 1}^B(\mathbf Z),\ldots,\Ss_1^B(\mathbf Z)) = \D_k^{(n)}(\Ss^B(\mathbf Z)).
\end{align*}
Hence, $\T^B(\eta) = \D^{(n)}(\Ss^B(\mathbf Z))$. Since $\eta = \D^{(n)}(\mathbf Z)$, we have that
\[
\T^B(\D^{(n)}(\mathbf Z)) = \D^{(n)}(\Ss^B(\mathbf Z)).
\]
By Theorem~\ref{multiline invariant distribution}, $\Ss^B(\mathbf Z) \overset{d}{=} \mathbf Z \sim \nu^\lambda$. Therefore, $\T^B(\eta) \overset{d}{=} \D^{(n)}(\mathbf Z) \sim \mu^\lambda$.
\end{proof}

\subsection{Uniqueness of the invariant measure} \label{sec:uniqueness}
The existence of an invariant measure for the Markov chain~\eqref{Busemann Markov chain} in the case $n = 1$ is recorded in Theorem~\ref{O Connell Yor BM independence theorem queues}, and is originally due to Harrison and Williams \cite{harrison1990}. In words, if $Z$ is a two-sided Brownian motion with drift $\lambda > 0$, and $B$ is an independent two-sided Brownian motion, then $D(Z,B)$ is also a two-sided Brownian motion with drift $\lambda$. However, it was not until 2019 that Cator, Lopez, and Pimentel~\cite{Cator-Lopez-Pimentel-2019} proved the uniqueness of this invariant measure. The proof comes from constructing a coupling $(\eta_m,\wt \eta_m)_{m \ge 0}$  of the Markov chain~\eqref{Busemann Markov chain} started from two different initial inputs but with the same driving Brownian motions. Below is the main theorem of~\cite{Cator-Lopez-Pimentel-2019}. We note that there is a typographical error in the statement of the theorem in~\cite{Cator-Lopez-Pimentel-2019} which has been confirmed to us by the authors. The corrected version is stated below. The version we state also reflects the fact that the queuing mappings we work with are the reverse-time versions of those given in~\cite{Cator-Lopez-Pimentel-2019}.  
\begin{theorem}
[\cite{Cator-Lopez-Pimentel-2019}, Theorem  3] \label{convergence Theorem in Cator 2019}
Let $\lambda \in (0,\infty)$ and let $X \in \CRpin$ be a random process such that, with probability one,  
\begin{equation} \label{lim_cond_conv}
\limsup_{t \rightarrow -\infty}\frac{X(t)}{t} \le \lambda\qquad\text{and}\qquad\liminf_{t \rightarrow \infty}\frac{X(t)}{t} \ge \lambda.
\end{equation}
Let $(X,Z)$ be a coupling of $X$ and $Z$ where $Z$ is a Brownian motion with drift $\lambda$, and $(X,Z)$ is independent of the field of independent two-sided Brownian motions $\mbf F = \{F_m: m \ge 1\}$. Consider the coupling $(\eta^{X}_m,\eta^{Z}_m)_{m \ge 0}$ defined by initial conditions $\eta^{X}_0 = X$ and $\eta^{Z}_0 = Z$, where the evolution of the process is defined by $\eta^{X}_m = D(\eta_{m - 1}^X,F_m)$, and $\eta^Z_m = D(\eta_{m - 1}^Z,F_m)$. Then, for all compact $K \subseteq \R$ and $\varepsilon > 0$,
\[
\lim_{m \rightarrow \infty} \Pp\Big(\sup_{t \in K}|\eta^{X}_m(-t-\lambda^{-2}m,-\lambda^{-2} m) - \eta^{Z}_m(-t - \lambda^{-2}m,-\lambda^{-2} m)| > \varepsilon   \Big) = 0.
\]
\end{theorem}
\begin{remark}
The above theorem holds true for any initial condition $\eta_0^X$ satisfying the given conditions, but in general, the conclusion only holds for the increments of the processes in the interval $(-t - \lambda^{-2}m,-\lambda^{-2}m)$. However, the queuing mappings preserve increment-stationarity, so if the initial condition is increment-stationary, the conclusion holds for an arbitrary increment. 
\end{remark}
 A straightforward generalization of this  proves distributional convergence of the Markov chain~\eqref{Busemann Markov chain} from an appropriate initial condition to the measure $\mu^\lambda$: 
 
\begin{corollary} \label{uniqueness of invariant measure for Busemann MC}
Let $\lambda = (\lambda_1,\ldots,\lambda_n) \in \R^n$ be such that $0 < \lambda_1 < \cdots < \lambda_n$. Let $\mbf X = (X^1,\ldots,X^n)$ be a random function in $\X_n$ such that, for $1 \le i \le n$, $X^{i}$ is increment-stationary with 
\begin{equation} \label{eqn:limprop}
\limsup_{t \rightarrow -\infty} \frac{X^i(t)}{t} \le \lambda_i\qquad\text{and}\qquad\liminf_{t \rightarrow \infty} \frac{X^i(t)}{t} \ge \lambda_i.
\end{equation}
Let $(\mbf X,\mbf Z)$ be a coupling of $\mbf X$ and $\mbf Z$ where $\mbf Z = (Z^1,\ldots,Z^n) \sim \mu^\lambda$, and, for $1 \le i \le n$, $(X^i, Z^i)$ is jointly increment-stationary. Consider the coupling 
\[
(\eta_m^{\mbf X},\eta_m^{\mbf Z})_{m \ge 0} = (\eta_m^{\mbf X,1},\ldots,\eta_m^{\mbf X,n},\eta_m^{\mbf Z,1},\ldots,\eta_m^{\mbf Z,n})_{m \ge 0},
\]
where $\eta_0^{\mbf X} = \mbf X$ and $\eta_0^{\mbf Z} = \mbf Z$ and the evolution of the processes is given by the Markov chain~\eqref{Busemann Markov chain}, run simultaneously with the same driving Brownian motions $\mathbf F = \{F_m\}_{m \ge 1}$, independent of $(\mbf X,\mbf Z)$.  Then, for all compact $K \subseteq \R^n$ and $\varepsilon > 0$,
\[
\limsup_{m \rightarrow \infty} P\Big(\sup_{t \in K}|\eta^{\mbf X}_m(t) - \eta^{\mbf Z}_m(t)|_1 > \varepsilon   \Big) = 0,
\]
where $|\cdot|_1$ is the $\ell^1$ norm on $\R^n$.
\end{corollary}
\begin{proof}
First, we note that since $(X^i,Z^i)$ is jointly increment-stationary and $F_1$ is an independent Brownian motion (and therefore increment stationary), then the process $(X^i,Z^i,F_1)$ is jointly increment-stationary. Then, as a translation respecting mapping of $(X^i,Z^i,F_1)$ (Lemma~\ref{lemma:D preserves increment stationarity}), $(\eta^{\mbf X,i}_1,\eta^{\mbf Z,i}_1)$  is jointly increment stationary. By induction, $(\eta^{\mbf X,i}_m,\eta^{\mbf Z,i}_m)$ is jointly increment-stationary for each $m$. The desired conclusion then follows by applying Theorem~\ref{convergence Theorem in Cator 2019} separately to each of the $n$ components of $\eta^{\mbf X}_m(t) - \eta^{\mbf Z}_m(t)$.
\end{proof}

\noindent We now use Corollary~\ref{uniqueness of invariant measure for Busemann MC} to prove the following precursor to Theorem~\ref{dist of Busemann functions and Bm}.

\begin{theorem} \label{joint distribution of Busemann functions}
If $ \theta_1 > \theta_2 > \cdots > \theta_n > 0$, then for each $m \in \Z$, the vector
\begin{equation} \label{hBusevec}
(h_m^{\theta_1},h_m^{\theta_2},\ldots,h_m^{\theta_n})
\end{equation}
almost surely lies in the space $\X_n \cap \Y_n$ and has distribution $\mu^\lambda$ with $\lambda_i = \frac{1}{\sqrt\theta_i}$ for $1 \le i \le n$.
\end{theorem}

\begin{proof}
In Corollary~\ref{uniqueness of invariant measure for Busemann MC},  we can  choose $\mbf Z \sim \mu^\lambda$ to satisfy the joint increment-stationarity of $(X^i,Z^i)$, for example,  by taking $\mbf Z$ independent of $\mbf X$. By the invariance of the measure $\mu^\lambda$ (Theorem~\ref{existence of an invariant measure for Busemann MC}), Corollary~\ref{uniqueness of invariant measure for Busemann MC} implies that, under the given assumptions on $\mbf X$, $\eta_m^{\mbf X}$ converges in distribution to $\mu^\lambda$, in the sense of uniform convergence on compact sets. Thus, $\mu^\lambda$ is the unique such invariant measure of the Markov chain~\eqref{Busemann Markov chain} among distributions whose marginal distributions are increment-stationary and satisfy~\eqref{lim_cond_conv}.

By Parts~\ref{general monotonicity Busemanns} and~\ref{Buse_marg_dist} of  Theorem~\ref{thm:summary of properties of Busemanns for all theta}, for $m \in \Z$ and $\theta_1 > \theta_2 > \cdots > \theta_n > 0$, $h_m^{\theta_{i}} \li h_m^{\theta_{i + 1}}$ for $1 \le i \le n - 1$, and each $h_m^{\theta_i}$ is a two-sided Brownian motion with drift $\f{1}{\sqrt{\theta_i}}$. Thus, for $1 \le i \le n$,
\be \label{eqn:BM_dir}
\lim_{t \rightarrow \pm \infty} \f{h_m^{\theta_i}(t)}{t} = \f{1}{\sqrt \theta_i}, \qquad\text{a.s.,}
\ee
and so $(h_m^{\theta_1},h_m^{\theta_2},\ldots,h_m^{\theta_n}) \in \X_n \cap \Y_n$ with probability one. 
By Theorem~\ref{thm:summary of properties of Busemanns for all theta}\ref{general queuing relations Busemanns}, for $1 \le i \le n$, $h_m^{\theta_i} = D(h_{m + 1}^{\theta_i},B_m)$ almost surely. Furthermore, by Theorem~\ref{thm:summary of properties of Busemanns for all theta}\ref{busemann functions agree for fixed theta}, with probability one, for each $1 \le i \le n$ and $t \in \R$,
\begin{align*}
h_m^\theta(t) &= \lim_{n \rightarrow \infty} [L_{(m,0),(n,n\theta)}(\mbf B) - L_{(m,t),(n,n\theta)}(\mbf B)]\\ &= \lim_{n \rightarrow \infty} [L_{(m,0),(n + m,(n + m)\theta)}(\mbf B) - L_{(m,t),(n + m,(n + m)\theta)}(\mbf B)],
\end{align*}
and
since the environment of i.i.d. Brownian motions $\{B_r\}_{r \in \Z}$ has the same distribution as the environment $\{B_{r + k}\}_{r \in \Z}$ for each $k \in \Z$, the distribution of $(h_m^{\theta_1},\ldots,h_m^{\theta_n})$ is independent of $m$. Therefore, $(h_m^{\theta_1},\ldots,h_m^{\theta_n})$ must be distributed as the unique invariant distribution of the Markov chain~\eqref{Busemann Markov chain}, under the limit condition~\eqref{eqn:BM_dir}. By Corollary~\ref{uniqueness of invariant measure for Busemann MC}, this distribution is $ \mu^{(\lambda_1,\ldots,\lambda_n)}$,
where $\lambda_i = \f{1}{\sqrt \theta_i}$ for $1 \le i \le n$.
\end{proof}

\begin{proof}[Proof of Lemma~\ref{weak continuity and consistency}]
\textbf{Part~\ref{weak continuity}}
We show the existence of $\eta_k \sim \mu^{\lambda^k}$ and $\eta \sim \mu^\lambda$ such that, for $1 \le i \le n$,  $\eta_k^i \rightarrow \eta^i$, uniformly on compact sets, almost surely. Let $\mbf Z = (Z^1,Z^2,\ldots,Z^n) \sim \nu^\lambda$, and  define $Z_k^i(t) = Z^i(t) + (\lambda_i^k - \lambda_i)t$. Then, $(Z_k^1,Z_k^2,\ldots,Z_k^n) \sim \nu^{\lambda^k}$. Set $\eta = \D^{(n)}(Z)$ and $\eta_k = \D^{(n)}(Z_k)$. By construction, for $1 \le i \le n$,  $Z_k^i \rightarrow Z^i$, uniformly on compact sets, and 
\[
\limsup_{\substack{t \rightarrow \infty \\ k \rightarrow \infty} } \Bigl|\f{1}{t} Z_k^i(t) - \lambda_i \Bigr| = 0.
\]
Thus, the convergence of $\eta_k^1 \rightarrow \eta^1$ is immediate. By Lemma~\ref{uniform convergence of queueing mapping}, $\eta_k^2 = D(Z_k^2,Z_k^1)$ converges to $\eta^2 = D(Z^2,Z^1)$ uniformly on compact sets, and
\[
\limsup_{\substack{t \rightarrow \infty \\ k \rightarrow \infty} } \Bigl|\f{1}{t} \eta_k^2(t) - \lambda_2 \Bigr| = 0.
\]
Now, assume by induction that for $i\ge 2$, $\eta_k^i = D^{(i)}(Z_k^i,\ldots,Z_k^1)$ converges uniformly on compact sets to $\eta^i$ and that 
\[
\limsup_{\substack{t \rightarrow \infty \\ k \rightarrow \infty} } \Bigl|\f{1}{t} \eta_k^i(t) - \lambda_i \Bigr| = 0.
\]
Then, by shifting indices and setting $\wt \eta_k^i = D^{(i)}(Z_k^{i + 1},\ldots,Z_k^2)$ and $\wt \eta^i = D^{(i)}(Z^{i + 1},\ldots,Z^2)$, it also holds that $\wt \eta_k^{i}$ converges uniformly on compact sets to $\wt \eta^{i}$, and
\[
\limsup_{\substack{t \rightarrow \infty \\ k \rightarrow \infty} } \Bigl|\f{1}{t} \wt \eta_k^i(t) - \lambda_i \Bigr| = 0.
\]
By definition of $D^{(i + 1)}$~\eqref{D iterated} and the $i = 2$ case,
\[
\eta_k^{i + 1} = D^{(i + 1)}(Z_k^{i + 1},\ldots,Z_k^1) = D(\wt \eta_k^i,Z_k^1) \to D(\wt \eta^i,Z^1) = D^{(i + 1)}(Z^{i + 1},\ldots,Z^1)= \eta^{i + 1},
\]
where the convergence is almost sure, uniformly on compact sets. Furthermore, the $i = 2$ case also guarantees
\[
\limsup_{\substack{t \rightarrow \infty \\ k \rightarrow \infty} } \Bigl|\f{1}{t} \eta_k^{i + 1}(t) - \lambda_{i+1} \Bigr| = 0.
\]

\medskip \noindent \textbf{Part~\ref{consistency}:} It suffices to show that if $(\eta^1,\ldots,\eta^n) \in \Xcomp_n$ has distribution $\mu^{\lambda_1,\ldots, \lambda_n}$, then \[(\eta^1,\ldots,\eta^{i - 1},\eta^{i + 1},\ldots, \eta^n) \sim \mu^{\lambda_1,\ldots,\lambda_{i - 1},\lambda_{i + 1},\ldots,\lambda_n}.\] Recall that $\mu^\lambda$ is the distribution of $\D^{(n)}(Z^1,\ldots,Z^n)$, where $Z^i$ are independent Brownian motions with drifts $\lambda_i,$ and the $j$th component of $\D^{(n)}(Z^1,\ldots,Z^n)$ is $D^{(j)}(Z^j,\ldots,Z^1)$~\eqref{definition of script D}.

\noindent For $i = n$, the statement is immediate from the definition of the map $\D^{(n)}$. Next, we show the case $i = 1$. For $2 \le j \le n$, we use~\eqref{intertwining} to write
\[
D^{(j)}(Z^j,\ldots,Z^1) = D^{(j - 1)}(D(Z^j,\wt Z^{j - 1}),\ldots, D(Z^3,\wt Z^{2}),D(Z^2,\wt Z^{1})),
\]
where $\wt Z^{1} = Z^1$, and for $i > 1$, $\wt Z^{i} = R(Z^{i},\wt Z^{i - 1})$. Then $(\eta^2,\ldots,\eta^n) = \D^{(n - 1)}(\hat Z^2,\ldots,\hat Z^n)$, where $\hat Z^i = D(Z^i,\wt Z^{i - 1})$ for $2 \le i \le n$. By Theorem~\ref{multiline invariant distribution}, $\hat Z^2,\ldots,\hat Z^n$ are independent, so
this completes the proof of the $i = 1$ case. By definition of $D^{(j)}$, for $i < j \le n$,
\begin{equation} \label{queue_iter}
D^{(j)}(Z^j,\ldots,Z^1) = D(D(\cdots D(D^{(j - i + 1)}(Z^j,\ldots,Z^{i}),Z^{i-1}),\ldots,Z^2),Z^1).
\end{equation}
 Similarly as in the $i = 1$ case, we apply~\eqref{intertwining} to get that
\begin{align} \label{queue_iter_2}
&\;\;\;D^{(j - i + 1)}(Z^j,\ldots,Z^i)\nonumber \\
&= D^{(j - i)}(D(Z^j,\wt Z^{j - 1}),\ldots,D(Z^{i + 1},\wt Z^{i})) = D^{(j - i)}(\hat Z^j,\ldots,\hat Z^{i + 1}),
\end{align}
where, $\wt Z^{i} = Z^i$, and for $j > i$, $\wt Z^{j} =R(Z^j,\wt Z^{j- 1})$. For $j > i$, we define $\hat Z^j = D(Z^j,\wt Z^{j - 1})$.
Then, by~\eqref{queue_iter} and~\eqref{queue_iter_2},  for $i < j \le n$,
\[D^{(j)}(Z^j,\ldots,Z^1) = D^{(j - 1)}(\hat Z^j,\ldots,\hat Z^{i + 1}, Z^{i - 1},\ldots,Z^1),
\]
 and so 
\be \label{eqn:marg}
(\eta^1,\ldots,\eta^{i - 1},\eta^{i + 1},\ldots,\eta^n) = \D^{(n - 1)}(Z^1,\ldots,Z^{i - 1},\hat Z^{i + 1},\ldots,\hat Z^n).
\ee 
By Theorem~\ref{multiline invariant distribution}, $\hat Z^{i + 1},\ldots,\hat Z^{n}$ are independent Brownian motions with drifts $\lambda_{i + 1},\ldots,\lambda_n$. Since these are functions of $Z^{i},\ldots,Z^n$, the functions $Z^1,\ldots,Z^{i - 1},\hat Z^{i + 1},\ldots,\hat Z^j$ are independent as well, and by~\eqref{eqn:marg}, \[
(\eta^1,\ldots,\eta^{i - 1},\eta^{i + 1},\ldots, \eta^n) \sim \mu^{(\lambda_1,\ldots,\lambda_{i - 1},\lambda_{i + 1},\ldots,\lambda_n)}. 
\]

\medskip \noindent \textbf{Part~\ref{scaling relations}}
We note that if $Z^1,\dotsc, Z^n$ are independent Brownian motions with drifts $\lambda_1, \dotsc, \lambda_n$ and $\wt Z^1,\ldots,\wt Z^n$ are independent Brownian motions with drifts $c(\lambda_1 + \nu),\ldots, c(\lambda_n + \nu)$,  then 
\be \label{eqn:BM_drift_dist_eq}
\{Z^1(t),\ldots,Z^n(t):t \in \R\} \deq \Big\{c\wt Z^1(t/c^2) - \nu t,\ldots, c\wt Z^n(t/c^2)-\nu t: t \in \R\Big\}. 
\ee
  Let $(\eta^1,\ldots,\eta^n) = \D^{(n)}(Z^1,\ldots,Z^2)$ and $(\wt \eta^1,\ldots,\wt \eta^n) = \D^{(n)}(\wt Z^1,\ldots,\wt Z^n)$. By Lemma~\ref{identity for multiple queuing mappings flipped},
\begin{align*}
   &\;\;\;\Big(c \tspa\wt \eta^k({t}/{c^2})- \nu t/c: t \in \R \Big)_{1 \le k \le n} \\
    &= \Big(c \wt Z^k({t}/{c^2}) - \nu t +  c\sup_{0 \leq t_1\leq t_{2} \cdots \leq t_{n - 1} < \infty} \sum_{i = 1}^{k - 1} (\wt Z^i(t_i) - \wt Z^{i + 1}(t_i) ) \\
    &\qquad\qquad\qquad - c\sup_{t/c^2 \leq t_1\leq t_{2} \cdots \leq t_{n - 1} < \infty} \sum_{i = 1}^{k - 1} (\wt Z^i(t_i)  - \wt Z^{i + 1}(t_i) ):t \in \R \Big)_{1 \le k \le n}. \\
    &= \Big(c \wt Z^k({t}/{c^2})- \nu t \\
    &+  \sup_{0 \leq t_1/c^2\leq \cdots \leq t_{n - 1}/c^2 < \infty} \sum_{i = 1}^{k - 1} (c \wt Z^i(t_i/c^2)- \nu t_i - c \wt Z^{i + 1}(t_i/c^2)+ \nu t_i) \\
    & - \sup_{t/c^2 \leq t_1/c^2\leq  \cdots \leq t_{n - 1}/c^2 < \infty} \sum_{i = 1}^{k - 1} (c \wt Z^i(t_i/c^2)- \nu t_i - c \wt Z^{i + 1}(t_i/c^2)+ \nu t_i): t \in \R \Big)_{1 \le k \le n} \\
    &\deq \Big(Z^k(t) +  \sup_{0 \leq t_1\leq t_{2} \cdots \leq t_{n - 1} < \infty} \sum_{i = 1}^{k - 1} (Z^i(t_i) - Z^{i + 1}(t_i)) \\
    &\qquad\qquad\qquad - \sup_{t \leq t_1\leq t_{2} \cdots \leq t_{n - 1} < \infty} \sum_{i = 1}^{k - 1} (Z^i(t_i) - Z^{i + 1}(t_i)):t \in \R \Big)_{1 \le k \le n} \\
    &= (\eta^k(t): t \in \R)_{1 \le k \le n},
\end{align*}
where in the second-to-last equality, we used the distributional equality~\eqref{eqn:BM_drift_dist_eq}.
\end{proof}

\begin{proof}[Proof of Theorem~\ref{dist of Busemann functions and Bm}] By Theorem~\ref{thm:summary of properties of Busemanns for all theta}\ref{general uniform convergence Busemanns}\ref{general uniform convergence:limits to infinity}, for each fixed $m \in \Z$, as $\theta \rightarrow \infty$, $h_m^{\theta}$ converges uniformly on compact sets to $B_m$.
The theorem then follows from Theorem~\ref{joint distribution of Busemann functions} and  Lemma~\ref{weak continuity and consistency}\ref{weak continuity}.
\end{proof}

\subsection{Proofs of results stated in Section~\ref{section:Busepp}} \label{sec:global_Buse_proofs}
We first prove Theorem~\ref{thm:Buse_inc} and then handle the remaining results from Section~\ref{section:Busepp}.

\begin{proof}[Proof of Theorem~\ref{thm:Buse_inc}]
 Set $(\eta^1,\eta^2) = \D^{(2)}(Z^{1},Z^2) = (Z^1,D(Z^2,Z^1))$ where $(Z^1,Z^2) \sim \nu^{0,\lambda}$, with $\lambda > 0$. Recall that this means that $Z^1$ and $Z^2$ are independent Brownian motions with drifts $0$ and $\lambda$, respectively.  Theorem~\ref{dist of Busemann functions and Bm} gives the first equality in distribution below:
\begin{align*}
X(\lambda;t) - X(0;t) &\deq \eta^2(t) - \eta^{1}(t) = D(Z^2,Z^{1})(t) - Z^{1}(t) \\
&= \sup_{0 \le s < \infty}\{ Z^{1}(s)-Z^2(s) \} - \sup_{t \le s < \infty}\{Z^{1}(s)-Z^2(s) \}.
\end{align*}
Since $Z^1$ and $Z^2$ are independent, the statement now follows from a direct application of Theorem~\ref{thm:dist of busemann increment}.
 \end{proof}

 \begin{proof}[Computation of the Laplace transform of the increment distribution]
  We first take $t = 1$. Denote the increment $Z = X(\lambda_2;1) - X(\lambda_1;1)$ with $\lambda = \lambda_2 - \lambda_1$. Consider the following lemma.
\begin{lemma} \label{Identity for Laplace Transform}
For a non-negative random variable $Z$ on $(\Omega,\F,\Pp)$ and $\alpha \in \R$, the following holds.
\[
\int_\Omega e^{-\alpha Z}\, d\Pp  = 1 - \alpha \int_0^\infty e^{-\alpha z}\Pp(Z > z)\, dz. 
\]
\end{lemma}
\begin{proof}
We use Tonelli's Theorem below.
\begin{align*}
    &-\alpha \int_0^\infty  e^{-\alpha z} \Pp(Z > z)\, dz = -\alpha \int_0^\infty \int_{\Omega}e^{-\alpha z}1(Z > z)\, d\Pp dz \\
    = &-\alpha \int_{\Omega}\int_0^\infty  e^{-\alpha z} 1(Z >z)\, dzd\Pp = -\alpha \int_\Omega \int_0^Z  e^{-\alpha z}\, dzd\Pp \\
    = &\int_\Omega (e^{-\alpha Z} - 1)\, d\Pp = \int_\Omega e^{-\alpha Z}d\Pp - 1.
\end{align*}
Rearranging this equation completes the proof.
\end{proof}

Using~\eqref{BuseCDF}, we get that 
\be \label{eqn:1-BuseCDF}
\Pp(Z > z) = \Phi\Big(-\zfraclambda\Big) - e^{\lambda z}(1 + \lambda z + \lambda^2)\Phi\Big(\negzfraclambda\Big) + \f{\lambda e^{\lambda z}}{\sqrt \pi} e^{-\f{(z + \lambda)^2}{4}}.
\ee
we compute the Laplace transform $\Ee[\exp(-\alpha Z)]$ via Lemma~\ref{Identity for Laplace Transform}, handling each of the three terms on in the sum on the right-hand side of~\eqref{eqn:1-BuseCDF}.  Using Fubini's theorem, we have
\begin{align}
    &\int_0^\infty -\alpha e^{-\alpha z} \Phi\Big(-\f{z - \lambda}{\sqrt 2}\Big)\,dz = \int_0^\infty \int_{-\infty}^{-\zfraclambda} -\alpha e^{-\alpha z} \f{1}{\sqtwopi} e^{-x^2/2}\, dxdz \nonumber \\
    = &\int_{-\infty}^{\fraclambda}\int_0^{\lambda - \sqrt 2 x}-\alpha e^{-\alpha z} \f{1}{\sqtwopi} e^{-x^2/2}\,dz\,dx = \int_{-\infty}^{\fraclambda} (e^{-\alpha(\lambda - \sqrt 2 x)} - 1)\f{1}{\sqtwopi} e^{-x^2/2}\, dx \nonumber  \\
    = &\int_{-\infty}^{\fraclambda} \f{e^{-\alpha(\lambda - \sqrt 2 x)}e^{-x^2/2}}{\sqtwopi} \, dx - \Phi\Big(\fraclambda\Big) 
    = e^{\alpha^2 - \alpha \lambda}\Phi\Big(\f{\lambda - 2\alpha}{\sqrt 2}\Big) - \Phi\Big(\fraclambda\Big).  \label{Laplace part 1}
\end{align}
Next, we compute the second term, using integration by parts in the second step
\begin{align}
    &\int_0^\infty \alpha e^{(\lambda - \alpha)z}(1 + \lambda z + \lambda^2)\Phi\Big(\negzfraclambda\Big)\,dz \nonumber  \\
    =&\int_{-\infty}^{-\fraclambda} \int_0^{-\sqrt 2 x - \lambda} \alpha e^{(\lambda - \alpha)z}\f{(1  + \lambda z + \lambda^2)e^{-x^2/2}}{\sqtwopi}\, dzdx \nonumber  \\
    = &\int_{-\infty}^{-\fraclambda} \Bigg((1 + \lambda(-\sqrt 2 x - \lambda) + \lambda^2)\f{\alpha}{\lambda - \alpha} e^{(\alpha - \lambda)(\sqrt 2 x + \lambda)} \nonumber  \\
    &\qquad\qquad - \f{\alpha(1 + \lambda^2)}{\lambda - \alpha} -\int_0^{-\sqrt 2 x - \lambda} \f{\alpha \lambda}{\lambda - \alpha}e^{(\lambda - \alpha)z}\, dz\Bigg) \f{e^{-x^2/2}}{\sqtwopi}\,dx \nonumber \\
    = &\int_{-\infty}^{-\fraclambda} \Bigg((1 -\sqrt 2 x \lambda)\f{\alpha}{\lambda - \alpha} e^{(\alpha - \lambda)(\sqrt 2 x + \lambda)} \nonumber \\
    &\qquad\qquad - \f{\alpha(1 + \lambda^2)}{\lambda - \alpha}-\f{\alpha \lambda}{(\lambda - \alpha)^2}e^{(\alpha - \lambda)(\sqrt 2 x + \lambda)} + \f{\alpha \lambda}{(\lambda - \alpha)^2}\Bigg)\f{e^{-x^2/2}}{\sqtwopi}\,dx \nonumber \\
    = &\f{\alpha(\alpha \lambda^2 + \alpha - \lambda^3)}{(\lambda - \alpha)^2}\Phi\Big(-\fraclambda\Big) \nonumber \\
    &\qquad\qquad+ \int_{-\infty}^{-\fraclambda}\f{\alpha(-\sqrt 2 x \lambda^2 - \alpha + \sqrt 2 x \alpha \lambda)e^{(\alpha - \lambda)(\sqrt 2 x + \lambda)}}{(\lambda - \alpha)^2}\f{e^{-x^2/2}}{\sqtwopi}\,dx \nonumber \\
    = &\f{\alpha(\alpha \lambda^2 + \alpha - \lambda^3)}{(\lambda - \alpha)^2}\Phi\Big(-\fraclambda\Big)  \nonumber \\
    &\qquad\qquad+ \int_{-\infty}^{-\fraclambda}\f{\alpha(-\sqrt 2 x \lambda^2 - \alpha + \sqrt 2 x \alpha \lambda)e^{-(x - \sqrt 2(\alpha - \lambda))^2/2 + \alpha(\alpha - \lambda)}}{(\lambda - \alpha)^2\sqtwopi}\,dx \nonumber\\
    = &\f{\alpha(\alpha \lambda^2 + \alpha - \lambda^3)}{(\lambda - \alpha)^2}\Phi\Big(-\fraclambda\Big) \nonumber\\
    &\qquad\qquad+ e^{\alpha^2 - \alpha \lambda}\Bigg( \f{ \lambda \alpha}{\sqrt \pi(\lambda - \alpha)}e^{-\lambda^2/4 + \alpha \lambda - \alpha^2} + \Big(2\lambda \alpha - \f{\alpha^2}{(\lambda - \alpha)^2}\Big) \Phi\Big(\f{\lambda - 2\alpha}{\sqrt 2}\Big) \Bigg) \nonumber\\
    = &\f{\alpha(\alpha \lambda^2 + \alpha - \lambda^3)}{(\lambda - \alpha)^2}\Phi\Big(-\fraclambda\Big) \nonumber \\
    &\qquad\qquad+ \f{\lambda \alpha}{\sqrt \pi(\lambda - \alpha)}e^{-\lambda^2/4} + e^{\alpha^2 - \alpha \lambda}\Big(2\lambda \alpha - \f{\alpha^2}{(\lambda - \alpha)^2}\Big)\Phi\Big(\f{\lambda - 2\alpha}{\sqrt 2}\Big). \label{Laplace part 2}
\end{align}
Lastly, we handle the third term:
\begin{align}
    \int_0^\infty -\alpha e^{(\lambda - \alpha)z}\f{\lambda}{\sqrt \pi }e^{-(z + \lambda)^2/4}\, dz = -2\alpha \lambda e^{\alpha^2 - \alpha \lambda} \Phi\Big(\f{\lambda - 2\alpha}{\sqrt 2}\Big). \label{Laplace part 3}
\end{align}

Adding $1$,~\eqref{Laplace part 1},\eqref{Laplace part 2}, and\eqref{Laplace part 3}, we get 
\begin{align*}
    \Ee[\exp(-\alpha Z)] &= e^{\alpha^2 - \alpha \lambda}\Phi\Big(\f{\lambda - 2\alpha}{\sqrt 2}\Big)\Big(1 - \f{\alpha^2}{(\lambda - \alpha)^2}\Big) \nonumber\\
    &\qquad\qquad+ \Phi\Big(-\fraclambda\Big)\Big(1 + \f{\alpha \lambda}{(\lambda - \alpha)^2} - \f{\alpha(1 + \lambda^2)}{\lambda - \alpha}\Big) + \f{\lambda \alpha}{\sqrt \pi (\lambda - \alpha)}e^{-\lambda^2/4}.
\end{align*}
The statement for general $t$ follows from Lemma~\ref{weak continuity and consistency}\ref{scaling relations}, replacing $\alpha$ with $\alpha \sqrt t$ and $\lambda$ with $\lambda \sqrt t$.
 \end{proof}

 Theorem~\ref{thm:Busemann jump process intro version} is proved by applying the following theorem, which gives a condition for a general increment-stationary process to be a jump process.  
\begin{theorem} \label{thm:jump process condition}
On a probability space $(\Omega,\F,\Pp)$, let $Y = \{Y(t): t \ge 0\}$ be a nondecreasing, increment-stationary process such that the following three conditions hold:
\begin{enumerate} [label=\rm(\roman{*}), ref=\rm(\roman{*})]  \itemsep=3pt    
    \item \label{finitemean} $I := \Ee[Y(1)-Y(0)] < \infty$.
    \item $\Pp[Y(t) = Y(0)] \in (0,1)$ for sufficiently small $t > 0$.
    \item \label{liminf assumption} $c := \liminf_{t \searrow 0} \Ee[Y(t) - Y(0)|Y(t) > Y(0)] > 0$.
\end{enumerate} 
Then, with probability one, the paths of  $t \mapsto Y(t)$ are step functions with finitely many jumps in each bounded interval. For each $t \ge 0$, there is a jump at $t$ with probability $0$. For  $a < b$, the expected number of jump points in the interval $[a,b]$ equals
\[
\f{(b - a)I}{c} = \f{\Ee[Y(b) - Y(a)]}{\liminf_{t \searrow 0}\Ee[Y(t) - Y(0)|Y(t) > Y(0)]  }.
\]
\end{theorem}
\begin{remark} \label{rmk:mean number jumps}
The claim that $(b - a)I = \Ee[Y(b) - Y(a)]$ follows from increment-stationarity and the fact that $Y$ is nondecreasing, as follows. By increment-stationarity, it suffices to show that $\Ee[Y(t) - Y(0)] = tI$ for all $t > 0$. Since $Y$ is nondecreasing, $t \mapsto \Ee[Y(t) - Y(0)]$ is nondecreasing, so it further suffices to show that $\Ee[Y(t) - Y(0)] = tI$ just for rational $t > 0$. For any integer $k$,
\[
\Ee[Y(k) - Y(0)] = \sum_{i = 1}^k \Ee[Y(i) - Y(i -1)] = k\Ee[Y(1) - Y(0)] = kI.
\]
Then for positive integers $r$ and $k$,
\[
rI = \Ee[Y(r) - Y(0)] = \sum_{i = 1}^k\Ee[Y(ri/k) - Y(r(i -1 )/k)] = k\Ee[Y(r/k) - Y(0)],
\]
and $\Ee[Y(r/k)-Y(0)] = \tf{r}{k} I$.
\end{remark}
\begin{remark}
Heuristically, we can think of Condition~\ref{liminf assumption} in the following way: on average, the size of the jumps are bounded away from $0$, and therefore the jumps cannot accumulate because an increment of the process itself has finite expectation by Condition~\ref{finitemean}.   
\end{remark}

\begin{proof}[Proof of Theorem~\ref{thm:jump process condition}]
We show there are finitely many jumps in the interval $[a,b] = [0,1]$, and the general case follows by increment-stationarity. Consider discrete versions of the process $Y$ as follows. For $n \in \Z_{>0}$, let $D_n = \{\f{j}{2^n}: j \in \Z, 0 \le j \le 2^n\}$, and consider the process $Y_n := \{Y(t): t \in D_n\}$. Let $J_n$ be the number of jumps of $Y_n$, i.e.,
\[
J_n = \sum_{j =1}^{2^n} \1\Bigl(Y\bigl(\tf{j}{2^n}\bigr) > Y\bigl(\tf{j -1}{2^n}\bigr)\Bigr)
\]
Then, $J_n$ is nondecreasing in $n$, so it has a limit, denoted as the random variable $J_\infty$. Let $K \in \{0,1,2,\ldots\} \cup \{\infty\}$ be the number of points of increase of $Y$ on the interval $[0,1]$. Specifically, a point $t \in (0,1)$ is a point of increase, if  $Y(t + \ve)  > Y(t -\ve)$ for all $\ve > 0$. We say $0$ is a point of increase if $Y(t) > Y(0)$ for all $t > 0$, and we likewise say that $1$ is a point of increase if $Y(t) < Y(1)$ for all $t < 1$.
We now show that $K \le J_\infty$. If $K < \infty$, let $k = K$, and otherwise, let $k$ be an arbitrary positive integer. It suffices to show that $J_\infty \ge k$.  By definition of $k$, we may choose $k$ points of increase $t_1<\dotsm<t_k$. First, we handle the case where $t_i \in (0,1)$ for all $i$. Then, for all sufficiently large $n$, there exist $n$-dependent positive integers $0 < j_1 < \cdots < j_k < 2^n$ so that for each $i$, $j_{i + 1} > j_i + 2$,  and 
\be \label{eqn:ti in interval}
\f{j_i -  1}{2^n} < t_i < \f{j_i + 1}{2^n}.
\ee
 Since $t_i$ is a point of increase and $Y$ is nondecreasing, $Y(\f{j_i + 1}{2^n}) > Y(\f{j_i - 1}{2^n})$. Therefore, $Y(\f{j_i + 1}{2^n}) > Y(\f{j_i}{2^n})$ or $Y(\f{j_i}{2^n}) > Y(\f{j_i - 1}{2^n})$. By assumption that $j_{i + 1} > j_i + 2$, the intervals $[\f{j_i - 1}{2^n},\f{j_i + 1}{2^n}]$ are mutually disjoint, so $J_n \ge k$ and therefore $J_\infty \ge k$. The case where $t_1 = 0$ or $t_k = 1$ is handled similarly.

Now, we show that $\Pp(J_\infty < \infty) = 1$. 
Let 
\[
c_n = \Ee[Y(2^{-n}) - Y(0)|Y(2^{-n}) > Y(0)]
\]
Then, using increment-stationarity,
\begin{align}
 \Ee[Y(1) - Y(0)] &= \sum_{j = 1}^{2^n} \Ee\Big[Y\big(\tf{j}{2^n}\big) - Y\big(\tf{j - 1}{2^n}\big)\Big]   \nonumber \\
    &= \sum_{j = 1}^{2^n} \Ee\Big[Y\big(\tf{j}{2^n}\big) - Y\big(\tf{j - 1}{2^n}\big)\Big|Y\big(\tf{j}{2^n}\big) > Y\big(\tf{j - 1}{2^n}\big) \Big]\Pp\Big(Y\big(\tf{j}{2^n}\big) > Y\big(\tf{j - 1}{2^n}\big)\Big) \nonumber \\
    &= c_n \sum_{j = 1}^{2^n}  \Pp\Big(Y\big(\tf{j}{2^n}\big) > Y\big(\tf{j - 1}{2^n}\big)\Big) \nonumber
    = c_n \Ee[J_n].
\end{align}
By assumptions~\ref{finitemean} and~\ref{liminf assumption} and the monotone convergence theorem,
\be \label{eqn:Jinf_mean}
\Ee[J_\infty] = \lim_{n \rightarrow \infty} \Ee[J_n] = \lim_{n \rightarrow \infty} \f{\Ee[Y(1) - Y(0)]}{c_n} < \infty. 
\ee
Therefore, $\Pp(J_\infty < \infty) = 1$. Since $K \le J_\infty$, with probability one, $Y$ has only finitely many points of increase on $[0,1]$. Therefore, with probability one, $Y:[0,1]\to \R$ is locally constant except at the finitely many jump points. Hence, for each $t \in (0,1)$, the left and right limits of $Y$ at $t, Y(t \pm)$ exist. The limits $Y(0+)$ and $Y(1-)$ exist as well. Since $Y$ is increasing, for each $t \in(0,1)$ and $\ve > 0$, we can apply Remark~\ref{rmk:mean number jumps} and Assumption~\ref{finitemean} to get
\[
\Ee[Y(t+) - Y(t-)] \le \Ee[Y(t + \ve) - Y(t - \ve)] = 2\ve \Ee[Y(1) - Y(0)] < \infty.
\]
Sending $\ve \searrow 0$, the left-hand side is $0$ and therefore, a jump occurs at time $t$ with probability $0$. Similar arguments apply to $t = 0$ and $t = 1$. Therefore, there exists an event of probability one, $\Omega_{\Q_2}$ on which $Y$ has no jumps at points of the form $\f{j}{2^n}$ for positive integers $j$ and $n$.

To compute the mean number of jumps, we show that $J_\infty = K$ on the event $\Omega_{\Q_2}$. We already showed that $K \le J_\infty$, so it remains to show $J_\infty \le K$.

 We start by showing that if $Y(b) > Y(a)$ for some $a < b$, there must be some point of increase in the interval $[a,b]$. We prove this as follows: let $c$ be the midpoint of $a$ and $b$. Then, since $Y$ is nondecreasing, either $Y(b) > Y(c)$ or $Y(c) > Y(a)$. If, without loss of generality, $Y(b) > Y(c)$, then we can bisect the interval again with midpoint $d$ and get that $Y(b) > Y(d)$ or $Y(d) > Y(c)$, where $d$ is the midpoint of $a$ and $b$. Inductively, this constructs a sequence of nested intervals $[a_n,b_n] \subseteq [a_{n - 1},b_{n - 1}]\subseteq [a,b]$, where $[a_n,b_n]$ is either the left or right half of the previous interval. Then, $a_n$ is nondecreasing and $b_n$ is nonincreasing and $b_n - a_n \rightarrow 0$. Then, set $t = \lim_{n \rightarrow \infty} a_n = \lim_{n \rightarrow \infty} b_n$, and we have that $t \in [a_n,b_n]$ for all $n$. If $t \in (0,1)$, then for all $\ve > 0$,  we may choose $n$ large enough so that, because $Y$ is nondecreasing,
\[
Y(t + \ve) - Y(t - \ve) \ge Y(b_n) - Y(a_n) > 0.
\]
Hence, $t$ is a point of increase. The case where $t = 0$ or $1$ is handled similarly.

Now, we show that on $\Omega_{\Q_2}$, $J_n \le K$ for all $n$. By definition, $J_n$ is the number of integers $0 < j \le 2^n$ such that $Y(j2^{-n}) > Y((j -1)2^{-n})$. For each such $j$, we just showed that there must be a point of increase in $[(j - 1)2^{-n},j2^{-n}]$, and on the event $\Omega_{\Q_2}$, that point of increase must lie in the interior of the interval. Thus, $J_n \le K$, and $J_\infty \le K$, so $J_\infty = K$ on $\Omega_{\Q_2}$. Equation~\eqref{eqn:Jinf_mean} computes the mean number of jump points. 
\end{proof}

\begin{proof}[Proof of Theorem~\ref{thm:Busemann jump process intro version}]
By Theorem~\ref{dist of Busemann functions and Bm} we can realize the distribution of the process as a function of independent Brownian motions  $Z^1,\ldots,Z^n$ with respective drifts $\lambda_1,\ldots,\lambda_n$: 
\[
\big(X(\lambda_1;t),\ldots,X(\lambda_n;t)\big) \deq (\eta^1(t),\ldots ,\eta^n(t)) = \D^{(n)}(Z^1,\ldots,Z^n)(t). 
\]
From this,  $\eta^1(t) = Z^1(t)$, and  by Lemma~\ref{identity for multiple queuing mappings flipped}, for $2 \le k \le n$,
\begin{align*}
\eta^k(t)
= Z^1(t) &+  \sup_{0 \leq t_1\leq t_{2} \cdots \leq t_{n - 1} < \infty} \sum_{i = 1}^{k - 1} (Z^i(t_i) - Z^{i + 1}(t_i)) \\
&\qquad- \sup_{t \leq t_1\leq t_{2} \cdots \leq t_{n - 1} < \infty} \sum_{i = 1}^{k - 1} (Z^i(t_i) - Z^{i + 1}(t_i)).
\end{align*}
Hence, the $Z^1(t)$ terms in $\eta^{k + 1}(1) - \eta^k(t)$ cancel out. 
With independent, standard Brownian motions $W^1,\ldots,W^n$, we can write 
\[
Z^i(t_i) - Z^{i + 1}(t_i) = W^i(t_i) - W^{i + 1}(t_i) -(\lambda_{i + 1} - \lambda_i)t_i.
\]
Hence the distribution of the vector of increments 
\[
\big(X(\lambda_2;t) - X(\lambda_1;t),X(\lambda_3;t) - X(\lambda_2;t),\ldots, X(\lambda_n;t) - X(\lambda_{n - 1};t)\big)
\]
depends only on the differences $\lambda_2 - \lambda_1,\ldots,\lambda_{n} - \lambda_{n - 1}$ and  not on the individual values  $\lambda_i$. This shows increment-stationarity of the process. 
To complete the proof, we show that $\Ee[X(\lambda;t) - X(0;t)]  = \lambda t$ and that 
\be \label{Xlambda_liminf}
\liminf_{\lambda \searrow 0} \Ee[X(\lambda;t) - X(0;t)|X(\lambda;t) > X(0;t)]  = \f{\sqrt{\pi t}}{2},
\ee
allowing us to invoke Theorem~\ref{thm:jump process condition}. The fact that $\Ee[X(\lambda;t) - X(0;t)] = \lambda t$ follows since $h_m^{1/\lambda^2}(t)$ is a Brownian motion with drift $\lambda$ and $t \mapsto X(0;t)$ is a Brownian motion with zero drift. Since $X(\lambda;t) - X(0;t)$ is nonnegative,
\begin{align*}
\Ee[X(\lambda;t) - X(0;t)|X(\lambda;t) > X(0;t)] 
= \f{\Ee[X(\lambda;t) - X(0;t)]}{\Pp(X(\lambda;t) > X(0;t))}
= \f{\lambda t}{\Pp(X(\lambda;t) > X(0;t))}. 
\end{align*}
 By~\eqref{BuseCDF},  for $t,\lambda > 0$,
\[
\Pp(X(\lambda;t)> X(0;t) ) =  1 - (2 + \lambda^2 t) \Phi\bigl(-\lambda\sqrt{{t}/{2}}\,\bigr) + \lambda e^{-\f{\lambda^2 t}{4}} \sqrt{{t}/{\pi}\,}. 
\]
Substitute this to the denominator above and 
apply L'H{\^o}pital's rule to deduce~\eqref{Xlambda_liminf}. Hence, we may apply Theorem~\ref{thm:jump process condition}. By Remark~\ref{rmk:any_interval_same}, for $0 < \gamma \le \infty$, the mean number of directions $\theta$ satisfying $h_m^{\theta+}(s,t) < h_m^{\theta -}(s,t)$ is distributed as the number of jumps of $\lambda \mapsto X(\lambda;s + (t - s)) - X(\lambda;s)$ in the interval $\lambda \in [\f{1}{\sqrt \delta},\f{1}{\sqrt \gamma}]$, which has mean
\[
\Big(\f{1}{\sqrt \gamma} - \f{1}{\sqrt \delta}\Big)\lambda (t - s)  \Bigg(\f{\sqrt{\pi (t - s)}}{2}\Bigg)^{-1} = 2\sqrt{\f{t - s}{\pi}} \Big(\f{1}{\sqrt \gamma} - \f{1}{\sqrt \delta}\Big).
\]
The almost sure existence of $\ve > 0$ such that $X(\lambda;t) = X(0;t) = B_0(t)$ for $\lambda \in [0,\ve)$ follows because Theorem~\ref{thm:jump process condition} states that $0$ is a jump point of $\lambda \mapsto X(\lambda;t)$ with zero probability.  The limit $\lim_{\lambda \rightarrow \infty} X(\lambda;t) = \infty$
follows by monotonicity and because $X(\lambda;t) = h_0^{1/\lambda^2}(t) \sim \Nor(t\lambda,t)$.
\end{proof}

\begin{proof}[Proof of Corollary~\ref{cor:inc_dist_conseq}]
\textbf{Part~\ref{itm:convNor}:}
In~\eqref{BuseCDF}, set $y = z + \lambda t$ and send $\lambda \rightarrow \infty$. The limit is $\Phi\bigl(\f{y}{\sqrt{2t}}\bigr)$.

\medskip \noindent \textbf{Part~\ref{itm:non_independent}:} 
 For $\lambda_2 > \lambda_1 > 0$, 
\be \label{Xlambda2}
X(\lambda_2;t) = X(\lambda_2;t) - X(\lambda_1;t) + X(\lambda_1;t).
\ee
Recall that $X(\lambda;t) = h_m^{(1/\lambda^2)-}(t)$, and for each fixed $\lambda$, $h_m^{1/\lambda^2}(t)$ is a two-sided Brownian motion with drift $\lambda$ by Theorem~\ref{thm:summary of properties of Busemanns for all theta}\ref{Buse_marg_dist}. Then, the left-hand side of~\eqref{Xlambda2} has variance $t$.  If, by way of contradiction, $X(\lambda_1;t)$ is independent of $X(\lambda_2;t) - X(\lambda_1;t)$, then the right-hand side of~\eqref{Xlambda2} has variance $\Var(X(\lambda_2;t) - X(\lambda_1;t)) + t$, implying that  $\Var(X(\lambda_2;t) - X(\lambda_1;t)) = 0$ and $X(\lambda_2;t)  = X(\lambda_1;t)$ is zero, almost surely. This cannot be true because $X(\lambda_1;t)$ and $X(\lambda_2;t)$ have different distributions: $X(\lambda_i;t)$ is normal with mean $\lambda_i t$ for $i = 1,2$.

For the second statement, by monotonicity, $X(\lambda_2;t) - X(0;t) = 0$ if and only if $X(\lambda_2;t) - X(\lambda_1;t) = 0$ and $X(\lambda_1;t) - X(0;t) = 0$. Then, if by contradiction, $\lambda \mapsto X(\lambda;t)$ has independent increments, then for $0 < \lambda_1 < \lambda_2$, 
\[
\Pp(X(\lambda_2;0) - X(0;t) = 0) = \Pp(X(\lambda_2;t) - X(\lambda_1;t) = 0)\Pp(X(\lambda_1;t)-X(0;t) = 0),
\]
or equivalently,
\be \label{meminc}
\Pp(X(\lambda_2;0) - X(0;t) = 0|X(\lambda_1;t) - X(0;t) = 0) = \Pp(X(\lambda_2;t) - X(\lambda_1;t) = 0).
\ee
If we let $T(t)$ be the time of the first jump of the process $\lambda \mapsto X(\lambda;t)$, then by the increment-stationarity of Theorem~\ref{thm:Busemann jump process intro version}, Equation~\eqref{meminc} is equivalent to 
\be \label{memoryless}
\Pp(T(t) > \lambda_2|T(t) > \lambda_1) = \Pp(T(t) > \lambda_2 - \lambda_1).
\ee
Note that $\Pp(T(t) > \lambda) = F(0,\lambda,t)$, which by~\eqref{eqn:0inc}, is not an exponential distribution and therefore not a memoryless distribution. Thus,~\eqref{memoryless} fails. 
\end{proof}

\subsection{Proof of Theorems~\ref{thm:qualitative_Buse} and~\ref{thm:Theta properties}}
We first prove a rather technical seeming theorem, but many of whose statements have natural geometric meaning. For example, Part \ref{convBM} says that uniformly for directions sufficiently close to horizontal, the rightmost geodesic must travel horizontally to some distance bounded away from $0$, and thereby the  horizontal Busemann process coincides with the environment of Brownian motions, throughout a given interval (See Lemma~\ref{lemma:equality of busemann to weights of BLPP}). Part \ref{itm:coupletoBM_smallt} gives   a dual statement: uniformly for directions bounded away from the vertical,  the  horizontal Busemann process coincides with the environment of Brownian motions at least for some nondegenerate interval.

\begin{theorem} \label{thm:coupled_BMs_technical}
There exists an event of full probability, on which the following hold. 
\begin{enumerate} [label=\rm(\roman{*}), ref=\rm(\roman{*})]  \itemsep=3pt
\item \label{itm:ineqboundary} For all $S < T \in \R$ and $\theta > 0$,  if $h_m^{\theta +}(s,t) < h_m^{\theta - }(s,t)$ for some $s < t \in [S,T]$, then $h_m^{\theta +}(S,T) < h_m^{\theta -}(S,T)$. 
\item \label{itm:jump_points_increasing} In particular, for every $m \in \Z$ and $s < t \in \R$, the paths of the process $\theta \mapsto h_m^{\theta \pm}(s,t)$ are the right and left continuous versions of a nonincreasing step function with discrete jumps. If, for some $s < t$, $\theta^\star$ is a jump point for $\theta \mapsto h_m^{\theta\pm}(s,t)$, then for all $S < s$ and $T > t$,  $\theta^\star$ is also a jump point for the process $\theta \mapsto h_m^{\theta\pm}(S,T)$. 
\item \label{convBM} For each $m \in \Z$ and each compact set $K \subseteq \R$, there exists a random $\eta = \eta(m,K) > 0$ such  that for all $\theta > \eta$ and $s,t \in K$, $h_m^{\theta+}(s,t) = h_m^{\theta-}(s,t) = B_m(s,t)$.
\item \label{itm:strong_unif_convergence} For all $\theta > 0,m \in \Z$ and compact $K \subseteq \R$, there exists a random $\ve = \ve(\theta,K,m) > 0$  such that whenever $\theta - \ve < \gamma <\theta < \delta < \theta  + \ve$, $\sigg \in \{-,+\}$, and $s < t \in K$, $h_m^{\gamma \sig}(s,t) = h_m^{\theta -}(s,t)$ and $h_m^{\delta \sig }(s,t) = h_m^{\theta + }(s,t)$.
\item \label{convBuse} More generally, for each compact set $K \subseteq \Z \times\R$ and $\theta > 0$,  there exists a random $\ve = \ve(K,\theta)> 0$ such that whenever $\mbf x,\mbf y \in K$, $\theta - \ve < \gamma < \theta < \delta < \theta + \ve$, and $\sigg \in \{-,+\}$,
\[
\B^{\gamma \sig}(\mbf x,\mbf y) = \B^{\theta -}(\mbf x,\mbf y),\qquad\text{and}\qquad\B^{\delta \sig}(\mbf x,\mbf y) = \B^{\theta +}(\mbf x,\mbf y).
\]
    \item \label{itm:coupletoBM_smallt} For all $\eta > 0$ and $m \in \Z$, there exists a random $\ve = \ve(m,\eta) > 0$ such that,  for all $|t| \le \ve$  and $\theta > \eta$,
    \[
    h_m^{\theta +}(t) = h_m^{\theta -}(t) = B_m(t).
    \]
    \item \label{itm:hm_to_infty}
    For all  $m \in \Z$, $\sigg \in \{-,+\}$, and $s < t \in \R$, 
    \[\lim_{\theta \searrow 0} h_m^{\theta \sig}(s,t) = + \infty, \qquad{and}\qquad \lim_{\theta \rightarrow \infty} v_m^{\theta \sig}(t) = +\infty.\]
    \item \label{itm:infinite_inc_points} For each $m \in \Z$ and  $s < t \in \R$, the process $\{h_m^{\theta \sig}(s,t): \theta > 0\}$ has infinitely many points of decrease,  whose unique  accumulation point  is $\theta = 0$.
\end{enumerate}  
\end{theorem}

\begin{proof}
\noindent \textbf{Part~\ref{itm:ineqboundary}} holds on the event $\Omega_2$ as follows. Theorem~\ref{thm:summary of properties of Busemanns for all theta}\ref{general monotonicity Busemanns} implies that for $S \le s \le t \le T$, $0 \le \theta \le \delta \le \infty$, and $\sigg_1,\sigg_2 \in \{-,+\}$ (if $\theta = \delta$, we require $\sigg_1 = -$ and $\sigg_2 = +$),
\[
h_m^{\delta \sig_2}(S,s) + h_m^{\delta \sig_2}(t,T) \le h_m^{\theta \sig_1}(S,s) + h_m^{\theta \sig_1}(t,T). 
\]
Here, we define $h_m^\infty(s,t) = B_m(s,t)$.
This inequality can be rearranged to get
\be \label{comp_mont}
0 \le  h_m^{\theta \sig_1}(s,t) - h_m^{\delta \sig_2}(s,t) \le h_m^{\theta \sig_1}(S,T) - h_m^{\delta \sig_2 }(S,T).
\ee
The case $\theta = \delta$ and $\sigg_1 = -,\sigg_2 = +$ proves~\ref{itm:ineqboundary}.

\be \label{omega3} \text{For the remaining parts, let $\Omega_3$ be the subset of $\Omega_2$ on which the following hold:}
\ee
\begin{enumerate}  \itemsep=3pt 
\item \label{itm:jumpallT} The paths of $\theta \mapsto h_m^{\theta \pm}(S,T)$ are step functions with discrete jumps for all $m,S<T \in \Z$. For such $m,S,T$, there exists $\eta = \eta(S,T) > 0$ such that for $\theta > \eta$ and $\sigg \in\{-,+\}$, $h_m^{\theta \sig}(S,T) = B_m(S,T)$. 
\item \label{itm:equalendpt} For each $\theta \in \Q_{>0}$ and $m \in \Z$, there exists $N \in \Z_{> 0}$ such that $h_m^{\theta}(\pm N^{-1}) = B_m(\pm N^{-1})$.
\item \label{itm:limit}For each $m \in \Z$, $s < t \in \R$, and $\sigg \in \{-,+\}$, $\lim_{\theta \searrow 0} h_m^{\theta \sig}(s,t) = +\infty$.
\item \label{itm:Buse_dir} For every  $m \in \Z$, $\theta > 0$, and $\sigg \in \{-,+\}$, $\lim_{t \rightarrow \infty} \f{h_m^{\theta \sig}(t)}{t} = \f{1}{\sqrt \theta}$.
\item \label{itm:supBm} For every $m \in \Z$, 
\[
\sup_{0 \le s < \infty}\{B_{m -1 }(s) - B_m(s)\} = +\infty.
\]
\end{enumerate}

We first show that $\Omega_3$ has probability one, and then show that the remaining parts of the theorem hold on this event. By Theorem~\ref{thm:Busemann jump process intro version}, Condition~\eqref{itm:jumpallT} holds with probability one. Next, rearranging $B_m(s,t) \le h_m^{\theta \sig}(s,t)$ for $s < t$ gives, for $a < t < b$,
\be \label{mont}
h_m^{\theta \sig}(a) - B_m(a) \le  h_m^{\theta \sig}(t) - B_m(t) \le h_m^{\theta \sig}(b) - B_m(b).
\ee
 Thus, for $\theta \in \Q_{>0}$, if $h_m^{\theta}(\pm N^{-1}) = B_m(\pm N^{-1})$, then also $h_m^{\theta}(\pm (N + 1)^{-1}) = B_m(\pm (N + 1)^{-1})$. Therefore, 
 \[
 \Pp\biggl(\bigcup_{N \in \Z_{>0}} \bigl\{h_m^{\theta}(\pm N^{-1}) = B_m(\pm N^{-1})\bigr\}\biggr) = \lim_{N \rightarrow \infty} \Pp\left(h_m^{\theta}(\pm N^{-1}) = B_m(\pm N^{-1})\right) = 1,
 \]
 where the last equality follows from Theorem~\ref{dist of Busemann functions and Bm} and \eqref{eqn:0inc} with $t \searrow 0$. Therefore, Condition~\ref{itm:equalendpt} holds with probability one. 
  Next, we show that Condition~\eqref{itm:limit} holds with probability one. For all $m \in \Z$ and a countable dense set of pairs  $s < t$, this follows from Theorem~\ref{thm:Busemann jump process intro version} and Remark~\ref{rmk:any_interval_same}. The monotonicity of Equation~\eqref{comp_mont} with $\theta \searrow 0$ and  $\delta$ fixed extends Condition~\eqref{itm:limit} to all $S < T \in \R$ on a single event of probability one.

 Condition~\eqref{itm:Buse_dir} holds with probability one for all $\theta \in \Q_{>0}$, since each $h_m^\theta$ is a Brownian motion with drift $\f{1}{\sqrt \theta}$ (Theorem~\ref{thm:summary of properties of Busemanns for all theta}\ref{Buse_marg_dist}). The monotonicity of~\eqref{weakmont} extends this to all $\theta > 0$ and $\sigg \in \{-,+\}$. Since $B_{m - 1}$ and $B_m$ are independent,    $B_{m - 1} - B_{m}$ is a variance $2$ Brownian motion. Hence,  Condition~\eqref{itm:supBm} holds with probability one, and $\Pp(\Omega_3) = 1$, as desired. We now prove the remaining parts of the theorem.

\medskip \noindent \textbf{Part~\ref{itm:jump_points_increasing}:} This now follows from Condition~\eqref{itm:jumpallT} of the event $\Omega_3$ and Part~\ref{itm:ineqboundary}. 
 
 \medskip \noindent \textbf{Parts~\ref{convBM}}:  Let $m \in \Z$, $S,T \in \Z_{>0}$, and, without loss of generality, let $K = [S,T]$. By Condition~\eqref{itm:equalendpt} and the $\delta = \infty$ case of~\eqref{comp_mont}: There exists $\eta = \eta(m,K) > 0$ such that, when $\theta > \eta$, $\sigg \in \{-,+\}$, and $s < t \in K$,
 \[
  0 \le h_m^{\theta \sig}(s,t) - B_m(s,t) \le h_m^{\theta \sig}(S,T) -B_m(S,T) = 0.
 \]
 
 \medskip \noindent \textbf{Part~\ref{itm:strong_unif_convergence}:} This follows similarly as the proof Part~\ref{convBM}: By Condition~\eqref{itm:jumpallT} and~\eqref{comp_mont}, when $m,S,T \in \Z$ with $S < T$ and $\theta > 0$, then setting $K = [S,T]$, there exists $\ve = \ve(\theta,m,K) > 0$ such that, when $\theta - \ve < \gamma < \theta < \delta < \theta + \ve$ and $\sigg \in \{-,+\}$,
 \begin{align*}
     0 &= h_m^{\theta+}(s,t) - h_m^{\delta \sig}(s,t) = h_m^{\theta+}(S,T) - h_m^{\delta \sig}(S,T),\quad\text{and} \\
     0 &= h_m^{\gamma \sig}(s,t) - h_m^{\theta-}(s,t) = h_m^{\gamma \sig}(S,T) - h_m^{\theta -}(S,T).
 \end{align*}

 \medskip \noindent \textbf{Part~\ref{convBuse}}: By Part~\ref{itm:strong_unif_convergence} and the additivity of Theorem~\ref{thm:summary of properties of Busemanns for all theta}\ref{general additivity Busemanns}, it is sufficient to prove that for each $m \in \Z$, compact  $K \subseteq \R$ and $\theta > 0$, there exists $\ve = \ve(m,K,\theta)$ such that $v_{m+1}^{\gamma \sig}(t) = v_{m + 1}^{\theta-}(t)$ and $v_{m +1}^{\delta \sig}(t) = v_{m + 1}^{\theta +}(t)$ whenever $t \in K$ and $\theta - \ve < \gamma < \theta < \delta < \theta + \ve$. By Theorem~\ref{thm:summary of properties of Busemanns for all theta}\ref{general queuing relations Busemanns} and Theorem~\ref{existence of semi-infinite geodesics intro version}\ref{itm:monotonicity of semi-infinite jump times}\ref{itm:monotonicity in theta}, for $\delta < \theta + 1$, 
 \be \label{vtrun}
 v_{m +1}^{\delta \sig}(t) = \sup_{t \le s < \infty}\{B_{m}(t,s) - h_{m + 1}^{\delta \sig}(t,s)\} = \sup_{t \le s \le \tau_{(m,t),m }^{(\theta + 1)+,R}} \{B_{m}(t,s) - h_{m + 1}^{\delta \sig}(t,s)\}.
 \ee
 By Part~\ref{itm:strong_unif_convergence}, there exists $\ve > 0$, such that, for $\delta \in (\theta,\theta + \ve)$, the right-hand side of~\eqref{vtrun} is equal to
 \[
 \sup_{t \le s \le \tau_{(m ,T),m }^{(\theta + 1)+,R}} \{B_{m}(t,s) - h_{m + 1}^{\theta +}(t,s)\} = v_{m + 1}^{\theta +}(t),
 \]
 where $T = \sup K$. The result for $\gamma < \theta$ is proved similarly.

\medskip \noindent  \textbf{Part~\ref{itm:coupletoBM_smallt}:} Let $m \in \Z$, $\eta > 0$ and let $\eta_1 < \eta$ be rational. By Condition~\eqref{itm:equalendpt} of $\Omega_3$ and~\eqref{mont}, there exists $\ve  = \ve(m,\eta_1)$ such that $B_m(t) = h_m^{\eta_1}(t)$ for all $|t| \le \ve$. Then, by the monotonicity of Theorem~\ref{thm:summary of properties of Busemanns for all theta}\ref{general monotonicity Busemanns}, for all $\theta > \eta > \eta_1$, $\sigg \in \{-,+\}$, and $|t| \le \ve$, $B_m(t) = h_m^{\theta \sig}(t)$ as well.   

 \medskip \noindent \textbf{Part~\ref{itm:hm_to_infty}:} The first limit is exactly Condition~\eqref{itm:limit} of the event $\Omega_3$. For the second limit, we use the representation of Theorem~\ref{thm:summary of properties of Busemanns for all theta}\ref{general queuing relations Busemanns} to get
 \[
 v_{m}^{\theta \sig}(t) = \sup_{t \le s < \infty}\{B_{m - 1}(t,s) - h_{m}^{\theta \sig}(t,s)\}.
 \]
 By Theorem~\ref{thm:summary of properties of Busemanns for all theta}\ref{general uniform convergence Busemanns}\ref{general uniform convergence:limits to infinity}, $h_m^{\theta \sig}$ converges to $B_m$ uniformly on compact sets as $\theta \rightarrow \infty$.  Therefore, for any $T \ge t$,
 \begin{align*}
 \liminf_{\theta \rightarrow \infty} v_m^{\theta \sig}(t) &=\liminf_{\theta \rightarrow \infty} \sup_{t \le s < \infty}\{B_{m - 1}(t,s) - h_{m}^{\theta \sig}(t,s)\} \\
 &\ge \liminf_{\theta \to \infty} [h_{m}^{\theta \sig}(t) - B_{m - 1}(t) + \sup_{t \le s \le T}\{B_{m - 1}(s) - h_m^{\theta \sig}(s)\}] \\
 &= B_{m}(t) - B_{m - 1}(t) + \sup_{t \le s \le T}\{B_{m -1 }(s) - B_m(s)\}.
 \end{align*}
 Since this holds for all $T \ge t$, the limit is $+\infty$ by Condition~\eqref{itm:supBm}   of the event $\Omega_3$~\eqref{omega3}.
 
 \medskip \noindent 
 \textbf{Part~\ref{itm:infinite_inc_points}:} This follows directly from Parts~\ref{itm:jump_points_increasing} and~\ref{itm:hm_to_infty}. 
\end{proof}

\begin{proof}[Proof of Theorem~\ref{thm:qualitative_Buse}]

The full probability event of this theorem is $\Omega_3$, defined in~\eqref{omega3}.

\medskip \noindent \textbf{Part~\ref{itm:dist_incr}:} The monotonicity of Theorem~\ref{thm:summary of properties of Busemanns for all theta}\ref{general monotonicity Busemanns} implies that for $0 < t < T$, 
\[
0 = h_0^{\gamma \sig_1}(0) - h_0^{\delta \sig_2 }(0) \le h_0^{\gamma \sig_1}(t) - h_0^{\delta \sig_2 }(t) \le h_0^{\gamma \sig_1}(T) - h_0^{\delta \sig_2 }(T).
\]

\medskip \noindent \textbf{Part~\ref{itm:stick_split}:} Let $\gamma < \theta$ and $\sigg_1,\sigg_2 \in \{-,+\}$. By Theorem~\ref{thm:coupled_BMs_technical}\ref{itm:coupletoBM_smallt}, there exists $\ve > 0$ small enough so that $B_m(t)= h_m^{\gamma \sig_1}(t) = h_m^{\theta \sig_2}(t)$ for $0 \le t \le \ve$. On the other hand, by definition of the event $\Omega_3$~\eqref{omega3}, 
\[
\lim_{t \rightarrow \infty} \f{h_0^{\gamma \sig_1}(t)}{t} = \f{1}{\sqrt \gamma} > \f{1}{\sqrt \theta} = \lim_{t \rightarrow \infty} \f{h_0^{\theta \sig_2}(t)}{t}.
\]
Hence  $h_0^{\gamma \sig_1}(t) > h_0^{\theta \sig_2}(t)$ for sufficiently large $t$. By the monotonicity of increments from Part~\ref{itm:dist_incr}, 
and continuity of Theorem~\ref{thm:summary of properties of Busemanns for all theta}\ref{general continuity of Busemanns}, 
separation happens at a unique time
$S\ge\ve$. 

\medskip \noindent \textbf{Part~\ref{itm:theta_split_h}:} Let $S = S(\gamma \sigg_1,\delta\sigg_2)$. We first consider the case $\sigg_1 = +$ and $\sigg_2 = -$. Then, the assumption is that $h_0^{\gamma +}(t) = h_0^{\delta - }(t)$ for $0 \le t \le S$ and $h_0^{\gamma +}(t) > h_0^{\delta -}(t)$ for $t > S$. Then, the jump process $\theta \mapsto h_0^{\theta \pm}(S)$ has no jumps in the interval $(\gamma,\delta)$, but for each $\ve > 0$, the process $\theta \mapsto h_0^{\theta \pm}(S + \ve)$ has a nonzero finite number of jumps in the open interval $(\gamma,\delta)$. Furthermore, by Theorem~\ref{thm:coupled_BMs_technical}\ref{itm:jump_points_increasing}, jumps are only added as $\ve$ increases. Hence, there must exist some $\theta^\star \in (\gamma,\delta)$ such that, for every $\ve > 0$, $\theta^\star$ is a jump point of the process $\theta \mapsto h_0^{\theta \pm}(S + \ve)$, i.e. $h_0^{\theta^\star\!-}(S + \ve) > h_0^{\theta^\star\!+}(S +\ve)$ for all $\ve > 0$. But since $h_0^{\gamma+}(S) = h_0^{\delta-}(S)$, the monotonicity of Equation~\eqref{weakmont} requires $h_0^{\theta^\star\!-}(S) = h_0^{\theta^*\!+}(S)$. Hence, $\theta^\star$ has the desired property and lies in $\Busedc$ by definition (Equation~\eqref{eqn:Theta}). 

Next, we prove the statement in the case $\sigg_1 = \sigg_2 = +$. The remaining cases follow similarly. The assumption is now that $h_0^{\gamma+}(t) = h_0^{\delta +}(t)$ for $0 \le t \le S$ and $h_0^{\gamma+}(t) > h_0^{\delta+}(t)$ for $t > S$. By~\eqref{weakmont}, this implies that $h_0^{\gamma+}(t) = h_0^{\delta-}(t) = h_0^{\delta+}(t)$ for $0 \le t \le S$. We may write
\[
h_0^{\gamma+}(t) - h_0^{\delta+}(t) = h_0^{\gamma+}(t) - h_0^{\delta-}(t) + h_0^{\delta-}(t) - h_0^{\delta+}(t),
\]
and therefore, by Part~\ref{itm:dist_incr}, either $h_0^{\gamma +}(t) > h_0^{\delta-}(t)$ for $t > S$ (allowing us to apply the previous case) or $h_0^{\delta-}(t) > h_0^{\delta+}(t)$ for $t > S$, in which case $\theta = \delta$ satisfies the desired property. 

\medskip \noindent \textbf{Parts~\ref{itm:count_trajectories}--\ref{itm:discrete_trajectories}:} By Theorem~\ref{thm:coupled_BMs_technical}\ref{itm:jump_points_increasing},\ref{itm:hm_to_infty}, for each $T > 0$, the path $\theta \mapsto h_0^{\theta \pm}(T)$ is a nonincreasing step function with discrete jump locations that tends to $\infty$ as $\theta \searrow 0$, so the set 
$\{h_0^{\theta \sig}(T): \theta > 0,\sigg \in \{-,+\}\}$
is discrete and infinite. 
Since distances between the Busemann  trajectories are nondecreasing by Part~\ref{itm:dist_incr}, each of these discrete values corresponds to a single trajectory from time $0$ up to time $T$.
\end{proof}

\begin{proof}[Proof of Theorem~\ref{thm:Theta properties}]
The fact that $\Pp(\theta \in \Theta)$ for each fixed $\theta > 0$ is a direct consequence of Theorem~\ref{thm:summary of properties of Busemanns for all theta}\ref{busemann functions agree for fixed theta}. We now prove the various parts of the theorem. The full probability event of the theorem is $\Omega_3$, constructed in~\eqref{omega3}.

\textbf{Part~\ref{itm:Theta_count_intro}:} 
 The density of $\Busedc$ is a direct consequence of Theorem~\ref{thm:qualitative_Buse}\ref{itm:theta_split_h}. Now, set 
 \[
 H_m = \{\theta > 0: h_m^{\theta+}(t) \neq h_m^{\theta -}(t) \text{ for some }t \in \R\}.
 \]
 By the additivity of Theorem~\ref{thm:summary of properties of Busemanns for all theta}\ref{general additivity Busemanns} and the relations of Theorem~\ref{thm:summary of properties of Busemanns for all theta}\ref{general queuing relations Busemanns}, the entire Busemann process can be obtained by a deterministic function of the process $\{h_m^{\theta \sig}(t): m \in \Z, t \in \R,\theta > 0,\sigg \in \{-,+\}\}$. Hence, $\Busedc = \bigcup_{m \in \Z} H_m$. It then suffices to prove that each $H_m$ is countably infinite. By Theorem~\ref{thm:coupled_BMs_technical}\ref{itm:ineqboundary},
 \begin{align}
H_m 
&= \bigcup_{T = 1}^\infty  \{\theta > 0: h_m^{\theta-}(t) \neq h_m^{\theta+}(t) \text{ for some }t \in [-T,T]\} \nonumber \\
&= \bigcup_{T \in \Z} \{\theta > 0: h_m^{\theta -}(T) \neq h_m^{\theta +}(T)\}. \label{eqn:at_T_bound}
\end{align}
For each $T \in \Z$, $\theta \mapsto h_m^{\theta \sig}(T)$ is monotone, so there are only countably many values of $\theta > 0$ such that $h_m^{\theta -}(T) \neq h_m^{\theta +}(T)$. Hence, $H_m$ is countable as well.

 \medskip \noindent \textbf{Part~\ref{itm:const}:} By Theorem~\ref{thm:coupled_BMs_technical}\ref{convBuse}, for every $\theta > 0$, there exists a random $\ve = \ve(\theta,\mbf x,\mbf y)$ such that for $\theta - \ve < \gamma < \theta < \delta < \theta +\ve$ and $\sigg \in \{-,+\}$, $\B^{\theta-}(\mbf x,\mbf y) = \B^{\gamma \sig}(\mbf x,\mbf y)$, and $\B^{\theta +}(\mbf x,\mbf y) = \B^{\delta \sig}(\mbf x,\mbf y)$. Hence, there are no points of $\Busedc_{\mbf x,\mbf y}$ in $(\theta - \ve,\theta + \ve)\setminus \{\theta\}$, and so $\Busedc_{\mbf x,\mbf y}$ has no nonzero limit points. As a result, the notion of two successive points of $\Busedc_{\mbf x,\mbf y}$ is well-defined.  Furthermore, if $\theta \notin \Busedc_{\mbf x,\mbf y}$, then $\B^{\theta-}(\mbf x,\mbf y) = \B^{\theta +}(\mbf x,\mbf y)$, so there exists a random $\ve = \ve(\mbf x,\mbf y) > 0$ such that $\theta \mapsto \B^{\theta \pm}(\mbf x,\mbf y)$ is constant in the interval $(\theta - \ve,\theta + \ve)$. Hence, if $\gamma < \delta$ are any two successive points of $\Busedc_{\mbf x,\mbf y}$, the function $\theta \mapsto \B^{\theta \pm}(\mbf x,\mbf y)$ is continuous and is everywhere locally constant on $(\gamma,\delta)$. Thus, $\theta \mapsto \B^{\theta \pm}(\mbf x,\mbf y)$ must be constant on the entire interval $(\gamma,\delta)$. 
 
 Lastly, set $\mbf x = (m,t)$, $\mbf y = (r,s)$ and $\mbf w = (r,t)$. Without loss of generality, assume $r \ge m$.  By~\eqref{hvdecomp},
\[
\B^{\theta \sig}(\mbf x,\mbf y) = \B^{\theta \sig}(\mbf x,\mbf w) + \B^{\theta \sig}(\mbf w,\mbf y) = \sum_{k = m + 1}^{r} v_k^{\theta \sig}(t)  -h_r^{\theta \sig}(s,t).
\]
By Theorem~\ref{thm:coupled_BMs_technical}\ref{itm:hm_to_infty} and Theorem~\ref{thm:summary of properties of Busemanns for all theta}\ref{general uniform convergence Busemanns}\ref{general uniform convergence:limits to infinity}-\ref{general uniform convergence:limits to 0}, for each $k \in \Z$ and $s < t \in \R$,
\begin{align*}
&\lim_{\theta \rightarrow \infty} h_k^{\theta \sig}(s,t) = B_k(s,t), \quad \lim_{\theta \searrow 0} h_k^{\theta \sig}(s,t) = + \infty, \\\
&\lim_{\theta \rightarrow \infty} v_k^{\theta \sig}(t) = +\infty,\quad \text{and}\quad \lim_{\theta \searrow 0} v_k^{\theta \sig}(t) = 0.
\end{align*}
Since $\mbf x \neq \mbf y$, $r > m$ or $s \neq t$ (or both). Hence, $\B^{\theta \sig}(\mbf x,\mbf y)$ converges to $+\infty$ or $-\infty$ either as $\theta \rightarrow \infty$ or $\theta \searrow 0$, and since $\theta \mapsto \B^{\theta \sig}(\mbf x,\mbf y)$ is constant on every open interval in $(0,\infty) \setminus \Busedc_{\mbf x,\mbf y}$, the set $\Busedc_{\mbf x,\mbf y}$ must be infinite.

\medskip \noindent \textbf{Part~\ref{itm:Theta=Vm}:} The fact that $\Busedc_{(m,-t),(m,t)}$ is nondecreasing follows by Theorem~\ref{thm:coupled_BMs_technical}\ref{itm:ineqboundary}. 
We prove the equality~\eqref{eqn:line_dc} by proving $H_m = H_{m + 1}$, i.e., that $h_m^{\theta+}(t) = h_m^{\theta-}(t)$ for all $t \in \R$ if and only if $h_{m + 1}^{\theta+}(s) = h_{m + 1}^{\theta-}(s)$ for all $s \in \R$. The if-part is immediate from $h_m^{\theta \sig} = D(h_{m + 1}^{\theta \sig},B_m)$ (Theorem~\ref{thm:summary of properties of Busemanns for all theta}\ref{general queuing relations Busemanns}).  

Next, for $\omega \in \Omega_4$, assume that $h_{m + 1}^{\theta +}(s) \neq h_{m + 1}^{\theta -}(s)$ for some $s \in \R$. Assume, by way of contradiction, that $h_{m}^{\theta +}(t) = h_m^{\theta -}(t)$ for all $t$. By Theorem~\ref{thm:summary of properties of Busemanns for all theta}\ref{general queuing relations Busemanns},  
\[
h_m^{\theta \sig}(t) =  B_m(t) + \sup_{0 \le u <\infty}\{B_m(u) - h_{m + 1}^{\theta \sig}(u)\} - \sup_{t \le u <\infty}\{B_m(u) - h_{m + 1}^{\theta \sig}(u)\},
\]
so the function
\[
t \mapsto f(t) := \sup_{t \le u <\infty}\{B_m(u) - h_{m + 1}^{\theta +}(u)\} - \sup_{t \le u <\infty}\{B_m(u) - h_{m + 1}^{\theta -}(u)\}
\]
is constant. First, consider the case $s > 0$. By Theorem~\ref{thm:summary of properties of Busemanns for all theta}\ref{general monotonicity Busemanns}, $h_{m + 1}^{\theta+} \li h_{m + 1}^{\theta -}$, so for $u \ge s$,
\[
h_{m + 1}^{\theta -}(u) - h_{m + 1}^{\theta+}(u) \ge h_{m + 1}^{\theta -}(s) - h_{m + 1}^{\theta+}(s) > 0.
\]
 Then, $f(s) > 0$. On the other hand, by Theorem~\ref{thm:summary of properties of Busemanns for all theta}\ref{limits of B_m minus h m + 1},
\be \label{eqn:pmlim}
\lim_{t \rightarrow \mp \infty} B_m(t) - h_{m + 1}^{\theta \sig}(t) = \pm \infty,
\ee
so we may choose $s_0 < 0$ to be sufficiently negative so that 
\[
\sup_{s_0 \le u < \infty}\{B_m(u)- h_{m + 1}^{\theta \sig}(u)\} =  \sup_{s_0 \le u \le 0}\{B_m(u)- h_{m + 1}^{\theta \sig}(u)\}
\]
for $\sigg \in \{-,+\}$.
By Equation~\eqref{weakmont}, $h_{m + 1}^{\theta +}(u) \ge h_{m + 1}^{\theta -}(u)$ for $u \le 0$. Thus, $f(s_0) \le 0$, a contradiction to the finding that $f(s) > 0$ and $f$ is constant.

Now, consider the case $s < 0$. Using $h_{m +1}^{\theta +}\li h_{m + 1}^{\theta -}$, as in the $s > 0$ case, $h_{m + 1}^{\theta +}(u) > h_{m +1}^{\theta -}(u)$ for all $u < s$. By~\eqref{eqn:pmlim}, we can choose $s_0$ to be sufficiently negative so that for $\sigg \in \{-,+\}$,
\[
\sup_{s_0 \le u < \infty}\{B_m(u)- h_{m + 1}^{\theta \sig}(u)\} =  \sup_{s_0 \le u \le s}\{B_m(u)- h_{m + 1}^{\theta \sig}(u)\},
\]
and hence $f(s_0) < 0$. On the other hand, by~\eqref{weakmont}, $h_{m + 1}^{\theta - }(u) \ge h_{m 
+ 1}^{\theta +}(u)$ for $u > 0$, so $f(0) \ge 0$, giving a similar contradiction.
\end{proof}

\section{Proofs of the results from Section~\ref{section:main_geometry} and Theorems~\ref{thm:Busedc_class} and~\ref{thm:Haus_comp_interface}} \label{sec:detailed_SIG_proofs}

We start by proving the results of Section~\ref{section:main_geometry}, and the proofs of Theorems~\ref{thm:Busedc_class} and~\ref{thm:Haus_comp_interface} are presented at the very end of this section. The full probability event of the results in Section~\ref{section:main_geometry} is denoted as $\Omega_4$, which we now define.
Recall the discussion of the events $\Omega_2$, $\Omega^{(\theta)}$, and $\Omega_{\mbf x}^{(\theta)}$ immediately before Theorem~\ref{thm:summary of properties of Busemanns for all theta}. 
Let $\Omega_3 \subseteq \Omega_2$ be the full probability event of Lemma~\ref{thm:coupled_BMs_technical}. Let $\wt \Omega$ be the full probability event of Lemma~\ref{lemma:geodesic_bound}. For each $m \in \Z$ and $\gamma > 0$, the function $s\mapsto B_m(s) - h_{m + 1}^\gamma(s)$ is a variance $2$ Brownian motion with negative drift by Theorems~\ref{thm:summary of properties of Busemanns for all theta}\ref{Buse_marg_dist} and~\ref{thm:summary of properties of Busemanns for all theta}\ref{independence structure of Busemann functions on levels}.   Let $C_{m,\gamma}$ be the full probability event on which the conclusions of Theorem~\ref{thm:countable non unique maximizers} hold, applied to $s \mapsto B_m(s) - h_{m + 1}^\gamma(s)$.  For $r \in \Z$, let $A_r$ be the event  on which the set 
    \[
    \Big\{s \in \R: B_{r -1 }(s) - B_{r}(s) = \sup_{s \le u \le t}\{B_{r - 1}(u) - B_{r}(u)\} \;\text{ for some } t> s\Big\}.
    \]
    has Hausdorff dimension $\f{1}{2}$. Since $B_r$ and $B_{r - 1}$ are independent, $B_r - B_{r - 1}$ is a variance $2$ Brownian motion, and Corollary~\ref{cor:hausdorff dimension for left maxes standard two-sided BM} implies that $\Pp(A_r) = 1$.

The full-probability event $\Omega_4$ is defined to be
\be \label{omega4}
\Omega_4 := \Omega_2 \cap \Omega_3 \cap \wt \Omega \cap \bigcap_{\theta \in \Q_{>0}} \Omega^{(\theta)} \cap \bigcap_{\theta \in \Q_{>0},\mbf x \in \Z \times \Q} \Omega_{\mbf x}^{(\theta)} \cap \bigcap_{m \in \Z,\gamma \in \Q_{>0}} C_{m,\gamma} \cap \bigcap_{r \in \Z} A_r.
\ee
For the sake of reference, the following properties hold on this event.
\begin{enumerate} [label=\rm(\roman{*}), ref=\rm(\roman{*})]  \itemsep=3pt
    \item \label{itm:Buse_eq} For $\theta \in \Q_{>0}$ and $\mbf x,\mbf y \in \Z \times \R$, $\B^{\theta -}(\mbf x,\mbf y) = \B^{\theta +}(\mbf x,\mbf y)$ (Theorem~\ref{thm:summary of properties of Busemanns for all theta}\ref{busemann functions agree for fixed theta}).   
    \item \label{itm:notNU} For each $\mbf x \in \Z \times \Q$ and $\theta \in \Q_{>0}$, $\mbf x \notin \NU_0^\theta$. In other words, for $\theta \in \Q_{>0}$ and $\mbf x \in \Z \times \Q$, there is a unique $\theta$-directed semi-infinite geodesic out of $\mbf x$. (See Theorem~\ref{thm:NU_paper1}\ref{itm:uniqueness of geodesic for fixed point and direction} and Remark \ref{rmk:NU_Sets}(a)).
    \item \label{itm:rational_non_unique} For each $\theta \in \Q_{>0}$, the sets $\NU_0^\theta$ and $\NU_1^\theta$ are countably infinite. (Theorem~\ref{thm:NU_paper1}\ref{itm:count_and_decomp}).
    \item The conclusions of Theorem~\ref{thm:coupled_BMs_technical} hold. 
    \item \label{itm:omega4_unique_geod} For every $(m,q_1) \le (r,q_2)$ with $q_1,q_2 \in \Q_{>0}$, there exists a unique geodesic between $(m,q_1)$ and $(r,q_2)$. That unique geodesic does not pass through $(k,q_1)$ for $k > m$ or $(n,q_2)$ for $n < r$ (Lemma~\ref{lemma:geodesic_bound}\ref{itm:unique_geod_rational}).
    \item \label{itm:finite_geodesics} For every pair of points $(m,s) \le (n,t)$, there are finitely many geodesics between the two points (Lemma~\ref{lemma:geodesic_bound}\ref{itm:geod_bound}). 
\end{enumerate} 
With this event in place, we have the following result.
\begin{lemma} \label{lem:all_theta_max_comp}
On the event $\Omega_4$, the following hold for all $\theta > 0$ and $\sigg \in \{-,+\}$.
\begin{enumerate} [label=\rm(\roman{*}), ref=\rm(\roman{*})]  \itemsep=3pt
    \item \label{itm:max_agree} For each compact set $K$, there exists $\gamma \in \Q_{>0}$ such that, for each $t \in K$, the functions $s \mapsto B_m(s) - h_{m + 1}^{\theta \sig}(s)$ and $s \mapsto B_m(s) - h_{m + 1}^\gamma(s)$ agree on the common compact set containing all maximizers of the functions over $s \in [t,\infty)$.
    \item \label{itm:all_no_3} There exist no points $t \in \R$ such that the function $s \mapsto B_m(s) - h_{m + 1}^{\theta \sig}(s)$ has three maximizers over $s \in [t,\infty)$. If the function has two maximizers over $s \in [t,\infty)$, one of them is $s = t$.
    \item \label{itm:all_not_monotone}The function $s \mapsto B_m(s) - h_{m + 1}^{\theta \sig}(s)$ is not monotone on any nonempty interval.
    \item \label{itm:approx_right} If, for some $s \in \R$, $\tau_{(m,s),m}^{\theta \sig,R} = s$, then for every $\ve > 0$, there exists $t \in (s,s + \ve)$ such that $(m,t) \in \NU_1^{\theta \sig}$.
    \item \label{itm:approx_left} If $(m,s) \in \NU_1^{\theta \sig}$, there exists $t \in (s - \ve,s)$ such that $\tau_{(m,t),m}^{\theta \sig,R} = t$.
    \item \label{itm:no_NU1_left} For all $(m,s) \in \NU_1^{\theta \sig}$, there exists $\delta  > 0$ such that $\NU_1^{\theta \sig} \cap (\{m\}\times (s,s+\delta)) = \varnothing$.
\end{enumerate}
\end{lemma}
\begin{proof}
\textbf{Part~\ref{itm:max_agree}:}
Let $t_0 = \min K$ and $t_1 = \max K$. By Theorem~\ref{thm:coupled_BMs_technical}\ref{itm:strong_unif_convergence}, for $\theta > 0$ and $\sigg \in \{-,+\}$, we may choose $\gamma \in \Q_{>0}$ to be sufficiently close to $\theta$ (from the right for $\sigg = +$ and from the left for $\sigg = -$) so that $h_{m + 1}^{\theta \sig}(s) = h_{m +1}^\gamma(s)$ for all $s \in [t_0,\tau_{(m,t_1),m}^{(\theta + 1)+,R}]$. Since $\tau_{(m,t),m}^{\theta \sig,R}$ is the right-most maximizer of $B_m(s) - h_{m + 1}^{\theta \sig}(s)$ over $s \in [t,\infty)$, by monotonicity of Theorem~\ref{existence of semi-infinite geodesics intro version}\ref{itm:monotonicity of semi-infinite jump times}\ref{itm:monotonicity in theta}--\ref{itm:monotonicity in t}, for all $t \in K$,  the set $[t_0,\tau_{(m,t_1),m}^{(\theta + 1)+,R}]$  contains all maximizers of $B_m(s) - h_{m + 1}^{\theta \sig}(s)$ and $B_m(s) - h_{m + 1}^\gamma(s)$  over $s \in [t,\infty)$. 

\medskip \noindent \textbf{Part~\ref{itm:all_no_3}:} If, by contradiction, for some $t \in \R$, $B_m(s) - h_{m + 1}^{\theta \sig}(s)$ has two maximizers over $s \in [t,\infty)$ that are both strictly greater than $t$, then by Part~\ref{itm:max_agree}, the same holds for $B_m(s) - h_{m + 1}^{\gamma}(s)$, where $\gamma$ is some rational direction. This contradicts the definition of $\Omega_4$ that the conclusion of Theorem~\ref{thm:countable non unique maximizers}\ref{itm:2_3_max} holds for the function $s \mapsto B_m(s) - h_{m + 1}^\gamma(s)$.

\medskip \noindent \textbf{Part~\ref{itm:not_mont}:}
Assume by way of contradiction that $s\mapsto B_m(s) - h_{m + 1}^{\theta\sig}(s)$ is monotone on some compact interval $I$. By Theorem~\ref{thm:coupled_BMs_technical}\ref{itm:strong_unif_convergence}, there exists a rational $\gamma \in \Q_{>0}$ such that $B_m(s) - h_{m + 1}^\gamma(s) = B_m(s) - h_{m + 1}^{\theta \sig}(s)$ for $s \in I$. Then, $s \mapsto B_m(s) - h_{m + 1}^{\gamma}(s)$ is monotone on $I$, contradicting the definition of the event $\Omega_4 \subseteq C_{m,\gamma}$~\eqref{omega4} and Theorem~\ref{thm:countable non unique maximizers}\ref{itm:not_mont}.

\medskip \noindent \textbf{Part~\ref{itm:approx_right}:} Assume that $\tau_{(m,s),m}^{\theta \sig,R} = s$ and let $\ve > 0$. By definition of right-most maximizers, $u = s$ is the unique maximizer of $B_m(u) - h_{m + 1}^{\theta \sig}(u)$ over $u \in [s,\infty)$.  Let $K = [s,s + \ve]$. Then, by Part~\ref{itm:max_agree}, there exists a rational $\gamma > 0$ such that $B_m(u) - h_{m + 1}^{\theta \sig}(u) = B_m(u) - h_{m + 1}^\gamma(u)$ for all $u$ in a common compact set containing all maximizers of both functions over $u \in [t,\infty)$, for each $t \in K$. By definition of $\Omega_4 \subseteq C_{m,\gamma}$ and Theorem~\ref{thm:countable non unique maximizers}\ref{non_discrete}, there exists $t \in (s,s + \ve)$ such that $B_m(u) - h_{m +1}^\gamma(u)$ (and therefore also $B_m(u) - h_{m + 1}^{\theta \sig}(u)$) has two maximizers over $u \in [t,\infty)$. Thus, $(m,t) \in \NU_1^{\theta \sig}$.

\medskip \noindent \textbf{Parts~\ref{itm:approx_left}--\ref{itm:no_NU1_left}:} These follow by an analogous proof as Part~\ref{itm:approx_right}, using Part~\ref{itm:max_agree}, the definition of $\Omega_4\subseteq C_{m,\gamma}$, and Theorem~\ref{thm:countable non unique maximizers}, Parts~\ref{non_discrete}--\ref{non_discrete MN}.
\end{proof}

\begin{proof}[Proof of Theorem~\ref{thm:geod_stronger}]
\noindent \textbf{Part~\ref{itm:stronger convergence theta}}: We prove the statement for limits from the right, and the other statement follows analogously. Without loss of generality, set $K = [t_0,t_1]$ for $t_0 < t_1$.
By Lemma~\ref{lemma:ptl_sig}\ref{itm:LR_geo_max}, for each  $t \in K$, $n \ge m$, and $\delta>0$,   the sequences $t = \tau_{(m,t),m - 1}^{\delta \sig,S} \le \tau_{(m,t),m}^{\delta \sig,S} \le \cdots \le \tau_{(m,t),n}^{\delta \sig,S}$ for  $S\in\{L,R\}$ are the leftmost and rightmost  maximizers of 
\be \label{eqn:ptl}
\sum_{r = m}^n B_r(s_{r - 1},s_r) - h_{n + 1}^{\delta \sig}(s_n)
\ee
over all sequences $t = s_{m - 1} \le s_m \le \cdots \le s_n < \infty$.
By Theorem~\ref{existence of semi-infinite geodesics intro version}\ref{itm:monotonicity of semi-infinite jump times}\ref{itm:monotonicity in theta}--\ref{itm:monotonicity in t}, for all $\delta \in [\theta,\theta + 1]$, $(m,t) \in \Z \times \R$, $\sigg \in \{-,+\}$, $S \in \{L,R\}$, and $r \ge m$, 
\be \label{gathmont}
\tau_{(m,t),r}^{\delta \sig ,S} \le \tau_{(m,t_1),r}^{(\theta+  1)+,R}.
\ee
Hence, for all $t \in K$, $\delta \in [\theta,\theta + 1]$ and $\sigg \in \{-,+\}$, 
the maximizers of \eqref{eqn:ptl} remain the same 
when the maximum is restricted to sequences $t = s_{m - 1} \le s_m \le \cdots \le s_n \le \tau_{(m,t_1),n}^{(\theta + 1)+,R}$.

By Theorem~\ref{thm:coupled_BMs_technical}\ref{itm:strong_unif_convergence} and the monotonicity of Theorem~\ref{existence of semi-infinite geodesics intro version}\ref{itm:monotonicity of semi-infinite jump times}\ref{itm:monotonicity in theta}--\ref{itm:monotonicity in t}, for each $n$, there exists a random $\ve > 0$ such that for $t \in K$, $\theta < \delta < \theta + \ve$, and $\sigg \in\{-,+\}$,  $h_{n + 1}^{\delta \sig }(s_n) = h_{n + 1}^{\theta +}(s_n)$ for $t_0 \le s_n \le \tau_{(m,t_1),n}^{(\theta + 1)+,R}$ (Recall that $h_{n +1}^{\delta \sig}(s_n) = h_{n +1}^{\delta \sig}(0,s_n)$). 
Hence, for $\theta < \delta < \theta + \ve$ and each $t = s_{m -1} \in K$, the functions
\[
\sum_{r = m}^n B_r(s_{r - 1},s_r) - h_{n + 1}^{\delta \sig}(s_n),\qquad\text{and}\qquad  \sum_{r = m}^n B_r(s_{r - 1},s_r) - h_{n + 1}^{\theta+}(s_n)
\]
are equal on the common compact set of sequences $t = s_{n - 1} \le \cdots \le s_n \le \tau_{(m,t_1),n}^{(\theta + 1)+,R}$ which contains all their maximizers. In particular, their left- and right-most maximizers coincide, and therefore 
$\tau_{(m,t),r}^{\delta\sig,S} = \tau_{(m,t),r}^{\theta +,S}$ for $S \in \{L,R\}$  and $m\le r \le n$ by Lemma~\ref{lemma:ptl_sig}\ref{itm:LR_geo_max}.

\medskip \noindent \textbf{Part~\ref{itm:global strong monotonicity in t}}. Assume, by way of contradiction, that there exists $\omega \in \Omega_4$ such that, for some $s < t \in \R,m \le r \in \Z$, $\theta > 0$, and $\sigg \in \{-,+\}$, $\tau_{(m,s),r}^{\theta \sig,R} > \tau_{(m,t),r}^{\theta \sig,L}$. Without loss of generality, take $\sigg = +$. Then, by Part~\ref{itm:stronger convergence theta}, for sufficiently close rational $\delta > \theta$,
\be \label{100}
\tau_{(m,s),r}^{\delta ,R} = \tau_{(m,s),r}^{\theta +,R} > \tau_{(m,t),r}^{\theta +,L} = \tau_{(m,t),r}^{\delta,L}.
\ee
By Theorem~\ref{existence of semi-infinite geodesics intro version}\ref{itm:monotonicity of semi-infinite jump times}\ref{itm:strong monotonicity in t}, $\tau_{(m,s),r}^{\delta,R} \le \tau_{(m,t),r}^{\delta,L}$ on the event $\Omega^{(\delta)}$, which contains $\Omega_4$ because $\delta$ is rational. This is a contradiction to~\eqref{100}. 

\medskip \noindent \textbf{Part~\ref{itm:stronger convergence in t}:} By the monotonicity of Theorem~\ref{existence of semi-infinite geodesics intro version}\ref{itm:monotonicity of semi-infinite jump times}\ref{itm:monotonicity in t}, the limits $\tau_r := \lim_{u \nearrow s} \tau_{(m,u),r}^{\theta \sig,S}$ exist  for $r \ge m$, and by Part~\ref{itm:global strong monotonicity in t}, $\tau_r \le \tau_{(m,s),r}^{\theta \sig,L}$. By Lemma~\ref{lemma:ptl_sig}\ref{itm:geo_maxes}, for $r \ge m$, $\tau_{(m,u),m}^{\theta \sig,S},\ldots,\tau_{(m,u),n}^{\theta \sig,S}$ is a maximizing sequence for 
\be \label{eqn:var}
 \sup\Biggl\{\sum_{r = m}^n B_r(s_{r - 1},s_r) - \h_{n +1}^{\theta\sig}(s_n): u = s_{m - 1} \le s_m \le \cdots \le s_n < \infty  \Biggr\}.
\ee
By Lemma~\ref{lemma:convergence of maximizers from converging sets}, $\tau_m,\ldots,\tau_n$ is a maximizing sequence for~\eqref{eqn:var}, replacing $u$ with $s$. By Lemma~\ref{lemma:ptl_sig}\ref{itm:LR_geo_max}, $\tau_{(m,s),m}^{\theta \sig,L},\ldots,\tau_{(m,s),n}^{\theta \sig,L}$ is the leftmost such maximizing sequence.  Since $\tau_r \le \tau_{(m,s),r}^{\theta \sig,L}$, we must have that $\tau_r = \tau_{(m,s),r}^{\theta \sig,L}$ for $m \le r \le n$.  The proof for limits as $t \searrow s$ is analogous.

\end{proof}

\subsection{Proof of the results from Section~\ref{sec:non_unique_coal}} \label{sec:pf_coal}

\begin{proof}[Proof of Theorem~\ref{thm:NU}]
We prove Part~\ref{itm:allcount} last.

\medskip \noindent \textbf{Part~\ref{itm:union}}: 
To  establish~\eqref{decomp}, we show that, for $i  =0,1$, if $(m,t) \in \NU_i^{\theta\sig}$, then $(m,t) \in \NU_i^\gamma$ for some $\gamma \in \Q_{>0}$. This follows from
Theorem~\ref{thm:geod_stronger}\ref{itm:stronger convergence theta}, by which on the event $\Omega_4$, if $\tau_{(m,t),r}^{\theta \sig,L} < \tau_{(m,t),r}^{\theta \sig,R}$, then $\tau_{(m,t),r}^{\gamma,L} < \tau_{(m,t),r}^{\gamma,R}$ for all rational $\gamma$ sufficiently close to $\theta$ on the appropriate side (greater than $\theta$ for $\sigg = +$ and less than $\theta$ for $\sigg = -$).  With~\eqref{decomp} established, Item~\ref{itm:rational_non_unique} below~\eqref{omega4} implies that, on $\Omega_4$, $\NU_0$ and $\NU_1$ are both countably infinite.

\medskip \noindent \textbf{Part~\ref{itm:NUp0}:} This follows from Part~\ref{itm:union} and Theorem~\ref{thm:NU_paper1}\ref{itm:uniqueness of geodesic for fixed point and direction} since $\Pp(\mbf x \in \NU_0^\theta) = 0$ for any $\theta \in \Q_{>0}$. Equation~\eqref{decomp} and Item~\ref{itm:notNU} below~\eqref{omega4} imply that, on $\Omega_4$, $\NU_0$ contains no points of $\Z \times \Q$.

\medskip \noindent \textbf{Part~\ref{itm:only_endpoint}:}  Let $(m,t)$ be such that $\tau_{(m,t),r}^{\theta \sig,L} < \tau_{(m,t),r}^{\theta \sig,R}$ for some $r \ge m$, and take $r$ to be the minimal such index. We show that $\tau_{(m,t),r}^{\theta \sig,L} = t$.  By  Theorem~\ref{thm:geod_stronger}\ref{itm:stronger convergence theta}, there exists rational $\gamma > 0$ sufficiently close to $\theta$ from the appropriate side such that $\tau_{(m,t),k}^{\theta \sig,S} = \tau_{(m,t),k}^{\gamma,S}$ for $m \le k \le r$ and $S \in \{L,R\}$. Then, $\tau_{(m,t),r}^{\gamma,L} < \tau_{(m,t),r}^{\gamma,R}$ and $\tau_{(m,t),k}^{\gamma,L} = \tau_{(m,t),k}^{\gamma,R}$ for $m \le k < r$. Since $\gamma$ is rational, $\Omega_4 \subseteq \Omega^{(\gamma)}$ by~\eqref{omega4}. Then, by Theorem~\ref{thm:NU_paper1}\ref{itm:count_and_decomp},
\[
t = \tau_{(m,t),r}^{\gamma,L} = \tau_{(m,t),r}^{\theta \sig,L}.
\]

\medskip \noindent \textbf{Part~\ref{itm:allcount}:}
Since $\NU_1^{\theta \sig} \subseteq \NU_0^{\theta \sig} \subseteq \NU_0$, and $\NU_0$ is countably infinite by Part~\ref{itm:union}, it suffices to show that for every $\theta > 0$ and $\sigg \in\{-,+\}$,  $\NU_1^{\theta \sig}$ is infinite. We start by showing that $\NU_1^{\theta \sig}$ is nonempty. To do this, we first show the existence of a point $(m,t) \in \Z \times \R$ such that $\tau_{(m,t),m}^{\theta \sig,L} > t$. If such a point does not exist, then for each $m \in \Z$, the function $B_m(s) - h_{m + 1}^{\theta \sig}(s)$ over $s \in [t,\infty)$ is maximized at $s = t$ for each $t \in \R$. But then, $B_m(s) - h_{m 
+ 1}^{\theta \sig}(s)$ is a nonincreasing function on $\R$, contradicting Lemma~\ref{lem:all_theta_max_comp}\ref{itm:all_not_monotone}.

Since $\tau_{(m,t),m}^{\theta \sig,L} > t$, all maximizers of the function $B_m(s) - h_{m + 1}^{\theta \sig}(s)$ over $s \in [t,\infty)$ are greater than $t$.  In other words, letting 
\[
M := \sup_{t \le s < \infty}\{B_m(s) - h_{m + 1}^{\theta \sig}(s)\},
\]
we have $B_m(t) - h_{m + 1}^{\theta \sig}(t) < M$.
By Theorem~\ref{thm:summary of properties of Busemanns for all theta}\ref{general continuity of Busemanns}--\ref{limits of B_m minus h m + 1}, $s \mapsto B_m(s) - h_{m + 1}^{\theta \sig}(s)$ is continuous and satisfies
\[
    \lim_{s \rightarrow \pm \infty} B_m(s) - \h_{m + 1}^{\theta \sig}(s) = \mp \infty.
    \]
Therefore, the quantity
\[
T := \sup\{s < t: B_m(s) - h_{m + 1}^{\theta \sig}(s) = M\} \qquad\text{is well-defined and finite.}
\]
In words, we go backwards from $t$ until we reach the first point $T$ where $B_m(T) - h_{m + 1}^{\theta \sig}(T) = M$. 
 Then, 
 \[
 \sup_{s \ge T} \{B_m(s) - h_{m + 1}^{\theta \sig}(s)\} = M,
 \]
 and the maximum is achieved at two locations in  $ [T,\infty)$, namely $T$ and $\tau_{(m,t),m}^{\theta \sig,L}$. Hence, \\ $T = \tau_{(m,T),m}^{\theta \sig,L} < \tau_{(m,T),m}^{\theta \sig,R}$, and $(m,T) \in \NU_1^{\theta \sig}$. 

Lastly, we show that $\NU_1^{\theta \sig}$ is infinite by showing that for $(m,s) \in \NU_1^{\theta \sig}$ and $\ve > 0$, there exists $t \in (s - \ve,s)$ such that $(m,t) \in \NU_1^{\theta \sig}$. By Lemma~\ref{lem:all_theta_max_comp}\ref{itm:approx_left}, there exists $t^\star \in (s - \ve,s)$ such that $\tau_{(m,t^\star),m}^{\theta \sig,R} = t^\star$. Then, by Lemma~\ref{lem:all_theta_max_comp}\ref{itm:approx_right}, there exists $t \in (t^\star,s)$ such that $(m,t) \in \NU_1^{\theta \sig}$.
 \end{proof}

\noindent 
We begin now to work towards the coalescence claims of  Theorem~\ref{thm:general_coalescence}. First, we present a technical lemma.

\begin{lemma} \label{lemma:SE_coal_pt}
Let $\omega \in \Omega_4$, $\mbf x \Se \mbf y$, $\theta > 0$, and $\sigg \in \{-,+\}$. Then, if $\Gamma_1 \in \mbf T_{\mbf x}^{\theta \sig}$ and $\Gamma_2 \in \mbf T_{\mbf y}^{\theta \sig}$ are such that $\Gamma_1 \cap \Gamma_2 \neq \varnothing$, then $\Gamma_1$ and $\Gamma_2$ coalesce, and the minimal point of intersection is the coalescence point of the two geodesics. 
\end{lemma}

 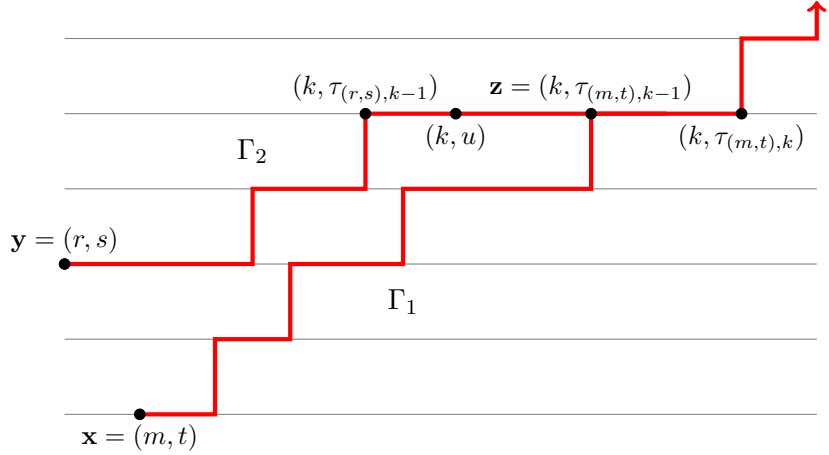
\begin{figure}[t]
 \centering
            \begin{tikzpicture}
            \draw[gray,thin] (0,0)--(10,0);
            \draw[gray,thin] (0,1)--(10,1);
            \draw[gray,thin] (0,2)--(10,2);
            \draw[gray,thin] (0,3)--(10,3);
            \draw[gray,thin] (0,4)--(10,4);
            \draw[gray,thin] (0,5)--(10,5);
            \draw[red,ultra thick,->] plot coordinates {(1,0)(2,0)(2,1)(3,1)(3,2)(4.5,2)(4.5,3)(7,3)(7,4)(9,4)(9,5)(10,5)(10,5.5)};
            \draw[red,ultra thick] plot coordinates {(0,2)(2.5,2)(2.5,3)(4,3)(4,4)(8,4)};
            \filldraw[black] (1,0) circle (2pt) node[anchor = north] {\small $\mbf x = (m,t)$};
            \filldraw[black] (0,2) circle (2pt) node[anchor = south] {\small $\mbf  y = (r,s)$};
            \filldraw[black] (7,4) circle (2pt) node[anchor = south] {\small $\mbf z = (k,\tau_{(m,t),k - 1})$};
            \filldraw[black] (5.2,4) circle (2pt) node[anchor = north] {\small $(k,u)$};
            \filldraw[black] (4,4) circle (2pt) node[anchor = south] {\small $(k,\tau_{(r,s),k - 1})$};
            \filldraw[black] (9,4) circle (2pt) node[anchor = north] {\small $(k,\tau_{(m,t),k})$};
            \node at (4.5,1.5) {$\Gamma_1$};
            \node at (2.5,3.5) {$\Gamma_2$};
            \end{tikzpicture}
            \caption{\small Coalescence of geodesics.}
            \label{fig:stay}
            \bigskip
        \end{figure}

\begin{proof}
For this proof, refer to Figure~\ref{fig:stay} for clarity. Set $\mbf x = (m,t)$ and $\mbf y = (r,s)$. Let $t = \tau_{(m,t),m - 1} \le \tau_{(m,t),m} \le \cdots$ denote the jump times for $\Gamma_1$, and let $s = \tau_{(r,s),r - 1} \le \tau_{(r,s),r} \le \cdots$ denote the jump times for $\Gamma_2$. Assume $\Gamma_1 \cap \Gamma_2 \neq \varnothing$, and let $\mbf z = (k,v)$ be the minimal point of intersection of the two geodesics. Since $\mbf x \Se \mbf y$, the geometry of BLPP paths requires that the first intersection occurs when $\Gamma_1$ makes an upward step to hit $\Gamma_2$. In terms of jump times, this means
\be \label{eqn:before_coal}
\tau_{(r,s),k - 1} < v = \tau_{(m,t),k - 1} \le \tau_{(m,t),k} \wedge \tau_{(r,s),k}.
\ee
Let $u \in [\tau_{(r,s),k - 1},\tau_{(m,t),k - 1}]$ be rational. Then, by~\eqref{eqn:before_coal} and construction of the semi-infinite geodesics in terms of maximizers (see Definition~\ref{def:semi-infinite geodesics}), both $\tau_{(m,t),k}$ and $\tau_{(r,s),k}$ maximize the function $B_k(w) - h_{k + 1}^{\theta \sig}(w)$ over $w \in [u,\infty)$. Inductively, the successive jump times $\tau_{(m,t),n}$ and $\tau_{(r,s),n}$ for $n > k$ maximize the function $B_n(w) - h_{n + 1}^{\theta \sig}(w)$ over $w \in [\tau_{(m,t),n - 1},\infty)$ and $w \in [\tau_{(r,s),n - 1},\infty)$, respectively. Therefore, the sequences
\[
u \le \tau_{(m,t),k} \le \tau_{(m,t),k + 1} \le \cdots \qquad\text{and}\qquad u \le \tau_{(r,s),k} \le \tau_{(r,s),k + 1} \le \cdots
\]
both define semi-infinite geodesics in $\mbf T_{\mbf (k,u)}^{\theta \sig}$. By Theorem~\ref{thm:NU}\ref{itm:NUp0}, because $(k,u) \in \Z \times \Q$, there is only one element in $\mbf T_{\mbf (k,u)}^{\theta \sig}$. Thus, $\tau_{(m,t),n} = \tau_{(r,s),n}$ for $n \ge k$, completing the proof.
\end{proof}

\noindent The following remark underscores the importance of the configuration and choice of $L/R$ geodesics in the lemmas and theorems that follow. 
\begin{remark} \label{rmk:coalcounter}
Let $s < t$ and $m \in \Z$. Consider the two initial points $(m,s)$ and $(m,t)$ which lie along the same horizontal line. The geodesics $\Gamma_{(m,s)}^{\theta \sig,L}$ and $\Gamma_{(m,t)}^{\theta \sig,L}$, which start from $(m,s)$ and $(m,t)$, respectively, coalesce at the point $(m,t)$ if and only if $\tau_{(m,s),m}^{\theta\sig,L} \ge t$.
However, by Theorem~\ref{thm:summary of properties of Busemanns for all theta}\ref{general queuing relations Busemanns}, for all $\theta > 0$ and $\sigg \in \{-,+\}$,
\begin{align*}
 h_m^{\theta \sig}(s,t) 
= B_m(s,t) + \sup_{s \le u < \infty}\{B_m(u) - h_{m + 1}^{\theta\sig}(u)\} - \sup_{t \le u < \infty}\{B_m(u) -h_{m + 1}^{\theta \sig}(u)\},
\end{align*}
and therefore, $h_m^{\theta \sig}(s,t) = B_m(s,t)$ if and only if $\tau_{(m,s),m}^{\theta \sig,R}$ (the rightmost maximizer of $B_m(u) - h_{m +1}^{\theta \sig}(u)$ over $u \in [s,\infty)$) is greater than or equal to $t$.  

Now, choose an arbitrary $\theta  \in (0,\infty) \setminus \Busedc$. 
We choose $(m,s) \in \NU_1^\theta$ and $t > s$ to be such that
\be \label{eqn:taus < t}
s = \tau_{(m,s),m}^{\theta,L} < t < \tau_{(m,s),m}^{\theta,R}. 
\ee
By~\eqref{eqn:limits in theta to infinity}, for $\delta > \theta$ large enough,
\be \label{eqn:tau > t} 
\tau_{(m,s),m}^{\delta \sig,L} > t.
\ee
By Theorem~\ref{thm:geod_stronger}\ref{itm:stronger convergence theta}, for $\gamma$ sufficiently close to  $\theta$ and $\sigg \in \{-,+\}$, 
\be \label{eqn:tau < t}
s = \tau_{(m,s),m}^{\gamma \sig,L} < t < \tau_{(m,s),m}^{\gamma \sig,R}.
\ee
Hence, there exists $\gamma < \theta$ such that, for all $\eta\in (\gamma,\infty)$,  $\tau_{(m,s),m}^{\eta \sig,R} > t$ and therefore $h_m^{\eta\sig}(s,t) = B_m(s,t)$. Furthermore, by~\eqref{eqn:tau > t}, $\Gamma_{(m,s)}^{\eta \sig,L}$ and $\Gamma_{(m,t)}^{\eta \sig,L}$ coalesce at $(m,t)$ for all $\eta$ sufficiently large, but by~\eqref{eqn:tau < t}, $\Gamma_{(m,s)}^{\eta \sig,L}$ and $\Gamma_{(m,t)}^{\eta \sig,L}$ do not coalesce at $(m,t)$ for $\eta\in (\gamma,\theta)$. In other words, if $\eta_1 \in (\gamma,\theta)$ and $\eta_2$ is large enough, then $\B^{\eta_1 \sig_1}((m,s),(m,t)) = \B^{\eta_2 \sig_2}((m,s),(m,t))$ for $\sigg_1,\sigg_2 \in \{-,+\}$, but $\Gamma_{(m,s)}^{\eta_1 \sig_1,L}$ and $\Gamma_{(m,t)}^{\eta_1 \sig_1,L}$ do not coalesce at the same place as $\Gamma_{(m,s)}^{\eta_2 \sig_2,L}$ and $\Gamma_{(m,t)}^{\eta_2 \sig_2,L}$. Hence, when $(m,s) \in \NU_1$, the $L/R$ distinction is essential in our statements about the connection between equality of Busemann functions and common coalesce points.
\end{remark}

\begin{lemma} \label{lemma:coaleqh}
Let $\omega \in \Omega_2$, $s < t \in \R$, $m \in \Z$, $0 < \gamma \le \theta$, and $\sigg_1,\sigg_2 \in \{-,+\}$. If $\gamma = \theta$, we require $\sigg_1 = -$ and $\sigg_2 = +$.  Then, $h_m^{\theta \sig_2}(s,t) = h_m^{\gamma \sig_1}(s,t)$ if and only if one of the following two conditions hold:
\begin{enumerate} [label=\rm(\roman{*}), ref=\rm(\roman{*})]  \itemsep=3pt
    \item $h_m^{\theta \sig_2}(s,t) = h_m^{\gamma \sig_1}(s,t) = B_m(s,t)$, and $\tau_{(m,s),m}^{\theta \sig_2,R} \ge \tau_{(m,s),m}^{\gamma \sig_1,R} \ge t$.  \label{itm:=Bm}
    \item  \label{itm: > Bm} $h_m^{\theta \sig_2}(s,t) = h_m^{\gamma \sig_1}(s,t) > B_m(s,t)$, and
    \[
    s' := \tau_{(m,s),m}^{\theta \sig_2,R} = \tau_{(m,s),m}^{\gamma \sig_1,R} < t \le \tau_{(m,t),m}^{\theta \sig_2,L} = \tau_{(m,t),m}^{\gamma \sig_1,L} =: t', \text{ and } h_{m + 1}^{\theta \sig_2}(s',t') = h_{m + 1}^{\gamma \sig_1}(s',t').
    \]
\end{enumerate}
\end{lemma}
\begin{proof}
The keys to this proof are the monotonicity of the Busemann functions and semi-infinite geodesics from Theorems~\ref{thm:summary of properties of Busemanns for all theta}\ref{general monotonicity Busemanns} and Theorem~\ref{existence of semi-infinite geodesics intro version}\ref{itm:monotonicity of semi-infinite jump times}\ref{itm:monotonicity in theta}. For simplicity of notation, we suppress the $\sigg_1$ and $\sigg_2$ notation in the superscripts, noting that Theorems~\ref{thm:summary of properties of Busemanns for all theta}\ref{general monotonicity Busemanns} and Theorem~\ref{existence of semi-infinite geodesics intro version}\ref{itm:monotonicity of semi-infinite jump times}\ref{itm:monotonicity in theta} still hold when we place $\sigg_1$ next to $\gamma$ and $\sigg_2$ next to $\theta$.  

By Lemma~\ref{lemma:equality of busemann to weights of BLPP}, $h_m^\theta(s,t) = B_m(s,t)$ if and only if $\tau_{(m,s),m}^{\theta,R} \ge t$, which covers the first of the two possible conditions (The inequality $\tau_{(m,s),m}^{\theta,R} \ge \tau_{(m,s),m}^{\gamma,R}$ follows by Theorem~\ref{existence of semi-infinite geodesics intro version}\ref{itm:monotonicity of semi-infinite jump times}\ref{itm:monotonicity in theta}).  By Theorem~\ref{thm:summary of properties of Busemanns for all theta}\ref{general queuing relations Busemanns},
\begin{equation} \label{hmr}
h_m^\theta(s,t)=  B_m(s,t) + \sup_{s \le u < \infty}\{B_m(u) - h_{m + 1}^\theta(u)\} - \sup_{t \le u < \infty}\{B_m(u) - h_{m + 1}^\theta(u)\}.
\end{equation}
Therefore, $h_m^\theta(s,t) = h_m^\gamma(s,t) > B_m(s,t)$ if and only if two conditions hold. The first is
\begin{multline*}
\sup_{s \le u < \infty}\{B_m(u) - h_{m +1}^\theta(u)\} - \sup_{s \le u < \infty}\{B_m(u) - h_{m +1}^\gamma(u)\}  \\
    =  \sup_{t \le u < \infty}\{B_m(u) - h_{m +1}^\theta(u)\} - \sup_{t \le u < \infty}\{B_m(u) - h_{m +1}^\gamma(u)\},
\end{multline*}
which comes from applying~\eqref{hmr} for both $\theta$ and $\gamma$.
By the previous case and the monotonicity of Theorem~\ref{existence of semi-infinite geodesics intro version}\ref{itm:monotonicity of semi-infinite jump times}\ref{itm:monotonicity in theta}, the second condition is 
\begin{equation} \label{mont_max}
\tau_{(m,s),m}^{\gamma,R} \le \tau_{(m,s),m}^{\theta,R} < t \le \tau_{(m,t),m}^{\gamma,L} \le \tau_{(m,t),m}^{\theta,L}.
\end{equation}
Then,
\begin{align}
    &B_m(\tau_{(m,s),m}^{\theta,R}) - h_{m + 1}^\theta(\tau_{(m,s),m}^{\theta,R}) - (B_m(\tau_{(m,s),m}^{\theta,R}) - h_{m + 1}^\gamma(\tau_{(m,s),m}^{\theta,R})) \nonumber \\
    \ge &\sup_{s \le u < \infty}\{B_m(u) - h_{m +1}^\theta(u)\} - \sup_{s \le u < \infty}\{B_m(u) - h_{m +1}^\gamma(u)\} \label{coaleq1} \\
    =  &\sup_{t \le u < \infty}\{B_m(u) - h_{m +1}^\theta(u)\} - \sup_{t \le u < \infty}\{B_m(u) - h_{m +1}^\gamma(u)\} \nonumber \\
    \ge &B_m(\tau_{(m,t),m}^{\gamma,L}) - h_{m + 1}^\theta(\tau_{(m,t),m}^{\gamma,L}) - (B_m(\tau_{(m,t),m}^{\gamma,L}) - h_{m + 1}^\gamma(\tau_{(m,t),m}^{\gamma,L})). \label{coaleq2}
\end{align}
Comparing the first and last lines above yields
\begin{equation} \label{hm+1ineq}
h_{m + 1}^\theta(\tau_{(m,s),m}^{\theta,R},\tau_{(m,t),m}^{\gamma,L}) \ge h_{m + 1}^\gamma(\tau_{(m,s),m}^{\theta,R},\tau_{(m,t),m}^{\gamma,L}).
\end{equation}
 
Theorem~\ref{thm:summary of properties of Busemanns for all theta}\ref{general monotonicity Busemanns} and~\eqref{mont_max} imply that~\eqref{hm+1ineq} is an equality. Thus, inequality~\eqref{coaleq1} is an equality, implying that
\[
B_m(\tau_{(m,s),m}^{\theta,R}) - h_{m + 1}^\gamma(\tau_{(m,s),m}^{\theta,R}) = \sup_{s \le u < \infty}\{B_m(u) - h_{m +1}^\gamma(u)\}.
\]
Thus, $\tau_{(m,s),m}^{\theta,R}$ is a maximizer of $B_m(u) - h_{m + 1}^\gamma(u)$ over $u \in [s,\infty)$. By definition, $\tau_{(m,s),m}^{\gamma,R}$ is the rightmost such maximizer, so $\tau_{(m,s),m}^{\theta,R} \le \tau_{(m,s),m}^{\gamma,R}$. However, by the first inequality of \eqref{mont_max},  $\tau_{(m,s),m}^{\gamma,R} = \tau_{(m,s),m}^{\theta,R}$. An analogous  argument using the equality in \eqref{coaleq2} shows that $\tau_{(m,t),m}^{\gamma,L} = \tau_{(m,t),m}^{\theta,L}$. The equality $h_{m + 1}^\theta(s',t') = h_{m + 1}^{\gamma}(s',t')$ follows since~\eqref{hm+1ineq} is an equality. 
\end{proof}

\begin{lemma} \label{lemma:vert_increments}
On the full-probability event $\Omega_2$, for all $m < r \in \Z$, $t \in \R$, $\theta > 0$, and $\sigg \in \{-,+\}$,
\begin{align} \label{vsum}
&\B^{\theta\sig}((m,t),(r,t)) \nonumber \\
= &\sup \Bigl\{\sum_{k = m}^{r - 1} B_k(t_{k - 1},t_k) - h_{r}^{\theta \sig}(t,t_{r - 1}): t = t_{m - 1} \le t_m \le \cdots \le t_{r - 1} <\infty\Bigr\}.
\end{align}
\end{lemma}

\begin{proof}
We proceed by induction. The case $r = m + 1$ is another way of stating $v_{m +1}^{\theta \sig}(t) = Q(h_{m +1}^{\theta \sig},B_m)(t)$, which is Theorem~\ref{thm:summary of properties of Busemanns for all theta}\ref{general queuing relations Busemanns}. Assume that~\eqref{vsum} holds for some $r > m$. By additivity and  Theorem~\ref{thm:summary of properties of Busemanns for all theta}\ref{general queuing relations Busemanns},
\begin{align*}
h_r^{\theta \sig}(t,t_{r - 1}) &= h_{r + 1}^{\theta \sig}(t,t_{r - 1}) + v_{r +1}^{\theta \sig}(t) - v_{r + 1}^{\theta \sig}(t_{r - 1}) \\
&=h_{r + 1}^{\theta \sig}(t,t_{r - 1}) + v_{r +1}^{\theta \sig}(t) - \sup_{t_{r - 1} \le t_r < \infty}\{B_r(t_{r - 1},t_r) - h_{r + 1}^{\theta \sig}(t_{r - 1},t_r)\} \\
&=v_{r + 1}^{\theta \sig}(t) - \sup_{t_{r - 1} \le t_r < \infty}\{B_r(t_{r - 1},t_r) - h_{r + 1}^{\theta \sig}(t,t_r)\}.
\end{align*}
Insert this into~\eqref{vsum} as follows. 
\begin{align*}
    &\sup_{t = t_{m - 1} \le \cdots \le t_r <\infty} \Bigl\{\sum_{k = m}^{r} B_k(t_{k - 1},t_k) - h_{r + 1}^{\theta \sig}(t,t_{r})\Bigr\} \\
    = &\B^{\theta \sig}((m,t),(r,t)) + v_{r+1}^{\theta \sig}(t) =\B^{\theta \sig}((m,t),(r + 1,t)),
\end{align*}
where the last equality above follows by additivity.
\end{proof}

\begin{lemma} \label{lemma:coaleqv}
Let $\omega \in \Omega_2$, $m  < r \in \Z$ and $t \in \R$. Assume that for some $0 < \gamma \le \theta$ and $\sigg_1,\sigg_2 \in \{-,+\}$,  $\B^{\theta \sig_2}((m,t),(r,t)) = \B^{\gamma \sig_1}((m,t),(r,t))$.  Then, $t_k := \tau_{(m,t),k}^{\theta \sig_2,L} = \tau_{(m,t),k}^{\gamma \sig_1,L}$ for $m \le k \le r - 1$, and $h_r^{\theta \sig_2}(t,t_{r -1 }) = h_r^{\gamma \sig_1}(t,t_{r - 1})$.    
\end{lemma}
\begin{proof}[Proof Lemma~\ref{lemma:coaleqv}]
If $\gamma = \theta$, without loss of generality, we may assume that $\sigg_1 = - $ and $\sigg_2 = +$. As in the proof of Lemma~\ref{lemma:coaleqh}, we suppress the $\sigg_1$ and $\sigg_2$ in the proof. 
By Lemmas~\ref{lemma:ptl_sig}\ref{itm:geo_maxes} and~\ref{lemma:vert_increments},
\begin{align}
&\sum_{k = m}^{r - 1} B_k(\tau_{(m,t),k - 1}^{\gamma,L},\tau_{(m,t),k}^{\gamma,L}) - h_{r}^\gamma(t,\tau_{(m,t),r - 1}^{\gamma,L}) \label{th1}\\
&=\sup_{t = t_{m - 1} \le t_m \le \cdots \le t_{r - 1} <\infty} \Bigl\{\sum_{k = m}^{r - 1} B_k(t_{k - 1},t_k) - h_{r}^\gamma(t,t_{r - 1})\Bigr\} \nonumber \\
&= \B^\gamma((m,t),(r,t)) = \B^\theta((m,t),(r,t)) \nonumber \\
&= \sup_{t = t_{m - 1} \le t_m \le \cdots \le t_{r - 1} <\infty} \Bigl\{\sum_{k = m}^{r - 1} B_k(t_{k - 1},t_k) - h_{r}^\theta(t,t_{r - 1})\Bigr\} \label{theta_max_vert} \\
&\ge \sum_{k = m}^{r - 1} B_k(\tau_{(m,t),k - 1}^{\gamma,L},\tau_{(m,t),k}^{\gamma,L}) - h_{r}^\theta(t,\tau_{(m,t),r - 1}^{\gamma,L}).  \label{ga1}
\end{align}
Comparing $h_r^\theta$ and $h_r^\gamma$ and using Theorem~\ref{thm:summary of properties of Busemanns for all theta}\ref{general monotonicity Busemanns}, the inequality~$\eqref{th1} \ge \eqref{ga1}$ must be an equality. Therefore, $t = \tau_{(m,t),m - 1}^{\gamma,L} \le\tau_{(m,t),m }^{\gamma,L} \le \cdots \le \tau_{(m,t),r - 1}^{\gamma,L}$ is maximal for the supremum in~\eqref{theta_max_vert}. Since $\tau_{(m,t),k}^{\gamma,L} \le \tau_{(m,t),k}^{\theta,L}$ for all $k \ge m$ (Theorem~\ref{existence of semi-infinite geodesics intro version}\ref{itm:monotonicity of semi-infinite jump times}\ref{itm:monotonicity in theta}) and $\tau_{(m,t),k}^{\theta,L}$ is the leftmost such maximizing sequence, $\tau_{(m,t),k}^{\theta,L} = \tau_{(m,t),k}^{\gamma,L}$ for $m \le k \le r - 1$, as desired. The equality $h_r^{\theta}(t,t_{r - 1}) = h_r^\gamma(t,t_{r- 1})$ then follows from the equality~\eqref{th1}=\eqref{ga1}.
\end{proof}

The previous lemma generalizes to points subject to southeast ordering. 

\begin{lemma} \label{lemma:coaleqdiag}
Let $\omega \in \Omega_2$ and  $(m,t) = \mbf x \SeS \mbf y = (r,s)$. Assume that for some $0 < \gamma \le \theta$ and $\sigg_1,\sigg_2 \in \{-,+\}$ that $\B^{\theta \sigg_2}(\mbf x,\mbf y) = \B^{\gamma \sigg_1}(\mbf x,\mbf y)$. Then, $t_k := \tau_{(m,t),k}^{\theta \sig_2,L} = \tau_{(m,t),k}^{\gamma \sig_1,L}$ for $m \le k \le r - 1$, and $h_r^{\theta \sig_2}(s,t_{r - 1}) = h_r^{\gamma \sig_1}(s,t_{r - 1})$.
\end{lemma}
\begin{proof}
If $\gamma = \theta$, without loss of generality, we may assume that $\sigg_1 = - $ and $\sigg_2 = +$. As in the proofs of Lemmas~\ref{lemma:coaleqh} and~\ref{lemma:coaleqv}, we suppress the $\sigg_1$ and $\sigg_2$ in the proof. By additivity,  $\B^\theta(\mbf x,\mbf y) = \B^\gamma(\mbf x,\mbf y)$ gives 
\be \label{hvthga}
\sum_{k = m + 1}^{r} v_k^{\theta }(t)  -h_r^{\theta }(s,t) = \sum_{k = m + 1}^{r} v_k^{\gamma }(t)  -h_r^{\gamma }(s,t). 
\ee
By Theorem~\ref{thm:summary of properties of Busemanns for all theta}\ref{general monotonicity Busemanns}, $v_k^{\theta} \ge v_k^{\gamma }$ and $h_m^{\theta } \li h_m^{\gamma}$. Hence, equality in~\eqref{hvthga} forces  $v_k^{\theta }(t) = v_k^{\gamma}(t)$ for $m + 1 \le k \le r$
and $h_r^\theta(s,t) = h_r^\gamma(s,t)$. Then, by Lemma~\ref{lemma:coaleqv}, $t_k := \tau_{(m,t),k}^{\theta,L} = \tau_{(m,t),k}^{\gamma,L}$ for $m \le k \le r - 1$ and $h_r^\theta(t,t_{r - 1}) = h_r^\gamma(t,t_{r - 1})$. Combining this equality with the equality $h_r^\theta(s,t) = h_r^\gamma(s,t)$ completes the proof.
\end{proof}

The next theorem contains the final step needed before we tackle the proof of Theorem~\ref{thm:general_coalescence}.  Recall the point $\mbf z^{\theta \sig}(\mbf x,\mbf y)$ defined in~\eqref{def:coalpt}. We have not yet shown that this quantity is well-defined. If, a priori, $\Gamma_{\mbf x}^{\theta \sig,L}$ and $\Gamma_{\mbf y}^{\theta \sig,R}$ do not intersect, we set $\mbf z^{\theta \sig}(\mbf x,\mbf y) = \infty$, interpreted as the point at $\infty$ in the one-point compactification of $\R^2$. Under this definition, $\mbf z^{\theta \sig}(\mbf x,\mbf y) \in \Z \times \R$ if and only if $\Gamma_{\mbf x}^{\theta \sig,L}$ and $\Gamma_{\mbf y}^{\theta \sig,R}$ intersect.

\begin{theorem} \label{thm:equiv_coal_pt}
Let $\omega \in \Omega_2$, $\gamma < \theta$,  $\sigg_1,\sigg_2 \in \{-,+\}$, and $\mbf x \SeS \mbf y$. Then, $\B^{\gamma \sig_1}(\mbf x,\mbf y) = \B^{\theta \sig_2}(\mbf x,\mbf y)$ if and only if $\mbf z^{\gamma \sig_1}(\mbf x,\mbf y) = \mbf z^{\theta \sig_2}(\mbf x,\mbf y) \in \Z \times \R$.  
\end{theorem}
\begin{remark}
The statement of this theorem is somewhat subtle. The ``only if" implication says two things: If there exists $\gamma < \theta$ such that $\B^{\gamma \sig_1}(\mbf x,\mbf y) = \B^{\theta \sig_2}(\mbf x,\mbf y)$, then $\mbf z^{\gamma \sig_1}(\mbf x,\mbf y)$ and $\mbf z^{\theta \sig_2}(\mbf x,\mbf y)$ are both in $\Z \times \R$, and they are equal. 
\end{remark}
\begin{proof}[Proof of Theorem~\ref{thm:equiv_coal_pt}]
 Set $\mbf x = (m,t)$ and $\mbf y = (r,s)$. The condition $\mbf x \SeS \mbf y$ gives $m \le r$ and $t \ge s$. Again, we suppress the $\sigg_1$ and $\sigg_2$ in the proofs. Assume first $\B^{\gamma}(\mbf x,\mbf y) = \B^{\theta}(\mbf x,\mbf y)$. By Lemma~\ref{lemma:coaleqdiag}, $\Gamma_{\mbf x}^{\theta,L}$ and $\Gamma_{\mbf x}^{\gamma,L}$ agree up to level $r$, and $h_r^\theta(s,t_{r - 1}) = h_r^\gamma(s,t_{r - 1})$. Therefore, by restarting the geodesics from level $r$, it suffices to assume $r = m$  so that $\mbf y = (r,s)$ and $\mbf x = (r,t)$ for some $r \in \Z$ and $s \le t$. If $s = t$, there is nothing to show, so we assume $s < t$.    
 
Now apply  Lemma~\ref{lemma:coaleqh} whose two cases are  illustrated in Figure~\ref{fig:samejump}.  In Case~\ref{itm:=Bm}, $\tau_{(r,s),r}^{\theta,R}$ and $\tau_{(r,s),r}^{\gamma,R}$ are both greater than or equal to $t$. Therefore, for $\eta = \theta,\gamma$, the first point of intersection of $\Gamma_{(r,s)}^{\eta,R}$ and $\Gamma_{(r,t)}^{\eta,L}$ is $(r,t)$. This precisely means that $\mbf z^{\theta}(\mbf x,\mbf y) = \mbf z^{\gamma}(\mbf x,\mbf y) = (r,t)$. 
 
In Case~\ref{itm: > Bm}, the paths $\Gamma_{(r,t)}^{\theta,L}$ and $\Gamma_{(r,t)}^{\gamma,L}$ both jump at time $t'$ to level $r + 1$, while $\Gamma_{(r,s)}^{\theta,R}$ and $\Gamma_{(r,s)}^{\gamma, R}$ both jump at time $s' < t$ to level $r + 1$. Furthermore, $h_{r + 1}^\theta(s',t') = h_{r + 1}^\gamma(s',t')$, so we may inductively repeat this procedure. By Theorem~\ref{existence of semi-infinite geodesics intro version}\ref{general limits for semi-infinite geodesics}, $\Gamma_{(r,t)}^{\theta,L}$ and $\Gamma_{(r,s)}^{\theta,R}$ are $\theta$-directed while  $\Gamma_{(r,t)}^{\gamma,L}$ and $\Gamma_{(r,s)}^{\gamma,R}$ are $\gamma$-directed. Hence, $\Gamma_{(r,t)}^{\theta,L}$ and $\Gamma_{(r,t)}^{\gamma,L}$ must separate eventually, and similarly for  $\Gamma_{(r,s)}^{\theta,R}$ and $\Gamma_{(r,s)}^{\gamma,R}$. Thus, this inductive procedure must terminate on some level at Case~\ref{itm:=Bm} of Lemma~\ref{lemma:coaleqh} and then $\mbf z^{\theta}(\mbf x,\mbf y) = \mbf z^{\gamma}(\mbf x,\mbf y) \in \Z \times \R$.
 
 For the converse claim of the theorem, if $\mbf z :=\mbf z^{\gamma}(\mbf x,\mbf y) = \mbf z^{\theta}(\mbf x,\mbf y) \in \Z \times \R$, then by Equation~\eqref{coaldiff},
$ 
\B^{\gamma}(\mbf x,\mbf y) = L_{\mbf x,\mbf z} - L_{\mbf y,\mbf z} = \B^{\theta}(\mbf x,\mbf y).
$ 
\end{proof}

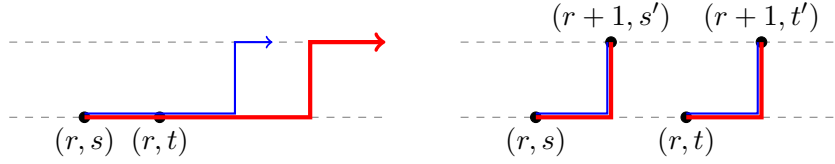
\begin{figure}[t]
    \centering
    \begin{tikzpicture}
    \draw[gray,thin,dashed] (0,0)--(5,0);
    \draw[gray,thin,dashed] (6,0)--(11,0);
    \draw[gray,thin,dashed] (0,1)--(5,1);
    \draw[gray,thin,dashed] (6,1)--(11,1);
    \filldraw[black] (1,0) circle (2pt) node[anchor = north] {$(r,s)$};
    \filldraw[black] (2,0) circle (2pt) node[anchor = north] {$(r,t)$};
    \filldraw[black] (7,0) circle (2pt) node[anchor = north] {$(r,s)$};
    \filldraw[black] (9,0) circle (2pt) node[anchor = north] {$(r,t)$};
    \filldraw[black] (8,1) circle (2pt) node[anchor = south] {$(r + 1,s')$};
    \filldraw[black] (10,1) circle (2pt) node[anchor = south] {$(r + 1,t')$};
    \draw[red,ultra thick,->] plot coordinates {(1,0)(4,0)(4,1)(5,1)};
    \draw[blue,thick,->] plot coordinates {(1,0.05)(3,0.05)(3,1)(3.5,1)};
    \draw[red,ultra thick,] plot coordinates {(7,0)(8,0)(8,1)};
    \draw[blue,thick] plot coordinates {(7,0.05)(7.95,0.05)(7.95,1)};
    \draw[red,ultra thick] plot coordinates {(9,0)(10,0)(10,1)};
    \draw[blue,thick] plot coordinates {(9,0.05)(9.95,0.05)(9.95,1)};
    \end{tikzpicture}
    \caption{\small $\theta$-directed geodesics are depicted in red/thick and $\gamma$-directed geodesics are depicted in blue/thin. The picture on the left depicts the case $h_r^\theta(s,t) = h_r^\gamma(s,t) = B_r(s,t)$ and the geodesics coalesce immediately at $(r,t)$. On the right, $h_r^\theta(s,t) = h_r^\gamma(s,t) > B_r(s,t)$. The geodesics from $(r,s)$ and $(r,t)$ make the same jumps to the next level$,h_{r + 1}^{\theta}(s',t') = h_{r + 1}^{\gamma}(s',t')$, and the induction continues.}
    \label{fig:samejump}
\end{figure}

\begin{proof}[Proof of Theorem~\ref{thm:general_coalescence}:]

\textbf{Part~\ref{itm:allsignscoalesce}}
First, we assume $\mbf x \Se \mbf y$ and show coalescence of the geodesics $\Gamma_{\mbf x}^{\theta +,L}$ and $\Gamma_{\mbf y}^{\theta+,R}$. By Theorem~\ref{thm:coupled_BMs_technical}\ref{convBuse}, $\B^{\theta +}(\mbf x,\mbf y) = \B^{\delta \sig}(\mbf x,\mbf y)$ for all $\delta > \theta$ sufficiently close. Then, by Theorem~\ref{thm:equiv_coal_pt}, $\mbf z^{\theta +}(\mbf x,\mbf y) \in \Z \times \R$, meaning that $\Gamma_{\mbf x}^{\theta+,L}$ and $\Gamma_{\mbf y}^{\theta+,R}$ intersect. By Lemma~\ref{lemma:SE_coal_pt}, the coalescence point is $\mbf z^{\theta+}(\mbf x,\mbf y)$. A symmetric argument applies to the geodesics $\Gamma_{\mbf x}^{\theta -,L}$ and $\Gamma_{\mbf y}^{\theta-,R}$ in the case $\mbf x \Se \mbf y$.

Now, let $\mbf x,\mbf y \in \Z \times \R$ be arbitrary and let $\Gamma_1 \in \mbf T_{\mbf x}^{\theta \sig}$ and $\Gamma_2 \in \mbf T_{\mbf y}^{\theta \sig}$.  Since $\Gamma_1$ and $\Gamma_2$ are infinite paths with direction $\theta > 0$, there exists $m \in \Z$ such that $(m,s) \in \Gamma_1$ and $(m,t) \in \Gamma_2$ for some $s,t \in \R$. Assume, without loss of generality, that $s \le t$. Let $s = s_{m - 1} \le s_m \le \cdots$ and $t = t_{m - 1} \le t_m \le \cdots$ denote the jump times from level $m$ of the semi-infinite geodesics $\Gamma_1$ and $\Gamma_2$, respectively. Then, for $r \ge m$, by Theorem~\ref{thm:geod_stronger}\ref{itm:global strong monotonicity in t},
\[
\tau_{(m,s - 1),r}^{\theta \sig,R} \le \tau_{(m,s),r}^{\theta\sig,L} \le s_r,t_r \le \tau_{(m,t),r}^{\theta \sig,R} \le \tau_{(m,t + 1),r}^{\theta \sig,L}.
\]
By the $\mbf x \Se \mbf y$ case, $\Gamma_{(m,s - 1)}^{\theta \sig,R}$ and $\Gamma_{(m,t + 1)}^{\theta \sig,L}$ coalesce. Therefore, for all sufficiently large $r$, $\tau_{(m,s - 1),r}^{\theta \sig,R} = \tau_{(m,t + 1),r}^{\theta \sig,L}$, so the above inequalities are all equalities for large $r$, and $\Gamma_1$ and $\Gamma_2$ indeed coalesce. The statement that the coalescence point is the first point of intersection in the case $\mbf x \Se \mbf y$ is then a direct consequence of Lemma~\ref{lemma:SE_coal_pt}.

\medskip \noindent \textbf{Part~\ref{itm:return_point}:}
By Theorem~\ref{thm:NU}\ref{itm:only_endpoint} and Remark~\ref{rmk:number_sIG}, for any two distinct geodesics $\Gamma_1, \Gamma_2 \in \mbf T_{\mbf x}^{\theta \sig}$, $\Gamma_1$ and $\Gamma_2$ exit the vertical line containing the point $\mbf x$ on different levels. Precisely, without loss of generality, there exists a level $r \ge m$ on which $\Gamma_1$ makes a vertical step from $(r,t)$ to $(r + 1,t)$, while $\Gamma_2$ makes a horizontal step to $(r,t +\ve)$ for some $\ve > 0$. Then, $(r,t + \ve) \Se (r + 1,t)$, and so by Part~\ref{itm:allsignscoalesce}, the minimal point of intersection of two $\theta \sig$ geodesics from $(r + 1,t)$ and $(r,t + \ve)$ is the coalescence point.  

\medskip \noindent \textbf{Part~\ref{itm:split_return}:} This follows because all geodesics in $\mbf T_{\mbf x}^{\theta \sig}$ lie between $\Gamma_{\mbf x}^{\theta \sig,L}$ and $\Gamma_{\mbf x}^{\theta \sig,R}$, which must coalesce by Part~\ref{itm:allsignscoalesce}.
\end{proof}

 \subsection{Proof the results from Section~\ref{sec:geometry_sub}} \label{section:geometry_proofs}

\begin{proof}[Proof of Theorem~\ref{thm:common_coalesce}]

\medskip \noindent \ref{itm:sameBuse}$\Leftrightarrow$\ref{itm:samecoal}: This is a direct application of Theorem~\ref{thm:equiv_coal_pt}.

\medskip \noindent 
\ref{itm:samecoal}$\Rightarrow$\ref{itm:allsamepath}:  If $\mbf z := \mbf z^{\gamma +}(\mbf x,\mbf y) = \mbf z^{\delta -}(\mbf x,\mbf y)$, then between $\mbf x$ and $\mbf z$, the portions of the paths $\Gamma_{\mbf x}^{\gamma +,L}$ and $\Gamma_{\mbf x}^{\delta -,L}$ are both leftmost geodesics between $\mbf x$ and $\mbf z$ by Theorem~\ref{existence of semi-infinite geodesics intro version}\ref{Leftandrightmost}. Hence, $\Gamma_{\mbf x}^{\gamma +,L}$ and $\Gamma_{\mbf x}^{\delta -,L}$ agree up to the point $\mbf z$. Let $\Gamma_1$ be this finite-length path from $\mbf x$ to $\mbf z$. Analogously, $\Gamma_{\mbf y}^{\gamma +,R}$ and $\Gamma_{\mbf y}^{\delta -,R}$  agree up to the point $\mbf z$. Let $\Gamma_2$ be the finite-length path from $\mbf x$ to $\mbf z$. By the monotonicity of Theorem~\ref{existence of semi-infinite geodesics intro version}\ref{itm:monotonicity of semi-infinite jump times}\ref{itm:monotonicity in theta}, for $\theta \in (\gamma,\delta)$ and $\sigg \in \{-,+\}$, $\Gamma_{\mbf x}^{\theta \sig,L}$ lies between $\Gamma_{\mbf x}^{\gamma +,L}$ and $\Gamma_{\mbf x}^{\delta -,L}$, so it must agree with $\Gamma_1$ between the two points. Similarly, $\Gamma_{\mbf y}^{\theta \sig,R}$ must agree with $\Gamma_2$ between $\mbf x$ and $\mbf z$. The point $\mbf z$ is the first point of intersection of $\Gamma_1$ and $\Gamma_2$ by definition. This is exactly~\ref{itm:allsamepath}.

\medskip \noindent \ref{itm:allsamepath}$\Rightarrow$\ref{itm:sameBuse}:  If~\ref{itm:allsamepath} holds, then by~\eqref{coaldiff}, 
\[
\B^{\gamma+}(\mbf x,\mbf y) = L_{\mbf x,\mbf z} - L_{\mbf y,\mbf z} = \B^{\delta-}(\mbf x,\mbf y). \qedhere
\]
\end{proof}


\begin{proof}[Proof of Theorem~\ref{thm:thetanotsupp}]
\ref{itm:pmBuse}$\Rightarrow$\ref{itm:pmcoal}: If $\B^{\theta -}(\mbf x,\mbf y) = \B^{\theta +}(\mbf x,\mbf y)$, then Theorem~\ref{thm:coupled_BMs_technical}\ref{convBuse} implies that for some $\gamma < \theta < \delta$, $\B^{\gamma +}(\mbf x,\mbf y) = \B^{\delta -}(\mbf x,\mbf y)$. Then,~\ref{itm:sameBuse}$\Rightarrow$\ref{itm:allsamepath} of Theorem~\ref{thm:common_coalesce} implies the existence of $\mbf z \in \Z \times \R$ and disjoint paths $\Gamma_1$ (from $\mbf x$ to $\mbf z$) and $\Gamma_2$ (from $\mbf y$ to $\mbf z$) such that for $\eta \in (\gamma,\delta)$ and $\sigg \in\{-,+\}$, $\Gamma_{\mbf x}^{\eta \sig,L}$ agrees with $\Gamma_1$ up to the point $\mbf z$, and $\Gamma_{\mbf y}^{\eta \sig,R}$ agrees with $\Gamma_2$ up to the point $\mbf z$. Applying this twice--both times with $\eta = \theta$, but once with $\sigg = +$ and once with $\sigg = -$, implies that $\mbf z = \mbf z^{\theta -}(\mbf x,\mbf y) = \mbf z^{\theta +}(\mbf x,\mbf y)$. 

\medskip \noindent \ref{itm:pmcoal}$\Rightarrow$\ref{itm:pmBuse}: If $\mbf z := \mbf z^{\theta -}(\mbf x,\mbf y) = \mbf z^{\theta +}(\mbf x,\mbf y)$, then by~\eqref{coaldiff}, 
\[
\B^{\theta -}(\mbf x,\mbf y) = L_{\mbf x,\mbf z} - L_{\mbf y,\mbf z} = \B^{\theta +}(\mbf x,\mbf y).
\]

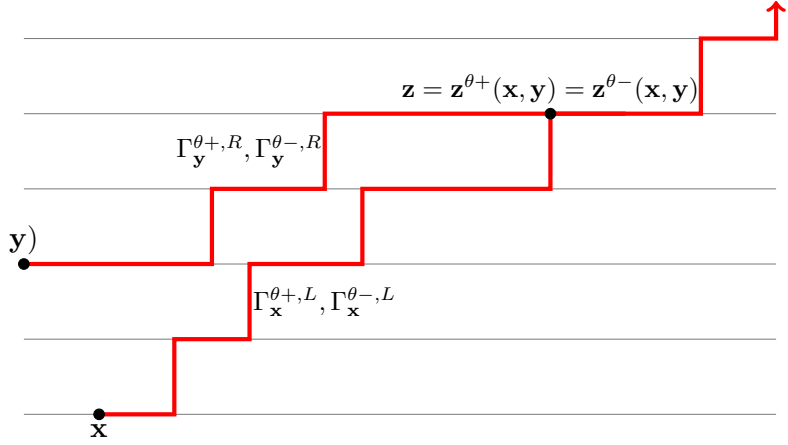
\begin{figure}[t]
            \begin{tikzpicture}
            \draw[gray,thin] (0,0)--(10,0);
            \draw[gray,thin] (0,1)--(10,1);
            \draw[gray,thin] (0,2)--(10,2);
            \draw[gray,thin] (0,3)--(10,3);
            \draw[gray,thin] (0,4)--(10,4);
            \draw[gray,thin] (0,5)--(10,5);
            \draw[red,ultra thick,->] plot coordinates {(1,0)(2,0)(2,1)(3,1)(3,2)(4.5,2)(4.5,3)(7,3)(7,4)(9,4)(9,5)(10,5)(10,5.5)};
            \draw[red,ultra thick] plot coordinates {(0,2)(2.5,2)(2.5,3)(4,3)(4,4)(8,4)};
            \filldraw[black] (1,0) circle (2pt) node[anchor = north] {$\mbf x$};
            \filldraw[black] (0,2) circle (2pt) node[anchor = south] {$\mbf y )$};
            \filldraw[black] (7,4) circle (2pt) node[anchor = south] {\small $\mbf z = \mbf z^{\theta +}(\mbf x,\mbf y) =\mbf z^{\theta -}(\mbf x,\mbf y) $};
            \node at (4,1.5) {\small $\Gamma_{\mbf x}^{\theta+,L},\Gamma_{\mbf x}^{\theta -,L}$};
            \node at (3,3.5) {\small $\Gamma_{\mbf y}^{\theta+,R},\Gamma_{\mbf y}^{\theta -,R}$};
            \end{tikzpicture}
            \caption{\small Common coalescence.}
            \label{fig:samecoalpt}
        \end{figure}
        
        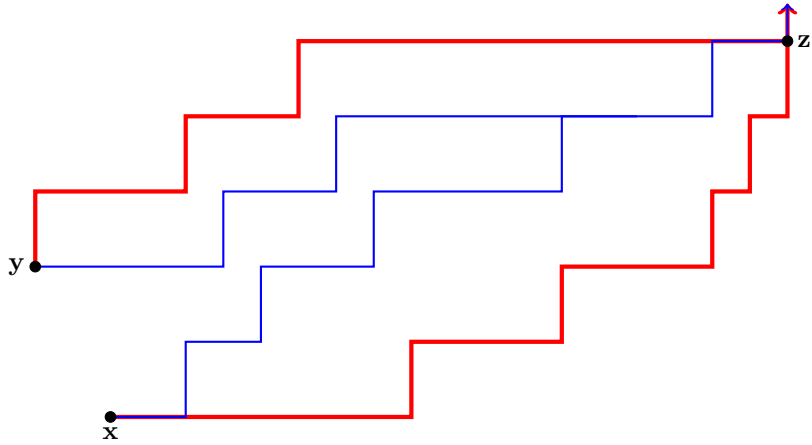
\begin{figure}[t]
            \begin{tikzpicture}
            \draw[red,ultra thick,->] plot coordinates {(1,0)(5,0)(5,1)(7,1)(7,2)(9,2)(9,3)(9.5,3)(9.5,4)(10,4)(10,5)(10,5.5)};
            \draw[red, ultra thick] plot coordinates
            {(0,2)(0,3)(2,3)(2,4)(2.5,4)(3.5,4)(3.5,5)(10,5)};
            \draw[blue, thick,->] plot coordinates {(1,0)(2,0)(2,1)(3,1)(3,2)(4.5,2)(4.5,3)(7,3)(7,4)(9,4)(9,5)(10,5)(10,5.5)};
            \draw[blue,thick] plot coordinates {(0,2)(2.5,2)(2.5,3)(4,3)(4,4)(8,4)};
            \filldraw[black] (1,0) circle (2pt) node[anchor = north] {\small $\mbf x$};
            \filldraw[black] (0,2) circle (2pt) node[anchor = east] {\small $\mbf y$};
            \filldraw[black] (10,5) circle (2pt) node[anchor = west] {\small $\mbf z$};
            \end{tikzpicture}
            \caption{\small The red/thick paths are $\Gamma_{\mbf x}^{\theta+,L}$ and $\Gamma_{\mbf y}^{\theta -,R}$, and the blue/thin paths are $\Gamma_{\mbf x}^{\theta-,L}$ and $\Gamma_{\mbf y}^{\theta +,R}$}
            \label{fig:ordering}
        \end{figure}

\medskip \noindent \ref{itm:pmcoal}$\Rightarrow$\ref{itm:disjointgeod}
Assume that $\mbf z := \mbf z^{\theta -}(\mbf x,\mbf y) = \mbf z^{\theta +}(\mbf x,\mbf y)$. Since leftmost Busemann geodesics are leftmost geodesics between their points and the same for rightmost (Theorem~\ref{existence of semi-infinite geodesics intro version}\ref{Leftandrightmost}), $\Gamma_{\mbf y}^{\theta +,R}$ agrees with $\Gamma_{\mbf y}^{\theta -,R}$ up to $\mbf z$, and $\Gamma_{\mbf x}^{\theta -,L}$ agrees with $\Gamma_{\mbf x}^{\theta +,L}$ up to $\mbf z$, as in Figure~\ref{fig:samecoalpt}. Therefore, $\Gamma_{\mbf x}^{\theta +,L}$ and $\Gamma_{\mbf y}^{\theta -,R}$ both contain the common point $\mbf z$.

\medskip \noindent \ref{itm:disjointgeod}$\Rightarrow$\ref{itm:pmcoal}:
Assume that $\mbf z \in \Gamma_{\mbf x}^{\theta+,L} \cap \Gamma_{\mbf y}^{\theta-,R}$ for some $\mbf z \in \Z \times \R$. Take $\mbf z$ to be the minimal point of intersection. In the degenerate case where $\mbf x$ lies directly below $\mbf y$, it is possible that $\Gamma_{\mbf x}^{\theta+,L}$ moves directly from $\mbf x$ to $\mbf y$, in which case $\Gamma_{\mbf x}^{\theta-,L}$ also moves directly up to $\mbf y$, so $\mbf y = \mbf z^{\theta -}(\mbf x,\mbf y) = \mbf z^{\theta +}(\mbf x,\mbf y)$. Otherwise, using monotonicity of the geodesics (Theorems~\ref{existence of semi-infinite geodesics intro version}\ref{itm:monotonicity of semi-infinite jump times}\ref{itm:monotonicity in theta} and~\ref{thm:geod_stronger}\ref{itm:global strong monotonicity in t}) the semi-infinite geodesics $\Gamma_{\mbf x}^{\theta-,L}$ and $\Gamma_{\mbf y}^{\theta+,R}$ both lie in between $\Gamma_{\mbf x}^{\theta+,L}$ and $\Gamma_{\mbf y}^{\theta-,R}$. See Figure~\ref{fig:ordering}. Thus, $\mbf z \in \Gamma_{\mbf x}^{\theta+,L} \cap \Gamma_{\mbf y}^{\theta-,R} \cap \Gamma_{\mbf x}^{\theta-,L}\cap \Gamma_{\mbf y}^{\theta+,R}$. Since $\Gamma_{\mbf x}^{\theta+,L}$ and $\Gamma_{\mbf x}^{\theta -,L}$ are both leftmost geodesics between $\mbf x$ and $\mbf z$ (Theorem~\ref{existence of semi-infinite geodesics intro version}\ref{Leftandrightmost}), they agree up to the point $\mbf z$. The same holds for rightmost geodesics between $\mbf y$ and $\mbf z$. Hence, the picture is given as in Figure~\ref{fig:samecoalpt}, not as in Figure~\ref{fig:ordering}, and $\mbf z = \mbf z^{\theta+}(\mbf x,\mbf y) = \mbf z^{\theta-}(\mbf x,\mbf y)$.
\end{proof}

\begin{proof}[Proof of Theorem~\ref{thm:geod_description}]
\ref{notTheta}$\Rightarrow$\ref{LRsame}: If $\theta \notin \Busedc$, then $h_m^{\theta+}(t) = h_m^{\theta -}(t)$ for all $m \in \Z$ and $t \in \R$, so \ref{LRsame} follows by the construction of semi-infinite geodesics from Busemann functions. 

\medskip \noindent 
\ref{LRsame}$\Rightarrow$\ref{coal}: Assuming that~\ref{LRsame} holds, we can dispense with the $\pm$ distinction. It suffices to show that for $m \in \Z$ and $s \le t \in \R$ , any $\theta$-directed semi-infinite geodesic, $\Gamma_1$, starting from $(m,s)$, coalesces with any $\theta$-directed semi-infinite geodesic $\Gamma_2$, starting from $(m,t)$. By Theorem~\ref{existence of semi-infinite geodesics intro version}\ref{itm:monotonicity of semi-infinite jump times},\ref{itm:all semi-infinite geodesics lie between leftmost and rightmost}, $\Gamma_1$ and $\Gamma_2$ lie between $\Gamma_{(m,s)}^{\theta,L}$ and $\Gamma_{(m,t)}^{\theta,R}$. Then, 
by Theorem~\ref{thm:general_coalescence}\ref{itm:allsignscoalesce}, $\Gamma_{(m,s)}^{\theta,L}$ and $\Gamma_{(m,t)}^{\theta,R}$ coalesce, so $\Gamma_1$ and $\Gamma_2$ also coalesce. 

\medskip \noindent \ref{coal}$\Rightarrow$\ref{notTheta}: We prove the contrapositive. If $\theta \in \Busedc$, then by Theorem~\ref{thm:Theta properties}\ref{itm:Theta=Vm}, $h_0^{\theta+}(s,t) < h_0^{\theta -}(s,t)$ for some $s < t$. By \ref{itm:disjointgeod}$\Leftrightarrow$\ref{itm:pmBuse} of Theorem~\ref{thm:thetanotsupp}, $\Gamma_{(0,t)}^{\theta +,L} \cap \Gamma_{(0,s)}^{\theta -,R} = \varnothing$, and these two $\theta$-directed geodesics do not coalesce.

\medskip \noindent \ref{LRsame}$\Rightarrow$\ref{countonegeod}: By definition~\eqref{NU_0_1}, for all $\mbf x \in (\Z\times \R) \setminus \NU_0$, $\theta > 0$, and $\sigg = \{-,+\}$, $\Gamma_{\mbf x}^{\theta \sig,R} = \Gamma_{\mbf x}^{\theta \sig,L}$. By assumption, we also have $\Gamma_{\mbf x}^{\theta -,R} = \Gamma_{\mbf x}^{\theta+,R}$. Therefore $\Gamma_{\mbf x}^{\theta -,L} = \Gamma_{\mbf x}^{\theta +,R}$, and the result then follows from Theorem~\ref{existence of semi-infinite geodesics intro version}\ref{itm:all semi-infinite geodesics lie between leftmost and rightmost}.

\medskip \noindent \ref{countonegeod}$\Rightarrow$\ref{onegeod}: This is immediate since $\NU_0$ is countable by Theorem~\ref{thm:NU}\ref{itm:union} and therefore not all of $\Z \times \R$. 

\medskip \noindent \ref{onegeod}$\Rightarrow$\ref{Ronegeod} and $\ref{onegeod}\Rightarrow\ref{Lonegeod}$ follow immediately by Theorem~\ref{existence of semi-infinite geodesics intro version}\ref{itm:all semi-infinite geodesics lie between leftmost and rightmost}.

\medskip \noindent \ref{Ronegeod}$\Rightarrow$\ref{LRsame}: Assume that $\Gamma_{\mbf x}^{\theta+,R} = \Gamma_{\mbf x}^{\theta -,R}$ for some $\mbf x \in \Z \times \R$. Let $\mbf y \in \Z \times \R$. By Theorem~\ref{thm:general_coalescence}\ref{itm:allsignscoalesce}, $\Gamma_{\mbf y}^{\theta-,L}, \Gamma_{\mbf y}^{\theta +,L},\Gamma_{\mbf y}^{\theta-,R}$, and $\Gamma_{\mbf y}^{\theta +,R}$ all coalesce with  $\Gamma_{\mbf x}^{\theta+,R} = \Gamma_{\mbf x}^{\theta -,R}$. Hence, $\Gamma_{\mbf y}^{\theta-,R}$ and $\Gamma_{\mbf y}^{\theta +,R}$ coalesce. Let $\mbf z$ be the coalescence point. By uniqueness of rightmost geodesics, $\Gamma_{\mbf y}^{\theta-,R}$ and $\Gamma_{\mbf y}^{\theta +,R}$ agree from $\mbf y$ to $\mbf z$, so $\mbf z = \mbf y$, and $\Gamma_{\mbf y}^{\theta-,R} = \Gamma_{\mbf y}^{\theta +,R}$. By a similar argument, $\Gamma_{\mbf y}^{\theta-,L} = \Gamma_{\mbf y}^{\theta +,L}$.

\medskip \noindent \ref{Lonegeod}$\Rightarrow$\ref{LRsame}: This follows analogously as for \ref{Ronegeod}$\Rightarrow$\ref{LRsame}.

\medskip \noindent \textbf{Part~\ref{allBuse}}:
Let $\theta \in (0,\infty) \setminus \Busedc$  and let $t = t_{m - 1} \le t_m \le \cdots$ be the jump times of  a $\theta$-directed geodesic  $\Gamma$  started from $\mbf x = (m,t)\in \Z \times \R$.   By the implication~\ref{notTheta}$\Rightarrow$\ref{coal}, $\Gamma$ coalesces with $\Gamma_{\mbf x}^{\theta,R}$. By Theorem~\ref{existence of semi-infinite geodesics intro version}\ref{Leftandrightmost}, for any point $\mbf z =(n,u) \in \Gamma \cap \Gamma_{\mbf x}^{\theta,R}$, the energy of a geodesic between $\mbf x$ and $\mbf z$ is given by $\B^\theta(\mbf x,\mbf z)$. Then, by the additivity of Theorem~\ref{thm:summary of properties of Busemanns for all theta}\ref{general additivity Busemanns},
\begin{align*}
\sum_{r = m}^{n - 1} B_r(t_{r - 1},t_r) + B_n(t_{n - 1},u) &= L_{\mbf x,\mbf z} \\
&= \B^{\theta}(\mbf x,\mbf z) = \sum_{r = m}^{n - 1} \bigl[ h_r^\theta(t_{r - 1},t_r) + v_{r + 1}^\theta(t_r)\bigr] + h_n^\theta(t_{n - 1},u).
\end{align*}
By the monotonicity of Theorem~\ref{thm:summary of properties of Busemanns for all theta}\ref{general monotonicity Busemanns}, $B_r(t_{r -1},t_r) = h_r^\theta(t_{r - 1},t_r)$ and $v_{r + 1}^\theta(t_r) = 0$ for $m \le r \le n - 1$. Combine this with the following consequences of Theorem~\ref{thm:summary of properties of Busemanns for all theta}\ref{general queuing relations Busemanns}:
\begin{align*}
h_r^\theta(t_{r - 1},t_r) &= B_m(t_{r - 1},t_r) + \sup_{t_{r - 1} \le u < \infty} \{B_r(u) - h_{r +1}^\theta(u)\} \\ & \qquad\qquad\qquad\qquad\qquad - \sup_{t_{r} \le u < \infty} \{B_r(u) - h_{r +1}^\theta(u)\}, \\
\text{and} \quad 
v_{r + 1}^\theta(t_r) &= \sup_{t_r \le u < \infty}\{B_r(t_r,u) - h_{r + 1}^\theta(t_r,u)\}. 
\end{align*}
From this, 
\begin{align*}
&\sup_{t_{r - 1} \le u < \infty} \{B_r(u) - h_{r +1}^\theta(u)\} =\sup_{t_{r} \le u < \infty} \{B_r(u) - h_{r +1}^\theta(u)\}  
=B_r(t_r) - h_{r + 1}^\theta(t_r),
\end{align*}
where the first equality comes from $h_r^\theta(t_{r - 1},t_r) = B_r(t_{r - 1},t_r)$, and the second inequality comes from $v_{r + 1}^\theta(t_r) = 0$.
Hence, $t_r$ maximizes $B_r(u) - h_{r + 1}^\theta(u)$ over $u \in [t_{r - 1},\infty)$ for $m \le r \le n - 1$. Letting $n \rightarrow \infty$ completes the proof. 
\end{proof}

\subsection{Proofs of results from Section~\ref{section:ci}}
For $n > m$, we recall the definitions 
\be \label{sigmaL}
\sigma_{(m,s),n}^L := \sup\big\{t \ge s: \Gamma_{(m,s),(n,t)}^L \text{ passes through }(m + 1,s)\big\}.
\ee
and
\be \label{sigmaR}
\sigma_{(m,s),n}^R := \sup\big\{t \ge s: \Gamma_{(m,s),(n,t)}^R \text{ passes through }(m + 1,s)\big\}.
\ee

 We now restate the definition of $\theta_{(m,s)}^L$ and $\theta_{(m,s)}^R$ given in Section~\ref{section:ci} for convenience of the reader. 
\begin{equation} \label{eqn:cidir}
\theta_{(m,s)}^L := \sup\{\theta \ge 0: \tau_{(m,s),m}^{\theta\sig,L} = s\},\;\;\text{and}\;\; \theta_{(m,s)}^R := \sup\{\theta \ge 0: \tau_{(m,s),m}^{\theta\sig,R} = s\}.  
\end{equation}

\begin{proof}[Proof of Lemma~\ref{lemma:ci_finite}]
 By~\eqref{eqn:limits in theta to infinity}, $\lim_{\theta \rightarrow \infty}\tau_{(m,s),m}^{\theta \sig,L} = \infty$ for all $(m,s) \in \Z \times \R$, so $\tau_{(m,s),m}^{\theta \sig,L} > s$ for sufficiently large $\theta > 0$. Hence, $\theta_{(m,s)}^L < \infty$. Further, by Theorem~\ref{existence of semi-infinite geodesics intro version}\ref{Leftandrightmost}, the portion of the semi-infinite geodesic $\Gamma_{(m,s)}^{\theta-,L}$ between the points $(m,s)$ and $(n,\tau_{(m,s),n}^{\theta -,L})$ is the leftmost geodesic between these two points. Hence, by~\eqref{sigmaL} and planarity, for such sufficiently large $\theta$, $\sigma_{(m,s),n}^L < \tau_{(m,s),n}^{\theta\sig,L} < \infty$.
\end{proof}

We first prove Theorem~\ref{thm:ci_equiv} before moving on to the proof of Theorem~\ref{thm:limiting direction of competition interface}.

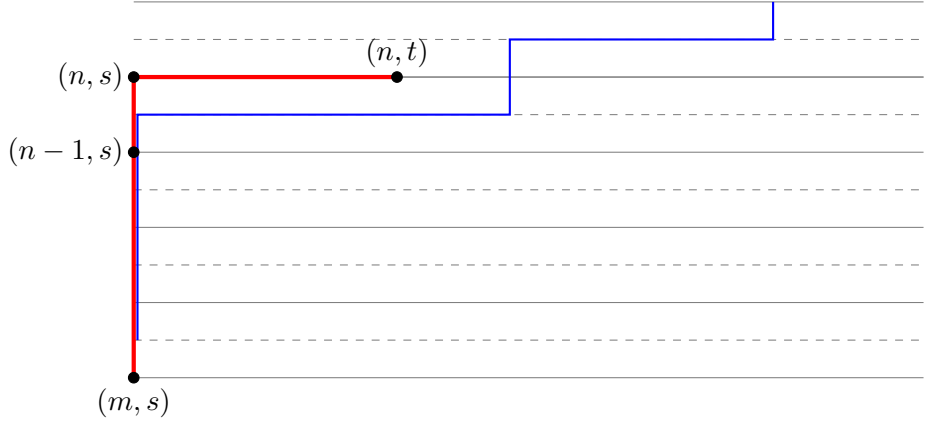
\begin{figure}[t]
\begin{tikzpicture}
\centering
\draw[gray,thin] (-0.5,0)--(10,0);
\draw[gray,thin] (-0.5,1)--(10,1);
\draw[gray,thin] (-0.5,2)--(10,2);
\draw[gray,thin] (-0.5,3)--(10,3);
\draw[gray,thin] (-0.5,4)--(10,4);
\draw[gray,thin] (-0.5,4)--(10,4);
\draw[gray,thin] (-0.5,5)--(10,5);
\draw[gray,thin,dashed] (-0.5,0.5)--(10,0.5);
\draw[gray,thin,dashed] (-0.5,1.5)--(10,1.5);
\draw[gray,thin,dashed] (-0.5,2.5)--(10,2.5);
\draw[gray,thin,dashed] (-0.5,3.5)--(10,3.5);
\draw[gray,thin,dashed] (-0.5,4.5)--(10,4.5);
\draw[blue,thick] (-0.45,0.5)--(-0.45,3.5)--(4.5,3.5)--(4.5,4.5)--(8,4.5)--(8,5);
\draw[red,ultra thick] (-0.5,0)--(-0.5,4)--(3,4);
\filldraw[black] (-0.5,0) circle (2pt) node[anchor = north] {$(m,s)$};
\filldraw[black] (-0.5,3) circle (2pt) node[anchor = east] {$(n - 1,s)$};
\filldraw[black] (-0.5,4) circle (2pt) node[anchor = east] {$(n,s)$};
\filldraw[black] (3,4) circle (2pt) node[anchor = south] {$(n,t)$};
\end{tikzpicture}
\caption{\small Example of a nontrivial competition interface (blue/thin) that is vertical for the first three steps. The red/thick path gives the geodesic between $(m,s)$ and $(n,t)$.}
\label{fig:nontrivial competition interface vertical then horizontal}
\end{figure}

\begin{proof}[Proof of Theorem~\ref{thm:ci_equiv}]
\medskip \noindent \ref{itm:some_vert_geod}$\Rightarrow$\ref{itm:Ltriv}: If, for some $n > m$ and $t > s$, one geodesic between $(m,s)$ and $(n,t)$ passes through $(m + 1,s)$, then the leftmost geodesic passes through $(m + 1,s)$, and $\sigma_{(m,s),n}^L \ge t > s$.

\medskip \noindent \ref{itm:Ltriv}$\Rightarrow$\ref{itm:some_vert_geod}: If $\sigma_{(m,s),n}^L > s$ for some $n > m$, then by definition~\eqref{sigmaL}, when $t \in [s,\sigma_{(m,s),n}^L]$, the leftmost geodesic between $(m,s)$ and $(n,t)$ passes through $(m +1,s)$

. 

\medskip \noindent\ref{itm:Ltriv}$\Rightarrow$\ref{itm:Rtriv}: Assume that $(m,s) \in \Z \times \R$ has trivial left competition interface, and let $n > m$ be the minimal index such that $\sigma_{(m,s),n}^L > s$. We show that $\sigma_{(m,s),n}^R > s$.  

Under this assumption, by~\eqref{eqn:LR_ci}, $\sigma_{(m,s),r}^R= \sigma_{(m,s),r}^L = s$ for $m \le r < n$. By definition, there exists some $t > s$ such that the leftmost geodesic between $(m,s)$ and $(n,t)$ passes through $(m + 1,s)$. Because of this and the assumption that $\sigma_{(m,s),n - 1}^L = s$, this geodesic cannot pass through $(n - 1,u)$ for any $u > s$. Hence, the left-most geodesic passes through $(n,s)$, as in Figure~\ref{fig:nontrivial competition interface vertical then horizontal}. Recall that the jump times for geodesics between $(m,s)$ and $(n,t)$ are defined as maximizers of the function
\[ 
B_m(s,s_m) + B_{m + 1}(s_m,s_{m + 1}) + \cdots + B_{n - 1}(s_{n - 2},s_{n - 1})+ B_n(s_{n -1},t)
\]
over all sequences $s \le s_m \le \cdots \le s_{n - 1} \le t$. Equivalently, they are maximizers of the function
\be \label{eqn:max_seq}
B_m(s_m) + B_{m + 1}(s_m,s_{m + 1}) + \cdots + B_{n - 1}(s_{n - 2},s_{n - 1})- B_n(s_{n -1}),
\ee
over the same set of sequences.
Since the left-most geodesic between $(m,s)$ and $(n,t)$ passes through $(n,s)$, this means that the sequence $s = s_m = s_{m + 1} = \cdots = s_{n - 1}$ is a maximizing sequence for~\eqref{eqn:max_seq} over all sequences $s \le s_m \le \cdots \le s_{n - 1} \le t$. By Item~\ref{itm:finite_geodesics} below the definition~\eqref{omega4} of $\Omega_4$, there are only finitely many maximizers. Choose $\hat t > s$ to be such that 
\[
\hat t  < t \wedge \min\{s_{n - 1} > s: s = s_{m - 1} \le \cdots \le s_{n - 1} \le t \text{ is a maximizing sequence for}~\eqref{eqn:max_seq}\},
\]
where we define the minimum of the empty set to be $\infty$.
Then, $s = s_{m - 1} = \cdots = s_{n  - 1}$ is the unique maximizing sequence of~\eqref{eqn:max_seq} over all sequences $s = s_{m - 1} \le \cdots \le s_{n - 1} \le \hat t$, so there is a unique geodesic between $(m,s)$ and $(n,\hat t)$, and it passes through $(m + 1,s)$. Hence, $\sigma_{(m,s),n}^R > s$, as desired.

\medskip \noindent \ref{itm:Rtriv}$\Rightarrow$\ref{itm:Ltriv}: This is immediate from~\eqref{eqn:LR_ci}.

\medskip \noindent \ref{itm:Ltriv}$\Leftrightarrow$\ref{itm:thetaL > 0} and \ref{itm:Rtriv}$\Leftrightarrow$\ref{itm:thetaR > 0}: We prove \ref{itm:Rtriv}$\Leftrightarrow$\ref{itm:thetaR > 0}, and \ref{itm:Ltriv}$\Leftrightarrow$\ref{itm:thetaL > 0} is analogous. For this, we choose an arbitrary point $(m,s)$ and use the shorthand notation $\wh \theta = \theta_{(m,s)}^R$. 

If $\wh \theta =0 $, then by the definition~\eqref{eqn:cidir}, $\tau_{(m,s),m}^{\theta \sig,R} > s$ for all $\theta > 0$ and $\sigg \in\{-,+\}$. Then, for all $\theta > 0$ and $\sigg \in \{-,+\}$, $\Gamma_{(m,s)}^{\theta \sig,R}$ does not pass through $(m + 1,s)$. For $n > m$, $(n,\tau_{(m,s),n}^{\theta\sig,R}) \in \Gamma_{(m,s)}^{\theta \sig,R}$, so $\sigma_{(m,s),n}^R \le \tau_{(m,s),n}^{\theta \sig,R}$ for all $\theta > 0$ and $\sigg \in \{-,+\}$. By~\eqref{eqn:limits in theta to infinity}, $\lim_{\theta \searrow 0} \tau_{(m,s),n}^{\theta \sig,R} = s$, so $\sigma_{(m,s),n}^R = s$ for all $n > m$. 

Now, assume  $\sigma_{(m,s),n}^R = s$ for all $n > m$. Then, by~\eqref{sigmaR}, for all $n > m$ and $t > s$, the rightmost geodesic between $(m,s)$ and $(n,t)$ passes through $(m,s + \ve)$ for some $\ve > 0$. 
By Theorem~\ref{existence of semi-infinite geodesics intro version}\ref{general limits for semi-infinite geodesics}, for each $\theta > 0$ and $\sigg \in \{-,+\}$,  we may choose $n$ large enough so that $\tau_{(m,s),n}^{\theta \sig,R} > s$. Then, the rightmost geodesic between $(m,s)$ and $(n,\tau_{(m,s),n}^{\theta \sig,R})$ passes through $(m,s + \ve)$ for some $\ve > 0$. By Theorem~\ref{existence of semi-infinite geodesics intro version}\ref{Leftandrightmost}, this rightmost geodesic agrees with the portion of $\Gamma_{(m,s)}^{\theta \sig,R}$ from $(m,s)$ to $(n,\tau_{(m,s),n}^{\theta \sig,R})$. Thus, for all $\theta > 0$ and $\sigg \in\{-,+\}$, $\Gamma_{(m,s)}^{\theta \sig,R}$ passes through $(m,s + \ve)$ for some $\ve >  0$, and $\tau_{(m,s),m}^{\theta \sig,R} > s$. By~\eqref{eqn:cidir}, $\wh \theta = 0$.

\medskip \noindent \ref{itm:thetaL > 0}$\Leftrightarrow$\ref{itm:Lvert} and \ref{itm:thetaR > 0}$\Leftrightarrow$\ref{itm:Rvert}: These are immediate from the definitions~\eqref{eqn:cidir}. The clarifications for Parts~\ref{itm:Lvert} and~\ref{itm:Rvert} are outlined in the proof of the next implications. 

\medskip \noindent \ref{itm:thetaL > 0}$\Rightarrow$\ref{itm:Lsplit} and \ref{itm:thetaR > 0}$\Rightarrow$\ref{itm:Rsplit}. We prove \ref{itm:thetaR > 0}$\Rightarrow$\ref{itm:Rsplit}, and \ref{itm:thetaL > 0}$\Rightarrow$\ref{itm:Lsplit} is analogous. Again, we use the shorthand notation $\wh \theta = \theta_{(m,s)}^R$ and assume $\wh \theta > 0$. By definition~\eqref{eqn:cidir} and the monotonicity of Theorem~\ref{existence of semi-infinite geodesics intro version}\ref{itm:monotonicity of semi-infinite jump times}\ref{itm:monotonicity in theta}, for $\gamma < \wh \theta$, and $\sigg \in \{-,+\}$, $\tau_{(m,s),m}^{\gamma\sig,R} = s$, while for  $\delta > \wh \theta$ and $\sigg \in \{-,+\}$, $\tau_{(m,s),m}^{\delta \sig,R} > s$.  By Theorem~\ref{thm:geod_stronger}\ref{itm:stronger convergence theta}, for $\gamma < \wh \theta < \delta$ sufficiently close to $\wh \theta$, and $\sigg \in \{-,+\}$, $\tau_{(m,s),m}^{\gamma \sig,R} = \tau_{(m,s),m}^{\wh \theta -,R}$ and $\tau_{(m,s),m}^{\delta \sig,R} = \tau_{(m,s),m}^{\wh \theta +,R}$. Therefore, $\tau_{(m,s),m}^{\wh \theta -,R} = s$ and  $\tau_{(m,s),m}^{\wh \theta+,R} > s$. In other words, $\Gamma_{(m,s)}^{\wh \theta-,R}$ makes a vertical step to $(m + 1,s)$ while $\Gamma_{(m,s)}^{\wh \theta+,R}$ moves horizontally to $(m,s + \ve)$ for some $\ve > 0$.
By Theorem~\ref{existence of semi-infinite geodesics intro version}\ref{Leftandrightmost}, $\Gamma_{(m,s)}^{\wh \theta-,R}$ and $\Gamma_{(m,s)}^{\wh \theta+,R}$  are both the rightmost geodesic between any two of their points. 
Thus, $\Gamma_{(m,s)}^{\wh \theta-,R}$ and $\Gamma_{(m,s)}^{\wh \theta+,R}$ cannot meet again after the initial point $(m,s)$, or else there would be two rightmost geodesics between $(m,s)$ and some point $(n,t) \ge (m,s)$.  Refer to Figure~\ref{fig:pf_semi-infinte geodesics split at initial point} for clarity. Hence, $\Gamma_{(m,s)}^{\wh \theta-,R} \cap \Gamma_{(m,s)}^{\wh \theta+,R} = \{(m,s)\}$.  By the monotonicity of Theorem~\ref{existence of semi-infinite geodesics intro version}\ref{itm:monotonicity of semi-infinite jump times}\ref{itm:monotonicity in theta}, for $\gamma < \wh \theta$, $\Gamma_{(m,s)}^{\gamma+,R}$ and $\Gamma_{(m,s)}^{\gamma -,R}$  both travel to $(m + 1,s)$ and thus contain $(x,s)$ for $x \in [m,m + 1]$, and for $\gamma > \wh \theta$, $\Gamma_{(m,s)}^{\gamma+,R}$ and $\Gamma_{(m,s)}^{\gamma -,R}$ both travel to $(m,s + \ve)$ for some $\ve > 0$ and therefore contain $(m,x)$ for $x \in [s,s + \ve]$. Therefore, $\wh \theta$ is indeed the unique direction $\gamma$ such that $\Gamma_{(m,s)}^{\gamma-,R} \cap \Gamma_{(m,s)}^{\gamma+,R}$ is a finite set.

In summary, for $S \in \{L,R\}$, $\tau_{(m,s),m}^{\theta \sig,S} = s$ for $\theta < \theta_{(m,s)}^{S}$, while $\tau_{(m,s),m}^{\theta \sig,S}$ for $\theta > \theta_{(m,s),m}^S$. Furthermore, $\tau_{(m,s),m}^{\theta_{(m,s)}^{S}-,S} = s < \tau_{(m,s),m}^{\theta_{(m,s)}^{S}+,S}$. This proves the clarifications stated in Parts~\ref{itm:Lvert} and~\ref{itm:Rvert}.

\medskip \noindent \ref{itm:Lsplit}$\Rightarrow$\ref{itm:Lvert} and \ref{itm:Rsplit}$\Rightarrow$\ref{itm:Rvert}: If $\Gamma_{(m,s)}^{\theta -,L} \cap \Gamma_{(m,s)}^{\theta +,L} = \{(m,s)\}$, then $\Gamma_{(m,s)}^{\theta -,L}$ must make an initial vertical step to $(m + 1,s)$. The same is true for $`L'$ replaced with $`R'$. 

\medskip \noindent \ref{itm:Lvert}$\Leftrightarrow$\ref{itm:v = 0}: By Theorem~\ref{thm:summary of properties of Busemanns for all theta}\ref{general queuing relations Busemanns}, 
$
v_{m + 1}^{\theta \sig}(s) = \sup_{s \le u <\infty}\{B_m(s,u) - h_{m + 1}^{\theta\sig}(s,u)\}
$,
so 
\be \label{eqn:v = 0 equiv}
v_{m +1}^{\theta \sig}(s) = 0 
\Leftrightarrow \text{$s$ maximizes $B_m(u) - h_{m + 1}^{\theta \sig}(u)$ over $u \in [s,\infty)$} \Leftrightarrow \tau_{(m,s),m}^{\theta \sig,L} = s,
\ee
thus completing the proof.
\end{proof}

\begin{figure}[t]
\begin{adjustbox}{max totalsize={5.5in}{5in},center}
\begin{tikzpicture}
\draw[gray,thin] (-0.5,0)--(15,0);
\draw[gray,thin] (-0.5,1)--(15,1);
\draw[gray,thin] (-0.5,2)--(15,2);
\draw[gray,thin] (-0.5,3)--(15,3);
\draw[gray,thin] (-0.5,4)--(15,4);
\draw[gray,thin] (-0.5,4)--(15,4);
\draw[gray,thin] (-0.5,5)--(15,5);
\draw[gray,thin,dashed] (-0.5,0.5)--(15,0.5);
\draw[gray,thin,dashed] (-0.5,1.5)--(15,1.5);
\draw[gray,thin,dashed] (-0.5,2.5)--(7,2.5);
\draw[gray,thin,dashed] (9,2.5)--(15,2.5);
\draw[gray,thin,dashed] (-0.5,3.5)--(3,3.5);
\draw[gray,thin,dashed] (5,3.5)--(15,3.5);
\draw[gray,thin,dashed] (-0.5,4.5)--(15,4.5);
\draw[blue,thick,<-] (14.5,5)--(14.5,4.5)--(13,4.5)--(13,3.5)--(6,3.5)--(6,2.5)--(4.5,2.5)--(4.5,1.5)--(1.5,1.5)--(1.5,0.5)--(-0.5,0.5);
\draw[red,ultra thick,->] (-0.5,0)--(2,0)--(2,1)--(5,1)--(5,2)--(7,2)--(7,3)--(14,3)--(14,4)--(15,4)--(15,4.5);
\draw[red,ultra thick,->] (-0.5,0)--(-0.5,1)--(1,1)--(1,2)--(3,2)--(3,3)--(5,3)--(5,4)--(10,4)--(10,5)--(13,5);
\node at (8,2.5) {$\Gamma_{(m,s)}^{\wh \theta+,R}$};
\node at (4,3.5) {$\Gamma_{(m,s)}^{\wh \theta-,R}$};
\filldraw[black] (-0.5,0) circle (2pt) node[anchor = north] {$(m,s)$};
\filldraw[black] (-0.5,1) circle (2pt) node[anchor = east] {$(m + 1,s)$};
\end{tikzpicture}
\end{adjustbox}
\caption{\small When the competition interface direction $\widehat\theta = \theta_{(m,s)}^R > 0$, $\Gamma_{(m,s)}^{\wh \theta-,R}$ (upper red/thick path) immediately splits from $\Gamma_{(m,s)}^{\wh \theta +,R}$ (lower red/thick path). These paths never touch after the initial point, and the competition interface (blue/thin) lies between the paths.}
\label{fig:pf_semi-infinte geodesics split at initial point}
\end{figure}
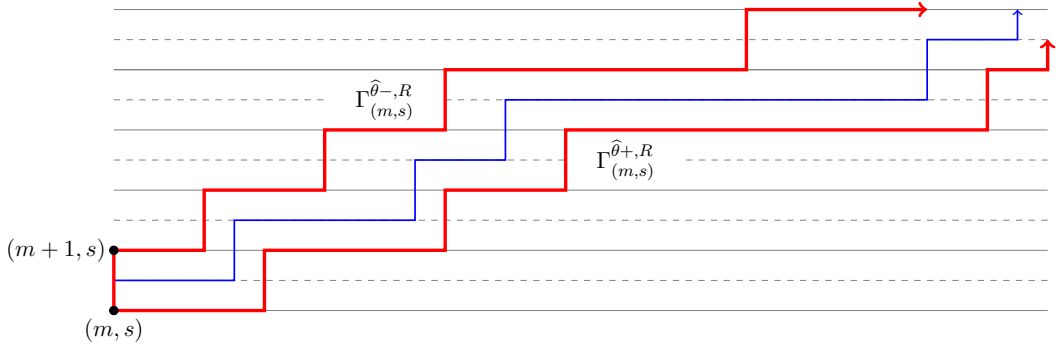

\begin{proof}[Proof of Theorem~\ref{thm:limiting direction of competition interface}] We prove the limit for $\wh \theta := \theta_{(m,s)}^R$, and the other limit follows analogously. We consider two cases: $\wh \theta = 0$ and $\wh \theta > 0$. If $\wh \theta = 0$, then by \ref{itm:Rtriv}$\Leftrightarrow$\ref{itm:thetaR > 0} of Theorem~\ref{thm:ci_equiv}, $\sigma_{(m,s),n}^{R} = s$ for all $n > m$, and $\lim_{n \rightarrow \infty}\tf{\sigma_{(m,s),n}^{R}}{n} = 0$.   

Now, assume $\wh \theta > 0$. By~\ref{itm:thetaR > 0}$\Rightarrow$\ref{itm:Rsplit} of Theorem~\ref{thm:ci_equiv}, $\Gamma_{(m,s)}^{\wh \theta-,R}$ makes an immediate vertical step to $(m + 1,s)$, while $\Gamma_{(m,s)}^{\wh \theta+,R}$ makes an immediate horizontal step. By Theorem~\ref{existence of semi-infinite geodesics intro version}\ref{Leftandrightmost}, $\Gamma_{(m,s)}^{\wh \theta-,R}$ and $\Gamma_{(m,s)}^{\wh \theta+,R}$ are rightmost geodesics between any of their points. In particular, $(n,\tau_{(m,s),n}^{\wh \theta \sig,R})$ lies on $\Gamma_{(m,s)}^{\wh \theta \sig,R}$ for $\sigg \in \{-,+\}$, so by definition~\eqref{sigmaR}, for each $n > m$,
\[
\tau_{(m,s),n}^{\wh \theta-,R} \le \sigma_{(m,s),n}^{R} \le \tau_{(m,s),n}^{\wh \theta+,R}.
\]
The result now follows by Theorem~\ref{existence of semi-infinite geodesics intro version}\ref{general limits for semi-infinite geodesics}.
\end{proof}

\begin{proof}[Proof of Theorem~\ref{thm:ci_and_NU}]
Part~\ref{contain}: By Theorem~\ref{thm:NU}\ref{itm:only_endpoint}, if $(m,s) \in \NU_0$, then \\$\tau_{(m,s),r}^{\theta \sig,L} = s$ for some $\theta > 0$, $\sigg \in \{-,+\}$, and $r \ge m$. By definition,
\[
s = \tau_{(m,s),m - 1}^{\theta \sig,L} \le \tau_{(m,s),m}^{\theta \sig,L} \le \cdots,
\]
so $\tau_{(m,s),m}^{\theta \sig,L} = s$, and $(m,s) \in \CI$ by Condition~\ref{itm:Lvert} of Theorem~\ref{thm:ci_equiv}.  

\medskip \noindent \textbf{Part~\ref{itm:cisd}}
The equality
\[
\{(m,s) \in \Z \times \R: \theta_{(m,s)}^R \neq \theta_{(m,s)}^L\} = \{(m,s) \in \Z \times \R: 0 < \theta_{(m,s)}^R < \theta_{(m,s)}^L \} 
\]
follows by~\eqref{eqn:LR_ci_dir} and~\ref{itm:thetaL > 0}$\Rightarrow$\ref{itm:thetaR > 0} of Theorem~\ref{thm:ci_equiv}. Now, let $\omega \in \Omega_4$ and $(m,s) \in \NU_1$. By Theorem~\ref{thm:NU}\ref{itm:union}--\ref{itm:only_endpoint}, there exists some rational $\delta > 0$ such that $s = \tau_{(m,s),m}^{\delta,L} < \tau_{(m,s),m}^{\delta,R}$. By~\eqref{eqn:cidir}, $\theta_{(m,s)}^L \ge \delta$. Since $\delta$ is rational, and we are working on the event $\Omega_4$, there is no $\pm$ distinction for direction $\delta$ (See Item~\ref{itm:Buse_eq} below~\eqref{omega4}). Then, by Theorem~\ref{thm:geod_stronger}\ref{itm:stronger convergence theta}, for all $\ve$ sufficiently small and $\sigg \in\{-,+\}$, $\tau_{(m,s),m}^{(\delta - \ve) \sig,R} = \tau_{(m,s),m}^{\delta,R} > s$. Hence, by~\eqref{eqn:cidir}, $\theta_{(m,s)}^R < \delta \le \theta_{(m,s)}^L$. Next, assume that $(m,s) \notin \NU_1$. By definition~\eqref{NU_0_1}, $\tau_{(m,s),m}^{\theta \sig,L} = \tau_{(m,s),m}^{\theta \sig,R}$ for all $\theta > 0$ and $\sigg \in \{-,+\}$. The definition~\eqref{eqn:cidir} implies that $\theta_{(m,s)}^R = \theta_{(m,s)}^L$.

\medskip \noindent \textbf{Part~\ref{itm:ci_int}:} We assume that $(m,s) \in \NU_1$ so that $0 < \theta_{(m,s)}^R < \theta_{(m,s)}^L$. Otherwise, by Part~\ref{itm:cisd}, the statement is vacuously true. 

By Theorem~\ref{thm:ci_equiv}\ref{itm:Lsplit}, $\tau_{(m,s),m}^{\theta\sig,L} = s$ if and only if either $\theta = \theta_{(m,s)}^L$ and $\sigg = -$ or $\theta < \theta_{(m,s)}^L$. By Theorem~\ref{thm:ci_equiv}\ref{itm:Rsplit}, $\tau_{(m,s),m}^{\theta \sig,R} > s$ if and only if either $\theta = \theta_{(m,s)}^R$ and $\sigg = +$ or $\theta > \theta_{(m,s)}^R$. Therefore, using Theorem~\ref{thm:NU}\ref{itm:only_endpoint},
\begin{align*}
&(m,s) \in \NU_1^{\theta-} \iff \tau_{(m,s),m}^{\theta-,L} = s < \tau_{(m,s),m}^{\theta-,R} \iff \theta \in (\theta_{(m,s)}^R,\theta_{(m,s)}^L], \qquad\text{and} \\
&(m,s) \in \NU_1^{\theta+} \iff \tau_{(m,s),m}^{\theta+,L} = s < \tau_{(m,s),m}^{\theta+,R} \iff \theta \in [\theta_{(m,s)}^R,\theta_{(m,s)}^L).
\end{align*}

\medskip \noindent \textbf{Part~\ref{itm:NU_dense_self}:}
Let $(m,s) \in \NU_1^{\theta \sig}$. By Lemma~\ref{lem:all_theta_max_comp}\ref{itm:approx_left}, there exists $t^\star \in (s - \ve,s)$ such that $\tau_{(m,t^\star),m}^{\theta \sig,R} = t^\star$. Then, by Lemma~\ref{lem:all_theta_max_comp}\ref{itm:approx_right}, there exists $t \in (t^\star,s)$ such that $(m,t) \in \NU_1^{\theta \sig}$.

Next, for $(m,s) \in \NU_1$, $\theta_{(m,s)}^R > 0$ by Part~\ref{itm:cisd}, and by Theorem~\ref{thm:ci_equiv}\ref{itm:Rsplit}, when $\theta < \theta_{(m,s)}^R$ and $\sigg \in \{-,+\}$ (or $\theta = \theta_{(m,s)}^R$ and $\sigg = -$), $\tau_{(m,s),m}^{\theta \sig,R} = s$.  By  Lemma~\ref{lem:all_theta_max_comp}\ref{itm:approx_right}, there exists $t \in (s,s + \ve)$ such that $(m,t) \in \NU_1^{\theta \sig}$.

\medskip \noindent \textbf{Part~\ref{itm:NU_struct}:}
Choose $\theta \ge \theta_{(m,s)}^R$ and $\sigg \in \{-,+\}$ or $\theta = \theta_{(m,s)}^{R}$ and $\sigg = +$. Set $\ve = \tau_{(m,s),m}^{\wh \theta+,R} - s$, where $\wh \theta  = \theta_{(m,s)}^R$. By Theorem~\ref{thm:ci_equiv}\ref{itm:Rsplit}, $\ve > 0$. By the monotonicity of Theorem~\ref{existence of semi-infinite geodesics intro version}\ref{itm:monotonicity of semi-infinite jump times}\ref{itm:monotonicity in theta}, $\tau_{(m,s),m}^{\theta \sig,R} \ge \tau_{(m,s),m}^{\wh \theta+,R} > s$. By Lemma~\ref{lem:all_theta_max_comp}\ref{itm:all_no_3}, for all $t \in (s,s + \ve]$, $\tau_{(m,s),m}^{\theta \sig,R}$ is the unique maximizer of $B_m(u) - h_{m + 1}^{\theta \sig}(u)$ over $u \in [t,\infty)$, and $(m,t) \notin \NU_1^{\theta \sig}$.

\medskip \noindent \textbf{Part~\ref{itm:CI_struct}:}
Let $(m,s) \in \CI$ and set $\wh \theta = \theta_{(m,s)}^R$. By Theorem~\ref{thm:ci_equiv}\ref{itm:Rsplit}, $\tau_{(m,s),m}^{\wh \theta-,R} = s < \tau_{(m,s),m}^{\wh \theta+,R}$. By Lemma~\ref{lem:all_theta_max_comp}\ref{itm:approx_right}, for $\ve > 0$, there exists $\hat t \in (s,(s + \ve)\wedge \tau_{(m,s),m}^{\wh \theta +,R})$ such that $(m,\hat t) \in \NU_1^{\wh \theta-}$. By Lemma~\ref{lem:all_theta_max_comp}\ref{itm:approx_left}, there exists $t \in (s,\hat t)$ such that $\tau_{(m,t),m}^{\wh \theta-,R} = t$. On the other hand, since $s < t < \tau_{(m,s),m}^{\wh \theta+,R}$, and  $\tau_{(m,s),m}^{\wh \theta+,R}$ is the rightmost maximizer of $B_m(u) - h_{m + 1}^{\wh \theta+}(u)$ over $u \in [s,\infty)$, we also have $\tau_{(m,t),m}^{\wh \theta+,R} = \tau_{(m,s),m}^{\wh \theta+,R} > t$. In summary, 
\[
\tau_{(m,t),m}^{\wh \theta-,R} = t < \tau_{(m,t),m}^{\wh \theta+,R},
\]
so by Theorem~\ref{thm:ci_equiv}\ref{itm:Rsplit}, $\theta_{(m,t)}^R = \wh \theta = \theta_{(m,s)}^R$.
\end{proof}

\begin{proof}[Proof of Theorem~\ref{thm:ci_and_Buse_directions}]
\textbf{Part~\ref{itm:cicount}:} We show that 
$
\{\theta_{(m,s)}^R\}_{(m,s) \in \CI} = \Busedc$. 
 The statement for the collection $\{\theta_{(m,s)}^L\}_{(m,s) \in \CI}$ has an analogous proof. First, let $(m,s) \in \CI$ and set $\wh \theta = \theta_{(m,s)}^{R}$. Then, by Parts~\ref{itm:thetaR > 0} and~\ref{itm:Rsplit} of Theorem~\ref{thm:ci_equiv}, $\wh \theta > 0$ and $\Gamma_{(m,s)}^{\wh \theta-,R} \ne \Gamma_{(m,s)}^{\wh \theta+,R}$. By~\ref{notTheta}$\Leftrightarrow$\ref{LRsame} of Theorem~\ref{thm:geod_description}, $\wh \theta \in \Busedc$.

Now, assume that $\theta \in \Busedc$. By \ref{Ronegeod}$\Leftrightarrow$\ref{notTheta} of Theorem~\ref{thm:geod_description}, for all $\mbf x \in \Z \times \R$, $\Gamma_{\mbf x}^{\theta+,R} \neq \Gamma_{\mbf x}^{\theta-,R}$. Since both $\Gamma_{\mbf x}^{\theta+,R}$ and $\Gamma_{\mbf x}^{\theta-,R}$ are rightmost geodesics between any two of their points (Theorem~\ref{existence of semi-infinite geodesics intro version}\ref{Leftandrightmost}), the two geodesics must split at some point $\mbf v \ge \mbf x$ and never come back together. See Figure~\ref{fig:splitting}. Then, $\Gamma_{\mbf v}^{\theta -,R} \cap \Gamma_{\mbf v}^{\theta +,R} = \{\mbf v\}$, and by Theorem~\ref{thm:ci_equiv}\ref{itm:Rsplit}, $\theta = \theta_{\mbf v}^R$. 

\begin{figure}[t]
            \begin{tikzpicture}
            \draw[red,ultra thick,->] plot coordinates {(1,0)(4,0)(4,1)(6,1)(6,2)(8,2)(8,3)(9.5,3)};
            \draw[blue,thick,->] plot coordinates {(1,0)(4,0)(4,1)(5,1)(5,2)(5.5,2)(5.5,3)(6,3)(6,4)};
            \filldraw[black] (1,0) circle (2pt) node[anchor = north] {\small $\mbf x$};
            \filldraw[black] (5,1) circle (2pt) node[anchor = north] {\small $\mbf v$};
            \end{tikzpicture}
            \caption{\small The red/thick path is the $\theta +$ geodesic and the blue/thin path is the $\theta-$ geodesic}
            \label{fig:splitting}
        \end{figure}
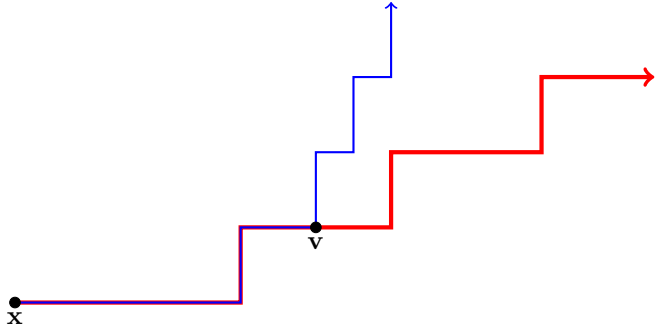

\medskip \noindent \textbf{Part~\ref{itm:ci_dir_supp}:}
First, we show that for $(m,s) \in \Z \times \R$, and $t > s$, 
\begin{multline} \label{eqn:square_dir}
\{\theta > 0: v_{m + 1}^{\theta -}(s) < v_{m + 1}^{\theta +}(s)\} \subseteq [\theta_{(m,s)}^L,\infty),
\quad\text{and}\\ \bigcap_{t: t > s}\{\theta > 0: h_m^{\theta +}(s,t) < h_m^{\theta -}(s,t)\} \subseteq [0,\theta_{(m,s)}^R].
\end{multline}
If $\theta_{(m,s)}^L = 0$, there is nothing to show for the first inclusion, so we assume that $\theta_{(m,s)}^L > 0$. If $\theta < \theta_{(m,s)}^L$, then by~\eqref{eqn:cidir}, $\tau_{(m,s),m}^{\theta \sig,L} = s$ for $\sigg \in \{-,+\}$. Thus, by~\eqref{eqn:v = 0 equiv},   $v_{m + 1}^{\theta \sig}(s)=0$ for  $\theta\in[0,\theta_{(m,s),m}^{\theta,L})$, and therefore, $\{\theta > 0: v_{m + 1}^{\theta -}(s) < v_{m + 1}^{\theta +}(s)\} \subseteq [\theta_{(m,s)}^L,\infty)$.

Now, for the second statement, we note by the last statement of Lemma~\ref{lemma:equality of busemann to weights of BLPP} that $B_m(s,t) =h_m^{\theta \sig}(s,t)$ if and only if $\tau_{(m,s),m}^{\theta \sig,R} \ge t$. Hence, $\theta \mapsto h_m^{\theta \sig}(s,t)$ is constant in the interval $(\theta_{(m,s),t}^R ,\infty)$, where 
\[
\theta_{(m,s),t}^R = \inf\{\theta > 0: \tau_{(m,s),m}^{\theta \sig,R}  \ge t\}.
\]
Therefore, for any $t>s$,  \[
\bigcap_{u: u > s}\{\theta > 0: h_m^{\theta +}(s,u) < h_m^{\theta -}(s,u)\} \subseteq [0,\theta_{(m,s),t}].
\]
The proof of the claim is complete once we show that
\be \label{convthetastar} 
\theta_{(m,s),t}^R \searrow \theta_{(m,s)}^R \text{ as } t \searrow s.
\ee
First, by Definition~\eqref{eqn:cidir} and monotonicity (Theorem~~\ref{existence of semi-infinite geodesics intro version}\ref{itm:monotonicity of semi-infinite jump times}\ref{itm:monotonicity in theta}), 
$\theta_{(m,s)}^R$ is equivalently defined as 
\[
\inf \{\theta > 0: \tau_{(m,s),m}^{\theta \sig,R} > s\},
\]
so $\theta_{(m,s)}^R \le \theta_{(m,s),t}^R$ for all $s < t$. However, setting $\wh \theta = \theta_{(m,s)}^R$,  monotonicity and the Definition~\eqref{eqn:cidir} imply that $s < \tau_{(m,s),m}^{\wh \theta +,R} \le \tau_{(m,s),m}^{\theta \sig,R} $ for $\theta > \wh \theta$. Hence, $\theta_{(m,s),t}^{R} = \theta_{(m,s)}^R$ for all $t \in (s,\tau_{(m,s),m}^{\theta \sig,R}]$. Specifically,~\eqref{convthetastar} holds, as desired. 

Since $\theta_{(m,s)}^R \le \theta_{(m,s)}^L$, the inclusions of~\eqref{eqn:square_dir} guarantee that
\be \label{compsetinclusion}
\Compset_{(m,s)} \subseteq \{\theta_{(m,s)}^R\} \cap \{\theta_{(m,s)}^L\}.
\ee
Hence, if $\theta_{(m,s)}^R < \theta_{(m,s)}^L$, $\Compset_{(m,s)} = \varnothing$. In the case that $\theta_{(m,s)}^R = \theta_{(m,s)}^L$, we show that $\theta_{(m,s)}^R \in \Compset_{(m,s)}$. We break this into two cases, one where $\theta_{(m,s)}^R = 0$ and the other where $\theta_{(m,s)}^R > 0$.

\medskip \noindent \textbf{Case 1: $\theta_{(m,s)}^R = 0$} : By~\eqref{compsetinclusion}, we just need to show that $0 \in \Compset_{(m,s)}$, which by definition holds if  and only if $v_{m + 1}^{\theta \sig}(s) > 0$ for all $\theta > 0$ and $\sigg \in \{-,+\}$. \ref{itm:v = 0}$\Leftrightarrow$\ref{itm:thetaR > 0} of Theorem~\ref{thm:ci_equiv} completes the proof of this case. 

\medskip \noindent \textbf{Case 2: $\theta_{(m,s)}^R > 0$}: In this proof, refer to Figure~\ref{fig:LRsplitting}. Setting $\wh \theta = \theta_{(m,s)}^R = \theta_{(m,s)}^L$, \ref{itm:thetaR > 0}$\Rightarrow$\ref{itm:Rsplit} and \ref{itm:thetaL > 0}$\Rightarrow$\ref{itm:Lsplit} of Theorem~\ref{thm:ci_equiv} imply that $\Gamma_{(m,s)}^{\wh \theta-,R}$ and $\Gamma_{(m,s)}^{\wh \theta-,L}$ make immediate vertical steps, while $\Gamma_{(m,s)}^{\wh \theta+,R}$ and $\Gamma_{(m,s)}^{\wh \theta+,L}$ make immediate horizontal steps (See also Remark~\ref{rmk:split}). Hence, both $\Gamma_{(m,s)}^{\wh \theta+,R}$ and $\Gamma_{(m,s)}^{\wh \theta+,L}$ pass through some $(m,u) \in \Z \times \Q_{>s}$. By Theorem~\ref{thm:NU}\ref{itm:NUp0}, on $\Omega_4$, $(m,u) \notin \NU_0$, and there is a single $\theta \sig$ Busemann geodesic starting from $(m,u)$. Since $\Gamma_{(m,s)}^{\wh \theta+,L}$  and $\Gamma_{(m,s)}^{\wh \theta+,R}$ both pass through $(m,u)$, it follows that $\Gamma_{(m,s)}^{\wh \theta+,L} = \Gamma_{(m,s)}^{\wh \theta+,R}$. The implication \ref{itm:thetaR > 0}$\Rightarrow$\ref{itm:Rsplit} of Theorem~\ref{thm:ci_equiv} also implies that the only common point between $\Gamma_{(m,s)}^{\wh \theta+,L} = \Gamma_{(m,s)}^{\wh \theta+,R}$ and $\Gamma_{(m,s)}^{\wh \theta-,R}$ is $(m,s)$.  Therefore, $\Gamma_{(m,s)}^{\wh \theta +,L} \cap \Gamma_{(m + 1,s)}^{\wh \theta-,R} = \varnothing$, and for all $t > s$, $\Gamma_{(m,t)}^{\wh \theta+,L} \cap \Gamma_{(m,s)}^{\wh \theta-,R} = \varnothing$. By~\ref{itm:pmBuse}$\Leftrightarrow$\ref{itm:disjointgeod} of Theorem~\ref{thm:thetanotsupp}, $\wh \theta \in \Compset_{(m,s)}$, as desired.
\end{proof}

        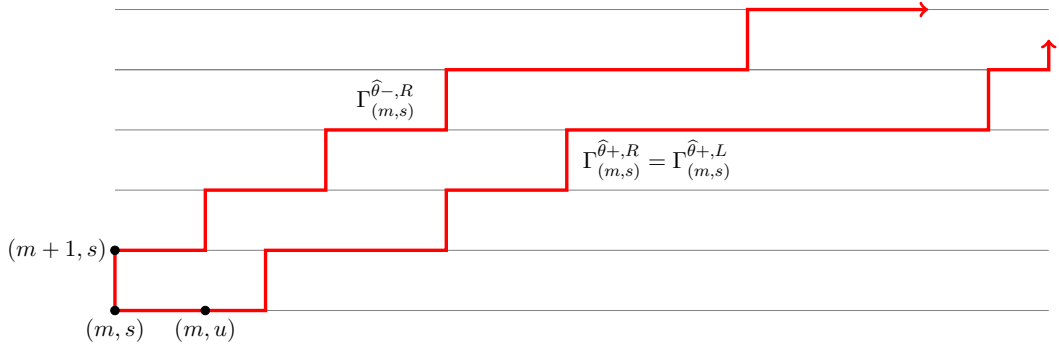
\begin{figure}[t]
\begin{adjustbox}{max totalsize={5.5in}{5in},center}
\begin{tikzpicture}
\draw[gray,thin] (-0.5,0)--(15,0);
\draw[gray,thin] (-0.5,1)--(15,1);
\draw[gray,thin] (-0.5,2)--(15,2);
\draw[gray,thin] (-0.5,3)--(15,3);
\draw[gray,thin] (-0.5,4)--(15,4);
\draw[gray,thin] (-0.5,4)--(15,4);
\draw[gray,thin] (-0.5,5)--(15,5);
\draw[red,ultra thick,->] (-0.5,0)--(2,0)--(2,1)--(5,1)--(5,2)--(7,2)--(7,3)--(14,3)--(14,4)--(15,4)--(15,4.5);
\draw[red,ultra thick,->] (-0.5,0)--(-0.5,1)--(1,1)--(1,2)--(3,2)--(3,3)--(5,3)--(5,4)--(10,4)--(10,5)--(13,5);
\node at (8.5,2.5) {$\Gamma_{(m,s)}^{\wh \theta+,R} = \Gamma_{(m,s)}^{\wh \theta+,L} $};
\node at (4,3.5) {$\Gamma_{(m,s)}^{\wh \theta-,R}$};
\filldraw[black] (-0.5,0) circle (2pt) node[anchor = north] {$(m,s)$};
\filldraw[black] (-0.5,1) circle (2pt) node[anchor = east] {$(m + 1,s)$};
\filldraw[black] (1,0) circle (2pt) node[anchor = north] {$(m,u)$};
\end{tikzpicture}
\end{adjustbox}
\caption{\small $\Gamma_{(m,s)}^{\wh \theta-,R}$ (upper red/thick path) travels to $(m + 1,s)$, while $\Gamma_{(m,s)}^{\wh \theta +,R}$ and $\Gamma_{(m,s)}^{\wh \theta +,R}$ (lower red/thick path) both pass through $(m,u)$ for some rational $u > s$.}
\label{fig:LRsplitting}
\end{figure}

\subsection{Proofs of Theorems~\ref{thm:Busedc_class} and~\ref{thm:Haus_comp_interface}} \label{section:main_proofs}

\begin{proof}[Proof of Theorem~\ref{thm:Busedc_class}]
Part~\ref{itm:BLPP_good_dir_coal} follows from the equivalences \ref{notTheta}$\Leftrightarrow$\ref{coal}$\Leftrightarrow$\ref{countonegeod} of Theorem~\ref{thm:geod_description}.
Part~\ref{itm:BLPP_bad_dir_split} follows from Remark~\ref{rmk:splitting_coalesceing_general}. 
\end{proof}

\begin{proof}[Proof of Theorem~\ref{thm:Haus_comp_interface}]
Recall that the last-passage time between $(m,s)$ and \\$(m + 1,t)$ is
\be \label{eqn:1levlpp}
\sup_{s \le u \le t}\{B_m(s,u) + B_{m + 1}(u,t)\} = B_{m +1}(t) - B_m(s) + \sup_{s \le u \le t}\{B_m(u) - B_{m + 1}(u)\},
\ee
and the maximizer $u$ gives the location of the jump from level $m$ to level $m + 1$.  For $r \in \Z$, define the random sets $\LM_r$ as
\begin{equation} \label{eqn:LM}
\Big\{s \in \R: B_{r -1 }(s) - B_{r}(s) = \sup_{s \le u \le t}\{B_{r - 1}(u) - B_{r}(u)\} \;\text{ for some } t> s\Big\}.
\end{equation}

As in the proof of \ref{itm:Ltriv}$\Rightarrow$\ref{itm:Rtriv} of Theorem~\ref{thm:ci_equiv}, if $\sigma_{(m,s),n}^L > s$ for some $n > m$, then letting $n$ be the smallest such integer, there exists $t > s$ such that the leftmost geodesic between $(m,s)$ and $(n,t)$ passes through $(n,s)$.  Again, refer to Figure~\ref{fig:nontrivial competition interface vertical then horizontal}. Therefore, 
\begin{align} \label{eqn:ntlm}
    \LM_{m + 1} \subseteq \{s \in \R: \sigma_{(m,s),n}^L > s \text{ for some }n > m \} \subseteq 
   \bigcup_{n > m }\LM_{n}.
\end{align}
On the event $\Omega_4 \subseteq A_r$~\eqref{omega4}, $\LM_r$ has Hausdorff dimension $\f{1}{2}$ for each $r \in \Z$. Using Condition~\ref{itm:Ltriv} of Theorem~\ref{thm:ci_equiv}, the set
\[
\{s \in \R: \sigma_{(m,s),n}^L > s \text{ for some }n > m \} = \{s \in \R: (m,s) \in \CI\}
\]
also has Hausdorff dimension $\f{1}{2}$.
By~\eqref{eqn:ntlm} and Corollary~\ref{cor:hausdorff dimension for left maxes standard two-sided BM}\ref{itm:left_max}, each point $s \in \R$ lies in the set $\CI_m$ with probability $0$.

Next, we show the density of the sets $\CI_m$.  On the event $\Omega_4$ (See Item~\ref{itm:omega4_unique_geod} below~\eqref{omega4}), for each $m \in \Z$ and rational $q_1 < q_2$,  there is a unique geodesic between $(m,q_1)$ and $(m +1,q_2)$ that jumps to level $m + 1$ at a time $u \in (q_1,q_2)$. Then, $(m,u) \in \CI$ by Condition~\ref{itm:some_vert_geod} of Theorem~\ref{thm:ci_equiv}.

Lastly, we show that $\CI$ is exactly the set of points $\mbf x \in \Z \times \R$ such that there exist two semi-infinite geodesics in the same direction, whose only common point is the initial point.  If such a pair of geodesics exists, then one of the geodesics must travel vertically, so $\mbf x \in \CI$ by definition~\eqref{CI_intro}. The reverse implication follows from \ref{itm:some_vert_geod}$\Rightarrow$\ref{itm:Lsplit} of Theorem~\ref{thm:ci_equiv}. 
\end{proof}

 \section{Proofs of the results in Section~\ref{sec:CGM}} \label{sec:CGM_proofs}
\subsection{Discrete queues and the proof of Theorem~\ref{thm:convCGM}}
We discuss the queuing setting from~\cite{Fan-Seppalainen-20} that will allow us to prove Theorem~\ref{thm:convCGM}. We note that in~\cite{Fan-Seppalainen-20}, the Busemann functions were constructed from limits of geodesics traveling to the southwest, so the notation is changed to reflect northeast geodesics. 

Let $I = (I_k)_{k \in \Z}$ and $\omega = (\omega_k)_{k \in \Z}$ be sequences that satisfy
\[
\lim_{m \rightarrow \infty} \sum_{i = 0}^m (\omega_{i + 1} - I_i) = -\infty.
\]
The sequence $I$ gives the inter-arrival times between customers in the queue, and $\omega$ gives the service times.  
Let $H_k$ be the sequence satisfying $H_0 = 0$ and $H_{k + 1} - H_{k} = I_k$, and let $S_k$ be the sequence satisfying $S_0 = 0$ and $S_k - S_{k - 1} = \omega_k$. We define the sequence $\wt I = (\wt I_k)_{k \in \Z}$ by 
\be \label{eqn:def of wt I k}
\wt I_k = \omega_k + \sup_{m: m \ge k}[S_m -H_m] - \sup_{m: m \ge k + 1}[S_m - H_m].
\ee
 In queuing terms, $\wt I$ is the process of departures from the queue.  We encode the mapping $(I,\omega) \rightarrow \wt I$ as the function $\wt I = \Dq(I,\omega)$, with subscript $d$ for discrete.  Similar to the spaces $\X_n,\Y_n$ defined in~\eqref{Yndef} and~\eqref{Xndef}, the following sequence spaces are defined in~\cite{Fan-Seppalainen-20} for the Busemann functions. Fix $n$, and define
\begin{align*}
    \Yd_n&=\Big\{I = (I^1,\ldots,I^n) \in (\R_{\ge 0}^{\Z})^n : \\ 
    &\qquad\qquad\qquad\text { for }2 \le i \le n, \lim_{m \rightarrow \infty} \f{1}{m} \sum_{k = 1}^m I_k^i > \lim_{m \rightarrow \infty} \f{1}{m}\sum_{k = 1}^m I_k^{i - 1} > 0 \Big\}, \\
    \Xd_n &=\Big\{\eta = (\eta^1,\ldots,\eta^n) \in (\R_{\ge 0}^{\Z})^n : \\ 
    &\qquad\qquad\qquad \text { for }2 \le i \le n,\;\; \eta^i \ge \eta^{i - 1}, \text{ and } \liminf_{m \rightarrow \infty} \f{1}{m} \sum_{k = 1}^m \eta_k^1 > 0 \Big\}.
\end{align*}

\noindent Above, each  component $I^i$ and $\eta^i$ is a nonnegative  sequence indexed by $\Z$, and $\eta^i \ge \eta^{i - 1}$ means coordinatewise ordering:   $\eta^i_k \ge \eta^{i - 1}_{k}$ for all $k \in \Z$. Similarly as in Section~\ref{section:fdd}, we iterate this map as follows:
\begin{align*}
    \Dq^{(1)}(I) =&D(I,0) =  I \\
    \Dq^{(n)}(I^n,I^{n - 1},\ldots,I^1) = &\Dq(\Dq^{(n - 1)}(I^n,\ldots,I^{2}),I^1) \text{ for } n \geq 2.
\end{align*}
We now define the map $\DqD^{(n)}:\Yd_n \rightarrow \Xd_n$ as follows: for $I = (I^1,\ldots,I^n) \in \Yd_n$, the image $\eta = (\eta^1,\ldots,\eta^n) = \DqD^{(n)}(I)$ is defined by 
\[
\eta^i = \Dq^{(i)}(I^i,I^{i -1 },\ldots,I^1)\qquad \text{for}\qquad i = 1,\ldots,n.
\]

Let $\alpha = (\alpha_1,\ldots,\alpha_n)$ be such that $\alpha_1> \cdots > \alpha_n > 0$. On the space $\Yd_n$, we define the measure $\nud^\alpha$ such that $(I^1,\ldots,I^n) \sim \nud^{\alpha}$ if all coordinates $I_k^i$  are independent and $I_k^i \sim \Exp(\alpha_i)$ for $k \in \Z$ and $1 \le i \le n$. On the space $\Xd_n$, we define the measure $\mud^\alpha = \nud^\alpha \circ (\DqD^{(n)})^{-1}$.  

For each level $m \in \Z$, define the level-$m$ sequence of weights $\overline Y_m = (Y_{(k,m)})_{k \in \Z}$, and for given $\alpha > 0$,  we denote the Busemann functions at level $m$ by $\overline U_m^{\alpha,\mbf e_1} = (U^{\alpha}((k - 1,m),(k,m))_{k \in \Z}$.
The joint distribution of Busemann functions along a horizontal edge is then described as follows. 
\begin{theorem}[\cite{Fan-Seppalainen-20}, Theorem 3.2] \label{CGM_dist}
  Let $1 > \alpha_1 >\alpha_2 > \cdots > \alpha_n > 0$. For each level $m \in \Z$, the $(n  + 1)-$tuple of sequences $(\overline Y_m,\overline U_m^{\alpha_1,\mbf e_1},\ldots,\overline U_m^{\alpha_n,\mbf e_1})$ has distribution $\mu^{(1,\alpha_1,\ldots,\alpha_n)}$. 
\end{theorem}

\noindent  The input sequences $I$ and $\omega$ can be encoded by the sequences $H$ and $S$ used in the definition of $\Dq$~\eqref{eqn:Dq}, where $H_{k + 1} - H_k = I_k$ and $S_k - S_{k - 1} = \omega_k$. For these sequences, we can define continuous functions $H,S \in \CRpin$ to be the piecewise linear interpolations such that $S(k) = S_k$ and $H(k) = H_k$ for $k \in \Z$. In this setting, $\Yd_n$ and $\Xd_n$ can be viewed as subspaces of $\Y_n$ and $\X_n$ as defined in Equations (4.1) and (4.2). Furthermore, the operators $\Dq^{(i)}$ and $\DqD^{(n)}$ can likewise be viewed as operators on spaces of functions. We define an output sequence $\wt H = (\wt H_k)_{k \in \Z}$ by $\wt H_0 = 0$ and $\wt H_{k + 1} - \wt H_k = \wt I_k$.  From this point of view,~\eqref{eqn:def of wt I k} implies that for $t \in \Z$,
\be \label{eqn:Dq}
\wt H_t = \Dq(S,H)_t = S(-1,t - 1) + \sup_{0 \le u <\infty}\{S(u) - H(u)\} - \sup_{t \le u < \infty}\{S(u) - H(u)\}.
\ee
The supremum over the integers can be replaced with the supremum over the reals because $S$ and $H$ are linear interpolations. Then, extend $\wt H(t)$ to all $t \in \R$ by~\eqref{eqn:Dq}, which is a continuous, piecewise linear interpolation of the sequence $\wt H$. Note that~\eqref{eqn:Dq} nearly matches the operator $D$~\eqref{definition of D}:
\be \label{eqn:D}
D(Z,B)(t) = B(t) + \sup_{0 \le u <\infty}\{B(u) - Z(u)\} - \sup_{t \le u < \infty}\{B(u) - Z(u)\}.
\ee
The only difference is that the term $B(t)$ has been replaced with $S(- 1,t - 1)$. However, when these random walks are appropriately scaled, this discrepancy is eliminated. This is made precise in the following proof.

\begin{proof}[Proof of Theorem~\ref{thm:convCGM}]
This follows by Donsker's Theorem, with some extra care. We show that distributional convergence is preserved under the queuing map. A similar proof is given for the Brownian queue in~\cite{harrison1990}, although the queues in that setting are only infinite in the positive direction. We appeal to tightness to prove the result in the bi-infinite setting. We start by showing that for fixed $T \in \Z$, 
\begin{multline} \label{eqn:truncated_conv}
\f{1}{\sqrt k}\Big(U^{\f{\sqrt k}{\sqrt k + \lambda_1}}((0,0),(tk,0)) -  t k, U^{\f{\sqrt k}{\sqrt k + \lambda_2}}((0,0),(tk,0)) - t k\Big)_{t \in (-\infty,T]}\\ \overset{k \rightarrow \infty}{\Longrightarrow} (h_0^{1/\lambda_1^2}(t), h_0^{1/\lambda_2^2}(t))_{t \in (-\infty,T]},
\end{multline}
and then the extension of the convergence to all $t \in \R$ follows by tightness, as discussed below. We may take $T > 0$ since the statement holds for any $T\le 0$ as long as it holds for some $T > 0$.
For $j \in \Z$, and $i = 1,2$, let $X_{j}^i \sim \Exp(1)$ be mutually independent random variables. For $k \in \Z_{>0}$, set $Z_k^1(0) = Z_k^2(0) = 0$, and for $tk \in \Z_{>0}$, set
\begin{align} \label{eqn:rw_rep_Z1}
Z_k^1(tk) &= \f{1}{\sqrt k}\sum_{j = 1}^{tk} \Big(\f{\sqrt k + \lambda_i}{\sqrt k}X_{j}^i - 1\Big),\quad \text{and}\\
Z_k^1(-tk) &= -\f{1}{\sqrt k}\sum_{j = -tk + 1}^{0} \Big(\f{\sqrt k + \lambda_i}{\sqrt k}X_{j}^i - 1\Big). \nonumber 
\end{align}
For general $t \in \R$, let $Z_k^i(tk)$ be the linear interpolation of the above. Equivalently, for $tk \in \Z_{>0}$,
\begin{align} \label{eqn:rw_rep2}
Z_k^i(tk) &= \Big(\f{\sqrt k + \lambda_i}{\sqrt k}\Big)\f{1}{\sqrt k} \sum_{j = 1}^{tk} \Big(X_{j}^i - 1\Big) + \lambda_i t, \quad \text{and}\\
Z_k^i(-tk) &= -\Bigg(\Big(\f{\sqrt k + \lambda_i}{\sqrt k}\Big)\f{1}{\sqrt k} \sum_{j = -tk + 1}^{0} \Big(X_{j}^i - 1\Big) + \lambda_i t\Bigg) \nonumber
\end{align}
Since $X_{j}^i$ has mean and variance $1$, this representation allows us to apply 
 Donsker's Theorem and conclude that, in the sense of uniform convergence on compact sets,
\be \label{eqn:Donsker}
(Z_k^1(tk),Z_k^2(tk))_{t \in \R} \overset{k\rightarrow \infty}{\Longrightarrow} (Z^1(t),Z^2(t)), 
\ee
where $Z^1$ and $Z^2$ are independent two-sided Brownian motions with drift $\lambda_1$ and $\lambda_2$, respectively. 
By~\eqref{eqn:Dq}, for $t \le T$ with $tk \in \Z$,
\begin{align*}
&\quad\Dq(Z_k^2,Z_k^1)(tk) \\
&= Z_k^1(-1,tk - 1) + \sup_{0 \le u <\infty}\{Z_k^1(u) - Z_k^2(u)\} - \sup_{tk \le u < \infty}\{Z_k^1(u) - Z_k^2(u)\} \\ 
&= Z_k^1(-1,tk - 1)  + \sup_{0 \le u \le Tk}\{Z_k^1(u) - Z_k^2(u)\}
\vee \sup_{Tk \le u < \infty}\{Z_k^1(u) - Z_k^2(u)\}  \\
&\qquad- \sup_{tk \le u \le Tk}\{Z_k^1(u) - Z_k^2(u)\} \vee \sup_{Tk \le u < \infty}\{Z_k^1(u) - Z_k^2(u)\} \\
&= Z_k^1(-1,tk - 1) \\
&\qquad+ \sup_{0 \le u \le Tk}\{Z_k^1(Tk,u) - Z_k^2(Tk,u)\}
\vee \sup_{Tk \le u < \infty}\{Z_k^1(Tk,u) - Z_k^2(Tk,u)\}  \\
&\qquad\qquad - \sup_{t \le u \le Tk}\{Z_k^1(Tk,u) - Z_k^2(Tk,u)\} \vee \sup_{Tk \le u < \infty}\{Z_k^1(Tk,u) - Z_k^2(Tk,u)\} \\
&= Z_k^1(-1,tk - 1) \\
&\qquad + \sup_{0 \le u \le T}\{Z_k^1(Tk,uk) - Z_k^2(Tk,uk)\}
\vee \sup_{T \le u < \infty}\{Z_k^1(Tk,uk) - Z_k^2(Tk,uk)\} \\
&\qquad \qquad - \sup_{t \le u \le T}\{Z_k^1(Tk,uk) - Z_k^2(Tk,uk)\} \vee \sup_{T \le u < \infty}\{Z_k^1(Tk,uk) - Z_k^2(Tk,uk)\},
\end{align*}
while by a similar computation, for $t \le T$, 
\begin{align*}
D(Z^2,Z^1)(t) = Z^1(t) &+ \sup_{0 \le u \le T}\{Z^1(T,u) - Z^2(T,u)\}
\vee \sup_{T \le u < \infty}\{Z^1(T,u) - Z^2(T,u)\}  \\
&- \sup_{t \le u \le T}\{Z^1(T,u) - Z^2(T,u)\} \vee \sup_{T \le u < \infty}\{Z^1(T,u) - Z^2(T,u)\}  .
\end{align*}
With $\Dq$ and $D$ represented this way, to prove~\eqref{eqn:truncated_conv}, it is sufficient to show the following. 
\begin{enumerate} [label=\rm(\roman{*}), ref=\rm(\roman{*})]  \itemsep=3pt 
    \item \label{itm:queue_rep} The following distributional equality holds.
    \begin{multline*}
    \f{1}{\sqrt k}\Big(U^{\f{\sqrt k}{\sqrt k + \lambda_1}}((0,0),(tk,0)) - tk, U^{\f{\sqrt k}{\sqrt k + \lambda_2}}((0,0),(tk,0)) - tk\Big)_{t \in \R}\\ \deq (Z_k^1(tk),\Dq(Z_k^2,Z_k^1)(tk))_{t \in \R}.
    \end{multline*}
    \item \label{itm:twopt_conv}The following distributional convergence holds in the topology of uniform convergence on compact sets. 
    \begin{multline*}
    \Big(Z_k^1(-1,tk - 1),\sup_{t \le u \le T}\{Z_k^1(Tk,uk) - Z_k^2(Tk,uk)\}\Big)_{t \le T} \\ \overset{k \rightarrow \infty}{\Longrightarrow} \Big(Z_k^1(t), \sup_{t \le u \le T}\{Z_k^1(T,u) - Z_k^2(T,u)\}\Big)_{t \le T}.
    \end{multline*}
    \item \label{itm:tail_conv} The following distributional convergence holds.
    \[
    \sup_{T \le u < \infty}\{Z_k^1(Tk,uk) - Z_k^2(Tk,uk)\} \overset{k \rightarrow \infty}{\Longrightarrow} \sup_{T \le u < \infty}\{Z^1(T,u) - Z^2(T,u)\}.
    \]
    \item \label{itm:indep} The process
    \[
    \Big(Z_k^1(-1,tk - 1),\sup_{t \le u \le T}\{Z_k^1(Tk,uk) - Z_k^2(Tk,uk)\}\Big)_{t \le T}
    \] is independent of $\sup_{T \le u < \infty}\{Z_k^1(Tk,uk) - Z_k^2(Tk,uk)\}$, while the process 
    \[\Big(Z_k^1(t), \sup_{t \le u \le T}\{Z_k^1(T,u) - Z_k^2(T,u)\}\Big)_{t \le T}
    \]
    is independent of $\sup_{T \le u < \infty}\{Z^1(T,u) - Z^2(T,u)\}$. 
\end{enumerate}
Then,~\eqref{eqn:CGM_conv} follows from~\eqref{eqn:truncated_conv} by showing that
\begin{enumerate} [resume, label=\rm(\roman{*}), ref=\rm(\roman{*})]  \itemsep=3pt
    \item  \label{itm:tight} The sequence 
    \[\Big\{\f{1}{\sqrt k}\Big(U^{\f{\sqrt k}{\sqrt k + \lambda_1}}((0,0),(tk,0)) -  t k, U^{\f{\sqrt k}{\sqrt k + \lambda_2}}((0,0),(tk,0)) - t k\Big)_{t \in \R}\Big\}_{k \ge 1}
    \]
    is tight in $C(\R \rightarrow \R^2)$. 
\end{enumerate}

\medskip \noindent \textbf{Item~\ref{itm:queue_rep}:} This follows by Theorem~\ref{CGM_dist} and the representation~\eqref{eqn:rw_rep_Z1} because $\f{\sqrt k + \lambda_i}{\sqrt k}$ times an $\Exp(1)$ random variable has distribution $\Exp\Big(\f{\sqrt k}{\sqrt k + \lambda_i}\Big)$.

\medskip \noindent \textbf{Item~\ref{itm:twopt_conv}:} Note that $Z_k^1(-1) \rightarrow 0$ since it is $O(k^{-1/2})$. The rest follows by the uniform convergence on compact sets of $(Z_k^1(tk),Z_k^2(tk))_{t \in \R}$ to $(Z^1(t),Z^2(t))_{t \in \R}$.

\medskip \noindent \textbf{Item~\ref{itm:tail_conv}:} By the Markov property, $\sup_{T \le u < \infty}\{Z_k^1(Tk,uk) - Z_k^2(Tk,uk)\}$ has the same distribution as $\sup_{0 \le u < \infty}\{Z_k^1(u) - Z_k^2(u)\}$, and $\sup_{T \le u < \infty}\{Z^1(T,u) - Z^2(T,u)\}$ has the same distribution as $\sup_{0 \le u <\infty}\{Z^1(u) - Z^2(u)\}$. Since $Z^1$ and $Z^2$ are independent Brownian motions with drift $\lambda_1 < \lambda_2$, $Z^1 - Z^2$ is a variance $2$ Brownian motion with negative drift. Then, the weak convergence
\[
\sup_{0 \le u < \infty}\{Z_k^1(u) - Z_k^2(u)\} \Longrightarrow \sup_{0 \le u <\infty}\{Z^1(u) - Z^2(u)\}
\]
follows by Proposition 6.9.4 in~\cite{resnick} (See also Chapter VIII, Section 6 in~\cite{Asmussen-1987}).

\medskip \noindent \textbf{Item~\ref{itm:indep}:} This follows from independence of increments of random walks and Brownian motion. 

\medskip \noindent \textbf{Item~\ref{itm:tight}:} It is sufficient to show that each of the components is tight in $C(\R \rightarrow \R)$. Each component is a scaled random walk converging to Brownian motion by Donsker's theorem, so tightness follows. 
\end{proof}

\subsection{The stationary horizon and proof of Theorems~\ref{thm:dist_of_SH} and~\ref{thm:conv_to_SH}} \label{sec:SH_proofs}
\noindent The following lemmas relate the mappings $\Phi$ and $\Phi^k$ to the mappings $D$, $D^{(k)}$, and $\D^{(k)}$.
\begin{lemma}
If $f(0) = g(0) = 0$, then for all $t \in \R$,
\[
\Phi(f,g)(t) = f(t) + \sup_{-\infty < s \le t}[g(s)- f(s)] - \sup_{-\infty < s \le 0}[g(s) - f(s)].
\]
\end{lemma}

\begin{proof}
We prove the statement for $t \ge 0$, and the statement for $t < 0$ follows similarly. We have
\begin{align*}
    \Phi(f,g)(t) &= f(t) + \Big[\sup_{s \le 0}[f(0) - g(0) - f(s) + g(s)] + \inf_{0 \le s \le t}[f(s) - g(s)]\Big]^{-} \\
    &= f(t) + \Big[\sup_{s \le 0}[g(s) - f(s)] - \sup_{0 \le s \le t}[g(s) - f(s)]\Big]^- \\
    &= f(t) + \sup_{-\infty < s \le t}[g(s)- f(s)] - \sup_{-\infty < s \le 0}[g(s) - f(s)]. \qquad\qquad\qquad\qquad \mbox{\qedhere}
\end{align*}
\end{proof}
\begin{lemma} \label{lem:iterated_map_Phi}
Define the mappings $\Psi^k: C(\R)^k \to C(\R)$ as follows:
\[
\Psi^1(f_1) = f_1,\qquad\text{and}\qquad \Psi^{k}(f_1,\ldots,f_k) = \Phi(f_1,\Psi^{k - 1}(f_2,\ldots,f_k)).
\]
Let $f_1,f_2,\ldots$ be an infinite sequence of continuous functions such that each of the operations below is well-defined. Let $(g_1,\ldots,g_k) = \Phi^k(f_1,\ldots,f_k)$. Then, for $1 \le i \le k$,
\[
g_i = \Psi^i(f_1,\ldots,f_i).
\]
\end{lemma}
\begin{proof}
The statements for $k = 1,2$ follow immediately from the definition. Assume the statement is true for some $k - 1$. For $i = 1$, $g_1 = f_1 = \Psi^1(f_1)$. For $2 \le i \le k$,
\[
g_i = \Phi(f_1,[\Phi^{k - 1}(f_2,\ldots,f_k)]_{i - 1}) = \Phi(f_1,\Psi^{i - 1}(f_{2},\ldots,f_i)) = \Psi^i(f_1,\ldots,f_i). \qedhere
\]
\end{proof}
\noindent Using these representations of the map $\Phi$ and $\Phi^k$, we have the following lemma from~\cite{Seppalainen-Sorensen-21a}. 
\begin{lemma}[\cite{Seppalainen-Sorensen-21a}, Lemma D.2] \label{lem:modified_time_reversal}
Let $Z,B:\R\rightarrow \R$ be continuous functions satisfying $Z(0) = B(0) = 0$ and
\[
\lim_{t \rightarrow \pm \infty} (B(t) - Z(t)) = \mp \infty.
\]
Then, for all $t \in \R$
\[
-D(Z,B)(-t) = \Phi(\wt B,\wt Z).
\]
\end{lemma}

\noindent We are now ready to prove Theorems~\ref{thm:dist_of_SH} and~\ref{thm:conv_to_SH}. 
 \begin{proof}[Proof of Theorem~\ref{thm:dist_of_SH}]
Let $f_1,\ldots,f_k$ be independent variance $4$ Brownian motions with drifts $4\alpha_1,\ldots,4\alpha_k$. For $1 \le i \le k$, let $Z^i = \wt f_i$, and note that $(Z^1,\ldots,Z^k) \deq (f_1,\ldots,f_k)$. Set $(g_1,\ldots,g_k) = \Phi^k(f_1,\ldots,f_k)$, and $(\eta^1,\ldots,\eta^k) = \D^{(k)}(Z^1,\ldots,Z^k)$ (Recall~\ref{definition of script D}). Since $(G_{\alpha_1},\ldots,G_{\alpha_k}) \deq (g_1,\ldots,g_k)$ and $(\eta^1,\ldots,\eta^k) \sim \mu^\alpha$ by Definition~\ref{definition of v lambda and mu lambda}, it suffices to show that, for $1 \le i \le k$, $\eta^i = \wt g_i$. For $i = 1$, this is immediate because $g_1 = f_1$ and $\eta^1 = Z^1$. The $i = 2$ case is Lemma~\ref{lem:modified_time_reversal}. Now, assume the statement holds for some $i < k$. By definition and Lemma~\ref{lem:iterated_map_Phi}, this means that for $t \in \R$, $-D^{(i)}(Z^i,\ldots,Z^1)(-t) = \Psi^i(f_1,\ldots,f_i)(t)$. Then, applying this assumption along with~\eqref{D iterated} and Lemmas~\ref{lem:iterated_map_Phi} and~\ref{lem:modified_time_reversal},
\[
g_{i + 1}(t) = \Phi(f_1,\Psi^{i}(f_{2},\ldots,f_{i + 1}))(t) = - D(D^{(i)}(Z^{i + 1},\ldots,Z^2),Z^1)(-t) =  -\eta^{i + 1}(-t). 
\]
For the second statement, Theorem~\ref{dist of Busemann functions and Bm} implies equality of the finite-dimensional distributions. Theorem~\ref{thm:summary of properties of Busemanns for all theta}\ref{general uniform convergence Busemanns} implies that the process $\{\wt h_0^{(1/\lambda^2)-}(4\abullet):\lambda \ge 0\}$ is right-continuous with left limits, in the sense of uniform convergence on compact sets. Thus, the process $\{\wt h_0^{(1/\lambda^2)-}(4\abullet)\}_{\lambda \ge 0}$ also lies in the 
Skorokhod space $D(\R,C(\R))$.
 \end{proof}

\begin{proof}[Proof of Theorem~\ref{thm:conv_to_SH}]
 Theorem~\ref{dist of Busemann functions and Bm} and the scaling relations of Lemma~\ref{weak continuity and consistency}\ref{scaling relations} imply that the vector
\[
n^{-1/3}(h_0^{1 - 2n^{-1/3}\alpha_1}(n^{2/3} \abullet) - n^{2/3}\abullet,\ldots,h_0^{1 - 2n^{-1/3}\alpha_k}(n^{2/3}\abullet)- n^{2/3}\abullet)
\]
has distribution $\mu^{\alpha^k}$, where, for $1 \le i \le k$,
\[
\alpha_i^k = \sqrt \f{n}{n^{1/3} - 2\alpha_i} - n^{1/3}.
\]
Noting that 
\begin{align*}
\sqrt \f{n}{n^{1/3} - 2\alpha_i} - n^{1/3} &= \f{2\alpha_i n^{2/3}}{(n^{1/3} - 2\alpha_i)\Big(\sqrt \f{n}{n^{1/3} - 2\alpha_i} + n^{1/3}\Big)} \\
&= \f{2\alpha_i}{(1 - 2\alpha_i n^{-1/3})\Big(\sqrt \f{n^{1/3}}{n^{1/3} - 2\alpha_i} + 1\Big)} \overset{n \to \infty}{\longrightarrow} \alpha_i,
\end{align*}
the continuity of the measures $\mu^\lambda$ from Theorem~\ref{weak continuity and consistency}\ref{weak continuity} completes the proof, via Theorem~\ref{thm:dist_of_SH}. We scale by a factor of $4$ to match Definition~\ref{def:SH}.
\end{proof}

\begin{appendix}
\section{Finite geodesics in BLPP} \label{sec:finite_geod}
\noindent Recall the uniqueness of geodesics for fixed initial and terminal point from Lemma~\ref{lemma:uniqueness of LPP time}. The following shows how to find random points in BLPP such that multiple geodesics exist.

\begin{lemma} \label{lemma:mult_geod}
The following hold.
\begin{enumerate}[label=\rm(\roman{*}), ref=\rm(\roman{*})]  \itemsep=3pt
    \item \label{2 geod}Fix an initial point $(m,s) \in \Z \times \R$. With probability one, there exist random points $(m + 1,t) \ge (m,s)$ such that there exist exactly two geodesics between $(m,s)$ and $(m + 1,t)$.
    \item \label{3 geod}With probability one, there exist random pairs of points $(m,s) \le (m + 1,t)$ such that there are exactly three geodesics between $(m,s)$ and $(m + 1,t)$.
\end{enumerate}
\end{lemma}
\begin{proof}
\textbf{Part~\ref{2 geod}}: For fixed $(m,s) \in \Z \times \R$, consider points of the form $(m + 1,t)$ for $t > s$. The last passage time is 
\[
B_{m +1}(t) - B_m(s) + \sup_{s \le u \le t}\{B_{m}(u) - B_{m + 1}(u)\}.
\]
Note that $B_{m} -B_{m + 1}$ is a variance $2$ Brownian motion. By Lemma~\ref{lemma:uniqueness of LPP time}, there is almost surely a unique maximizer of $B_{m}(u) - B_{m + 1}(u)$ over $u \in [s,s + 1]$, and that maximizer $u^\star \in (s,s + 1)$. Since Brownian motion is recurrent, there exists $v > s + 1$ such that $B_{m}(v) - B_{m + 1}(v) = B_{m}(u^\star) - B_{m + 1}(u^\star)$. Letting
\[
t = \inf\{v > s + 1: B_{m }(v) - B_{m +1 }(v) = B_{m}(u^\star) - B_{m + 1}(u^\star)\},
\]
there exist two geodesics between $(m,s)$ and $(m + 1,t)$: one that jumps to level $m + 1$ at $u^\star$ and another that jumps at the right endpoint $t$. 

\medskip \noindent \textbf{Part~\ref{3 geod}}: Similarly, start with fixed $(m,s)$ and define $t$ as in the previous case, but then set 
\[
s' = \sup\{v < s: B_{m}(v) - B_{m +1}(v) = B_{m}(u^\star) - B_{m + 1}(u^\star)  \}.
\]
Then, there are three geodesics between $(m,s')$ and $(m + 1,t)$: one that jumps at $s'$, another that jumps at $u^\star$, and another that jumps at $t$. 
\end{proof}

The following gives a crude bound on the maximum number of geodesics that grows as the vertical distance between the two points increases. By Lemma~\ref{lemma:mult_geod}\ref{3 geod}, the bound is sharp for $n = m + 1$, but we do not know if the bound is sharp for $n > m + 1$, or even whether there exist random points with an arbitrarily large number of geodesics between them.  For the present paper, we need only  the fact that, between any two points, there are only finitely many geodesics. 
\begin{lemma} \label{lemma:geodesic_bound}
There exists an event $\wt \Omega$ of full probability, on which the following hold.
\begin{enumerate} [label=\rm(\roman{*}), ref=\rm(\roman{*})]  \itemsep=3pt
    \item \label{itm:unique_geod_rational} Between any two points $(m,a)\le(n,b)$, both in  $\Z \times \Q$, there is a unique geodesic between the two points. That unique geodesic does not pass through $(k,a)$ for $k > m$ or $(r,b)$ for $r < n$.
    \item \label{itm:geod_bound} There exist no pairs $(m,s) \le (n,t) \in \Z \times \R$, with more than
\be \label{eqn:geodesics_bound}
1 + 2(n - m) + \f{(n - m - 1)(n - m)}{2}
\ee
geodesics between the two points.
\end{enumerate} 
\end{lemma}
\begin{proof}
\textbf{Part~\ref{itm:unique_geod_rational}:} Lemma~\ref{lemma:uniqueness of LPP time} guarantees that on an event of probability one, there exists a unique geodesic between any two points $(m,a) \le (n,b)$, both in $\Z \times \Q$. Let $\wt \Omega$ be the intersection of this event with the event on which, for each rational pair $q_1 < q_2$ and $k \in \Z$, the maximum of $B_k(s) - B_{k + 1}(s)$ over $s \in [q_1,q_2]$ is uniquely achieved at a point in the interior of the interval. By Lemma~\ref{lemma:BM unique max}, $\Pp(\wt \Omega) = 1$.

We show that on $\wt \Omega$,  for $(m,a) \le (n,b) \in \Z \times \Q$, the unique geodesic does not pass through $(k,a)$ or $(r,b)$ for any $k > m$ or $r < n$.  If, by contradiction, the converse fails, then the geodesic makes an upward step from $(m,a)$ to $(k,a)$ or from $(r,b)$ to $(n,b)$, or both. We show that the first case cannot hold on $\wt \Omega$, and the second case follows analogously. Let $k > m$ be the maximal index such that $(k,a)$ lies on the geodesic. Then, the geodesic passes through $(k - 1,a),(k,a),$ and $(k,q)$ for some rational $q > a$. See Figure~\ref{fig:vert_geod}. The portion of the geodesic between $(k - 1,a)$ and $(k,q)$ is also a geodesic, and the last passage time between the two points is
\[
\sup_{s \in [a,q]} \{B_{k - 1}(s,a) + B_k(a,q)\} = B_k(q) - B_{k - 1}(a) + \sup_{s \in [a,q]}\{B_{k - 1}(s) - B_k(a)\}.
\]
 Since the geodesic passes through $(k,a)$, the maximum is achieved at $s = q$. This contradicts the definition of $\wt \Omega$.

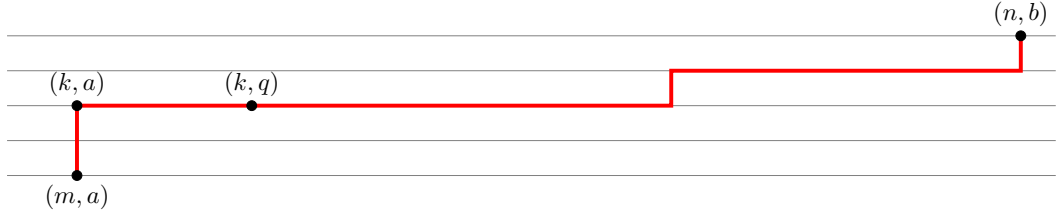
\begin{figure}[t]
\begin{adjustbox}{max totalsize={5.5in}{5in},center}
\begin{tikzpicture}
\draw[gray,thin] (0.5,0) -- (15.5,0);
\draw[gray,thin] (0.5,0.5) --(15.5,0.5);
\draw[gray, thin] (0.5,1)--(15.5,1);
\draw[gray,thin] (0.5,1.5)--(15.5,1.5);
\draw[gray,thin] (0.5,2)--(15.5,2);
\draw[red,ultra thick] (1.5,0)--(1.5,1)--(4.5,1)--(10,1)--(10,1.5)--(12,1.5)--(15,1.5)--(15,2);
\filldraw[black] (1.5,0) circle (2pt) node[anchor = north] {\small $(m,a)$};
\filldraw[black] (1.5,1) circle (2pt) node[anchor = south] {\small $(k,a)$};
\filldraw[black] (4,1) circle (2pt) node[anchor = south] {\small $(k,q)$};
\filldraw[black] (15,2) circle (2pt) node[anchor = south] {\small $(n,b)$};
\end{tikzpicture}
\end{adjustbox}
\caption{\small Geodesic that passes through $(k,a)$ on the way to $(n,b)$}
\label{fig:vert_geod}
\end{figure}

\medskip \noindent \textbf{Part~\ref{itm:geod_bound}:} Let $s = s_{m - 1} \le \cdots \le s_{n - 1} \le s_n = t$ denote the jump times of an arbitrary geodesic between $(m,s)$ and $(n,t)$. We prove the following:
\begin{enumerate}[label=\rm(\alph{*}), ref=\rm(\alph{*})]  \itemsep=3pt
    \item 
    \label{itm:hh} There is at most one geodesic satisfying $s_m > s$ and $s_{n - 1} < t$. 
    \item \label{itm:vh} There are at most $n - m$ geodesics satisfying $s_m > s$ and $s_{n - 1} = t$.
    \item \label{itm:hv} There are at most $n - m$ geodesics satisfying $s_m = s$ and $s_{n-1} < t$.
    \item \label{itm:vv} There are at most $\f{(n - m)(n - m - 1)}{2}$ geodesics satisfying $s_m = s$ and $s_{n- 1} = t$. 
\end{enumerate}

\medskip \noindent \textbf{Part~\ref{itm:hh}:} If two geodesics $\Gamma_1$ and $\Gamma_2$ both satisfy $s_m > s$ and $s_{n - 1} < t$, then $\Gamma_1$ and $\Gamma_2$ are also geodesics between $(m,a)$ and $(n,b)$ for some rational $a,b \in \Q$, so $\Gamma_1 = \Gamma_2$. See Figure~\ref{fig:hhgeodesic}.

\medskip \noindent \textbf{Part~\ref{itm:vh}:} For a geodesic $\Gamma$ satisfying $s_m > s$ and $s_{n - 1} = t$, let $r$ be the smallest index such that $s_r = t$. Geometrically, $r$ is the level at which the geodesic enters the right boundary, and the geodesic passes through $(r,t)$. See Figure~\ref{fig:hvgeodesic}. For each such $r \in \{m,\ldots,n - 1\}$, there is at most one geodesic satisfying $s_m > s$ and $s_{r - 1} < s_r = t$, by the previous case, giving at most $n - m$ geodesics of this type. 

\medskip \noindent \textbf{Part~\ref{itm:hv}:} The proof is analogous to Part~\ref{itm:vh}.

\medskip \noindent \textbf{Part~\ref{itm:vv}:} For a geodesic satisfying $s_m = s$ and $s_{n - 1} = t$, we let $k$ be the smallest index such that $s_{k} > s$ and $r \ge k$ be the smallest index such that $s_r = t$. Then, the geodesic passes through both $(k,s)$ and $(r,t)$. Geometrically, $k$ is the level at which the geodesic exits the left boundary, and $r$ is the level at which the geodesic enters the right boundary. See Figure~\ref{fig:vvgeodesic}. By Part~\ref{itm:hh}, for each pair $(k,r)$ with  $m < k \le r \le n - 1$, there is at most one geodesic that exits the left boundary at level $k$ and enters the right boundary at level $r$. There are $\f{(n - m)(n - m - 1)}{2}$ of these pairs $(k,r)$. 
\end{proof}

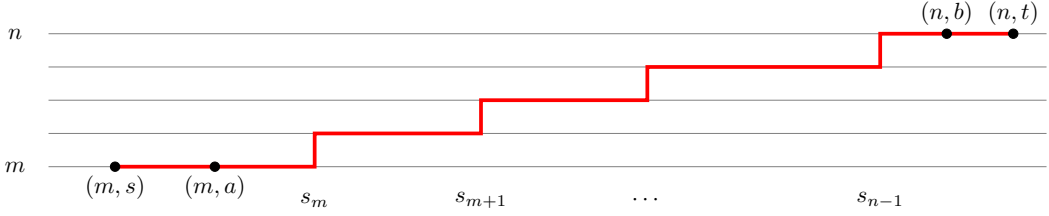
\begin{figure}[t]
\begin{adjustbox}{max totalsize={5.5in}{5in},center}
\begin{tikzpicture}
\draw[gray,thin] (0.5,0) -- (15.5,0);
\draw[gray,thin] (0.5,0.5) --(15.5,0.5);
\draw[gray, thin] (0.5,1)--(15.5,1);
\draw[gray,thin] (0.5,1.5)--(15.5,1.5);
\draw[gray,thin] (0.5,2)--(15.5,2);
\draw[red,ultra thick] (1.5,0)--(4.5,0)--(4.5,0.5)--(7,0.5)--(7,1)--(9.5,1)--(9.5,1.5)--(13,1.5)--(13,2)--(15,2);
\filldraw[black] (1.5,0) circle (2pt) node[anchor = north] {\small $(m,s)$};
\filldraw[black] (3,0) circle (2pt) node[anchor = north] {\small $(m,a)$};
\filldraw[black] (14,2) circle (2pt) node[anchor = south] {\small $(n,b)$};
\filldraw[black] (15,2) circle (2pt) node[anchor = south] {\small $(n,t)$};
\node at (4.5,-0.5) {\small $s_m$};
\node at (7,-0.5) {\small $s_{m + 1}$};
\node at (9.5,-0.5) {\small $\cdots$};
\node at (13,-0.5) {\small $s_{n - 1}$};
\node at (0,0) {\small $m$};
\node at (0,2) {\small $n$};
\end{tikzpicture}
\end{adjustbox}
\caption{\small Example of geodesics between $(m,s)$ and $(n,t)$ with $s_m > s$ and $s_{n - 1} < t$.}
\label{fig:hhgeodesic}
\end{figure}

\begin{figure}[t]
\begin{adjustbox}{max totalsize={5.5in}{5in},center}
\begin{tikzpicture}
\draw[gray,thin] (0.5,0) -- (15.5,0);
\draw[gray,thin] (0.5,0.5) --(15.5,0.5);
\draw[gray, thin] (0.5,1)--(15.5,1);
\draw[gray,thin] (0.5,1.5)--(15.5,1.5);
\draw[gray,thin] (0.5,2)--(15.5,2);
\draw[red,ultra thick] (1.5,0)--(4.5,0)--(4.5,0.5)--(10,0.5)--(10,1)--(15,1)--(15,2);
\filldraw[black] (1.5,0) circle (2pt) node[anchor = north] {$(m,s)$};
\filldraw[black] (15,1) circle (2pt) node[anchor = north] {$(r,t)$};
\filldraw[black] (15,2) circle (2pt) node[anchor = south] {$(n,t)$};
\node at (4.5,-0.5) {\small $s_m$};
\node at (10,-0.5) {\small $s_{r - 1}$};
\node at (0,0) {\small $m$};
\node at (0,2) {\small $n$};
\node at (16,1) {\small $r$};
\end{tikzpicture}
\end{adjustbox}
\caption{\small Example of a geodesic between $(m,s)$ and $(n,t)$ with $s_m > s$ and $s_{n - 1} = t$. The level $r$ is denoted in the right. }
\label{fig:hvgeodesic}
\end{figure}
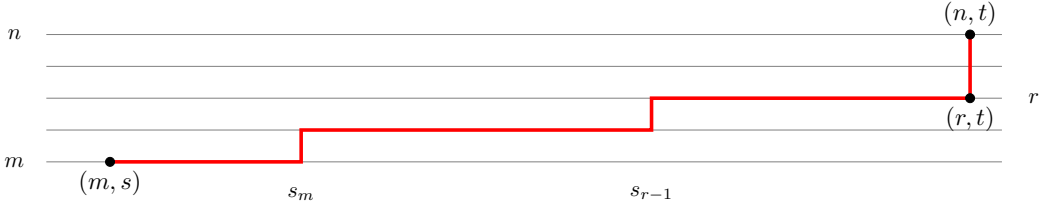

\begin{figure}[t]
\begin{adjustbox}{max totalsize={5.5in}{5in},center}
\begin{tikzpicture}
\draw[gray,thin] (0.5,0) -- (15.5,0);
\draw[gray,thin] (0.5,0.5) --(15.5,0.5);
\draw[gray, thin] (0.5,1)--(15.5,1);
\draw[gray,thin] (0.5,1.5)--(15.5,1.5);
\draw[gray,thin] (0.5,2)--(15.5,2);
\draw[red,ultra thick] (1.5,0)--(1.5,0.5)--(4.5,0.5)--(10,0.5)--(10,1)--(12,1)--(12,1.5)--(15,1.5)--(15,2);
\filldraw[black] (1.5,0) circle (2pt) node[anchor = north] {\small $(m,s)$};
\filldraw[black] (1.5,0.5) circle (2pt) node[anchor = south] {\small $(k,s)$};
\filldraw[black] (15,1.5) circle (2pt) node[anchor = north] {\small $(r,t)$};
\filldraw[black] (15,2) circle (2pt) node[anchor = south] {\small $(n,t)$};
\node at (10,-0.5) {\small $s_k$};
\node at (12,-0.5) {\small $s_{r - 1}$};
\node at (0,0) {\small $m$};
\node at (0,2) {\small $n$};
\node at (16,1.5) {\small $r$};
\node at (0,0.5) {\small $k$};
\end{tikzpicture}
\end{adjustbox}
\caption{\small Example of a geodesic between $(m,s)$ and $(n,t)$ with $s_m = s$ and $s_{n - 1} = t$. The level $k$ is denoted in the left and the level $r$ is denoted on the right. }
\label{fig:vvgeodesic}
\end{figure}
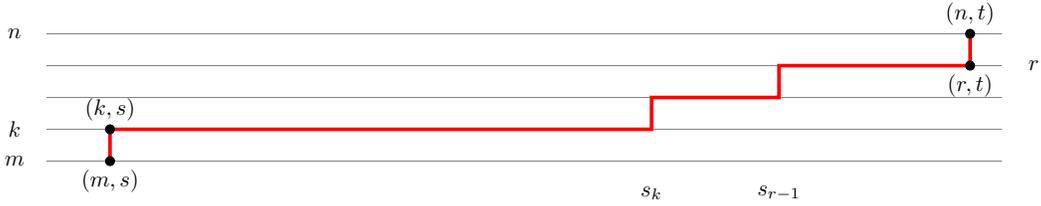

\section{Prior results on Busemann functions and semi-infinite geodesics} \label{sec:previous_paper}
In addition to the results stated in Section~\ref{section:Defs_results}, we use several other results about Busemann functions and semi-infinite geodesics that were proven in~\cite{Seppalainen-Sorensen-21a}. Recall the definition of the mappings $Q$ and $D$ from~\eqref{definition of Q} and~\eqref{definition of D}. Also, recall the discussion above Theorem~\ref{thm:summary of properties of Busemanns for all theta} regarding different full-probability events.

In the following theorem, recall the definitions of $\NU_0^{\theta \sig}$ and $\NU_1^{\theta \sig}$ from~\eqref{NU_0 theta def}--\eqref{NU_1 theta def}. Also, recall Remark~\ref{rmk:NU_Sets}, which states that when working on the event $\Omega^{(\theta)}$, there is no $\pm$ distinction, and we write $\NU_i^{\theta} = \NU_i^{\theta \pm}$ for $i = 0,1$.

    \begin{theorem}[\cite{Seppalainen-Sorensen-21a}, Theorem 3.1(iii) and 4.7] \label{thm:NU_paper1}
    Fix $\theta > 0$. Then, the following hold.
    \begin{enumerate} [label=\rm(\roman{*}), ref=\rm(\roman{*})]  \itemsep=3pt
    \item \label{itm:uniqueness of geodesic for fixed point and direction}  For each fixed $\mbf x \in \Z \times \R$, on the full probability event $\Omega_{\mbf x}^{(\theta)}$, we have $\mbf x \notin \NU_0^\theta$.
    \item \label{itm:count_and_decomp} On the event $\Omega^{(\theta)}$, the sets $\NU_0^\theta$ and $\NU_1^\theta$ are countably infinite and can be written as 
    \begin{align*}
    \NU_0^\theta &= \{(m,t) \in \Z \times \R: t= \tau_{(m,t),r}^{\theta,L} < \tau_{(m,t),r}^{\theta,R} \text{ for some }r \ge m\}, \text{ and} \\
    \NU_1^\theta &= \{(m,t) \in \Z \times \R: t = \tau_{(m,t),m}^{\theta,L} < \tau_{(m,t),m}^{\theta,R}\}.
    \end{align*}
    In other words, Busemann geodesics emanating from $(m,t)$ and in a fixed direction $\theta$ can separate only along the upward vertical ray from  $(m,t)$. 
    \item \label{non-discrete or dense} On the event $\Omega^{(\theta)}$, the set $\NU_1^\theta$ is neither discrete nor dense in $\Z \times \R$. More specifically, for each point $(m,t) \in \NU_1^\theta$ and every $\ve > 0$, there exists $s \in (t - \ve,t)$ such that $(m,s) \in \NU_1^\theta$. For each $(m,t) \in \NU_1^\theta$, there exists $\delta > 0$ such that $(m,s) \notin \NU_0^\theta$ for all $s \in (t,t+\delta)$. 
    \end{enumerate}
    \end{theorem}

\noindent The following lemmas provide useful characterizations of the Busemann geodesics. The first one utilizes point-to-line last-passage problems. 

    \begin{lemma}[\cite{Seppalainen-Sorensen-21a}, Lemma 7.3] \label{lemma:ptl_sig}
Let $\omega \in \Omega_2$, $(m,t) \in \Z \times \R$ and $\theta > 0$, $\sig \in \{-,+\}$. Then, the following hold.
\begin{enumerate} [label=\rm(\roman{*}), ref=\rm(\roman{*})]  \itemsep=3pt
    \item \label{itm:geo_maxes} Let $\{\tau_r\}_{r = m - 1}^\infty$ be any sequence in $\mbf T_{(m,t)}^{\theta\sig}$. Then, for each $n \ge m$, the jump times $t = \tau_{m - 1} \le \tau_m \le \cdots \le \tau_n$ are a maximizing sequence for 
\begin{equation} \label{ptl_BLPP}
 \sup\Biggl\{\sum_{r = m}^n B_r(s_{r - 1},s_r) - \h_{n +1}^{\theta\sig}(s_n): t = s_{m - 1} \le s_m \le \cdots \le s_n  < \infty \Biggr\}.
\end{equation}
\item \label{itm:maxes_geo}
 Conversely, for each $n \ge m$, whenever $t = t_{m - 1} \le t_m \le \cdots\le t_n$ is a maximizing sequence for~\eqref{ptl_BLPP}, there exists $\{\tau_r\}_{r = m - 1}^\infty \in \mbf T_{(m,t)}^{\theta \sig}$ such that $t_r = \tau_r$ for $m \le r \le n$.  
 \item \label{itm:LR_geo_max}
 For each $n \ge m$, the sequences $t = \tau_{(m,t),m - 1}^{\theta \sig,L} \le \cdots \le \tau_{(m,t),n}^{\theta \sig,L}$ and $t = \tau_{(m,t),m - 1}^{\theta \sig,R} \le \cdots \le \tau_{(m,t),n}^{\theta \sig,R}$ are, respectively, the leftmost and rightmost maximizing sequences for~\eqref{ptl_BLPP}.
 \end{enumerate}
\end{lemma}

\noindent
The next lemma indicates how the L/R distinction of geodesics can be characterized by the Busemann functions. 

\begin{lemma}[\cite{Seppalainen-Sorensen-21a}, Lemma 7.4] \label{lemma:equality of busemann to weights of BLPP}
Let $\omega \in \Omega_2$, $(m,t) \in \Z \times \R$, $ \theta > 0$, $\sigg \in \{-,+\}$, and $\{\tau_r\}_{r \ge m - 1} \in \mbf T_{(m,t)}^{\theta \sig}$. Then, for all $r \ge m$,
\[
\vv_{r + 1}^{\theta \sig}(\tau_r) = 0,\qquad\text{and}\qquad \h_r^{\theta \sig}(u,v) = B_r(u,v) \text{ for all }u,v \in [\tau_{r - 1},\tau_r].
\]
Furthermore, the following identities hold for $r \ge m$.
     \begin{align}
     &\tau_{(m,t),r}^{\theta \sig,L} = \inf\bigl\{u \ge \tau^{\theta \sig,L}_{(m,t),r - 1}: \vv_{r + 1}^{\theta \sig}(u) = 0\bigr\},\qquad \text{and} \label{eqn:inf_sig} \\
     &\tau^{\theta \sig,R}_{(m,t),r} = \sup\bigl\{u \ge \tau^{\theta \sig,R}_{(m,t),r - 1}: \h_r^{\theta \sig}(\tau^{\theta \sig,R}_{(m,t),r - 1},u) = B_r(\tau^{\theta \sig,R}_{(m,t),r - 1},u)\bigr\} \label{eqn:sup_sig}
     \end{align}
     More specifically, if $u \ge \tau^{\theta \sig,R}_{(m,t),r - 1}$, then $\h_r^{\theta \sig}(\tau^{\theta \sig,R}_{(m,t),r - 1},u) = B_r(\tau^{\theta \sig,R}_{(m,t),r - 1},u)$ if and only if $u \le \tau^{\theta \sig,R}_{(m,t),r}$.
\end{lemma}

\noindent
As the last of the results from \cite{Seppalainen-Sorensen-21a} we cite a Brownian calculation.  

\begin{theorem}[arXiv version of \cite{Seppalainen-Sorensen-21a}, Theorem B.2] \label{thm:dist of busemann increment}
Let $B$ be a standard Brownian motion, and for $t > 0$, let 
\begin{align*}
D(t) = \sup_{0 \le s < \infty}\bigl\{\sqrt 2 B(s) - \lambda s\bigr\} - \sup_{t \le s < \infty}\bigl\{\sqrt 2 B(s) - \lambda s\bigr\}.
\end{align*}
Then, for all $z \ge 0$,
\[
\Pp(D(t) \le z) = \Phi\Bigl(\frac{z - \lambda t}{\sqrt{2 t}}\Bigr) + e^{\lambda z}\Biggl( (1 + \lambda z + \lambda^2 t)\Phi\Bigl(-\frac{z + \lambda t}{\sqrt{2 t}}\Bigr) - \lambda \sqrt{\frac{t}{ \pi}}\,e^{-\frac{(z + \lambda t)^2}{4 t}}    \Biggr).
\]
\end{theorem}

\section{Auxiliary technical inputs} \label{sec:aux_tec}
\begin{lemma} \label{lemma:convergence of maximizers from converging sets}
Let $S_n$ for $n \ge 0$ be subsets of some set $\wt S \subseteq \R^n$, on which the function $f:\wt S \rightarrow \R$ is continuous. Assume that each point $x \in S_0$ is the limit of a sequence $\{x_n\}$, where $x_n \in S_n$ for each $n$. Assume that $\{c_n\}$ is a sequence of maximizers of $f$ over $S_n$. Assume further that $c_n$ converges to some $c \in S_0$. Then, $c$ is a maximizer of $f$ over $S_0$. 
\end{lemma}
\begin{proof}
For each $x_0 \in S_0$, write $x_0 = \lim_{n\rightarrow \infty} x_n$, where $x_n \in S_n$ for each $n$. Then, $f(c_n) \ge f(x_n)$ for all $n \ge 1$, and the result follows by taking limits.
\end{proof}

\begin{theorem}[\cite{brownian_queues}, Theorem 4, this formulation found in~\cite{Seppalainen-Sorensen-21a}, Theorem C.2 and Lemma D.2] \label{O Connell Yor BM independence theorem queues} 
Let $Z$ be a two-sided Brownian motion with drift $\lambda > 0$, independent of the two-sided Brownian motion $B$ (with no drift). Then, $D(Z,B)$ is a two-sided Brownian motion with drift $\lambda$, independent of the two-sided Brownian motion $R(Z,B)$. Furthermore, for all $s \in \R$, $\{(D(Z,B)(s,t),R(Z,B)(s,t)): s \le t < \infty \}$ is independent of $\{Q(Z,B)(u): u \le t\}$. 
\end{theorem}

\begin{lemma} \label{lemma:D preserves increment stationarity}
Let $Z,B \in \CRpin$ be such that 
\[
\limsup_{t \rightarrow \infty} Z(t) - B(t) = -\infty.
\]
For $s \in \R$, denote by $Z^s$ the shifted process
\[
(Z(s,t + s))_{t \in \R}.
\]
Then, we have
\begin{equation} \label{D respects translations}
D(Z^s,B^s) = D(Z,B)^s.
\end{equation}
\end{lemma}
\begin{proof}
This is a straightforward verification, using the definition of $D$~\eqref{definition of D}.
\end{proof}

\begin{lemma}[\cite{morters_peres_2010}, Theorem 2.11] \label{lemma:BM unique max}
Let $B$ be a standard Brownian motion on $[0,1]$. With probability one, $B$ has a unique maximizer, and the maximizer lies in $(0,1)$.
\end{lemma}

We use Lemma~\ref{lemma:BM unique max} to derive the following corollary.
\begin{corollary} \label{cor:max_in_interior}
Let $B$ be a Brownian motion (could be one or two-sided). Then, there exists a full event of probability one, on which,
for all $s < t$, at most one maximizer of $B$ over $[s,t]$ lies in $(s,t)$.
\end{corollary}
\begin{proof}
Take the full probability event on which $B$ has a unique maximizer over $[a,b]$ for all rational endpoints $a < b$.
\end{proof}

\begin{lemma}[\cite{taylor_1955}, page 270. See also~\cite{morters_peres_2010}, Theorem 4.24] \label{lemma:Hausdorff dimension facts for standard Brownian motion}
Let $B:[0,\infty)\rightarrow \R$ be a standard Brownian motion. Then, with probability one, the following sets have Hausdorff dimension $\f{1}{2}$. Furthermore, for each fixed $t \in [0,\infty)$, $t$ lies in either of the following sets with probability zero. 
\begin{enumerate} [label=\rm(\roman{*}), ref=\rm(\roman{*})]  \itemsep=3pt
\item $
\{t \in [0,\infty): B(t) = 0\}$ \label{itm: zero set BM}
\item 
$\{t \in [0,\infty): B(t) = \underset{0 \le s \le t}{\sup} B(s)\}$ \label{itm: right maximizer set BM}
\end{enumerate}
\end{lemma}

\begin{corollary} \label{cor:hausdorff dimension for left maxes standard two-sided BM}
Let $B:\R \rightarrow \R$ be a standard, two-sided Brownian motion. Then, the following sets are equal. These sets almost surely have Hausdorff dimension $\f{1}{2}$, and for any fixed $s \in \R$, the point $s$ lies in either set with probability zero.
\begin{enumerate} [label=\rm(\roman{*}), ref=\rm(\roman{*})]  \itemsep=3pt
    \item \label{itm:left_max} $\{s \in \R: B(s) = \underset{s \le u \le t}{\sup}B(u) \;\text{for some }t > s \}$ 
    \item \label{itm:right_max} $\{s \in \R: \text{ for some } t > s, B(s) > B(u)\; \text{ for all }u \in (s,t]\}$
\end{enumerate}
\end{corollary}
\begin{proof}
First, we show that the sets are equal. The inclusion $\ref{itm:right_max} \subseteq \ref{itm:left_max}$ is immediate.  Now, assume that $s \in\ref{itm:left_max}$, and let $t$ be such that $B(s) = \sup_{s \le u \le t}B(u)$. By Corollary~\ref{cor:max_in_interior}, $B$ has at most one maximizer, $\hat s$, in the interior of $[s,t]$. If no such maximizer exists, set $\hat s = t$. Choose $\hat t \in (s,\hat s).$ Then, $B(s) > B(u)$ for all $u \in (s,\hat t]$, and $s \in\ref{itm:right_max}$. Hence,~\ref{itm:left_max}=\ref{itm:right_max}. Next, for a two-sided Brownian motion $B$ and any point $t \in \R$, the process
\[
\{\wt B_t(u) := B(t - u) - B(t): u \ge 0\}
\]
is a standard Brownian motion.  Observe that 
\begin{align*}
\ref{itm:left_max}
&= \bigcup_{q \in \Q} \{s < q: B(s) = \underset{s \le u \le q}{\sup}B(u)\} \\
&= \bigcup_{q \in \Q} \Big\{s < q: B(q - (q - s)) - B(q) = \underset{0 \le u \le q - s}{\sup}[B(q - u) - B(q)]\Big\}.
\end{align*}
Since Hausdorff dimension is preserved under countable unions, translations, and reflections, Lemma~\ref{lemma:Hausdorff dimension facts for standard Brownian motion} completes the proof. 
\end{proof}

\begin{theorem}[\cite{Seppalainen-Sorensen-21a}, Lemma B.4 (Lemma B.5 in the arXiv version)] \label{thm:countable non unique maximizers}
Let $X$ be a two-sided Brownian motion with strictly negative drift. Let 
\[
M = \{t \in \R: X(t) = \sup_{t \le s < \infty} X(s) \}.
\]
Furthermore, let  
\[
M^U = \big\{t \in M: X(t) > X(s) \text{ for all }s > t   \big\}
\]
 be the set of points $t\in M$ that are unique maximizers of $X(s)$ over $s \in [t,\infty)$.  Define
$
    M^N = M \setminus M^U
$
to be the set of $t\in M$ that are non-unique maximizers of $X(s)$ over $s \in [t,\infty)$.
Then, there exists an event of probability one, on which the following hold.
\begin{enumerate} [label=\rm(\roman{*}), ref=\rm(\roman{*})]  \itemsep=3pt 
    \item $M$ is a closed set. \label{Mclosed}
    \item \label{itm:2_3_max} There exists no points $t \in \R$ such that $X(s)$ has three maximizers over $s \in [t,\infty).$ If there exist two maximizers over $s \in [t,\infty)$, one of them is $s = t$.
    \item  \label{itm:not_mont} The function $s \mapsto X(s)$ is not monotone on any nonempty interval.
    \item $M^N$ is a countably infinite set. \label{countable}
    \item For all $\hat t \in M^U$ and $\ve > 0$, there exists $t \in M^N$ satisfying $\hat t < t < \hat t + \ve$. For all $t \in M^N$ and $\ve > 0$, there exists $\hat t \in M^U$ satisfying $t - \ve < \hat t < t$. \label{non_discrete}
    \item For all $t \in M^N$ and $\ve > 0$, there exists $t^\star \in M^N$ with $t - \ve < t^\star < t$. For each $t \in M^N$, there exists $\delta > 0$ such that $M \cap (t,t+ \delta) = \varnothing$.\label{non_discrete MN}
\end{enumerate}
\end{theorem}
\begin{remark}
Parts~\ref{itm:2_3_max} and~\ref{itm:not_mont} are not stated as a part of the theorem in~\cite{Seppalainen-Sorensen-21a}. However, the proof of Part~\ref{itm:2_3_max} is fairly simple and is contained in~\cite{Seppalainen-Sorensen-21a}: If, for some, $t \in \R$ there are two maximizers of $X(s)$ over $s \in [t,\infty)$  which are strictly greater than $t$, then for some rational $q > t$, there are two maximizers of $X(s)$ over $s \in [q,\infty)$. The proof is complete by showing that there exists a full probability event on which, for every $q \in \Q$, there is a unique maximizer of $X(s)$ over $s \in [q,\infty)$. Part~\ref{itm:not_mont} is an immediate fact about Brownian motion with drift. 
\end{remark}

\end{appendix}

\bibliographystyle{alpha}
\bibliography{references_file}

\end{document}